\documentclass[11pt]{amsart}

\usepackage{epsf,graphics,graphicx,amssymb,overpic,amscd,diagrams,psfrag,times,stmaryrd, rotating,pkmacros,xypic,mathrsfs}
\usepackage[all]{xy}
\usepackage[mathscr]{euscript}

\usepackage{xr}
\newtheorem{thm}{Theorem}[section]
\newtheorem{prop}[thm]{Proposition}
\newtheorem{lemma}[thm]{Lemma}
\newtheorem{cor}[thm]{Corollary}

\newtheorem{obs}[thm]{Observation}
\newtheorem{claim}[thm]{Claim}

\numberwithin{equation}{subsection}
\numberwithin{thm}{subsection}

\theoremstyle{definition}
\newtheorem{defn}[thm]{Definition}
\newtheorem{convention}[thm]{Convention}

\theoremstyle{remark}

\newtheorem{rmk}[thm]{Remark}

\DeclareMathAlphabet{\mathpzc}{OT1}{pzc}{m}{it}
\newcommand{\hh}{\mathpzc{h}}

\newcommand{\D}{\mathbb{D}}
\renewcommand{\H}{\mathbb{H}}
\newcommand{\C}{\mathbb{C}}

\newcommand{\R}{\mathbb{R}}
\newcommand{\Z}{\mathbb{Z}}

\newcommand{\N}{\mathbb{N}}
\renewcommand{\P}{\mathbb{P}}
\newcommand{\bdry}{\partial}
\newcommand{\s}{\vskip.1in}
\newcommand{\n}{\noindent}

\newcommand{\F}{\mathbb{F}}

\newcommand{\nom}{\nomenclature}
\newcommand{\bs}{\boldsymbol}

\newcommand{\be}{\begin{enumerate}}
\newcommand{\ee}{\end{enumerate}}
\newcommand{\op}{\operatorname}

\usepackage[refpage]{nomencl}
\makenomenclature

\begin{document}

\title[HF=ECH via open book decompositions I]
{The equivalence of Heegaard Floer homology and embedded contact homology via open book decompositions I}

\author{Vincent Colin}
\address{Universit\'e de Nantes, 44322 Nantes, France}
\email{Vincent.Colin@univ-nantes.fr}

\author{Paolo Ghiggini}
\address{Universit\'e de Nantes, 44322 Nantes, France}
\email{paolo.ghiggini@univ-nantes.fr}
\urladdr{http://www.math.sciences.univ-nantes.fr/\char126 Ghiggini}

\author{Ko Honda}
\address{University of Southern California, Los Angeles, CA 90089}
\email{khonda@usc.edu} \urladdr{http://www-bcf.usc.edu/\char126 khonda}

\date{This version: June 16, 2015.}



\thanks{VC supported by the Institut Universitaire de France, ANR Symplexe, ANR Floer Power, and ERC Geodycon. PG supported by ANR Floer Power and ANR TCGD. KH supported by NSF Grants DMS-0805352 and DMS-1105432.}

\begin{abstract}
Given an open book decomposition $(S,\hh)$ adapted to a closed, oriented $3$-manifold $M$, we define a chain map $\Phi$ from a certain Heegaard Floer chain complex associated to $(S,\hh)$ to a certain embedded contact homology chain complex associated to $(S,\hh)$, as defined in \cite{CGH2}, and prove that it induces an isomorphism on the level of homology. This implies the isomorphism between the hat version of Heegaard Floer homology of $-M$ and the hat version of embedded contact homology of $M$.
\end{abstract}

\maketitle

\setcounter{tocdepth}{2}
\tableofcontents

\newpage
\section{Introduction and main results}
\label{section: introduction}

There is a plethora of Floer-type homology theories that can be associated to a three-manifold. This paper and its sequels are concerned with three specific Floer-type homology theories: {\em monopole Floer homology}, {\em embedded contact homology}, and {\em Heegaard Floer homology}. Monopole Floer homology, constructed by Kronheimer and Mrowka~\cite{KM}, counts solutions of the Seiberg-Witten equations. Since its definition requires an auxiliary Riemannian metric, it has strong connections with geometry (e.g., positive curvature). Embedded contact homology, abbreviated ECH, is due to Hutchings~\cite{Hu1,Hu2} and Hutchings-Taubes~\cite{HT1,HT2} and is dynamical in nature. It counts periodic orbits of a Reeb vector field associated with a contact form.  Finally, Heegaard Floer homology, due to Ozsv\'ath and Szab\'o~\cite{OSz1,OSz2}, is defined in the most topological way: from a Heegaard diagram. Of the three homologies, it is the easiest to compute and admits a combinatorial description~\cite{SW}. These theories have had spectacular applications over the last decade, including a proof of the Gordon conjecture by Kronheimer, Mrowka, Ozsv\'ath and Szab\'o~\cite{KMOS} and progress on the exceptional surgery problem due to Ghiggini~\cite{Gh} and Ni~\cite{Ni}.

In 2006, Taubes obtained a breakthrough result which extended his celebrated correspondence between solutions of the Seiberg-Witten equations and $J$-holo\-morphic curves to the relative case. This enabled him to prove the Weinstein conjecture in dimension three~\cite{T1} and to establish the equivalence between monopole Floer cohomology and ECH shortly thereafter~\cite{T2}. A byproduct of this equivalence was the proof of Arnold's chord conjecture in dimension three~\cite{HT3}.

\s
The goal of our series of papers \cite{CGH2}, \cite{CGH-I}, \cite{CGH-II}, \cite{CGH-III} is to prove the equivalence of Heegaard Floer homology and ECH. The first paper \cite{CGH2} can be viewed as a stand-alone paper.  In this paper \cite{CGH-I} and the sequel~\cite{CGH-II}, we establish an isomorphism between the hat versions of the Heegaard Floer homology and ECH groups associated to a closed, oriented $3$-manifold $M$. This isomorphism is compatible with the splitting of Heegaard Floer homology according to Spin$^c$-structures and of ECH according to first homology classes. {\em For simplicity we will use $\F=\Z/2\Z$ coefficients (or coefficients in a module $\Lambda$ over $\F[H_2(M;\Z)]$) for both Heegaard Floer homology and ECH.}  However, we expect the equivalence to hold over the integers.  The isomorphism between the plus version of Heegaard Floer homology and the usual version of ECH will be given in ~\cite{CGH-III}.

The results of this paper were announced in \cite{CGH1}. An isomorphism between monopole Floer homology and Heegaard Floer homology was proved, more or less simultaneously, by Kutluhan, Lee and Taubes~\cite{KLT1}--\cite{KLT5}.

\subsection{Some background}

Before stating the main result of this paper and describing the ideas involved in its proofs, we will give a brief summary of the definitions of Heegaard Floer homology
and ECH.

We first describe Heegaard Floer homology in its ``cylindrical reformulation'' due to Lipshitz~\cite{Li}. The starting point of its construction is a {\em pointed Heegaard diagram} $(\Sigma, \boldsymbol{\alpha}, \boldsymbol{\beta}, z)$ which describes a three-manifold $M$. Here $\Sigma$ is a genus $g>0$ Heegaard surface associated to some self-indexing Morse function with a unique maximum and a unique minimum, $\boldsymbol{\alpha}$ is the collection of the attaching circles for the index one critical points, $\boldsymbol{\beta}$  is the collection of the attaching circles for the index two critical points, and $z$ is a basepoint in the complement of $\boldsymbol{\alpha}$ and $\boldsymbol{\beta}$.
The Heegaard Floer complex $\widehat{CF}(\Sigma, \boldsymbol{\alpha}, \boldsymbol{\beta}, z)$ is generated by $g$-tuples of intersection points between the $\boldsymbol{\alpha}$-curves and the $\boldsymbol{\beta}$-curves and the differential counts certain $J$-holomorphic curves in $\R \times [0,1] \times \Sigma$ with boundary on $\R \times \{ 0 \} \times \boldsymbol{\beta}$ and $\R \times \{ 1 \} \times \boldsymbol{\alpha}$; see Section~\ref{section: variation of HF adapted to OB} or \cite{Li} for a more detailed exposition. While Heegaard Floer homology {\em a priori} depends on the choice of a pointed Heegaard diagram $(\Sigma, \boldsymbol{\alpha}, \boldsymbol{\beta}, z)$ and an almost complex structure on $\R \times [0,1] \times \Sigma$, it was shown to be independent of those choices, i.e., is a topological invariant of $M$.

The starting point for embedded contact homology is a contact form $\alpha$ on $M$. The contact form determines the {\em Reeb vector field} $R$ by
$$\iota_R d \alpha =0 \quad \text{and} \quad \alpha(R)=1.$$
The complex $ECC(M, \alpha)$ is generated by finite sets of simple Reeb orbits with finite multiplicities and its differential counts certain $J$-holomorphic
curves in the symplectization $(\R \times M, d(e^s \alpha))$; see \cite{Hu1, Hu2, Hu3}
for more details. The hat version of ECH is defined as the homology
of the mapping cone of a chain map $U: ECC(M, \alpha) \to ECC(M, \alpha)$.
ECH {\em a priori} depends on the choice of a contact form $\alpha$ on $M$ and an adapted almost complex structure $J$ on the symplectization $\R\times M$. There is
currently no direct proof of the fact that the ECH groups are
topological invariants of $M$ (or even invariants of the contact
structure $\ker\alpha$, for that matter); the only known proof is
due to Taubes~\cite{T1,T2}, and is a consequence of the equivalence
between Seiberg-Witten Floer cohomology and ECH.
A direct proof of the invariance would provide, combined with the present work and computations in Heegaard Floer homology, an alternative proof of the Weinstein conjecture.

\s
A natural setting for relating Heegaard Floer homology and ECH is that of
{\em open book decompositions} (see Definition~\ref{defn: open book decomposition}) because an open book decomposition determines both a Heegaard splitting and a contact structure. The Heegaard splitting is obtained by taking as Heegaard surface
the union of two opposite pages. The contact structure is provided by the Thurston-Winkelnkemper construction~\cite{TW}. In the foundational work
~\cite{Gi2}, Giroux proved the equivalence of contact structures up
to isomorphism and (abstract) open book decompositions modulo
positive stabilization. While we are not using its full strength, Giroux's
correspondence was an important source of inspiration for this work.

We use open book decompositions and their relationship with both
Heegaard splittings and contact structures to build symplectic cobordisms
relating the geometric setting of Heegaard Floer homology to that of
embedded contact homology. Then the isomorphisms are defined by counting
certain $J$-holomorphic maps in those cobordisms. Their definition is similar in spirit
to the definition of the open-closed and closed-open maps in \cite{Abo}. The similarity
is more direct for the open-closed map, while our closed-open map require some extra
twist, as we will see.

\subsection{Main result}

Let $M$ be a closed, connected, oriented three-manifold. We fix an open book decomposition of $M$ with connected binding, page $S$, and monodromy $\hh$, and denote by $\xi_{(S, \hh)}$ the contact structure supported by it. The contact structure $\xi_{(S, \hh)}$ determines a Spin$^c$-structure which we denote by $\mathfrak{s}_{(S, \hh)}$.

The main result of \cite{CGH-I} and \cite{CGH-II} is the following:

\begin{thm} \label{thm: main}
There is an isomorphism
$$\Phi_0: \widehat{HF}(-M,\mathfrak{s}_{(S,\hh)}+PD(A))\stackrel\sim\longrightarrow \widehat{ECH}(M,\xi_{(S,\hh)},A),$$
where $A\in H_1(M;\Z)$, defined via the open book decomposition $(S,\hh)$ of $M$. Moreover, $\Phi_0$ sends the Heegaard Floer contact invariant for $\xi_{(S,\hh)}$ to the ECH contact invariant for $\xi_{(S,\hh)}$.
\end{thm}

The definition of  $\Phi_0$ depends on many choices, and first of all on the choice of the open book decomposition. In this series of papers we will not address the
question of its naturality, but we conjecture that it only depends on the contact structure $\xi_{(S, \hh)}$. Even the dependence on the contact structure should be
very mild, and due only to the fact that ECH groups defined from different contact structures are isomorphic after a shift in their decompositions according to homology classes of orbit sets.

There is a similar map for the so-called twisted coefficients. Let $\Lambda$ be any module over the group ring $\F[H_2(M;\Z)]$.  We denote the versions of Heegaard Floer homology and ECH with twisted coefficients in $\Lambda$ by $\underline{\widehat{HF}}(-M,\mathfrak{s}; \Lambda)$ and $\underline{\widehat{ECH}}(M,\xi,A; \Lambda)$. Twisted coefficients are defined in \cite[Section~8]{OSz2} for Heegaard Floer homology and in \cite[Section~11]{HS2} for ECH.

\begin{thm}\label{thm: main twisted}
There is an $\F[H_2(M;\Z)]$-module isomorphism
$$\underline{\Phi}_0: \underline{\widehat{HF}}(-M,\mathfrak{s}_{(S,\hh)}+PD(A); \Lambda)\stackrel\sim\longrightarrow \underline{\widehat{ECH}}(M,\xi_{(S,\hh)},A; \Lambda),$$
where $A\in H_1(M;\Z)$, defined via an open book decomposition $(S,\hh)$ of $M$. Moreover, $\underline{\Phi}_0$ sends the Heegaard Floer contact invariant for $\xi_{(S,\hh)}$ to the ECH contact invariant for $\xi_{(S,\hh)}$.
\end{thm}

On the other hand, Taubes~\cite{T2} has proven that Seiberg-Witten Floer cohomology and ECH are isomorphic. Let $\widetilde{HM}(M)$ be the homology of the mapping cone of $U_\dagger: \check{C}(M) \to \check{C}(M)$, where $\check C(M)$ is the chain complex for $\Hto(M)$. Combining Taubes' theorem with Theorem~\ref{thm: main}, we obtain the following ``corollary'':

\begin{cor}
$\widehat{HF}(M,\mathfrak{s})\simeq \widetilde{HM}(M,\mathfrak{s})$ for any ${\frak s}\in \mbox{Spin}^c(M)$.
\end{cor}

\subsection{Outline of proof}

Fix an open book decomposition $(S, \hh)$ for $M$. The first step in the construction of the map $\Phi_0$ is to adapt the definitions of $\widehat{HF}(-M)$ and $\widehat{ECH}(M)$ to the open book decomposition $(S,\hh)$. For Heegaard Floer homology this is achieved in Section~\ref{section: variation of HF adapted to OB}, essentially by pushing all interesting intersection points between the $\boldsymbol{\alpha}$- and $\boldsymbol{\beta}$-curves to one side of the Heegaard surface obtained from $(S, \hh)$.

On the ECH side this was achieved in the first paper~\cite{CGH2} of our series, where we introduced the group $\widehat{ECH}(N,\bdry N)$ for the mapping torus $N$ of
$(S, \hh)$. This group was defined as a direct limit
$$\widehat{ECH}(N,\bdry N)=\lim_{j\to\infty}ECH_j(N),$$
where $j$ is the number of intersections of an orbit set in $N$ with a page  of an open book
and the direct limit is taken with respect to maps
$$(\mathfrak{I}_j)_*:ECH_j(N)\to ECH_{j+1}(N),$$
defined by increasing the multiplicity of an elliptic orbit on $\partial N$ which
can be regarded intuitively as a receptacle for the $J$-holomorphic curves in
$\R \times M$ which intersect the cylinder over the binding. The following was
proved in \cite{CGH2}:

\begin{thm}\label{thm: from first paper of series}
$\widehat{ECH}(M)\simeq \widehat{ECH}(N,\bdry N).$
\end{thm}

The map $\Phi_0$ is, roughly speaking, the composition of a map
$$\Phi_*:\widehat{HF}(-M) \to ECH_{2g}(N),$$
induced by a symplectic cobordism $W_+$, followed by the natural map
$$ECH_{2g}(N)\to \lim_{j\to\infty} ECH_j(N),$$
induced by the maps $(\mathfrak{I}_j)_*$. Here $g$ is the genus of $S$.
(Strictly speaking, we need to use a ``perturbed'' version of $ECH_j(N)$, as explained in Section~\ref{elimination of elliptic orbits}, and replace the ECH groups by certain periodic Floer homology groups, as explained in Section~\ref{section: periodic Floer homology}.)

The cobordism $W_+$ is a symplectic fibration with fiber $S$ and monodromy
$\hh$ over a Riemann surface with a strip-like end and a cylindrical end, which is biholomorphic to a disk with a puncture at the center and a puncture on the boundary. On the boundary of $W_+$ there is a (disconnected) Lagrangian submanifold $\Lambda$ which, over the strip-like end, roughly speaking coincides with the Lagrangians used in the definition of Heegaard Floer homology (restricted to the nontrivial half of the Heegaard splitting). Then $\Phi_*$ is defined by counting degree $2g$ $J$-holomorphic multisections of $W_+$ with boundary on $\Lambda$ which converge to generators of the Heegaard Floer complex over the strip-like end and to generators of the ECH complex over the cylindrical end.

We also define a map $\Psi_*: ECH_{2g}(N) \to \widehat{HF}(-M)$ by counting degree
$2g$ $J$-holomorphic multisections of a symplectic cobordism $\overline{W}_-$
obtained by closing all fibers of $W_+$ with a disk and turning it upside down. The
Lagrangian boundary condition on $\overline{W}_-$ becomes singular, and this leads
to many more potential degenerations of $J$-holomorphic sections. For this reason,
the proof that $\Psi_*$ is defined is the longest and most difficult part of this paper.

In the next paper \cite{CGH-II} we prove that $\Phi_*$ and $\Psi_*$ are inverses of each other by composing the two cobordisms and degenerating them in a different way.
The proof that $\Phi_* \circ \Psi_*$ and $\Psi_* \circ \Phi_*$ are the identity is thus
reduced to a computation of some relative Gromov-Witten (or relative Gromov-Taubes) invariants.
Finally we prove that the maps $(\mathfrak{I}_j)_*$ are isomorphisms for $j\geq 2g$ by
an argument based on stabilizing the open book decomposition.

\subsection{Organization of the paper}

Papers \cite{CGH-I} and \cite{CGH-II} should be read as a single paper which has been split for practical reasons. References from \cite{CGH-II} will be written as ``II.$x$''; for example  ``Section II.$x$'' will mean ``Section $x$'' of \cite{CGH-II}.

In Section~\ref{section: adapting ECH to open book} we recall some results of \cite{CGH2}, including the definition of $\widehat{ECH}(N,\bdry N)$. In
Section~\ref{section: periodic Floer homology} we replace the ECH chain complexes $ECC_j(N)$ by the periodic Floer homology chain complexes $PFC_j(N)$, which are technically a little easier to use when defining chain maps to and from Heegaard Floer homology. Then in Section~\ref{section: variation of HF adapted to OB} we (i) review Lipshitz' reformulation of Heegaard Floer homology, (ii) restrict the Heegaard Floer chain complex to a page $S$ as in
\cite{HKM1} and obtain the chain group $\widehat{CF}(S,\mathbf{a}, \hh(\mathbf{a}))$ whose homology is isomorphic to $\widehat{HF}(-M)$, and (iii) introduce an ECH-type index $I_{HF}$ for Heegaard Floer homology. Section~\ref{section: moduli spaces of multisections} is devoted to describing the moduli spaces of multisections which are used in the definitions of the chain maps  $\Phi$ and $\Psi$ between the Heegaard Floer chain complex $\widehat{CF}(S,\mathbf{a}, \hh(\mathbf{a}))$ and the periodic Floer homology chain complex $PFC_{2g}(N)$.  Then in Sections~\ref{section: chain map phi} and \ref{section: chain map psi} we show that $\Phi$ and $\Psi$ are indeed chain maps.  The proof that $\Psi$ is a chain map is substantially more involved than the proof that $\Phi$ is a chain map.

The proofs of the chain homotopies between the chain maps $\Psi\circ\Phi$ and $id$, and between the chain maps $\Phi\circ\Psi$ and $id$, are rather involved and occupy almost all (Sections~II.\ref{P2-section: Gromov-Witten computation}--II.\ref{P2-section: homotopy of cobordisms II}) of \cite{CGH-II}.  The necessary Gromov-Witten (or Gromov-Taubes) type calculations which are used in the proof of the chain homotopy are carried out in Section~II.\ref{P2-section: Gromov-Witten computation}. Finally, in Section~II.\ref{P2-section: stabilization} we prove that (variants of)
the maps $(\mathfrak{I}_j)_*$ are isomorphisms for $j\geq 2g$.

\section{Adapting $ECH$ to an open book decomposition}
\label{section: adapting ECH to open book}

In this section we briefly recall the results of \cite{CGH2}. The reader is referred to \cite{CGH2} for a more complete discussion; the notation here is the same as that of \cite{CGH2}, unless indicated otherwise.

\subsection{The first return map}
\label{subsection: first return map}

We start by recalling the definition of an open book decomposition
in order to set the notation. Let $(S,\hh)$
\nom[SS2]{$(S, \hh)$}{Open book decomposition of $M$ with connected binding, page $S$, and monodromy $\hh$}
\nom[SS1]{$S$}{Compact oriented connected surface of genus $g$ with connected boundary}
\nom[h4]{$\hh$}{Monodromy $\hh:S\stackrel\sim\to S$}
be a pair consisting of a compact oriented surface $S$ with nonempty boundary
(sometimes called a {\em bordered surface}) and a diffeomorphism
$\hh:S\stackrel\sim\to S$ which restricts to the identity on $\bdry S$.
{\em In this paper we will always
assume that $\bdry S$ is connected, unless stated otherwise.}

We define the mapping torus
$$N_{(S,\hh)}=S\times[0,1]/\sim,$$
where $(x,1)\sim (\hh(x),0)$.
\nom[N1]{$N=N_{(S, \hh)}$}{Mapping torus of $\hh: S \stackrel\sim\to S$; from Section~\ref{section: moduli spaces of multisections} onwards, is given by
$N=S\times[0,2]/(x,2)\sim (\hh(x),0)$}
\begin{defn}\label{defn: open book decomposition}
Let $K\subset M$ be a link. Then $M$ admits an {\em open book decomposition
$(S,\hh)$ with binding $K$} if there is a diffeomorphism
$$M \cong N_{(S, \hh)} \cup V,$$
where $V \cong D^2 \times S^1$ is a tubular neighborhood of $K$
and, if we parametrize $\partial D^2$ by $t \in [0,1]$, then $\partial S \times \{ t \}
\subset \partial N_{(S, \hh)}$  is glued to $\{ t \} \times S^1$.
\end{defn}
The decomposition we give here is the standard one and is slightly different from the
one in \cite{CGH2}, where we introduced a ``no man's land'' diffeomorphic to $T^2 \times [0,1]$ between $N_{(S, \hh)}$ and $V$. Of course the two decompositions are
equivalent.

Let $S$ be a bordered surface.  Let $\nu(\bdry S)\simeq
[-\varepsilon,0]\times \R/\Z$ be a neighborhood of $\bdry S$ with
coordinates $(y,\theta)$ so that $\bdry S=\{y=0\}$. (The slight
difference with \cite{CGH2} is that, in \cite[Section~9.3]{CGH2},
$\bdry S=\{y=1\}$ instead of $\{y=0\}$.)

The following was essentially proved in \cite{CGH2}:

\begin{lemma} \label{lemma: change from h to h_0}
Given $\hh\in \op{Diff}(S,\bdry S)$, there exists $\hh_0\in \op{Diff}(S,\bdry S)$ in the same connected component as $\hh$, and which satisfies the following:
\begin{enumerate}
\item there exists a contact form $\alpha$ on $N_{(S,\hh)}$ such that the Reeb vector field $R_\alpha$ of $\alpha$ is transverse to $S\times\{t\}$, $t\in[0,1]$, and $\hh_0$ is the first return map of $R_\alpha$ on $S\times\{0\}$; and
\item the diffeomorphism $\hh_0$ restricts to $(y,\theta)\mapsto (y, \theta-y)$ on $\nu(\bdry S)$.
\end{enumerate}
Moreover, $\alpha=f_tdt+\beta_t$, where $f_t$ is a positive function on $S$ and $\beta_t$ is a Liouville $1$-form on $S$.
\end{lemma}

\begin{proof}
 (1) and the last sentence of the lemma follow from combining Lemmas~9.3.2 and 9.3.3 from \cite{CGH2}. To verify (2), we refer to Section~9.3.1 and the discussion after the proof of Lemma~9.3.3 in \cite{CGH2} and take a slight modification of $\alpha$ of the form
$$\alpha=g(y)d\theta+f(y)dt$$
on a neighborhood $\nu(\bdry N_{(S,\hh)})$ of $\bdry N_{(S,\hh)}$. Here $\nu(\bdry N_{(S,\hh)})$ is the quotient of
$$\nu(\bdry S)\times[0,1]\simeq [-\varepsilon,0]\times(\R/\Z)\times[0,1]$$
with coordinates $(y,\theta,t)$,
\nom[y]{$(y,\theta,t)$}{Coordinates of a collared neighborhood of $\partial N_{(S, \hh)}$}
by the equivalence relation $(y,\theta,1)\sim (y,\theta,0)$. If we take
$$(f(y),g(y))=\left(f(0)+y^2/2, g(0)+y\right),$$
with $(f(0),g(0))$ in the interior of the first quadrant, then the
Reeb vector field $R_\alpha$ is parallel to
\begin{equation}\label{Reeb near the boundary}
-f'(y)\bdry_\theta+g'(y)\bdry_t=-y\bdry_\theta+\bdry_t.
\end{equation}
Its first return map then satisfies (2).
\end{proof}

{\em From now on, we assume that $\hh=\hh_0$ as given by Lemma~\ref{lemma: change from h to h_0}.}

\subsection{$ECH(N,\bdry N,\alpha)$ and $\widehat{ECH}(N,\bdry N,\alpha)$}

Let $N=N_{(S,\hh)}$ and $\alpha$ be as in the previous subsection. We recall the definitions of the variants $ECH(N,\bdry N,\alpha)$ and $\widehat{ECH}(N,\bdry N,\alpha)$ and the main result concerning them from \cite{CGH2}. In particular, we carry over the Morse-Bott terminology from \cite[Section~4]{CGH2}. \marginpar{binding reference updated} We will assume that the almost complex structure $J$ on $\R\times N$ is Morse-Bott regular.

The boundary $\bdry N$ is foliated by a Morse-Bott family $\mathcal{N}$
\nom[N]{$\mathcal{N}$}{Negative Morse-Bott family of Reeb orbits on $\partial N$; from Section~\ref{section: moduli spaces of multisections} onwards, is a negative Morse-Bott family of orbits of the Hamiltonian vector field $\overline{R}_0$ on $\bdry N$}
of simple orbits of $R_\alpha$ of the form $\theta=const$. We may assume without loss of generality that $\alpha$ is nondegenerate away from $\bdry N$, after a $C^k$-small perturbation for $k\gg 0$. We pick two orbits from $\mathcal{N}$ and label them $h$ and $e$.  The orbits $h$ and $e$
\nom[e]{$e$}{Elliptic orbit of Conley-Zehnder index $-1$ on $\partial N_{(S, \hh)}$}
\nom[h]{$h$}{Hyperbolic orbit of Conley-Zehnder index $0$ on $\partial N_{(S, \hh)}$;  assumed to satisfy Convention~\ref{convention for h} from Section~\ref{subsection: variant widetide psi} onwards}
are meant to become hyperbolic and elliptic after a small, controlled perturbation of $\alpha$. The Morse-Bott family $\mathcal{N}$ is {\em negative}. Since $\mathcal{N}$ is a Morse-Bott family on $\bdry N$, this means that $\mathcal{N}$ plays the role of a sink and that no holomorphic curve (besides a trivial cylinder) is asymptotic to an orbit of $\mathcal{N}$ at the positive end.

\subsubsection{$ECH(N,\bdry N,\alpha)$}

Let $\mathcal{P}$ be the set of simple Reeb orbits of $\alpha$ in $int(N)$.
\nom[P]{$\mathcal{P}$}{Set of simple Reeb orbits of $\alpha$ in $int(N)$; from Section~\ref{section: moduli spaces of multisections} onwards, is the set of simple orbits of the Hamiltonian vector field $\overline{R}_0$ in $int(N)$}
We write $ECC_j^\flat(N,\alpha)$ for the chain complex generated by orbit sets $\boldsymbol{\gamma}$ constructed from $\mathcal{P}\cup\{e\}$, whose homology class $[\boldsymbol{\gamma}]$ intersects the page $S\times\{t\}$ exactly $j$ times. The differential for $ECC_j^\flat(N,\alpha)$ counts ECH index $1$ Morse-Bott buildings in $(\R\times N,J)$ between orbit sets which are constructed from $\mathcal{P}\cup\{e\}$; for more details see \cite[Section~7.3]{CGH2}. In particular, if $\tilde u$ is an ECH index $1$ Morse-Bott building which is counted in the differential, then $e$ can appear only at a negative end of $\tilde u$ and no single end of $\tilde u$ can multiply cover $e$ with multiplicity $>1$. (It is still possible that there are many ends of $\tilde u$ which simply cover $e$.)

There are inclusions of chain complexes:
$$\mathfrak{I}_j^\flat:ECC_j^\flat(N,\alpha)\to ECC_{j+1}^\flat(N,\alpha),$$
$$\boldsymbol{\gamma}\mapsto e\boldsymbol{\gamma},$$
where we are using multiplicative notation for orbit sets. Let us
write $ECH_j^\flat(N,\alpha)$ for the homology of the chain complex
$ECC_j^\flat(N,\alpha)$.  We then define
$$ECH(N,\bdry N,\alpha)=\lim_{j\to\infty}
ECH_j^\flat(N,\alpha).$$

\subsubsection{$\widehat{ECH}(N,\bdry N,\alpha)$}
\label{subsection: defn of ECH hat}

Let $ECC_j(N,\alpha)$
\nom[ECC ]{$ECC(N,\alpha)$}{ECH chain complex for $(N,\alpha)$}
\nom[ECC ]{$ECC_j(N,\alpha)$}{Subcomplex of $ECC(N,\alpha)$
generated by orbit sets which intersect a page $j$ times} be the chain complex generated by orbit sets $\boldsymbol{\gamma}$ constructed from $\widehat{\mathcal{P}}=\mathcal{P}\cup\{e,h\}$,
\nom[P]{$\widehat{\mathcal{P}}$}{$\mathcal{P}\cup\{e,h\}$}
whose homology class $[\boldsymbol{\gamma}]$ intersects $S\times\{t\}$ exactly $j$ times. The differential for $ECC_j(N,\alpha)$ counts ECH index $1$ Morse-Bott buildings in $(\R\times N,J)$. There are inclusions:
\begin{equation}\label{eqn: new year}
\mathfrak{I}_j: ECC_j(N,\alpha)\to ECC_{j+1}(N,\alpha),
\end{equation}
$$\boldsymbol{\gamma}\mapsto e\boldsymbol{\gamma},$$
\nom[I]{$\mathfrak{I}_j$}{ECH chain maps given by Equation~\eqref{eqn: new year}}
as before.
Writing $ECH_j(N,\alpha)$ for the homology of
$ECC_j(N,\alpha)$, we define
$$\widehat{ECH}(N,\bdry N,\alpha)=\lim_{j\to\infty}ECH_j(N,\alpha).$$

The following was the main result of \cite{CGH2}:

\begin{thm}\label{thm: from first paper of series bis}
We have the isomorphisms:
$$ECH(M)\simeq ECH(N,\bdry N,\alpha),$$
$$\widehat{ECH}(M)\simeq \widehat{ECH}(N,\bdry N,\alpha).$$
\end{thm}

\subsection{Splitting of $ECH$ according to homology classes}

Given an orbit set $\boldsymbol{\gamma} = \prod_{j=1}^l\gamma_j^{m_j}$ in $M$ or $N$, its {\em total homology class} $[\boldsymbol{\gamma}]$ is defined as
\[ [\boldsymbol{\gamma}] = \sum_{j=1}^l m_j [\gamma_j], \]
where $ [\gamma_j] \in H_1(M; \Z)$ if $M$ is a closed manifold and $[\gamma_j] \in H_1(N;\Z)$ or $H_1(N, \partial N; \Z)$ (as appropriate) if $N$ has torus boundary. We then have the direct sum decomposition:
\[ ECC(M) = \bigoplus_{A \in H_1(M)} ECC(M,A), \]
where $ECC(M,A)$ is the subcomplex of $ECC(M)$ generated by orbit sets with total homology class $A$. The direct sum of chain groups descends to a direct sum of homology groups $ECH(M,A)$.

There is an analogous splitting for $ECC(N,\bdry N,\alpha)$.  In fact,
\[ ECC_j^\flat(N,\alpha) = \bigoplus_{A \in H_1(N, \partial N)}
ECC_j^\flat(N, \alpha,A), \]
and the inclusion $\boldsymbol{\gamma} \mapsto e \boldsymbol{\gamma}$ respects this splitting since $[e]=0$ in $H_1(N, \partial N)$.

\begin{lemma} \label{varpi}
If $M$ has an open book decomposition with binding $K$ and $N$ is the mapping torus of a page, then there is an isomorphism
\[ \varpi: H_1(M) \stackrel{\sim} \longrightarrow  H_1(N, \partial N). \]
\end{lemma}

\begin{proof}
We use the long exact sequence for the pair $(M, K)$:
\[ H_1(K) \to H_1(M) \stackrel{i}\to H_1(M, K) \to H_0(K). \]
Since $[K]=0$ in $H_1(M)$ and $K$ is connected, the map $i$ is an
isomorphism. By excision and homotopy invariance, we obtain the
isomorphism $H_1(M, K) \stackrel{\sim} \longrightarrow H_1(N,
\partial N).$ Combining the two isomorphisms gives us $\varpi$.
\end{proof}

The isomorphism $ECH(M) \cong ECH(N,\bdry N,\alpha)$ respects the
splitting into total homology classes.  In fact
$$ECH(M, A) \cong ECH(N,\bdry N, \alpha, \varpi(A)).$$
The same also holds for the hat versions.

\subsection{Twisted coefficients in ECH}
\label{subsection: twisted coefficients ECH}

In this subsection we describe the construction of ECH with twisted coefficients. We will adapt the analogous construction in Heegaard Floer homology from \cite[Section~8]{OSz2} instead of following the original construction in \cite[Section~11]{HS2}.

Fix a homology class $A$ and a closed curve $\Gamma \subset M$ such that
$[\Gamma]=A$.  Let $\boldsymbol{\gamma}^+$ and $\boldsymbol{\gamma}^-$ be orbit sets with $[\boldsymbol{\gamma}^+]=[\boldsymbol{\gamma}^-]= A$. \nom[H]{$H_2(M, \boldsymbol{\gamma}^+, \boldsymbol{\gamma}^-)$}{Homology classes of chains in $N$ with boundary in $\boldsymbol{\gamma}^-$ and $\boldsymbol{\gamma}^-$} We denote by $H_2(M, \boldsymbol{\gamma}^+, \boldsymbol{\gamma}^-)$ the set of relative homology classes of $2$-chains $C$ in $M$ with $\partial C= \boldsymbol{\gamma}^+ - \boldsymbol{\gamma}^-$ and by ${\mathcal M}^{I=1}(\boldsymbol{\gamma}^+, \boldsymbol{\gamma}^-, C)$ the moduli spaces of $I=1$ holomorphic curves in $\R \times M$ from $\boldsymbol{\gamma}^+$ to
$\boldsymbol{\gamma}^-$ representing the homology class $C$.

A {\em complete set of paths for $A$ based at $\Gamma$} is the choice, for every orbit set $\boldsymbol{\gamma}$ such that $[\boldsymbol{\gamma}]=A$, of a surface $C_{\boldsymbol{\gamma}} \subset M$ such that
$\partial C_{\boldsymbol{\gamma}} = \boldsymbol{\gamma} - \Gamma$. A complete set of paths for $A$ induces maps
$${\mathfrak A}': H_2(M, \boldsymbol{\gamma}^+, \boldsymbol{\gamma}^-) \to H_2(M)$$
for all $\boldsymbol{\gamma}^+$ and $\boldsymbol{\gamma}^-$ in $A$ by $\mathfrak{A}'(C)= [C_{\boldsymbol{\gamma}^-} \cup C \cup - C_{\boldsymbol{\gamma}^+}]$. This map is compatible with the action of $H_2(M)$ on $H_2(M, \boldsymbol{\gamma}^+, \boldsymbol{\gamma}^-)$ and with the concatenation of chains with matching ends.

We denote  the group ring of $H_2(M; \Z)$ by $\F[H_2(M; \Z)]$ and the generator corresponding to $c \in H_2(M; \Z)$ by $e^c$. We define
$$\underline{ECC}(M, \alpha, A) = ECC(M, \alpha, A) \otimes_{\F} \F[H_2(M; \Z)]$$
as an $\F[H_2(M; \Z)]$-module, with differential
$$ \partial \boldsymbol{\gamma}^+ = \sum_{\boldsymbol{\gamma}^-} \sum_{C \in H_2(M, \boldsymbol{\gamma}^+, \boldsymbol{\gamma}^-)} \# \left ( {\mathcal M}^{I=1}(\boldsymbol{\gamma}^+, \boldsymbol{\gamma}^-, C) / \R \right ) e^{\mathfrak{A}'(C)} \boldsymbol{\gamma}^-.$$

The homology of this complex is ECH with twisted coefficients $\underline{ECH}(M, A)$. The $U$-map can be defined in a similar manner and
$\widehat{\underline{ECH}}(M, A)$ is the homology of its mapping cone.
The construction of $\underline{ECH}(N, \partial N, \varpi(A))$ with coefficient ring
$\F[H_2(N, e; \Z)] \cong \F[H_2(M; \Z)]$ is similar, and there are isomorphisms
\begin{equation} \label{twisted binding 1}
\underline{ECH}(M, A) \simeq \underline{ECH}(N,\bdry N, \varpi(A)).
\end{equation}
\begin{equation} \label{twisted binding 2}
\underline{\widehat{ECH}}(M,A)\simeq \underline{\widehat{ECH}}(N,\bdry N, \varpi(A)).
\end{equation}
Moreover, by considering only orbit sets that intersect a page $j$ times we can define
$\underline{ECH}_j(N)$ and we have
$$\underline{\widehat{ECH}}(N,\bdry N, \alpha, \varpi(A)) = \lim \limits_{j \to \infty}\underline{ECH}_j(N, \varpi(A)).$$

\subsection{Elimination of elliptic orbits} \label{elimination of elliptic orbits}

The goal of this subsection is to show how to locally replace elliptic orbits by hyperbolic orbits with the same parity (i.e., with negative
eigenvalues). The main result, which is used in \cite{CGH-II} but is also of independent interest, is Theorem~\ref{thm: elimination} below.  Let us first give the following definition:

\begin{defn}[Filtration $\mathcal{F}$]
If $N=N_{(S,\hh)}$ is the mapping torus of $(S,\hh)$ and $\gamma\subset N$ is a link which is everywhere transverse to $S\times\{t\}$, $t\in[0,1]$, then we define $\mathcal{F}(\gamma)=\langle\gamma,S\times\{0\}\rangle$.
\end{defn}

\begin{thm}[Elimination of elliptic orbits]\label{thm: elimination}
Let $\alpha$ be a contact form on the mapping torus $N=N_{(S,\hh)}$, $\hh\in \op{Diff}(S,\bdry S)$, such that the Reeb vector field $R_\alpha$ is transverse to $S\times\{t\}$, $t\in[0,1]$, and $\hh$ is the first return map of $R_\alpha$ on $S\times \{ 0\}$. Then, given $m\in \Z^+$ and $\varepsilon>0$, there exists a smooth function $f :N\rightarrow (0,+\infty)$ which is $\varepsilon$-close to $1$ with respect to a fixed $C^1$-norm and whose Reeb vector field $R_{f \alpha}$ has no elliptic orbits $\gamma$ in $int(N)$ satisfying $\mathcal{F}(\gamma)\leq m$.
\end{thm}

\subsubsection{Model situation on the solid torus}\label{modelsituation}

Fix a constant $\delta>0$. Consider the solid torus
$$V=D^2 \times S^1=D^2\times (\R /\Z) =\{ (r,\theta ,z) ~|~r\leq \delta \}$$ with the contact structure $\xi_0=\ker \alpha_0$, where $\alpha_0=dz+r^2 d\theta$. We write $D_{z_0} =\{z=z_0\}\subset V$ and $T_{r_0}=\{r=r_0\}\subset V$.

Given a function $f: [0,\delta] \rightarrow (0,+\infty )$, the Reeb vector field $R_{f(r)\alpha_0}$ is given by:
\begin{equation} \label{reeb}
R_{f\alpha_0} = \frac{1}{2rf^2} ((r^2 f'+2rf) \partial_z -f'\partial_\theta ).
\end{equation}
In particular, $R_{f \alpha_0}$ is transverse to each $D_z$, provided $r^2f'+2rf>0$, and $R_{f\alpha_0}$ is tangent to each $T_r$. The first return map $\Phi_{f\alpha_0}:D_0\stackrel\sim\to D_0$ is a rotation on each circle $\{r=r_0\}$ and can be written as $(r,\theta)\mapsto (r, \theta+\phi_{f\alpha_0}(r))$.

\begin{claim} \label{explicit formula}
Let $C:[0,\delta]\to (0,+\infty)$ be a smooth function. If
\begin{equation}\label{formula}
f(r)=A\exp \left(-\int_0^r \frac{2sC(s)}{1+C(s)s^2} ds\right),
\end{equation}
where $A$ is a positive constant,
then the first return map of $R_{f\alpha_0}$ satisfies $\phi_{f\alpha_0}(r)=C(r)$ for each $r\in(0,\delta]$.
\end{claim}
 Here we are ignoring the differentiability at $r=0$.

\begin{proof}
In view of Equation~\eqref{reeb}, $R_{f\alpha_0}$ is parallel to $\bdry_z +C(r) \bdry_\theta$ if and only if
\begin{equation} \label{diff eq}
C(r)(r^2f'+2rf)=-f'
\end{equation}
is satisfied.
\end{proof}

Let $f_0: [0,\delta] \rightarrow (0,+\infty )$ be a function such that $\phi_{f_0\alpha_0}(r)=\phi_0$, where $\phi_0 \in (0,2\pi )$. By Claim~\ref{explicit formula}, we may take $$f_0(r)=\exp \left(-\int_\delta^r \frac{2s\phi_0}{1+\phi_0 s^2} ds \right)= {1+\phi_0\delta^2\over 1+\phi_0 r^2}.$$ In particular, $\gamma_0=\{ 0\} \times S^1$ is the only orbit $\gamma$ of $R_{f_0\alpha_0}$ satisfying $\mathcal{F}(\gamma)=1$, where $\mathcal{F}(\gamma)=\langle \gamma, D_0\rangle$, and the orbit $\gamma_0$ is elliptic and nondegenerate.

\begin{lemma}[Modification lemma]\label{lemma: elimination}
There exists a function $f_2: V\rightarrow (0,+\infty )$ such that $f_2\alpha_0$ is arbitrarily $C^1$-close to $f_0\alpha_0$ and the Reeb vector field $R_{f_2\alpha_0}$ is equal to $R_{f_0\alpha_0}$ near $\partial V$, is transverse to $D_z$ for all $z\in S^1$, and has only one orbit $\gamma$ satisfying $\mathcal{F}(\gamma)=1$. Moreover the orbit $\gamma$ is hyperbolic.
\end{lemma}

\begin{proof}
Without loss of generality we take $\phi_0\in (0,\pi]$; the case $\phi_0\in[\pi,2\pi)$ is similar.

Let $0< \delta'\ll \delta$ and let $C_{\delta'}:[0,\delta]\to [\phi_0,\pi]$ be a smooth function such that:
\begin{itemize}
\item $C_{\delta'}(r)=\pi$ on $[0,{\delta'\over 2}]$; and
\item $C_{\delta'}(r)=\phi_0$ on $[\delta',\delta]$.
\end{itemize}
Let $f_1: [0,\delta] \rightarrow (0,+\infty )$ be the smooth function
$$f_1(r) = \exp\left( -\int_\delta^r {2sC_{\delta'}(s)\over 1+C_{\delta'}(s) s^2} ds  \right).$$
The function $f_1$ satisfies the following:
\begin{enumerate}
\item $f_1=f_0$ and $R_{f_1\alpha_0} =R_{f_0\alpha_0}$ for $r\in[\delta',\delta]$;
\item $R_{f_1\alpha_0}\pitchfork D_z$ for all $z\in S^1$;
\item $\phi_{f_1\alpha_0}(r)=C_{\delta'}(r)$ for all $r\in[0,\delta]$.
\end{enumerate}

We compute that
\begin{equation}  \label{exp integral}
{f_1(r)\over f_0(r)}= \exp \left(\int_\delta^r \left( {2s\phi_0\over 1+\phi_0 s^2}-{2sC_{\delta'}(s)\over 1+C_{\delta'}(s)s^2} \right) ds\right).
\end{equation}
By taking $\delta'>0$ to be arbitrarily small, the absolute value of the integrand of Equation~\eqref{exp integral} can be made arbitrarily small. Hence $f_1(r)\approx f_0(r)$ for all $r\in[0,\delta]$, in view of (1). (Here $f\approx g$ means $|f-g|$ is bounded above by a continuous function of $\delta'$ which approaches $0$ as $\delta'\to 0$.)

Next we consider $\bdry_x f_i ={xf_i'(r)\over r}$ and $\bdry_y f_i={yf'_i(r)\over r}$, where $r=\sqrt{x^2+y^2}$. We appeal to Equation~\eqref{diff eq} and write
$$f_1'(r)={2rf_1(r)\over 1+C_{\delta'}(r) r^2}, \quad f_0'(r)={2rf_0(r)\over 1+\phi_0 r^2}.$$
For $\delta'>0$ sufficiently small, $f_1'(r)$ and $f_0'(r)$ are both arbitrarily close to $0$. Moreover $f_1'(r)=f_0'(r)$ for $r\in[\delta',\delta]$. Hence $f_1'\approx f_0'$ for all $r\in[0,\delta]$. This implies that $\bdry_* f_1\approx\bdry_* f_0$ and $\bdry_* (f_1/ f_0) \approx 0$ for $*=x,y$.

The Reeb vector field $R_{f_1\alpha_0}$ has exactly one orbit $\gamma_0$ satisfying $\mathcal{F}(\gamma_0)=1$. The linearized first return map $d\Phi_{f_1\alpha_0}(0)$ of $\gamma_0$ has eigenvalues $-1$.

We now claim there exists a $C^k$-small perturbation $f_2$ of $f_1$ for any $k\gg 0$ such that the linearized first return map $d\Phi_{f_2\alpha_0}(0)$ of the corresponding orbit has eigenvalues $-\lambda$ and $-{1\over \lambda}$ with $\lambda\in \R^{>0}-\{1\}$. This follows from a local model on $D^2\times[0,1]$ with coordinates $(x,y,z)$:
Suppose the Reeb vector field is $R=\bdry_z$. Then the contact $1$-form can be written as $\alpha=dz+\beta$, where $\beta$ is a $1$-form on $D^2$. We consider $R_{\mathscr{G}\alpha}$, where $\mathscr{G}(x,y)=1+\varepsilon(x^2-y^2)$ and $\varepsilon>0$ is sufficiently small. The component of $R_{\mathscr{G}\alpha}$ in the $xy$-direction is parallel to $y\bdry_x-x\bdry_y$. Hence the derivative at zero of the holonomy map $D^2\times\{0\}\dashrightarrow D^2\times\{1\}$,\footnote{The dashed arrow indicates that the map is only partially defined.} obtained by flowing along $R_{\mathscr{G}\alpha}$, has eigenvalues $\lambda_0$ and ${1\over \lambda_0}$ with $\lambda_0\in \R^{>0}-\{1\}$. By appropriately damping  $\mathscr{G}$ out to $1$ outside a small neighborhood of $(0,0,0)\in D^2\times[0,1]$, the above model can be grafted into $V$ to give $f_2$. This procedure does not introduce any extra $\mathcal{F}=1$ orbits, since the graph of $\Phi_{f_1\alpha_0}$ in $D_0\times D_0$ intersects the diagonal transversely in one point and this property is stable under a $C^k$-small perturbation of $\Phi_{f_1\alpha_0}$ for any $k\gg 0$.
\end{proof}

\begin{rmk}
The modification in Lemma~\ref{lemma: elimination} introduces many elliptic orbits satisfying $\mathcal{F}>1$.
\end{rmk}

\subsubsection{Proof of Theorem~\ref{thm: elimination}}

We now prove Theorem~\ref{thm: elimination}. Starting with the contact form $\alpha$ on $N$, we make a $C^2$-small perturbation of $\alpha$ relative to $\bdry N$ such that the resulting Reeb vector field --- also called $R_\alpha$ --- satisfies the following:
\begin{enumerate}
\item $R_\alpha$ is nondegenerate away from $\bdry N$;
\item for each $\mathcal{F}=1$ elliptic orbit $\gamma\subset int(N)$, the first return map is a rotation by an irrational angle and $\alpha$ is of the form $C_0f_0\alpha_0$ on a tubular neighborhood $V_\gamma$ of $\gamma$.
\end{enumerate}
Here $f_0$ and $\alpha_0$ are as in Section~\ref{modelsituation} and $C_0$ is some constant.
We then use Lemma~\ref{lemma: elimination} on the tubular neighborhoods $V_\gamma$ to replace the $\mathcal{F}=1$ elliptic orbits by hyperbolic orbits
plus $\mathcal{F}>1$ orbits.  Next, we perturb the form so that the $\mathcal{F}=1$ orbits and the $\mathcal{F}=2$ hyperbolic orbits are left unchanged and the $\mathcal{F}=2$ elliptic orbits satisfy (2), with $\mathcal{F}=1$ replaced by $\mathcal{F}=2$.
Using Lemma~\ref{lemma: elimination} again, we replace the $\mathcal{F}=2$ elliptic orbits by hyperbolic orbits and $\mathcal{F}>2$ orbits.  Continuing in this manner, we obtain $\alpha$ without any $\mathcal{F}\leq m$ elliptic orbits in $int(N)$.

\subsubsection{Direct limits} \label{direct limits}

Starting with $(S,\hh)$ and $\alpha$ from Section~\ref{subsection: first return map}, we define $f_j:N\to (0,+\infty)$, $j\in \N$, inductively as follows.  Let $f_0=1$.  Suppose we have chosen up to $f_j$ so that $R_{f_j\alpha}$ has no elliptic orbits with $\mathcal{F}\leq j$.  Let $V_{j+1}$ be a small tubular neighborhood of the elliptic orbits of $R_{f_j\alpha}$ with $\mathcal{F}= j+1$. Then we choose $f_{j+1}$ such that the following hold:
\begin{enumerate}
\item $f_{j+1}$ is $C^1$-close to $f_j$;
\item $f_{j+1}=f_j$ on $N-V_{j+1}$;
\item $R_{f_{j+1}\alpha}$ has no elliptic orbits with $\mathcal{F}\leq j$.
\end{enumerate}
The existence of $f_{j+1}$ is given by Theorem~\ref{thm: elimination}.

Next we consider the ECH chain maps
\begin{equation} \label{eqn: frak J}
{\frak K}_j: ECC_j(N,f_j\alpha)\to ECC_{j+1}(N,f_{j+1}\alpha),
\end{equation}
\nom[K ]{${\frak K}_j$}{ECH chain maps given by Equation~\eqref{eqn: frak J}}
given by composing two maps:
$${\frak K}'_j: ECC_j(N,f_j\alpha)\to ECC_{j}(N,f_{j+1}\alpha)$$
and $\mathfrak{I}_j: ECC_j(N,f_{j+1}\alpha)\to ECC_{j+1}(N,f_{j+1}\alpha)$ given by $\boldsymbol{\gamma}\mapsto e\boldsymbol{\gamma}$.  The map ${\frak K}'_j$ is defined by suitably completing $f_j\alpha$ and $f_{j+1}\alpha$ to $M$ and applying the ECH cobordism map given by \cite[Theorem~2.4]{HT3}.
It is important to remember that the ECH cobordism map is defined by using cobordism maps in Seiberg-Witten Floer cohomology.

Then we have:
\begin{thm} \label{thm: direct limit}
$\displaystyle\widehat{ECH}(M)\simeq \lim_{j\to\infty} ECH_j(N,f_j\alpha),$ where the direct limit is taken with respect to the maps ${\frak K}_j$.
\end{thm}

This theorem will be used in \cite{CGH-II}. Its proof is omitted, since it is similar to that of Theorem~\ref{thm: from first paper of series}.

\section{Periodic Floer homology}
\label{section: periodic Floer homology}

In order to simplify some technicalities, we would like to replace
the ECH groups by the {\em periodic Floer homology groups} of
Hutchings~\cite{Hu1,Hu2}, abbreviated PFH in this paper. The PFH
groups are defined in a manner completely analogous to the ECH
groups, with stable Hamiltonian vector fields replacing the Reeb
vector fields.

If $M$ is a closed manifold which fibers over the circle, then the
PFH groups of $M$ are equivalent to the Seiberg-Witten Floer
cohomology groups of $M$ by the work of Lee-Taubes~\cite{LT}.

\subsection{Interpolating between Reeb and stable Hamiltonian
vector fields} \label{subsection: interpolation}

Consider the contact form $\alpha=f_tdt+\beta_t$ on $S\times[0,1]$,
as defined in Section~\ref{subsection: first return map}. We may
assume that $R_\alpha$ is parallel to $\bdry_t$ on $S\times[0,1]$.
Since
$$d\alpha= d_Sf_t\wedge dt+ d_S\beta_t+dt\wedge{d\beta_t\over dt},$$
where $d_S$ is the exterior derivative in the $S$-direction, it
follows that ${d\beta_t\over dt}=-d_Sf_t$. (Hence $d_S\beta_t$ is an
area form which does not depend on $t$.) Also, the form $\alpha$ is
a contact form as long as $d_S\beta_t>0$. Hence, for $C\gg 0$, the
form $(C+f_t)dt+\beta_t$ is a contact form with Reeb vector field
parallel to $\bdry_t$.

Now consider the $1$-parameter family of $1$-forms
\begin{equation} \label{eqn: defn of alpha s}
\alpha_\varsigma=Cdt+\varsigma(f_tdt+\beta_t),
\end{equation}
$\varsigma\in[0,1]$, on $N$. It interpolates between the contact form
$\alpha_1=Cdt+ (f_tdt+\beta_t)$ and the stable Hamiltonian form
$\alpha_0=Cdt$. The Reeb vector fields $R_{\varsigma}=R_{\alpha_\varsigma}$
are directed by $\bdry_t$ and hence are parallel for all $\varsigma>0$.

The pair $(\alpha_\varsigma,\omega=d_S\beta_t)$
\nom[1a$\alpha$]{$(\alpha_\varsigma, \omega)$}{Stable Hamiltonian structures on $N_{(S, \hh)}$}
is a {\em stable Hamiltonian structure} on $N$, i.e., $\omega$ is closed, $\alpha_\varsigma\wedge \omega>0$, and $\ker d\alpha_\varsigma\supset \ker \omega$; see for example \cite{CV}.
When $\varsigma=0$, the Hamiltonian vector
field $R_0$ equals ${1\over C} \bdry_t$ and hence is parallel to all
the $R_\varsigma$, $\varsigma>0$. Also let $\xi_\varsigma$ be the $2$-plane field
on $N$ given by the kernel of $\alpha_\varsigma$. The closed $2$-form
$\omega$ can either be viewed as an area form on $S$ or as a
(maximally nondegenerate) $2$-form on $N$.

\subsection{Definitions}
Consider the infinite cylinder $\R\times N$ with coordinates $(s,x)$. We will also
use the notation $W'=\R \times N$.

\begin{defn}\label{defn: almost complex structures J_tau}
An almost complex structure $J_\varsigma$
\nom[J10]{$J_\varsigma$}{Almost complex structure on $\R \times N$ adapted to $(\alpha_\varsigma, \omega)$}
\nom[J11]{$J'=J_0$}{Another name for $J_\varsigma$ when $\varsigma=0$}
on $\R\times N$ is {\em adapted to the stable Hamiltonian structure} $(\alpha_\varsigma,\omega)$ if $J_\varsigma$ satisfies the following:
\begin{itemize}
\item $J_\varsigma$ is $s$-invariant;
\item $J_\varsigma(\bdry_s)=R_{\varsigma}$ and $J_\varsigma(\xi_\varsigma)=\xi_\varsigma$;
\item $J_\varsigma$ is tamed by the symplectic form $\Omega'=ds\wedge dt+ \omega$.
\end{itemize}
\end{defn}

Our goal is to replace the ECH chain complexes
$$ECC_j(N,\alpha_\varsigma,J_\varsigma),\quad ECC_j^\flat(N,\alpha_\varsigma,J_\varsigma)$$
for $\varsigma>0$, by the analogously defined PFH chain complexes
$$PFC_j(N,\alpha_0,\omega,J_0),\quad PFC_j^\flat(N,\alpha_0,\omega,J_0).$$
The orbit sets of the ECH chain groups are constructed using the Reeb vector fields $R_\varsigma$ and the orbit sets of the PFH chain groups are constructed using the Hamiltonian vector field $R_{0}$.
\nom[PFC]{$PFC_j(N,\alpha_0,\omega)$}{Periodic Floer homology chain complex on $N$ analogous to $ECC_j(N, \alpha)$}

We introduce some notation. Let $J_\varsigma$, $\varsigma\in[0,1]$, be a
smooth family of $(\alpha_\varsigma,\omega)$-adapted almost complex
structures, let $\boldsymbol{\gamma},\boldsymbol{\gamma}'$ be orbit sets of
$PFC_j(N,\alpha_0,\omega)$ or $ECC_j(N,\alpha_\varsigma)$, and let $Z\in
H_2(N,\boldsymbol{\gamma},\boldsymbol{\gamma}')$.  We write:
\begin{itemize}
\item $\mathcal{M}_{J_\varsigma}(\boldsymbol{\gamma},\boldsymbol{\gamma}')$ for the moduli space of
$J_\varsigma$-holomorphic curves from $\boldsymbol{\gamma}$ to $\boldsymbol{\gamma}'$;
\item $\mathcal{M}_{J_\varsigma}(\boldsymbol{\gamma},\boldsymbol{\gamma}',Z)\subset
\mathcal{M}_{J_\varsigma}(\boldsymbol{\gamma},\boldsymbol{\gamma}')$ for the subset of curves in
the class $Z$;
\item $\mathcal{M}_{J_\varsigma }^s(\boldsymbol{\gamma},\boldsymbol{\gamma}')\subset
\mathcal{M}_{J_\varsigma}(\boldsymbol{\gamma},\boldsymbol{\gamma}')$ for the subset of
simply-covered curves; and
\item $\mathcal{M}^{nc}_{J_\varsigma}(\boldsymbol{\gamma},\boldsymbol{\gamma}',Z)\subset
\mathcal{M}_{J_\varsigma}(\boldsymbol{\gamma},\boldsymbol{\gamma}',Z)$ for the subset of curves
without connector components.
\nom[2nc]{$*=nc$}{Modifier ``without connector components''}
\end{itemize}
Here a {\em connector over an orbit set $\delta=\prod_i
\delta_i^{m_i}$} is a collection of branched covers of trivial
cylinders where the branching is optional and the total covering
degree over $\R\times \delta_i$ is $m_i$.

Given two orbit sets $\delta=\prod_i \delta_i^{m_i}$ and
$\delta'=\prod_i \delta_i^{m_i'}$, we set $\delta/\delta'= \prod_i
\delta_i^{m_i-m_i'}$ if $m_i'\leq m_i$ for all $i$; otherwise we set
$\delta/\delta'=0$. Here some $m_i$ and $m_i'$ may be zero. We then
write $u\in \mathcal{M}_{J_\varsigma}(\boldsymbol{\gamma},\boldsymbol{\gamma}',Z)$ as $u^0\cup
u^1$, where $u^0$ is a connector over the orbit set $\boldsymbol{\gamma}_0$ and
$u^1\in \mathcal{M}^{nc}_{J_\varsigma}(\boldsymbol{\gamma}/\boldsymbol{\gamma}_0,\boldsymbol{\gamma}'/\boldsymbol{\gamma}_0)$. By \cite[Proposition~7.5]{HT1}, if $I_{ECH}(\boldsymbol{\gamma},\boldsymbol{\gamma}',Z)=1$ (resp.\ $2$), then $u^1\in \mathcal{M}_{J_\varsigma}^s(\boldsymbol{\gamma}/\boldsymbol{\gamma}_0,\boldsymbol{\gamma}'/\boldsymbol{\gamma}_0)$ and has Fredholm index $1$ (resp.\ $2$).
\nom[IECH]{$I_{ECH}$}{ECH index}

\subsection{The flux} \label{subsection: flux}

For more details, see for example~\cite[Section~2]{CHL}.  Let
$$\hh_*: H_1(S;\Z)\to H_1(S;\Z)$$ be the map on homology induced by
$\hh:(S,\omega)\stackrel\sim\to (S,\omega)$ and let $K$ be the kernel
of $\hh_*-id$. Then the {\em flux} $F_{\hh}: K\to \R$ is defined as
follows: Let $[\zeta]\in K$. Since $[\zeta]=[\hh(\zeta)]\in
H_1(S;\Z)$, there is a $2$-chain $\mathcal{S}\subset S$ such that
$\bdry \mathcal{S}=\zeta-\hh(\zeta)$.  We then define:
$$F_{\hh}(\zeta)=\int_{\mathcal{S}}\omega.$$

Since the map $\hh$ is the first return map of a Reeb vector field
$R_\varsigma$, it has zero flux (cf.\ \cite[Lemma~2.2]{CHL}). This
implies that $\int_{[C]}\omega=0$ for every $[C]\in H_2(N)$, where
$\omega$ is now viewed as a closed $2$-form on $N$. Indeed, $[C]$
can be represented by a surface of the form $(\zeta\times[0,1])\cup
(\mathcal{S}\times\{0\})$, where the relevant boundary components
are glued. Hence, if $\boldsymbol{\gamma}$ and $\boldsymbol{\gamma}'$ are orbit sets of
$PFC_j(N,\alpha_0,\omega_0,J_0)$, then the $\omega$-area of any
$Z\in H_2(N,\boldsymbol{\gamma},\boldsymbol{\gamma}')$ only depends on $\boldsymbol{\gamma}$ and $\boldsymbol{\gamma}'$.

\subsection{Compactness}
\label{subsection: compactness PFH case}

The vanishing of the flux is an important ingredient in establishing
that $\mathcal{M}_{J_0}(\boldsymbol{\gamma},\boldsymbol{\gamma}')$ admits a compactification
in the sense of \cite{BEHWZ}.

We briefly outline the argument from \cite[Section~9]{Hu1}:  (1) There is a bound on the $\omega$-area for all elements of $\mathcal{M}_{J_0}(\boldsymbol{\gamma},\boldsymbol{\gamma}')$. This was done above. (2) Given
a sequence of holomorphic curves $u_i\in
\mathcal{M}_{J_0}(\boldsymbol{\gamma},\boldsymbol{\gamma}')$, $i\in \N$, there is a
subsequence which converges weakly as currents to a holomorphic
building $u_\infty$. This is due to the Gromov-Taubes compactness
theorem~\cite{T3}, which works in dimension four and does not
require any a priori bound on the genus of $u_i$. The extraction of
the holomorphic building is treated in some detail in
\cite[Lemma~9.8]{Hu1}. Hence we may assume that the homology classes
$[u_i]\in H_2(N,\boldsymbol{\gamma},\boldsymbol{\gamma}')$ are fixed. (3) There is a bound on
the genus of the curve, provided the homology classes $[u_i]$ are
fixed. This follows from the adjunction inequality and will be
discussed below. (4) Once there is a genus bound, apply the SFT
compactness theorem of \cite{BEHWZ}.

We now explain in some detail how to obtain genus bounds from bounds
on the homology classes $[u_i]$, especially since similar arguments
will appear in later sections. But first let us introduce some
notation.

Let $(F,j)$ be a closed Riemann surface and ${\bf p}^+$ and ${\bf p}^-$ be disjoint finite sets of punctures of $F$.  Then let
$$u: \dot F= F-\mathbf{p}^+-\mathbf{p}^- \to \R\times N,$$
be a $(j,J_0)$-holomorphic map from $\boldsymbol{\gamma}=\prod_i \gamma_i^{m_i}$ to $\boldsymbol{\gamma}'=\prod_i \gamma_i^{n_i}$. Here the punctures of ${\bf p}^\pm$ are asymptotic to the $\pm$ ends of $u$.  The positive ends of $u$
partition $m_i$ into $(m_{i1},m_{i2},\dots)$ and the negative ends
of $u$ partition $n_i$ into $(n_{i1},n_{i2},\dots)$.  (We ignore the
partition terms that are zero.)  Pick a trivialization $\tau$
of $TS$ in a neighborhood of all the $\gamma_i$, and let
$\mu_\tau(\gamma_i^k)$ be the usual Conley-Zehnder index of the
$k$-fold cover of $\gamma_i$ with respect to $\tau$. Then we define
the {\em total Conley-Zehnder indices} at the positive and negative
ends of $u$ as follows:
$$\mu^+_\tau(u)= \sum_i \sum_r \mu_\tau(\gamma_i^{m_{ir}}),$$
$$\mu^-_\tau(u)=\sum_i\sum_r \mu_\tau(\gamma_i^{n_{ir}}),$$
and also write $\mu_\tau(u)= \mu^+_\tau(u)-\mu^-_\tau(u)$. The {\em
symmetric Conley-Zehnder index} of Hutchings~\cite{Hu1}, so called
because of its motivation from studying the ``symplectomorphism'' $Sym^k(\hh)$ of
$Sym^k(S)$ induced by $\hh$,\footnote{$Sym^k(\hh)$ is a symplectomorphism away from the (multi)-diagonal.} is defined as:
$$\widetilde\mu_\tau(\boldsymbol{\gamma})=\sum_i\sum_{r=1}^{m_i} \mu_\tau(\gamma_i^r),$$
and does not depend on the choice of $u$ from $\boldsymbol{\gamma}$ to $\boldsymbol{\gamma}'$.
\nom[1l$\mu$]{$\widetilde\mu_\tau(\boldsymbol{\gamma})$}{Symmetric Conley-Zehnder index}
We write
$\widetilde\mu_\tau(u)=\widetilde\mu_\tau(\boldsymbol{\gamma})-\widetilde\mu_\tau(\boldsymbol{\gamma}')$.
We also recall the {\em writhe}
$$w_\tau(u)=w_\tau^+(u)-w_\tau^-(u),$$ where $w^+_\tau(u)$ is the total
writhe of braids $u(\dot F)\cap ( \{s\}\times N)$, $s\gg 0$, viewed
in the union of solid torus neighborhoods of $\gamma_i$ and computed with
respect to the framing $\tau$; and $w^-_\tau(u)$ is defined similarly.

The key ingredient in establishing genus bounds is the relative
adjunction formula from \cite[Equation~(18)]{Hu1} for simple curves
$u$ with a finite number of singularities and no connector
components:
\begin{equation} \label{eqn: relative adjunction formula for ECH}
c_1(u^*TS,\tau) = \chi(\dot F) + w_\tau(u) +Q_\tau(u) -2\delta(u),
\end{equation}
where $Q_\tau(u)$ is the relative intersection pairing with respect to $\tau$ and
$\delta(u)$ is a nonnegative integer which is a count of the
singularities.  In particular, $\delta(u)=0$ if and only if $u$ is
an embedding (see \cite{M1,MW}). Together with the writhe bounds
\begin{align} \label{eqn: writhe bound}
w_\tau^+(u) &\leq  \widetilde\mu_\tau(\boldsymbol{\gamma})-\mu_\tau^+(u),\\
\notag w_\tau^-(u) &\geq  \widetilde\mu_\tau(\boldsymbol{\gamma}')-\mu_\tau^-(u),
\end{align}
from \cite[Lemma~4.20]{Hu2}, we obtain:
\begin{equation}
\chi(\dot F)\geq c_1(u^*TS,\tau)+ \mu_\tau(u)-\widetilde\mu_\tau(u) -Q_\tau(u)+2\delta(u).
\end{equation}
(See \cite[Theorem~10.1]{Hu1}.) Since all of the terms on the
right-hand side are either homological quantities or depend on the
data near $\boldsymbol{\gamma}$ and $\boldsymbol{\gamma}'$, we have a lower bound on
$\chi(\dot F)$, which implies an upper bound on the genus of $\dot
F$.

\subsection{Transversality}

Let $\mathcal{J}_{W', \varsigma}$ be the space of almost complex structures $J_\varsigma$ on $W'= \R\times N$ in the class $C^\infty$ which are adapted to $(\alpha_\varsigma,\omega)$.

\begin{defn}
An almost complex structure $J_\varsigma\in \mathcal{J}_{W', \varsigma}$ is {\em $j$-regular} if, for all orbit sets $\boldsymbol{\gamma}$ and $\boldsymbol{\gamma}'$ of $R_\varsigma$ which intersect $S\times\{0\}$ at most $j$ times, the moduli space $\mathcal{M}^s_{J_\varsigma}(\boldsymbol{\gamma},\boldsymbol{\gamma}')$ is transversely cut out.
\end{defn}

Let $\mathcal{J}_{W', \varsigma}^{reg,j}\subset \mathcal{J}_{W', \varsigma}$
be the subset of $j$-regular $J_\varsigma$.
In order to simplify the notation, from now on we will write $\mathcal{J}_{W'}$ for $\mathcal{J}_{W', 0}$ and $\mathcal{J}^{reg,j}_{W'}$ for $\mathcal{J}^{reg,j}_{W', 0}$. The following lemma states that $\mathcal{J}^{reg,j}_{W'}\subset\mathcal{J}_{W'}$ is dense.

\begin{lemma}[Transversality] \label{lemma: PFH transversality}
A generic $J_0 \in \mathcal{J}_{W'}$ is $j$-regular.
\end{lemma}

\begin{proof}
This follows from \cite[Lemma~9.12(b)]{Hu1}, which states that a generic $J_0\in \mathcal{J}_{W'}$ is regular away from holomorphic curves which have a fiber $\{(s,t)\}\times S$ as an irreducible component.  (Observe that the fibers are holomorphic for any $J_0 \in \mathcal{J}_{W'}$.) In our case, the fibers are not closed and cannot occur as irreducible components of curves in $\mathcal{M}^s_{J_0}(\boldsymbol{\gamma},\boldsymbol{\gamma}')$.
\end{proof}

\subsection{The equivalence of certain ECH and PFH groups}

In this section we prove the following theorem:

\begin{thm}
\label{thm: equiv of ECH and PFH} Given $j>0$, there exist $J_0\in \mathcal{J}_{W'}^{reg,j}$ and $\varsigma_0=\varsigma_0(j,J_0)>0$ such that there are isomorphisms of chain complexes
$$PFC_j(N,\alpha_0,\omega,J_0)\simeq ECC_j(N,\alpha_\varsigma,J_\varsigma),$$
$$PFC_j^\flat(N,\alpha_0,\omega,J_0)\simeq ECC_j^\flat(N,\alpha_\varsigma,J_\varsigma),$$
for all $0< \varsigma\leq \varsigma_0$. Here $J_\varsigma\in \mathcal{J}_{W', \varsigma}^{reg,j}$ is sufficiently close to $J_0$. Similar isomorphisms hold with twisted coefficients.
\end{thm}

\begin{proof}
We will prove the first equivalence, leaving the second to the reader.

Since there is a one-to-one correspondence between the generators of the chain groups $PFC_j(N,\alpha_0,\omega)$ and $ECC_j(N,\alpha_\varsigma)$, we have
$$PFC_j(N,\alpha_0,\omega)\simeq ECC_j(N,\alpha_\varsigma)$$
as $\F$-vector spaces, but not necessarily as chain complexes. In other words, we may view any orbit set $\boldsymbol{\gamma}$ for $R_\varsigma$ as an orbit set of any other $R_{\varsigma'}$.

Let $J_0$ be an almost complex structure in $\mathcal{J}^{reg,j}_{W'}$. By Lemma~\ref{lemma: PFH transversality}, $\mathcal{J}^{reg,j}_{W'}$ is a dense subset of $\mathcal{J}_0$, and in particular is nonempty. The moduli spaces $\mathcal{M}^{nc}_{J_0}(\boldsymbol{\gamma},\boldsymbol{\gamma}',Z)$ of ECH index $1$ and $2$ are transversely cut out since they are simple by the ECH index inequality. Now let $J_\varsigma$, $\varsigma\in[0,1]$, be a smooth family of $(\alpha_\varsigma,\omega)$-adapted almost complex structures which extend $J_0$.

We claim that the ECH index $1$ moduli spaces $\mathcal{M}^{nc}_{J_\varsigma}(\boldsymbol{\gamma},\boldsymbol{\gamma}',Z)$ are transversely cut out and diffeomorphic to $\mathcal{M}^{nc}_{J_0}(\boldsymbol{\gamma},\boldsymbol{\gamma}',Z)$, when $\varsigma>0$ is sufficiently small. Indeed, if $u_\varsigma\in \mathcal{M}^{nc}_{J_\varsigma}(\boldsymbol{\gamma},\boldsymbol{\gamma}',Z)$ is sufficiently close to $\mathcal{M}^{nc}_{J_0}(\boldsymbol{\gamma},\boldsymbol{\gamma}',Z)$, then the moduli space is regular at $u_\varsigma$. Hence it suffices to prove that, if $\varsigma>0$ is sufficiently small, then every $u_\varsigma\in \mathcal{M}^{nc}_{J_\varsigma}(\boldsymbol{\gamma},\boldsymbol{\gamma}',Z)$ is sufficiently close to $\mathcal{M}^{nc}_{J_0}(\boldsymbol{\gamma},\boldsymbol{\gamma}',Z)$. Indeed, this follows from the compactness argument from Section~\ref{subsection: compactness PFH case}. Let $u_i \in \mathcal{M}^{nc}_{J_{\varsigma_i}}(\boldsymbol{\gamma},\boldsymbol{\gamma}',Z)$ be a sequence of ECH index $1$ holomorphic curves with $\varsigma_i\to 0$. By the compactness theorem and incoming/outgoing partition considerations, $u_i$ converges to $u\in \mathcal{M}_{J_0}(\boldsymbol{\gamma},\boldsymbol{\gamma}',Z)$ with $I_{ECH}(u)=1$, after possibly taking a subsequence. In particular, the limit $u$ is not a holomorphic building with multiple levels. If $u$ has connector components $u^0$ over $\boldsymbol{\gamma}_0$, then $u_i$ must also have connector components $u^0_i$ over $\boldsymbol{\gamma}_0$, a contradiction. Hence $u\in \mathcal{M}^{nc}_{J_0}(\boldsymbol{\gamma},\boldsymbol{\gamma}',Z)$, which proves the claim.

Since the chain groups are isomorphic as vector spaces and the differentials agree for sufficiently small $\varsigma>0$, the theorem follows.
\end{proof}

\subsection{Remarks about Morse-Bott theory}\label{subsection: Morse-Bott}

There is an equivalent approach to the definition of the groups $PFC_j(N, \alpha_0, \omega)$  which does not use the Morse-Bott language. A similar construction is possible also for the groups  $PFC^\flat_j(N, \alpha_0, \omega)$, but will be left to the reader.

We add a small collar $[0, \eta] \times \partial S$ to $S$ and extend the monodromy
such that $\hh$ restricts to $(y, \theta) \mapsto (y, \theta -y)$ on $[0, \eta] \times \partial S$.
(See Lemma~\ref{lemma: change from h to h_0}). We denote by $(S^+, \hh^+)$ the extensions. The mapping torus of $(S^+, \hh^+)$ is
a manifold $N^+$ such that $N \subset N^+$. The Hamiltonian structure $(\alpha_0, \omega)$ can
be extended in an obvious way from $N$ to $N^+$, and the extension will still be denoted by
$(\alpha_0, \omega)$. The groups $PFC_j(N^+, \alpha_0, \omega)$ are defined because $\partial N^+$ is foliated by orbits of the Hamiltonian vector field. If $\eta$ is small enough, all Hamiltonian orbits in $N^+-N$ intersect a fibre more than $j$ times, and therefore $PFC_j(N^+, \alpha_0, \omega) \cong PFC_j(N, \alpha_0, \omega)$ and $PFC^\flat_j(N^+, \alpha_0, \omega) \cong PFC^\flat_j(N, \alpha_0, \omega)$.

\begin{rmk}
In the above construction, $\eta$ depends on $j$, and therefore the extension
is not compatible with direct limits. However this is not a problem, because in this
article we will always work with $PFC_{2g}(N, \alpha_0, \omega)$, where $g$ is the genus of $S$.
\end{rmk}

Given a function $f : N^+ \to \R$ and a real number $\lambda>0$, we define a new stable Hamiltonian structure
$$\alpha_0' = \alpha_0, \quad \omega' = \omega + \lambda df \wedge \alpha_0.$$
Note that $(\alpha_0', \omega')$ is stable Hamiltonian since $d \alpha_0=0$. If $R_0$ is the Hamiltonian vector field of $(\alpha_0, \omega)$ and $R_0'$ of $ (\alpha_0', \omega')$, then $R_0'= R_0 + \lambda X$, where $X$ is a vector field satisfying the equations
$$\left \{ \begin{array}{l}
\alpha_0(X)=0, \\
\iota_X \omega = df - df(R_0)\alpha_0.
\end{array} \right. $$
We construct a function $f$ on $N^+$ as follows: We fix a Morse function on $\partial S$ with a unique minimum and a unique maximum, pull it back to $\partial N$, and extend it to a nonpositive function $f$ supported on a small neighborhood of $\partial N$ using a bump function centered at $\partial N$.  We choose $f$ to be supported away from the (finitely many and nondegenerate) orbits of $R_0$ in $int(N)$ with $\mathcal{F}\leq j$, so that those orbits will also be orbits of $R_0'$.

On $\partial N$, the vector field $R_0'$ has a pair of nondegenerate orbits $e$ and $h$, where $e$ corresponds to the minimum of the Morse function on $\partial S$, and $h$ to the maximum. We will choose $\lambda$ small enough so that the perturbation creates no new orbits with $\mathcal{F}\leq j$. Since $\partial N^+$ is foliated by the flow of  $R_0'$, we can define the periodic Floer homology group $PFC_j(N^+, \alpha_0', \omega')$ and, as explained in \cite[Section~4]{CGH2}, there is an isomorphism $PFC_j(N^+, \alpha_0, \omega) \cong PFC_{2g}(N^+, \alpha_0', \omega')$. (\cite[Section~4]{CGH2} is mostly concerned with contact structures, but there is no problem in applying those arguments here.) The isomorphism holds already on the level of complexes because there is no direct limit involved in its definition. This is a consequence of the topological constraint $\mathcal{F}\leq j$ on the Hamiltonian orbits, which implies that only finitely many of them enter the definition of the groups $PFH_j$.

\section{A variation of $\widehat{HF}(M)$ adapted to open book decompositions}
\label{section: variation of HF adapted to OB}

In this section we recall the cylindrical reformulation of Heegaard Floer homology. This reformulation was suggested by Eliashberg and worked out in detail by Lipshitz~\cite{Li}.  The discussion is slightly different from that of \cite{Li} in that we introduce an ECH-type index $I_{HF}$ and define the Heegaard Floer groups in terms of $I_{HF}$.  We then use the work of \cite{HKM1} to restrict the Heegaard Floer data to the page $S$ of an open book decomposition.

\subsection{Heegaard data}

A {\em pointed Heegaard diagram} is a quadruple $(\Sigma,\boldsymbol{\alpha}, \boldsymbol{\beta}, z)$ which consists of the following:
\begin{itemize}
\item a closed oriented surface $\Sigma$ of genus $k$;
\item two collections $\boldsymbol{\alpha} =\{\alpha_1 ,\dots, \alpha_k \}$ and
$\boldsymbol{\beta} =\{\beta_1 ,\dots,\beta_k \}$ of $k$ pairwise disjoint simple closed curves in $\Sigma$; and
\item a point $z\in \Sigma  -\boldsymbol{\alpha} -\boldsymbol{\beta}$;
\end{itemize}
where each of $\boldsymbol{\alpha}$ and $\boldsymbol{\beta}$ forms a basis of $H_1(\Sigma;\Z)$ and $\boldsymbol{\alpha}$ and $\boldsymbol{\beta}$ intersect transversely in $\Sigma$.

If $M$ is a closed oriented $3$-manifold, then $(\Sigma,\boldsymbol{\alpha},\boldsymbol{\beta},z)$
\nom[1r$\sigma$]{$(\Sigma,\boldsymbol{\alpha},\boldsymbol{\beta},z)$}{Pointed Heegaard diagram for $M$}
is a {\em pointed Heegaard diagram for $M$} if $\Sigma$ decomposes $M$ into two handlebodies, i.e., $M=H_{\boldsymbol{\alpha}} \cup_\Sigma H_{\boldsymbol{\beta}}$ and $\Sigma=\bdry H_{\boldsymbol{\alpha}}=-\bdry H_{\boldsymbol{\beta}}$, where the $\boldsymbol{\alpha}$-curves (resp.\ $\boldsymbol{\beta}$-curves) bound compression disks in the handlebody $H_{\boldsymbol{\alpha}}$ (resp.\ $H_{\boldsymbol{\beta}}$).

Let $\omega$ be an area form on $\Sigma$. We consider $[0,1] \times \Sigma$ with the stable Hamiltonian structure $(dt,\omega)$, where $t$ is the $[0,1]$-coordinate. The Hamiltonian vector field is $\bdry_t$, and the $2$-plane field $\ker dt$ will be written as $T\Sigma$, at the risk of some confusion.  (If $Y$ is a topological space and $f:Y\to \Sigma$ is a continuous map, then we write $T\Sigma_{Y}$ or $T\Sigma$ for the pullback bundle $f^*T\Sigma$.)

Let $\mathcal{S}_{\boldsymbol{\alpha},\boldsymbol{\beta}}$
\nom[S ]{$\mathcal{S}_{\boldsymbol{\alpha},\boldsymbol{\beta}}$}{Set of generators of $\widehat{CF}(\Sigma,\boldsymbol{\alpha},\boldsymbol{\beta},z)$}
be the set of $k$-tuples of chords $\{[0,1] \times \{y_1\},\dots, [0,1] \times
\{y_k\}\}$ in $[0,1] \times \Sigma$, where there exists a
permutation $\sigma\in \mathfrak{S}_k$ for which $y_i \in \alpha_i
\cap \beta_{\sigma (i)}$, $i=1,\dots,k$. Here the chords
$[0,1]\times\{y_i\}$ are orbits of the Hamiltonian vector field
$\bdry_t$ which connect from $\{0\}\times\beta$ to $\{1\}\times
\alpha$.

\s\n {\em Terminology.} We will often write elements of $\mathcal{S}_{\boldsymbol{\alpha},\boldsymbol{\beta}}$ as ${\bf y}=\{y_1,\dots,y_k\}$ and refer to ${\bf y}$ as a {\em $k$-tuple of intersection points}.  Also, if $l\leq k$, then an {\em $l$-tuple of chords/intersection points} ${\bf y}=\{y_1,\dots,y_l\}$ is a collection of points in $\boldsymbol{\alpha} \cap \boldsymbol{\beta}$ where each $\alpha_i$ is used at most once and each $\beta_i$ is used at most once.

\subsection{Almost complex structures}

Consider the natural projection $\pi_B: X \to B$,
\nom[1p$\pi$1]{$\pi_B: X \to B$ or $W\to B$}{Symplectic fibration with fiber $\Sigma$ or $S$ used in the definition the Heegaard Floer complex}
where $X =\R \times [0,1] \times \Sigma$ and $B=\R\times[0,1]$.
\nom[B]{$B$}{Alternate notation for $\R\times[0,1]$}
\nom[X ]{$X$}{$\R\times[0,1]\times\Sigma$}
We also write
$\pi_{\R}$, $\pi_{[0,1]\times\Sigma}$ and $\pi_\Sigma$ for the
natural projections of $X$ onto $\R$, $[0,1]\times\Sigma$ and
$\Sigma$. Let $(s,t)$ be the coordinates on the base $B=\R \times
[0,1]$. We then define the symplectic form
$$\Omega_X =ds\wedge dt+\omega$$ on $X$.
\nom[1yy$\omega$1]{$\Omega_X$}{Symplectic form on $X$}
The
submanifolds $L_{\boldsymbol{\alpha}} =\R \times \{ 1\} \times \boldsymbol{\alpha}$
and $L_{\boldsymbol{\beta}} =\R \times \{ 0\} \times \boldsymbol{\beta}$
\nom[L0]{$L_{\boldsymbol{\alpha}}$, $L_{\boldsymbol{\beta}}$}{Lagrangian submanifolds used in the definition of the Heegaard Floer complex $\widehat{CF}(\Sigma,\bs{\alpha},\bs{\beta},z)$}
are Lagrangian submanifolds of the symplectic manifold $(X,\Omega_X)$. We denote
by $L_{\alpha_i}$ or $L_{\beta_i}$ their connected components.

\begin{defn}[$\Omega_X$-admissibility]\label{admissible}
An almost complex structure $J$ on $X$ is {\em $\Omega_X$-admissible}
(or simply {\em admissible}) if it satisfies the following:
\begin{enumerate}
\item $J$ is $s$-invariant;
\item $J(\bdry_s)=\bdry_t$ and $J(T\Sigma)=T\Sigma$;
\item $J$ is tamed by the symplectic form $\Omega_X$;
\item there is a point $z_i$ in each component of $\Sigma -\boldsymbol{\alpha} -
\boldsymbol{\beta}$ such that $J$ is a product complex structure $j_B \times j_\Sigma $
in a small neighborhood of $\R \times [0,1] \times \{ z_i \}$ in
$X$; some $z_i$ coincides with the basepoint $z$.
\end{enumerate}
\end{defn}

\begin{rmk}
The fibers $\{(s,t)\}\times \Sigma$ of an admissible $J$ on $X$ are
holomorphic and the projection $\pi_B$ is $(J,j_B)$-holomorphic,
where $j_B$ is the standard complex structure on $B=\R\times[0,1]$.
\end{rmk}

We write $\mathcal{J}_X$
\nom[Jset]{$\mathcal{J}_X$}{Space of $C^\infty$-smooth admissible almost complex structures on $X$}
for the space of $C^\infty$-smooth $\Omega_X$-admissible almost complex structures $J$ on $X$.

\subsection{Holomorphic curves and moduli spaces}
\label{subsection: HF holomorphic curves and moduli spaces}

Let $(F,j)$ be a compact Riemann surface, possibly disconnected, with two sets of punctures ${\bf q}^+=\{q_1^+,\dots,q_k^+ \}$ and ${\bf q}^-=\{q_1^-,\dots,q_k^-\}$ on $\bdry F$, such that (i) each component of $F$ has nonempty boundary, (ii) on each boundary component there is at least one puncture from each of $\mathbf{q}^+$ and $\mathbf{q}^-$, and (iii) the punctures on $\mathbf{q}^+$ and $\mathbf{q}^-$ alternate around each boundary component. We write $\dot{F}= F-{\bf q}^+-{\bf q}^-$ and $\bdry \dot F= \bdry F- {\bf q}^+-{\bf q}^-$.

\begin{defn}\label{HF-curve}
Let $J\in \mathcal{J}_X$. A {\em degree $l\leq k$ multisection of $(W,J)$} is a holomorphic map
$$u :(\dot{F} ,j) \rightarrow (X,J)$$
which is a degree $l$ multisection of $\pi_B: X\to B=\R\times[0,1]$ and which {\em additionally satisfies the following:}
\begin{itemize}
\item[(1)] $u(\partial \dot{F} )\subset L_{\boldsymbol{\alpha}} \cup L_{\boldsymbol{\beta}}$;
\item[(2)] for each $i\in \{ 1,...,k\}$, $u^{-1} (L_{\alpha_i} )$ (resp.\ $u^{-1} (L_{\beta_i} )$) consists of at most one component of $\partial \dot{F}$, which we call $\alpha_i^*$ (resp.\ $\beta_i^*$);
\item[(3)] $\displaystyle {\lim_{w\rightarrow q_i^-} \pi_{\R}\circ  u(w)=-\infty}$ and $\displaystyle{\lim_{w\rightarrow q_i^+} \pi_{\R}\circ  u(w)=+\infty}$;
\item[(4)] the energy of $u$ (see Definition~\ref{defn: energy of Lipshitz curve} below) is finite.
\end{itemize}
A {\em Heegaard Floer curve} (or {\em HF curve}) is a degree $k$ multisection of $(X,J)$.
\end{defn}

By the compactness theorem of \cite{BEHWZ} (adapted to the Lagrangian case), a holomorphic curve $u$ satisfying (1), (2) and (4) converges to cylinders over Reeb chords as $s\to \pm\infty$. By the work of Abbas~\cite{Abb}, an HF curve $u$ converges exponentially to cylinders over Reeb chords near the ends. Components of $u$ may map to $\R\times [0,1]\times \{y_i\}$; such components will be called {\em trivial strips}.

\begin{defn} \label{defn: energy of Lipshitz curve}
The {\em energy} of $u$ is the quantity
\nom[E1]{$E(u)$}{Hofer energy of the holomorphic curve $u$}
\begin{equation}
\label{eqn: energy of Lipshitz curve}
E(u )=\int_{\dot F}  u^* \omega + \sup_{\phi \in \mathcal{C}} \int_{\dot F} u^* d(\phi(s)dt),
\end{equation}
where $\mathcal{C}$ is the set of nondecreasing smooth functions $\phi :\R \rightarrow [0,1]$.
\end{defn}

We now define some moduli spaces of HF curves with respect to $J\in \mathcal{J}_X$. Let ${\bf y} =\{y_1,\dots,y_k\}$ and ${\bf y'}=\{y'_1,\dots,y'_k\}$ be $k$-tuples of $\boldsymbol{\alpha}\cap\boldsymbol{\beta}$. Let $\mathcal{M}^X_J ({\bf y},{\bf y'})$
\nom[M1]{$\mathcal{M}^{X}_J ({\bf y},{\bf y'})$}{Moduli space of Heegaard Floer curves in $(X,J)$ from ${\bf y}$ to ${\bf y'}$}
be the moduli space of HF curves $u$ which are asymptotic to $\R\times [0,1] \times \{y_i\}$ near $q_i^+$ and to $\R\times [0,1] \times \{ y_i' \}$ near $q_i^-$.\footnote{In \cite{Li}, $W$ has the
opposite orientation, $\mathbf{y}$ is at $-\infty$, and $\mathbf{y'}$ is at $+ \infty$. The moduli spaces, however, are diffeomorphic.} Such a curve $u$ is said to be an {\em HF curve from $\mathbf{y}$ to $\mathbf{y'}$}. Also let $\widehat{\mathcal{M}}^X_J (\mathbf{y},\mathbf{y'})\subset
\mathcal{M}^X_J(\mathbf{y},\mathbf{y'})$
\nom[M2]{$\widehat{\mathcal{M}}^{X}_J (\mathbf{y},\mathbf{y'})$}{Moduli space of Heegaard Floer curves in $(X,J)$ from ${\bf y}$ to ${\bf y'}$ that do not cross the basepoint}
be the subset consisting of curves $u$ which additionally satisfy $\pi_{\Sigma}\circ u (\dot{F} ) \cap \{ z \} = \varnothing$.

Let $\check{X}=[-1,1]\times [0,1]\times \Sigma$ be the compactification of $W=\R\times[0,1]\times \Sigma$, obtained by attaching $[0,1]\times \Sigma$ at the positive and negative ends, and let $\check{L}_{\boldsymbol{\alpha}}=[-1,1]\times\{1\}\times \boldsymbol{\alpha}$ and $\check{L}_{\boldsymbol{\beta}}=[-1,1]\times\{0\}\times \boldsymbol{\beta}$ be the compactifications of $L_{\boldsymbol{\alpha}}$ and $L_{\boldsymbol{\beta}}$. We then define $Z_{{\bf y},{\bf y'}} \subset \check{X}$ as the subset
\begin{equation}
Z_{\mathbf{y}, \mathbf{y'}} = \check{L}_{\boldsymbol{\alpha}}\cup\check{L}_{\boldsymbol{\beta}} \cup (\{ 1 \} \times [0,1] \times \mathbf{y}) \cup (\{ -1 \} \times [0,1] \times \mathbf{y'}).
\end{equation}
Similarly define
\begin{equation} \label{eqn: Z alpha beta}
Z_{\boldsymbol{\alpha},\boldsymbol{\beta}} = \check{L}_{\boldsymbol{\alpha}}\cup \check{L}_{\boldsymbol{\beta}}\cup (\{-1,1\}\times[0,1]\times(\boldsymbol{\alpha}\cap\boldsymbol{\beta})).
\end{equation}
The exponential decay of HF curves in $ \R \times [0,1] \times \Sigma$ implies that an HF curve $u: \dot{F} \to W$ from $\mathbf{y}$ to $\mathbf{y'}$  can be compactified to a continuous map
$$\check{u}: (\check{F}, \partial \check{F}) \to (\check{X}, Z_{\mathbf{y}, \mathbf{y'}}).$$
Here $\check{F}$ is obtained from $\dot F$ by performing a real blow-up at its boundary punctures.

We define $H_2(X,\mathbf{y},\mathbf{y'})$ as
\nom[H1]{$H_2(\mathbf{y},\mathbf{y'})$}{Homology classes of maps in HF (denoted $\pi_2(\mathbf{y},\mathbf{y'})$ elsewhere in the literature)}
the set of {\em homology} classes of continuous maps $u: \dot F\to X$ which satisfy (1), (2) and (3) of Definition~\ref{HF-curve} and are positively asymptotic to $[0,1]\times {\bf y}$ and negatively asymptotic to $[0,1]\times{\bf y'}$; here two maps $u_1$ and $u_2$ are equivalent in $H_2(X,\mathbf{y},\mathbf{y'})$ if their compactifications $\check u_1$ and $\check u_2$ are homologous in $H_2(\check{X},Z_{{\bf y},{\bf y'}})$. To any HF curve from $\mathbf{y}$ to $\mathbf{y'}$ we can then associate a class in $H_2(X,\mathbf{y},\mathbf{y'})$. If we consider moduli spaces of HF curves $u$ in the homology class $A \in H_2(X,{\bf y},{\bf y'})$, we will write $\mathcal{M}^X_J({\bf y},{\bf y'},A)$ or $\widehat{\mathcal{M}}^X_J({\bf y},{\bf y'},A)$.

\begin{rmk}
The sets $H_2(X,{\bf y},{\bf y'})$ are denoted $\pi_2({\bf y},{\bf y'})$ in the Heegaard Floer literature. We opted for this change of notation because to be consistent with the notation we introduced for ECH and for the maps to the open-closed cobordisms introduced in Section \ref{section: moduli spaces of multisections}.
\end{rmk}

\subsection{The Fredholm index}
\label{subsection: the Fredholm index W}

In this subsection and the next, we fix $J\in \mathcal{J}_X$.
Now we discuss the Fredholm index of an HF curve $u:\dot{F}\to X$, which is the expected dimension of a neighborhood $\mathcal{U}$ of $u\in \mathcal{M}^X_J({\bf y},{\bf y'})$, modulo reparametrizations of the domain. 
The Fredholm index of $u$ will be denoted by $\op{ind}(u)=\op{ind}_{HF}(u)$.
\nom[ind]{$\op{ind}$}{Fredholm index}

\subsubsection{The Fredholm index, first version}
\label{subsubsection: Fredholm index first version}

We start with Lipshitz's formula \cite[Equation~5]{Li} for the Fredholm index of $u$:
\begin{equation} \label{index-lipshitz}
{\rm ind} (u)= -\chi (F)+k+ \sum_{i=1}^{k} \mu (\alpha_i^*) - \sum_{i=1}^{k} \mu(\beta_i^*),
\end{equation}
where $k$ is the genus of $\Sigma$, $\chi(F)=\chi(\dot F)$ is the Euler characteristic of $F$ or $\dot F$, and $\alpha_i^*$ and $\beta_i^*$ are as in Definition~\ref{HF-curve}.

We now define the Maslov indices $\mu (\alpha_i^*)$ and $\mu (\beta_i^*)$ which appear in Equation~\eqref{index-lipshitz}: Choose a trivialization $\tau_0'$ of $T\Sigma \simeq \C$ on a neighborhood of the points $w\in \boldsymbol{\alpha} \cap \boldsymbol{\beta}$ so that $\R$ corresponds to $T_w \boldsymbol{\beta}$ and $i \R$ corresponds to $T_w \boldsymbol{\alpha}$. Then let $\tau_0$ be a trivialization of $u^* T\Sigma_X$ which coincides with the one already given near ${\bf p}$ and ${\bf q}$ by pulling back $\tau_0'$. Along each component of $\partial \dot{F}$ --- called $\alpha_i^*$ or $\beta_i^*$, depending on whether it is mapped to $\alpha_i$ or to $\beta_i$, and  oriented in the same way as $\bdry F$ for $\alpha_i$ and in the opposite way for $\beta_i$ --- we have a loop of unoriented real lines in $\C$, given by the pullback of $T\alpha_i$ or $T\beta_i$. The Maslov index $\mu(\alpha_i^*)$ (resp.\ $\mu(\beta_i^*)$) is the degree of the loop along $\alpha_i^*$ (resp.\ $\beta_i^*$) with respect to $\tau_0$.

\subsubsection{The Fredholm index, second version}
\label{subsubsection: Fredholm index second version}

For the purposes of computing indices, we replace $X= \R \times
[0,1] \times \Sigma$ by the compactification $\check{X}= [-1,1]
\times [0,1] \times \Sigma$ from Section~\ref{subsection: HF
holomorphic curves and moduli spaces}.

Recall $Z_{\boldsymbol{\alpha}, \boldsymbol{\beta}}\subset \check X$ given by Equation~\eqref{eqn: Z alpha beta}. We define a
trivialization $\tau$ of $T \Sigma_{\check X}$ along
$Z_{\boldsymbol{\alpha},\boldsymbol{\beta}}\subset \check X$ as
follows: First choose a nonsingular tangent vector field along each
component of $\boldsymbol{\alpha}$ and $\boldsymbol{\beta}$. This
induces a trivialization $\tau'$ of $T\Sigma_{[0,1]\times\Sigma}$ on
$(\{ 0 \} \times \boldsymbol{\beta}) \cup (\{ 1 \} \times
\boldsymbol{\alpha})$. We then extend $\tau'$ arbitrarily to $[0,1]
\times (\boldsymbol{\alpha} \cap \boldsymbol{\beta})$. Finally we pull
$\tau'$ back to $Z_{\boldsymbol{\alpha},\boldsymbol{\beta}} \subset
\check{X}$ using the projection $\pi_{[0,1]\times\Sigma}:\check X \to
[0,1]\times\Sigma$ to obtain $\tau$.

Given an HF curve $u: \dot F\to X$, we define its {\em Maslov index}
$\mu_{\tau}(u)$ as follows: Let
$$\check{u}: (\check{F}, \partial \check{F}) \to
(\check{X}, Z_{\boldsymbol{\alpha},\boldsymbol{\beta}})$$
be the compactification of $u$. We then construct a (not necessarily
oriented) real rank one subbundle $\mathcal{L}$ of $\check u^*T\Sigma$
on $\bdry \check F$. The bundle $\mathcal{L}$ is given by
$\check u^*T\boldsymbol{\alpha}$ and $\check u^*T\boldsymbol{\beta}$
along $\bdry \dot{F}$. We extend $\mathcal{L}$ to $\bdry \check{F} -
\bdry \dot F$ by rotating in the counterclockwise direction from $
\check u^*T\boldsymbol{\beta}$ to $\check u^*T\boldsymbol{\alpha}$
by the minimum amount possible. (Assuming orthogonal intersections,
this is a ${\pi\over  2}$-rotation.) Then $\mu_\tau(u)$ is the sum of the Maslov indices
of $\mathcal{L}$ with respect to the trivialization $\tau$, where the
sum is over all the connected components of $\bdry \check F$.

\begin{lemma}
If $u: \dot F\to X$ is an HF curve, then
\begin{equation}
\mu_{\tau}(u) + 2 c_1(u^*T\Sigma, \tau)=\sum_{i=1}^k \mu(\alpha_i^*) - \sum_{i=1}^k \mu(\beta_1^*).
\end{equation}
\end{lemma}

\begin{proof}
By standard Maslov index theory, we have
$$\mu_{\tau} (u) + 2 c_1(u^*T\Sigma, \tau)=\mu_{\tau_0}(u)
+2c_1(u^*T\Sigma,\tau_0),$$ where $\tau_0$ denotes the
trivialization of $u^*T\Sigma$ from Section~\ref{subsubsection:
Fredholm index first version}. We immediately obtain
$c_1(u^*T\Sigma,\tau_0)=0$ since $\tau_0$ is a trivialization on all
of $\dot F$. Hence it suffices to prove that:
\begin{equation} \label{pipi}
\mu_{\tau_0}(u) = \sum_{i=1}^k \mu (\alpha_i^*) - \sum_{i=1}^k\mu(\beta_1^*).
\end{equation}

The difference between the two sides of Equation~\eqref{pipi} is the
total amount of rotation of the real lines introduced at $\boldsymbol{\alpha}
\cap \boldsymbol{\beta}$ in the definition of $\mu_{\tau_0}$: if we go from
$\boldsymbol{\beta}$ to $\boldsymbol{\alpha}$ we rotate by $\frac{\pi}{2}$,
while if we go from $\boldsymbol{\alpha}$ to $\boldsymbol{\beta}$ we rotate
by $- \frac{\pi}{2}$; hence the total amount of rotation is $0$.
\end{proof}

We can now rephrase the Fredholm index as follows:
\begin{equation}\label{eqn: Fredholm index for HF, version 2}
\op{ind}(u)=-\chi(F)+k+\mu_\tau(u)+2c_1(u^*T\Sigma,\tau).
\end{equation}

\begin{rmk}
Since the Maslov index of $\mathcal{L}$ with respect to $\tau$ is an integer along each chord $[0,1]\times \{y_i\}$, it makes sense to write $\mu_\tau(y_i)\in \Z$. If we let
$$\mu_\tau({\bf y})= \sum_{i=1}^k \mu_\tau(y_i),$$
then $\mu_{\tau}(u)=\mu_{\tau}({\bf y})-\mu_{\tau}({\bf y'})$. In particular, $\mu_\tau(u)$ only depends on ${\bf y}$, ${\bf y'}$, and the choice of $\tau$.
\end{rmk}

\subsection{The ECH-type index}

In this subsection we define an ECH-type index $I_{HF}(u)$ and prove an index inequality which is analogous to the ECH index inequality of \cite{Hu1}.

\subsubsection{The relative intersection form}

Let $\tau$ be a trivialization of $T\Sigma_{\check X}$ along
$Z_{\boldsymbol{\alpha},\boldsymbol{\beta}}$ as defined in Section~\ref{subsubsection:
Fredholm index second version}. We will define the notions of a {\em
representative} and a {\em $\tau$-trivial representative} of a
homology class $A\in H_2(X,{\bf y},{\bf y'})$, where ${\bf
y}=\{y_1,\dots,y_k\}$ and ${\bf y'}=\{y_1',\dots,y_k'\}$ are
$k$-tuples in $\mathcal{S}_{\boldsymbol{\alpha},\boldsymbol{\beta}}$.

\begin{defn} \label{defn: representative of A}
An oriented immersed compact surface $$\check C\subset \check
X=[-1,1]\times[0,1]\times\Sigma$$ in the homology class $A \in
H_2(X,\mathbf{y}, \mathbf{y'})$ is an {\em immersed representative
of $A$} if:
\begin{enumerate}
\item $\check C$ is positively transverse to the fibers $\{(s,t)\}\times
\Sigma$ along all of $\bdry \check C$.
\item $(\check\pi_{[0,1]\times\Sigma})|_{\check C}$ is an embedding near
$\bdry \check C\cap (\{-1,1\}\times[0,1]\times\Sigma)$, where
$\check \pi_{[0,1]\times\Sigma}$ is the projection $\check X \to
[0,1]\times\Sigma$.
\end{enumerate}
If $\check C$ is embedded in addition, then $\check C$ is a {\em
representative} of $A$.
\end{defn}

\begin{defn}[$\tau$-trivial representative]
\label{defn: tau-trivial representative HF} A representative $\check
C$ of $A$ is {\em $\tau$-trivial} if, for all sufficiently small
$\varepsilon>0$, $\check C\cap \{s=\pm (1-\varepsilon)\}$ is the
union of single-stranded braids $\zeta_i^\pm$, $i=1,\dots,k$, where
$\zeta_i^+$ (resp.\ $\zeta_i^-$) lies in a tubular neighborhood of
$[0,1]\times \{y_i\}$ (resp.\ $[0,1]\times \{y_i'\}$), is disjoint
from $[0,1]\times \{y_i\}$ (resp.\ $[0,1]\times \{y_i'\}$), and
induces a framing which agrees with $\tau$ along $[0,1]\times
\{y_i\}$ (resp.\ $[0,1]\times \{y_i'\}$).
\end{defn}

Let $A$ be a homology class in $H_2(X,{\bf y},{\bf y'})$.  Then we
define
\begin{equation} \label{eqn: jan 6}
n_{z_j}(A) =  \langle A, [-1,1]\times[0,1]\times\{z_j\}\rangle,
\end{equation}
\nom[n]{$n_{z_j}(A)$}{Intersection number given by Equation~\eqref{eqn: jan 6}}
where $z_j\in \Sigma-\boldsymbol{\alpha}-\boldsymbol{\beta}$
are given in Definition~\ref{admissible} and $\langle\cdot,\cdot\rangle$
is the signed intersection number. We say that $A\in H_2(X,{\bf y},{\bf
y'})$ is {\em positive} if $n_{z_j}(A)\geq 0$ is nonnegative for all
$z_j$.

\begin{lemma}
A positive $A\in H_2(X,{\bf y}, {\bf y'})$ admits a $\tau$-trivial
representative $\check C$.
\end{lemma}

\begin{proof}
Let $A$ be a positive homology class in $H_2(X,{\bf y}, {\bf y'})$.  Then we can glue closures of connected components of $\Sigma-\boldsymbol{\alpha}-\boldsymbol{\beta}$ with multiplicity $n_{z_j}(A)$ as in Rasmussen~\cite[Lemma~9.3]{Ra} (also see \cite[Lemma 4.1]{Li} and its correction \cite[Lemma~4.1']{Li2}) to construct a smooth map $u=(u_1,u_2)$, where $u_1: \dot F\to \R\times[0,1]$ is a branched cover with interior branch points and $u_2: \dot F\to \Sigma$ admits an extension $u_2: F\to\Sigma$ such that
\begin{itemize}
\item $u_2(q_i^+)=y_i$ and $u_2(q_i^-)=y_i'$;
\item each component of $\bdry \dot F$ is mapped to some $\alpha_i$ or $\beta_i$ so that each $\alpha_i,\beta_i$, $i=1,\dots,k$, is used exactly once; and
\item $u_2$ is holomorphic on a neighborhood of $\bdry F$ and is locally given by $w\mapsto w^{\ell(q_i^\pm)/2}$ for some positive odd integer $\ell(q_i^\pm)$ near $q_i^\pm$, provided $q_i^\pm$ is not the end of a trivial strip of $u_2$.
\end{itemize}
We extend $u$ to $\check u: \check F\to X$, where $\check F$ is the real blow-up of $F$ given in
Section~\ref{subsection: HF holomorphic curves and moduli spaces}. Condition (1) of
Definition~\ref{defn: representative of A} is immediately satisfied.
We can resolve all the (interior) singularities to make $\check u$
embedded. It is a local exercise to modify $\check u$ in a
neighborhood of $\bdry \check F- \bdry \dot F$ so that $\check u$
becomes $\tau$-trivial.
\end{proof}

\begin{lemma} \label{lemma: no branch points on boundary} If
$u: \dot{F}\to X$ is an HF curve and $\check{C}$ is the image of the compactification $\check{u}:\check F\to \check X$, then the following hold:
\begin{enumerate}
\item $\pi_B\circ u: \dot{F} \to B$ has no branch points along $\bdry B$.
\item $\check u$ is positively transverse to the fibers $\{(s,t)\}\times \Sigma$ along all of $\bdry \check F$ and $\check\pi_{[0,1]\times\Sigma}|_{\check C}$ is an embedding near $\bdry \check C$.
\end{enumerate}
In other words, $\check C$ satisfies all the conditions of a $\tau$-trivial representative for some $\tau$, with the exception of the embeddedness of $\check C$.
\end{lemma}

\begin{proof}
(1) follows from the fact that $\pi_B\circ u$ is a $k$-fold branched
cover of $B$. Let $\H=\{\mbox{Im}(z)\geq 0\}$ be the upper
half-plane and $U\subset\H$ be an open subset which contains $0$. If
$f$ is a holomorphic map $U\to \R\times[0,1]$ which maps $0$ to
$(0,0)$ and $U\cap \bdry \H$ to $\R\times\{0\}$, then it can be
extended to a holomorphic map $f: U\cup
\overline{U}\to\R\times[-1,1]$ by Schwarz reflection, where
$\overline{U}=\{\overline{z}~|~ z\in U\}$. If $df(0)=0$, then $f$ is
locally a composition of $z\mapsto z^l$ for some integer $l>1$ and a
biholomorphism. This contradicts the requirement that $f(\H)$ stay
on one side of $\R\times\{0\}$.

(2) follows from (1), together with the asymptotics of $u$ as $s\to
\pm \infty$.
\end{proof}

We now define the relative intersection form $Q_\tau(A)$, which is analogous to the relative intersection form which appears in the definition of the ECH index $I_{ECH}$, but is easier.

\begin{defn}[Relative intersection form $Q_\tau(A)$] \label{defn of Q}
Let $A\in H_2(X,\mathbf{y}, \mathbf{y'})$ be a positive homology
class and let $\check C$ be a $\tau$-trivial representative of $A$.
Let $\psi$ be a section of the normal bundle $\nu$ to $\check C$ such
that $\psi|_{\bdry \check C}=J\tau$, and let $\check C'$ be a
pushoff of $\check C$ in the direction of $\psi$. Then the {\em
relative intersection form} $Q_\tau(A)$ is given by:
\[ Q_{\tau}(A) = \langle\check C, \check C'\rangle \]
where $\langle \cdot,\cdot\rangle$ denotes the algebraic count of
intersection points.
\end{defn}

Note that, since a representative $\check C$ is positively transverse to the fibers
$\{(s,t)\}\times \Sigma$ along all of $\bdry \check C$, we may take
the normal bundle $\nu$ to $\check C$ to satisfy $\nu|_{\bdry \check
C}=T\Sigma|_{\bdry \check C}$. Also, since $J$ is
$\Omega_X$-admissible, it takes $T\Sigma$ to itself. Hence
$(\tau,J\tau)$ is a trivialization of $\nu|_{\bdry \check C}$.
Although $\tau$ and $J\tau$ are homotopic, we will often use $J\tau$
due to its appearance in the definition of $Q_\tau(A)$.

Let $\tau'$ and $\tau$ be trivializations of $T\Sigma_{\check X}$ along
$Z_{\boldsymbol{\alpha},\boldsymbol{\beta}}$ which differ only on
$\{\pm 1\}\times [0,1]\times (\boldsymbol{\alpha}\cap\boldsymbol{\beta})$.
Let $\op{deg}(\tau,\tau',y_i)$ (resp.\ $\op{deg}(\tau,\tau',y_i')$) be the degree
of $\tau$ with respect to the trivialization $\tau'$ along $[0,1]\times \{y_i\}$
(resp.\ $[0,1]\times \{y_i'\}$), oriented in the $\bdry_t$-direction. We
then have the following:

\begin{lemma}[Change of trivialization]
\label{lemma: change of trivialization}
If $A\in H_2(X,{\bf y},{\bf y'})$ is positive, then
\begin{equation} \label{eqn: change of trivialization}
Q_\tau(A)-Q_{\tau'}(A)= \sum_{i=1}^k\op{deg}(\tau,\tau',y_i) -\sum_{i=1}^k\op{deg}(\tau,\tau',y_i').
\end{equation}
\end{lemma}

\begin{proof}
Let $\check C_{\tau'}$ be a $\tau'$-trivial representative of $A$.
Let $\varepsilon>0$ be small and let
$$\check C_{\tau',0}= \check C_{\tau'}\cap ([-1+\varepsilon,1-\varepsilon]\times[0,1]\times\Sigma).$$
We can extend $\check C_{\tau',0}$ to $\check C_\tau$ in $\check X$ by
gluing disks $D_i$, $D_i'$, $i=1,\dots,k$, corresponding to $y_i$,
$y_i'$, so that $\check C_\tau$ becomes $\tau$-trivial.

Let $\psi$ be a section of the normal bundle to $\check C_\tau $
such that $\psi|_{\bdry \check C_\tau}=J\tau$ and $\psi|_{\bdry
\check C_{\tau',0}}=J\tau'$. (Here we are assuming that $\tau'$ has
been extended to a neighborhood of the $[0,1]\times \{y_i\}$.) Then
$$Q_\tau(A)-Q_{\tau'}(A)= \sum_{i=1}^k \# (\psi|_{D_i})^{-1}(0)
+\sum_{i=1}^k \# (\psi|_{D'_i})^{-1}(0),$$ where $\#$ is a signed
count. A local calculation gives
$$\# (\psi|_{D_i})^{-1}(0)= \op{deg}(\tau,\tau',y_i), \quad
\# (\psi|_{D'_i})^{-1}(0)= -\op{deg}(\tau,\tau',y_i'),$$ which
proves the lemma.
\end{proof}

The following is immediate from the definition of $Q_\tau(A)$.

\begin{lemma}[Additivity]
\label{lemma: additivity of Q} If $A_1\in H_2(X,{\bf y},{\bf y'})$
and $A_2\in H_2(X,{\bf y'},{\bf y''})$ are positive, then
$$Q_\tau(A_1\# A_2)= Q_\tau(A_1)+Q_\tau(A_2),$$ where $A_1\#A_2\in
H_2(X,{\bf y},{\bf y''})$ is obtained from stacking two copies of
$\check X$.
\end{lemma}

\subsubsection{The relative adjunction formula}

In this subsection we prove a relative adjunction formula for
HF curves (Lemma~\ref{lemma: adjunction formula for X}).

If $\check C \subset \check X$ is a properly immersed surface with only transverse double points in its interior, we denote by $\delta(\check C)$ the signed count of the double points of $\check C$.

\begin{lemma}\label{lemma: c1 and Q}
If $\check C$ is an immersed representative of $A \in
H_2(X,\mathbf{y}, \mathbf{y'})$ with only transverse
double points in its interior, then
\[ c_1(\nu, J\tau) = Q_{\tau}(A)  - 2 \delta(\check C). \]
\end{lemma}

\begin{proof}
Let us first assume that $\check C$ is embedded. If $\check C$ is
$\tau$-trivial, then there is a section $\psi$ of the normal bundle
to $\check C$ such that $\psi|_{\partial \check C} = J\tau$, and
\[ Q_{\tau}(A) = \# \psi^{-1}(0) = c_1(\nu, J\tau). \]
Next let $\tau'$ and $\tau$ be trivializations of $T\Sigma_{\check X}$ along
$Z_{\boldsymbol{\alpha},\boldsymbol{\beta}}$, which differ only on $\{\pm 1\}\times
[0,1]\times (\boldsymbol{\alpha}\cap\boldsymbol{\beta})$. By Equation~\eqref{eqn: change of trivialization}, together with an analogous equation for
$c_1(\nu,J\tau)-c_1(\nu,J\tau')$, we have
$$Q_\tau(A)-Q_{\tau'}(A)=c_1(\nu,J\tau)-c_1(\nu,J\tau'),$$
which proves the lemma for embedded $\check C$.

Suppose now that $\check C$ has a single positive transverse double
point $d$. (The case of a negative double point is similar.) We
resolve the intersection in the following way: Let ${\mathcal B} \subset
\check{X}$ be a small ball centered at $d$. Then $\check C \cap
\partial {\mathcal B}$ is a Hopf link, and $\check C \cap {\mathcal B}$
is the union of two slice disks for the components which intersect at $d$.
We can construct a new surface $\check C_{sm}$ by replacing the two disks
with a Hopf band connecting the two components of the Hopf link. By definition,
we have $Q_{\tau}(A)= \langle\check C_{sm}, \check C_{sm}'\rangle$.
On the other hand, if $\nu_{sm}$ is the normal bundle to $\check
C_{sm}$, then
\[ c_1(\nu_{sm}, J\tau) = c_1(\nu, J\tau) + 2. \]
This can be seen easily by embedding $\mathcal{B}$ into $S^2 \times S^2$ and
using the properties of the intersection product for closed
$4$-manifolds.

In general, $$c_1(\nu,J\tau)+2\delta(\check C)=
c_1(\nu_{sm},J\tau)=\langle\check C_{sm}, \check C_{sm}'\rangle=
Q_\tau(A),$$ and the lemma follows.
\end{proof}

Let $u: \dot F \to X$ be an HF curve. Then $u$ has no singular points on $\partial \dot F$ because $\pi_B \circ u$ has no branch point on the boundary. By \cite{M1,MW}
(see also \cite[Appendix E]{MS}), there exists a modification $v:\dot F\to X$ of
$u:\dot F\to X$ in a neighborhood of its finitely many singular
points so that $v$ is an immersion with only positive transverse double
points. We define $\delta(u)$ the number of positive double points of $v$, which
depends only on $u$ and not on the modification $v$ used to define it. Then
$\delta(u) \ge 0$ and $\delta(u)=0$ if and only if $u$ is an embedding.
We can now state and prove the relative adjunction formula:

\begin{lemma}[Relative adjunction formula]
\label{lemma: adjunction formula for X} If $u: \dot{F} \to X$ is an
HF curve in the homology class $A \in H_2(X,{\mathbf y}, {\mathbf
y'})$, then
\begin{align}\label{goofy}
c_1(\check u^*T\Sigma, \tau) &= c_1(T\check F, \partial_t)  + Q_{\tau}(A) - 2 \delta(u)\\
\notag &= \chi(F) - k+ Q_{\tau}(A) - 2 \delta(u).
\end{align}
Here $\bdry_t$ is the pullback to $\bdry \check F$ of the trivialization
$\bdry_t$ on $[-1,1]\times [0,1]$.
\end{lemma}

\begin{proof}
Let $v:\dot F\to X$ be the immersion with transverse double points
obtained by modifying $u:\dot F\to X$.
Since the modification is purely local and is away from $\bdry \dot
F$, it follows that $\check u$ and $\check v$ belong to the same
homology class in $H_2(X,\mathbf{y},\mathbf{y'})$ and $c_1(\check
u^*T\Sigma, \tau) = c_1(\check v^*T\Sigma, \tau)$. Hence
Equation~\eqref{goofy} for $u$ is equivalent to Equation~\eqref{goofy}
for $v$, and we may assume without loss of generality that $u$ is immersed
with positive transverse double points.

The vector field $\partial_t$ is a global trivialization of the
complex line bundle $T ([-1,1]\times [0,1])$ over
$\check{X}=[-1,1]\times [0,1]\times \Sigma$. Hence
$$c_1(\check{u}^*T\check{X}, (\tau, \partial_t))=
c_1(\check u^*T\Sigma, \tau).$$ On the other hand,
$$c_1(\check{u}^*T \check{X}, (\tau, \partial_t)) = c_1(T\check F,
\partial_t)+ c_1(\nu, J\tau).$$

The first line of the relative adjunction formula now follows from Lemma~\ref{lemma: c1 and Q}. The equivalence of the first and second lines is a consequence of Claim~\ref{claim: calculation of c_1 of TF}, proved below.
\end{proof}

\begin{claim} \label{claim: calculation of c_1 of TF}
$c_1(T\check F, \partial_t) =\chi(F) - k$.
\end{claim}

\begin{proof}
Let $\tau_{\bdry F}$ be the trivialization of $TF|_{\bdry F}$ which
is given by an oriented nonsingular vector field tangent to
$\partial F$. We then have
$$c_1(T\check F,\partial_t) = \chi(F) + \deg(\partial_t,\tau_{\bdry F}),$$
where $\deg(\partial_t,\tau_{\bdry F})$ is the degree of
$\partial_t$ with respect to $\tau_{\bdry F}$. By an easy direct
calculation we obtain $\deg(\bdry_t,\tau_{\bdry F})=-k$.
\end{proof}

\subsubsection{The index $I_{HF}$ and the index inequality}

We are now ready to define the ECH-type index $I_{HF}$ and prove the ECH-type index inequality (Theorem~\ref{thm: index inequality for HF}).

\begin{defn}[ECH-type index] \label{defn: ECH-type index for HF}
Let $A \in H_2(X,\mathbf{y},\mathbf{y'})$ be a positive homology class. Then the {\em ECH-type index} $I_{HF}$ of $A$ is given as follows:
\nom[IHF]{$I_{HF}$}{ECH-type index for Heegaard Floer homology}
\begin{equation} \label{eqn: ECH-type index for HF}
I_{HF}(A)=  c_1(T\Sigma|_A, \tau) + Q_{\tau}(A)+ \mu_{\tau}({\bf y})-\mu_\tau({\bf y'}).
\end{equation}
\end{defn}

We observe that $I_{HF}(A)$ does not depend on the choice of $\tau$: Suppose $\tau$ and $\tau'$ differ only at $y_i$ and $\op{deg}(\tau,\tau',y_i)=1$. Then we compute (i) $Q_\tau(A)=Q_{\tau'}(A)+1$ by Lemma~\ref{lemma: change of trivialization}, (ii) $\mu_{\tau}({\bf y})=\mu_{\tau'}({\bf y}) -2$
and (iii) $c_1(T\Sigma|_A,\tau)= c_1(T\Sigma|_A,\tau')+1$.

The index $I_{HF}$ satisfies the following additivity property under the stacking
operation $\#: H_2(X,{\bf y},{\bf y'}) \times H_2(X,{\bf y'},{\bf y''}) \to H_2(X,{\bf y},
{\bf y''})$.

\begin{lemma}[Additivity of $I_{HF}$]
If $A_1\in H_2(X,{\bf y},{\bf y'})$ and $A_2\in H_2(X,{\bf y'},{\bf y''})$ are positive, then
$$I_{HF}(A_1\# A_2)= I_{HF}(A_1)+I_{HF}(A_2).$$
\end{lemma}

\begin{proof}
Each of the terms $c_1(T\Sigma|_A, \tau)$, $Q_\tau(A)$, and $\mu_{\tau}({\bf y})-\mu_\tau({\bf y'})$ in the definition of $I_{HF}(A)$ is additive under stacking; see Lemma~\ref{lemma: additivity of Q} for the additivity of $Q_\tau(A)$.
\end{proof}

The following index inequality is analogous to (but much easier than) the ECH index inequality, due to Hutchings~\cite[Theorem~1.7]{Hu1}. We remark that $u$ is required to be simply-covered in the statement of the usual ECH index inequality. This is automatically satisfied for HF-curves.

\begin{thm}[ECH-type index inequality] \label{thm: index inequality for HF}
Let $u: \dot{F} \to X$ be an HF curve in the class $A \in H_2(X,\mathbf{y},\mathbf{y'})$. Then
\begin{equation}\label{eqn: index equality for HF}
\op{ind}(u)+2\delta(u)=I_{HF}(A),
\end{equation}
where $\delta(u)\geq 0$ is an integer count of the singularities. Hence
\begin{equation}\label{eqn: index inequality for HF}
{\rm ind}(u) \le I_{HF}(A),
\end{equation}
with equality if and only if $u$ is an embedding.
\end{thm}
Note that $u$ has no boundary singularities because $\pi_B \circ u$ has no branch points at
the boundary.
\begin{proof}
We calculate:
\begin{align*}
\op{ind}(u) &= -\chi(F)+k+\mu_\tau({\bf y})-\mu_\tau({\bf y'})+2c_1(\check u^*T\Sigma,\tau)\\
&= c_1(\check u^*T\Sigma,\tau) + Q_\tau(A) +\mu_\tau({\bf y})-\mu_\tau({\bf y'}) -2\delta(u).
\end{align*}
The first line is Equation~\eqref{eqn: Fredholm index for HF, version 2}. The equivalence of the first and second lines follows from the
relative adjunction formula. Hence
$${\rm ind}(u) +2\delta(u) = I_{HF}(A).$$
The index inequality~\eqref{eqn: index inequality for HF} follows immediately.
\end{proof}

\subsection{Compactness}

We now discuss the requisite compactness issues. The key notion is that of {\em weak admissibility} from \cite[Definition 4.10]{OSz1}, which is analogous to the vanishing of the flux in the PFH situation (see Section~\ref{subsection: flux}). A pointed Heegaard diagram $(\Sigma,\boldsymbol{\alpha},\boldsymbol{\beta},z)$ is {\em weakly admissible} if, for every Spin$^c$-structure $\mathfrak{s}$ and nontrivial periodic domain $\mathcal{Q}$ which satisfies $\langle c_1(\mathfrak{s}), \mathcal{Q} \rangle=0$, there exist $j_1$ and $j_2$ for which $n_{z_{j_1}}(\mathcal{Q})>0$ and $n_{z_{j_2}}(\mathcal{Q})<0$. (The points $z_{j_1}$ and $z_{j_2}$ are the points from Definition \ref{admissible}(4).) Equivalently, by \cite[Lemma~4.12]{OSz1}, $(\Sigma, \boldsymbol{\alpha}, \boldsymbol{\beta}, z)$ is weakly admissible if and only if there is an area form $\omega$ on $\Sigma$ such that each periodic domain has total signed $\omega$-area zero.

Let $N>0$ be a fixed constant. We consider the subset $H_2^N(X,{\bf
y},{\bf y'})$ consisting of homology classes of $H_2(X,{\bf y},{\bf
y'})$ which intersect $[-1,1]\times[0,1]\times \{z\}$ at most $N$
times.  (This is sufficient for $\widehat{CF}$ and $CF^+$, defined
in Section~\ref{subsection: HF def}.) The difference of two homology classes
$A_1,A_2\in H_2^0(X,{\bf y},{\bf y'})$ is a periodic domain
$\mathcal{Q}$ and has zero $\omega$-area. This implies that the
$\omega$-areas of any two $A_1,A_2\in H_2^N(X,{\bf y},{\bf y'})$
differ by $i\cdot\omega(\Sigma)$ where $0\leq i\leq N$. Let
$\phi_1,\dots,\phi_r$ be the connected components of
$\Sigma-\boldsymbol{\alpha}-\boldsymbol{\beta}$ with $z_j\in \phi_j$. If $A$ is represented by a holomorphic curve,
then the projection of $A$ to $\Sigma$ can be written as $\sum_j
n_{z_j}(A)\phi_j$ with $n_{z_j}(A)\geq 0$. Since each $\phi_i$ has finite
area, there must only be a finite number of homology classes $A\in
H_2^N(X,{\bf y}, {\bf y'})$ for which the moduli space
$\mathcal{M}_J^X({\bf y},{\bf y'},A)$ is nonempty.

We now prove the existence of a compactification of
$\mathcal{M}^X_J({\bf y},{\bf y'},A)/\R$. It suffices to show that if
$u: \dot F\to X$ is an element of $\mathcal{M}_J^X({\bf y},{\bf
y'},A)$, then the genus of $\dot F$ is bounded as long as $A$ is
fixed. This will be carried out in Lemma~\ref{lemma: HF genus bound
given fixed homology class}. Once we have a genus bound, the SFT
compactness theorem from ~\cite{BEHWZ} can be applied to give a
compactification of $\mathcal{M}_J^X({\bf y},{\bf y'},A)/\R$.

\begin{lemma}
\label{lemma: HF genus bound given fixed homology class} There is an
upper bound on the genus of a holomorphic curve $u:\dot F\to X$ in a
fixed homology class $A\in H_2(X,{\bf y},{\bf y'})$.
\end{lemma}

\begin{proof}
The proof is analogous to the proof in the PFH case. In view of the
relative adjunction formula (Lemma~\ref{lemma: adjunction formula
for X}) and the nonnegativity of $\delta(u)$, we have
\begin{equation}
\chi(\dot F)\geq c_1(\check u^*T\Sigma,\tau)+k-Q_\tau(A).
\end{equation}
The lemma follows by observing that the terms on the right-hand side
depend only on the homology class $A$.
\end{proof}

\subsection{Transversality}

\begin{defn}
An almost complex structure $J\in\mathcal{J}_X$ is {\em regular} if the moduli spaces $\mathcal{M}^X_J({\bf y}, {\bf y'})$ are transversely cut out for all ${\bf y}, {\bf y'}\in \mathcal{S}_{\alpha,\beta}$.
\end{defn}

Note that if $u$ is an HF curve, then it does not have any closed irreducible components by definition.  In particular, $u$ cannot have any fibers $\{(s,t)\}\times \Sigma$ as irreducible components.

We write $\mathcal{J}_X^{reg}\subset \mathcal{J}_X$
\nom[Jset1]{$\mathcal{J}_\star^*$}{Subspace of $\mathcal{J}_\star$ satisfying $*$}
\nom[2reg]{$*=reg$}{Modifier ``regular''}
for the subset of regular almost complex structures $J$. For $J\in \mathcal{J}_X^{reg}$, the dimension of $\mathcal{M}_J^X(\mathbf{y},\mathbf{y'})$ near $u$ is equal to the {\em Fredholm index} $\op{ind}(u)$. The moduli space $\mathcal{M}_J^X({\bf y},{\bf y'})$ carries a natural $\R$-action given by translations in the $s$-direction, and the quotient $\mathcal{M}_J^X({\bf y},{\bf y'})/\R$ is a manifold.

\begin{lemma} \label{lemma: HF regularity}
A generic $J\in\mathcal{J}_X$ is regular.
\end{lemma}

\begin{proof}
This follows from \cite[Proposition~3.8]{Li} (see also \cite[Section~2]{Li2}), by noting that an HF curve $u$ does not have any fibers as irreducible components.  Lemma~\ref{lemma: HF regularity} can also be proved in the same way as in \cite[Lemma 9.12(b)]{Hu1}.  Note that the transversality theory is relatively straightforward because HF curves are never multiply-covered, i.e., all the moduli spaces $\mathcal{M}^X_J({\bf y}, {\bf y'})$ consist of simple curves.
\end{proof}

We will use the notation $\mathcal{M}_J^{X, I=r}({\bf y},{\bf y'})$
to denote the moduli space of HF curves from ${\bf y}$ to ${\bf y'}$ with ECH index $I_{HF}=r$.

\begin{cor}[Corollary of Theorem~\ref{thm: index inequality for HF}]
If $I_{HF}(A) = 0,1$ and $J\in \mathcal{J}_X^{reg}$, then every HF curve $u$ in the class $A$ satisfies $\op{ind}(u) = I_{HF}(A)$ and is therefore embedded.
\end{cor}

\begin{proof}
This follows from Equation~\eqref{eqn: index equality for HF} by observing that the term $2\delta(u)$ is even and nonnegative and that
$\op{ind}(u)\geq 0$ since $J$ is regular and $u$ is not multiply-covered.
\end{proof}

\subsection{Definition of the Heegaard Floer homology groups} \label{subsection: HF def}

Let $(\Sigma,\boldsymbol{\alpha},\boldsymbol{\beta},z)$ be a weakly admissible Heegaard diagram and let $J\in \mathcal{J}_X^{reg}$.  We define the Heegaard Floer chain complexes
\nom[CF0]{$\widehat{CF}(\Sigma,\boldsymbol{\alpha},\boldsymbol{\beta},z)$}{Heegaard Floer hat complex for pointed Heegaard diagram $(\Sigma,\boldsymbol{\alpha},\boldsymbol{\beta},z)$}
$$(\widehat{CF}(\Sigma,\boldsymbol{\alpha},\boldsymbol{\beta},z,J),\widehat\bdry) \ \mbox{ and } \ (CF^+(\Sigma,\boldsymbol{\alpha},\boldsymbol{\beta},z,J),\bdry^+),$$ whose corresponding homology groups are
$$\widehat{HF}(\Sigma,\boldsymbol{\alpha},\boldsymbol{\beta},z,J) \ \mbox{ and } \ HF^+(\Sigma,\boldsymbol{\alpha},\boldsymbol{\beta},z,J).$$

The hat group $\widehat{CF}(\Sigma,\boldsymbol{\alpha},\boldsymbol{\beta},z,J)$ is the $\F$-vector space generated by $\mathcal{S}_{\boldsymbol{\alpha},\boldsymbol{\beta}}$ and the plus group $CF^+(\Sigma,\boldsymbol{\alpha},\boldsymbol{\beta},z,J)$ is the $\F$-vector space generated by $\mathcal{S}_{\boldsymbol{\alpha},\boldsymbol{\beta}} \times\Z^{\geq 0}$. Elements of $\mathcal{S}_{\boldsymbol{\alpha},\boldsymbol{\beta}}$ will be written as ${\bf y}$ and elements of $\mathcal{S}_{\boldsymbol{\alpha},\boldsymbol{\beta}} \times \Z^{\geq 0}$ will be written as $[{\bf y},i]$.

We now define the differentials $\widehat\bdry$ and $\bdry^+$. The differential $\widehat\bdry$ is given by
$$\widehat\bdry {\bf y}= \sum_{{\bf y'} \in \mathcal{S}_{\boldsymbol{\alpha},\boldsymbol{\beta}}} \langle \widehat\bdry \mathbf{y},\mathbf{y'}\rangle \cdot {\bf y'},$$
where $\langle \widehat\bdry \mathbf{y},\mathbf{y'}\rangle$ is the count of $\widehat{\mathcal{M}}_J^{X, I=1} ({\bf y} ,{\bf y'})/\R$.  The differential $\bdry^+$ is given by
$$\bdry^+ ([{\bf y},i]) = \sum_{[{\bf y'},j] \in \mathcal{S}_{\boldsymbol{\alpha},\boldsymbol{\beta}}\times \Z^{\geq 0}} \langle \bdry^+ ([\mathbf{y},i]),[\mathbf{y'},j]\rangle \cdot [{\bf y'},j],$$
where $\langle \bdry^+( [\mathbf{y},i]),[\mathbf{y'},j]\rangle$ is the count of $\mathcal{M}_J^{X, I=1} ({\bf y} ,{\bf y'})/\R$ whose representatives
have intersection number $i-j$ with $\R\times [0,1]\times\{z\}$. By Theorem~\ref{thm: index inequality for HF}, the count of $I_{HF}(u)=1$ curves is equivalent to the count of embedded $\op{ind}(u)=1$ curves.  Hence our definition is the same as that of Lipshitz.

The differentials $\widehat\bdry$ and $\bdry^+$ indeed satisfy $\widehat\bdry^2=0$ and $(\bdry^+)^2=0$ by \cite{Li}. A tricky issue which arises for $\bdry^+$ (but not for $\widehat\bdry$) is that an element $u$ of the boundary of $\mathcal{M}_J^{X, I=2}({\bf y},{\bf y'})/\R$ might a priori have a fiber $\{(s,t)\}\times \Sigma$ as an irreducible component. (In that case, $u$ consists of one copy of $\Sigma$, together with $k$ trivial strips.)  This possibility is eliminated in \cite[Lemma 8.2]{Li}.

Both $\widehat{HF}(\Sigma,\boldsymbol{\alpha},\boldsymbol{\beta},z,J)$ and
$HF^+(\Sigma,\boldsymbol{\alpha},\boldsymbol{\beta},z,J)$ are independent of the choices and can be written as $\widehat{HF}(M)$ and $HF^+(M)$. In this paper, we
are interested in $\widehat{HF}(-M)$, where $-M$ is the manifold $M$
with the opposite orientation. The group $\widehat{HF}(-M)$ is the
homology of the chain complex $\widehat{CF}(\Sigma, \boldsymbol{\beta}, \boldsymbol{\alpha}, z,J)$, i.e., the complex for which the roles of the $\boldsymbol{\alpha}$- and $\boldsymbol{\beta}$-curves are exchanged.

\subsection{Restricting the complex to a page}

In this subsection we describe a pointed Heegaard diagram for $-M$ which is adapted to an open book and which has the property that $\widehat{HF}(-M)$ can be computed from a single page. The Heegaard diagram constructed here is a slight modification of the Heegaard diagram constructed by Honda, Kazez and Mati\'c in \cite{HKM1}.

Let $S$ be a bordered surface of genus $g$ and connected boundary, and let $(S,\hh)$ be an open book decomposition of $M$ with binding $K$. Then there is a homeomorphism
$$((S \times[0,1]) /\sim,\- (\partial S \times[0,1])/\sim)\simeq (M,K),$$
where $(x,1) \sim (\hh(x),0)$ for $x\in S$ and $(y,t) \sim (y,t')$ for $y \in \partial S$ and $t,t'\in[0,1]$.
\nom[g]{$g$}{Genus of page $S$}

\subsubsection{A Heegaard diagram compatible with $(S,\hh)$}
\label{subsubsection: Heegaard diagram compatible with S h}

We define a pointed Heegaard diagram $(\Sigma,\boldsymbol{\beta}, \boldsymbol{\alpha}, z)$ for $-M$ which is compatible with $(S,\hh)$. (In particular, this means that $L_{\boldsymbol{\beta}}=\R\times\{1\}\times \boldsymbol{\beta}$ and $L_{\boldsymbol{\alpha}}=\R\times\{0\}\times \boldsymbol{\alpha}$.)

The open book decomposition $(S,\hh)$ gives a natural Heegaard decomposition of $M$ into two handlebodies $H_1 =(S\times [0,\frac{1}{2} ])/\sim$ and $H_2 =(S\times [\frac{1}{2} , 1 ])/\sim$. The Heegaard surface $\Sigma$ is $S_{1/2} \cup -S_0$ and has genus $2g$. Here we abbreviate $S\times \{ t\}$ by $S_t$.

A {\it basis} of $S$ is a collection of properly embedded pairwise
disjoint arcs ${\bf a}= \{ a_1 ,...,a_{2g} \}$
\nom[a00 ]{${\bf a}={\bf a}(m)$}{Basis $\{ a_1 ,...,a_{2g} \}$ of arcs for $S$; depends on $m$ starting from Section~\ref{section: moduli spaces of multisections}}
\nom[a01 ]{$a_i=a_i(m)$}{Basis arc in $S$; depends on $m$ starting from Section~\ref{section: moduli spaces of multisections}}
of $S$ such that
$S-{\bf a}$ is a connected $8g$-gon.  Given a basis ${\bf a}$ of
$S$, there is a natural collection of compression curves $\alpha_i
=\partial (a_i \times [0,\frac{1}{2} ])$ for $H_1$. We write
$\alpha_i=a_i^\dagger\cup a_i$, where the presence of $\dagger$
indicates a copy of an arc in $S_{1/2}$ and the absence indicates a
copy of the arc in $S_0$. Recall the monodromy $\hh$ maps
$(y,\theta)\mapsto (y,\theta-y)$ near $\bdry S$. We then construct a
collection of compression curves $\beta_i= b_i^\dagger \cup \hh(a_i)$
for $H_2$, where $b_i$ is the simplest arc (i.e., with the fewest number of
intersections with the $a_j$) in $S_{1/2}$ which is parallel to
$a_i$ and extends $\hh(a_i)$ to smooth curve in $\Sigma$. See
Figure~\ref{darkside}.

\begin{figure}[ht]
\begin{overpic}[width=7cm]{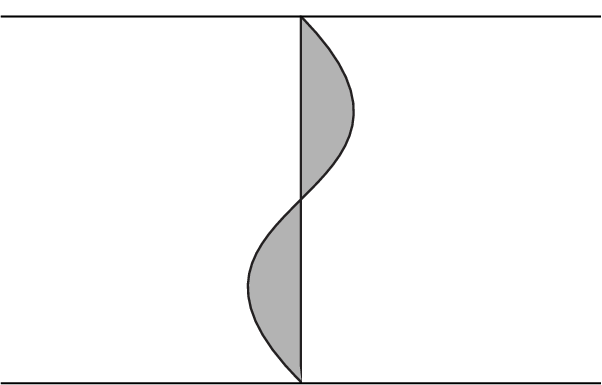}
\put(44.5,45){\tiny $a_i^\dagger$} \put(61,45) {\tiny $b_i^\dagger$} \put(52,27)
{\tiny $x_i''$} \put(52,2.1){\tiny $x_i$} \put(44.3,57.3){\tiny
$x_i'$}
\end{overpic}
\caption{A portion of $S_{1/2}$ whose normal orientation points out of the page.  The shaded regions are the disks
$D_i$ and $D_i'$.} \label{darkside}
\end{figure}
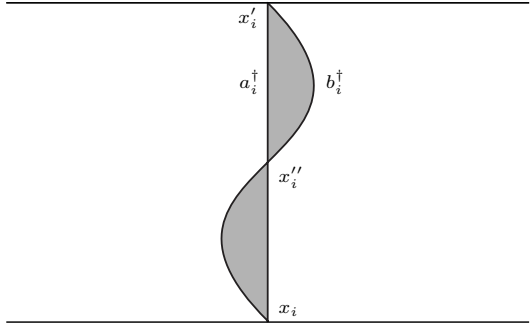

The arcs $a_i$ and $\hh(a_i)$ intersect at their endpoints $x_i$ and
$x_i'$
\nom[x ]{$x_i$, $x_i'$}{Intersection points on $\partial S$ forming the contact class}
by the definition of $\hh$ near $\bdry S$, and the arcs $a_i^\dagger$ and $b_i^\dagger$ intersect at a unique point $x_i''$
in $int(S_{1/2})$. This means that all the intersection points of $\alpha_i \cap
\beta_j$ lie in $S_0$, except for one intersection point $x_i''$ of
$\alpha_i \cap \beta_i$ for each $i$. We then place the basepoint
$z$ on the binding, away from all the intersection points
$x_i,x_i'$. The regions of $\Sigma-\boldsymbol{\alpha}-\boldsymbol{\beta}$
which nontrivially intersect $S_{1/2}$ are the following:
\begin{itemize}
\item the ``forbidden region'' containing the basepoint $z$;
\item for each $i=1,\dots,2g$, a bigon $D_i$ from
$x_i''$ to $x_i$ and a bigon $D_i'$ from $x_i''$ to $x_i'$.
\end{itemize}
By the placement of the basepoint $z$, it is clear that any periodic
domain must have terms of the form $k (D_i - D_i')$, where $k$ is an
integer. This implies the weak admissibility of the Heegaard diagram $(\Sigma,\boldsymbol{\beta},\boldsymbol{\alpha},z)$.

\begin{rmk}
The point $x_i$ or $x_i'$ (either one) is a component of the contact
class $c(\xi_{(S,\hh)})\in \widehat{HF}(\Sigma,\boldsymbol{\beta}, \boldsymbol{\alpha}, z)$, where $\xi_{(S,\hh)}$ is the contact structure
which corresponds to $(S,\hh)$.
\end{rmk}

\subsubsection{Holomorphic curves in the region $\R\times[0,1]\times S_{1/2}$}
\label{subsubsection: description of some holomorphic curves}

Let $J\in \mathcal{J}_X$ with the additional property:
\begin{itemize}
\item[(\&)] $J$ is a product complex structure on $\R \times [0,1] \times S_{1/2}$.
\end{itemize}

All the holomorphic curves and moduli spaces in this subsection are for the Heegaard diagram $(\Sigma, \boldsymbol{\beta}, \boldsymbol{\alpha}, z)$.

\begin{claim} \label{claim: on S one half}
Let $u\in \widehat{\mathcal{M}}_J^X({\bf y},{\bf y'})$ for some ${\bf y},{\bf y'}\in \mathcal{S}_{\boldsymbol{\beta}, \boldsymbol{\alpha}}$. Then the following hold:
\begin{enumerate}
\item If $u$ is not asymptotic to any $x_i''$, then its image is contained in $\R \times [0,1] \times S_0$.
\item If $u$ is asymptotic to some $x_i''$, then $u$ has $x_i''$ at the positive end and a component of $u$ is either (i) a trivial strip over $x_i''$ or (ii) a ``thin strip'' from $x_i''$ to $x_i$ or $x_i'$, whose projection to $\Sigma$ is $D_i$ or $D_i'$.
\item If $u$ is asymptotic to $x_i$ or $x_i'$ at the positive end, then a component of $u$ is a trivial strip over $x_i$ or $x_i'$.
\end{enumerate}
\end{claim}

The only nontrivial components of $u$ which intersect $\R\times[0,1]\times S_{1/2}$ are the ``thin strips'' in (2) and are easily seen to satisfy automatic transversality. Hence a generic $J$ which satisfies (\&) is in $\mathcal{J}_X^{reg}$.

\subsubsection{The variant $\widehat{CF}(S,{\bf a},\hh({\bf a}))$}
\label{subsubsection: variant CF of S}

Let $J\in \mathcal{J}_X^{reg}$ which satisfies (\&). We now define the chain complex $\widehat{CF}(S,{\bf a},\hh({\bf a}),J)$, which can be defined on a page of the open book $(S,\hh)$ and whose homology is isomorphic to $\widehat{HF}(\Sigma,\boldsymbol{\beta}, \boldsymbol{\alpha}, z, J)$.  The almost complex structure $J$ will usually be suppressed from the notation.

Let $\mathcal{S}_{{\bf a}, \hh({\bf a})}$
\nom[S ]{$\mathcal{S}_{{\bf a}, \hh({\bf a})}$}{$2g$-tuples of intersection points of ${\bf a}$ and $\hh({\bf a})$}
be the set of $2g$-tuples of intersection points of ${\bf a}$ and $\hh({\bf a})$; equivalently, $\mathcal{S}_{{\bf a}, \hh({\bf a})}=
\{{\bf y}\in \mathcal{S}_{\boldsymbol{\beta}, \boldsymbol{\alpha}}~|~ {\bf y}\subset S_0\}$. Then we define $(\widehat{CF'}(S, \mathbf{a}, \hh(\mathbf{a})),\bdry')$ as the subcomplex of $(\widehat{CF}(\Sigma,\boldsymbol{\beta}, \boldsymbol{\alpha} ,z),\bdry)$ generated by $\mathcal{S}_{{\bf a}, \hh({\bf a})}$.  The differential $\bdry$ restricts to $\bdry'$ by Claim~\ref{claim: on S one half}(1).

Next define an equivalence relation $\sim$ on $\widehat{CF'}(S, \mathbf{a}, \hh(\mathbf{a}))$ as follows: if we write $\mathbf{y}= \mathbf{y}_0 \cup \mathbf{y}_1$, where $\mathbf{y}_0$ consists of chords of type $x_i,x_i'$, $i=1,\dots,2g$, and $\mathbf{y}_1$ does not contain any $x_i, x_i'$, $i=1,\dots,2g$, then $\mathbf{y}\sim \mathbf{y}'$ if and only if $\mathbf{y}=\mathbf{y}_0\cup \mathbf{y}_1$, $\mathbf{y}'=\mathbf{y}'_0\cup \mathbf{y}'_1$ and $\mathbf{y}_1=\mathbf{y}'_1$.  We then take the quotient complex
$$\widehat{CF}(S, \mathbf{a}, \hh(\mathbf{a}))=\widehat{CF'}(S, \mathbf{a}, \hh(\mathbf{a}))/\sim,$$
with the differential $\widehat{\bdry}$ induced from $\bdry'$.  The differential $\bdry'$ descends to the quotient $\widehat{\bdry}$ by Claim~\ref{claim: on S one half}(3).
\nom[CF1]{$\widehat{CF}(S, \mathbf{a}, \hh(\mathbf{a}))$}{Heegaard Floer chain complex isomorphic to $\widehat{HF}(-M)$, obtained by quotienting $\widehat{CF'}(S, \mathbf{a}, \hh(\mathbf{a}))$}
\nom[CF1]{$\widehat{CF'}(S, \mathbf{a}, \hh(\mathbf{a}))$}{Subcomplex of Heegaard Floer chain complex $\widehat{CF}(\Sigma,\boldsymbol{\beta},\boldsymbol{\alpha},z)$ obtained from the open book decomposition $(S,\hh)$}

\begin{rmk}
Since $\Sigma$ and $S=S_0$ have opposite orientations, the order $(\boldsymbol{\beta}, \boldsymbol{\alpha})$ is switched to $(\mathbf{a},\hh(\mathbf{a}))$.
\end{rmk}

The following theorem allows us to restrict from the Heegaard
surface $\Sigma$ to the page $S$:

\begin{thm}\label{t:hf}
$H_*(\widehat{CF}(S,\mathbf{a},\hh(\mathbf{a})),\widehat{\bdry})\simeq \widehat{HF}(\Sigma,\boldsymbol{\beta}, \boldsymbol{\alpha}, z)$.
\end{thm}

\begin{proof}
Let us write $\widehat{CF}$ for $\widehat{CF}(\Sigma, \boldsymbol{\beta},
\boldsymbol{\alpha}, z)$. Also let $\widehat{CF}_k$ be the subgroup of $\widehat{CF}$ generated by $2g$-tuples of chords, exactly $k$ of which are of the form $x_i''$.  Using Claim~\ref{claim: on S one half}, we can write the differential $\bdry$ on $\widehat{CF}$ as $\partial = \partial_0 + \partial_1$, where  $\partial_0: \widehat{CF}_k \to \widehat{CF}_{k}$ counts $I_{HF}=1$ curves whose nontrivial part is contained in $S_0$ and $\partial_1: \widehat{CF}_k \to \widehat{CF}_{k-1}$ counts $I_{HF}=1$ curves whose nontrivial part is contained in $S_{1/2}$. In particular, $\bdry_1$ counts HF curves which correspond to the domains $D_i$ and $D_i'$. Since $\partial^2=0$, it follows that
$$\partial_0^2 = \partial_1^2 = \partial_0 \partial_1 + \partial_1 \partial_0 =0,$$
i.e., $\widehat{CF}$ becomes a double complex.

The $\partial_1$-homology of $\widehat{CF}$ is:
\begin{equation*}
H_k(\widehat{CF}, \partial_1) = \left \{
\begin{array}{cl} \widehat{CF}(S, \mathbf{a},
\hh(\mathbf{a})), & \text{ if } k=0; \\ 0, & \text{ if } k>0.
\end{array} \right.
\end{equation*}
This claim will be proved in Lemma \ref{d1-homology}. For the moment
we assume it to finish the proof of the theorem. The double complex
gives rise to a spectral sequence converging to
$\widehat{HF}(\Sigma, \boldsymbol{\beta}, \boldsymbol{\alpha}, z)$ such that:
\begin{align*}
E^1 & = H_*(\widehat{CF}, \partial_1) = \widehat{CF}(S, \mathbf{a},
\hh(\mathbf{a})) \\
E^2 & = H_*(H_*(\widehat{CF}, \partial_1), [\partial_0]) = H_*
(\widehat{CF}(S, \mathbf{a}, \hh(\mathbf{a})),\widehat\bdry).
\end{align*}
Since $E^2$ is concentrated in degree $k=0$, the spectral sequence degenerates at the second step and $\widehat{HF}(\Sigma, \boldsymbol{\beta}, \boldsymbol{\alpha}, z) \cong E^2$. This proves the theorem.
\end{proof}

Before proceeding to Lemma~\ref{d1-homology}, let us introduce some
notation. Let $\mathcal{I}$ be a $k$-element subset of $\{1,\dots, 2g \}$ and
let $\mathcal{I}^c$ be its complement. Then let ${\mathcal S}_{\mathcal{I}^c}$ be the
set of $(2g-k)$-tuples $\mathbf{y}_1$ of chords from $\hh(\mathbf{a})$
to $\mathbf{a}$ such that each $a_j$ and $\hh(a_j)$, $j\in \mathcal{I}^c$, is
used exactly once and no $x_j,x_j',x_j''$ is in $\mathbf{y}_1$. In
particular, $a_i$ and $\hh(a_i)$ remain unoccupied for all $i \in \mathcal{I}$.

\begin{lemma}\label{d1-homology}
The homology of $(\widehat{CF}, \partial_1)$ is:
\begin{equation*}
H_k(\widehat{CF}, \partial_1) = \left \{ \begin{array}{cl}
\widehat{CF}(S, \mathbf{a}, \hh(\mathbf{a})), & \text{ if } k=0;
\\ 0, & \text{ if } k>0. \end{array} \right.
\end{equation*}
\end{lemma}

\begin{proof}
Let $(\widehat{CF}(\mathbf{y}_1),\bdry_1) \subset (\widehat{CF}, \bdry_1)$ be the subcomplex generated by $2g$-tuples of chords of the form $\mathbf{y}_0\cup \mathbf{y}_1$, where $\mathbf{y}_0$ is a $k$-tuple of chords consisting of one of $x_j,x_j',x_j''$ for each $j \in \mathcal{I}$ and $\mathbf{y}_1\in {\mathcal S}_{\mathcal{I}^c}$. Since $(\widehat{CF}, \partial_1)$ is the direct sum of chain complexes of the form $(\widehat{CF}(\mathbf{y}_1),\bdry_1)$, it suffices to treat each $(\widehat{CF}(\mathbf{y}_1),\bdry_1)$ separately.

Consider the chain complexes $(C(j),d)=(C_0(j) \oplus C_1(j),d)$, where
$$C_1(j) = \F\{ x_j''\}, \quad C_0(j)= \F\{x_j, x_j'\}, \quad d(x_j'')= x_j - x_j'.$$
The homology groups of those complexes are:
\begin{equation}
H_k(C(j),d) = \left\{
\begin{array}{cl}
\langle x_j, x_j' \rangle / \langle x_j -x_j' \rangle, & \mbox{ if } k=0;\\ 0, & \mbox{ if } k=1.
\end{array}
\right.
\end{equation}

By Claim~\ref{claim: on S one half}(2), we have
$$(\widehat{CF}(\mathbf{y}_1), \bdry_1) \cong \bigotimes_{j \in \mathcal{I}} (C(j),d).$$
By the K\"unneth formula, $H_*(\widehat{CF}(\mathbf{y}_1), \bdry_1)$ is generated by the equivalence class $\{{\bf y}'_0\cup{\bf y}_1\}$, where ${\bf y}_1$ is fixed and ${\bf y}'_0$ ranges over all $k$-tuples of chords which consist of one of $x_j$, $x_j'$ for each $j\in \mathcal{I}$. The lemma then follows.
\end{proof}

\subsection{${\rm Spin}^c$-structures}\label{spin-c revisited}

Let $\mathcal{S}_{\boldsymbol{\alpha},\boldsymbol{\beta}}$ be the set of $k$-tuples of intersection points of the pointed Heegaard diagram $(\Sigma,\boldsymbol{\alpha}, \boldsymbol{\beta},z)$ and let Spin$^c(M)$ be the set of Spin$^c$-structures on $M$. In \cite[Section~2.6]{OSz1}, Ozsv\'ath and Szab\'o defined a map
$$s_z: \mathcal{S}_{\boldsymbol{\alpha},\boldsymbol{\beta}} \to {\rm Spin}^c(M).$$
Although the precise definition of $s_z$ will not be given here, we review an important property of $s_z$ which is more or less equivalent to the definition. Given $\mathbf{y}=\{y_i\}_{i=1}^k,\mathbf{y'}=\{y'_i\}_{i=1}^k\in \mathcal{S}_{\boldsymbol{\alpha},\boldsymbol{\beta}}$, the difference between the Spin$^c$-structures corresponding to $\mathbf{y}$ and $\mathbf{y'}$ is given by:
$$\epsilon(\mathbf{y}, \mathbf{y'}) = PD(s_z(\mathbf{y})-s_z(\mathbf{y'})) \in H_1(M),$$
where a cycle representing $\epsilon(\mathbf{y}, \mathbf{y'})$ can be constructed on the Heegaard diagram as follows: For each $i=1,\dots, k$, choose an arc $\alpha_i^\star$ on $\alpha_i$ from $y_i$ to $y'_i$, where $y_i,y_i'\in \alpha_i$. Similarly, we choose arcs $\beta_i^\star$ on $\beta_i$, $i=1,\dots,k$, which connect ${\bf y'}$ to ${\bf y}$. Then $\epsilon({\bf y},{\bf y'})$ is the homology class of $\cup_{i=1}^k (\alpha^\star_i\cup\beta^\star_i)$, which is a union of closed curves; see \cite[Definition~2.11 and Lemma~2.19]{OSz1}. It is easy to verify that $\epsilon({\bf y},{\bf y'})$ does not depend on the choice of arcs $\alpha_i^\star$ and $\beta_i^\star$ and provides a topological obstruction to the existence of $HF$ curves connecting $\mathbf{y}$ and $\mathbf{y'}$.

The Heegaard Floer chain complex $\widehat{CF}(\Sigma,\boldsymbol{\alpha}, \boldsymbol{\beta},z)$ therefore splits into a direct sum
$$\widehat{CF}(\Sigma,\boldsymbol{\alpha},\boldsymbol{\beta},z)= \bigoplus_{\mathfrak{s} \in {\rm Spin}^c(M)} \widehat{CF}(\Sigma, \boldsymbol{\alpha}, \boldsymbol{\beta} ,z,\mathfrak{s}),$$
where the subgroup $\widehat{CF}(\Sigma, \boldsymbol{\alpha}, \boldsymbol{\beta} ,z,\mathfrak{s})$ is generated by $\mathbf{y}\in\mathcal{S}_{\boldsymbol{\alpha}, \boldsymbol{\beta}}$ with $s_z(\mathbf{y})=\mathfrak{s}$ and is a subcomplex.

We now interpret the above discussion in a way which relates more easily to the splitting of ECH in terms of homology classes of orbit sets. Consider the chain complex $\widehat{CF}(S,\mathbf{a},\hh(\mathbf{a}))$ which is generated by the set $\mathcal{S}_{{\bf a},\hh({\bf a})}$ of $2g$-tuples of
intersection points of ${\bf a}$ and $\hh({\bf a})$, i.e., we are restricting to a page $S$. The homology groups $H_1(M;\Z)\simeq H_1(N, \partial N;\Z)$ are identified via the isomorphism $\varpi$, given in Lemma~\ref{varpi}.

We then define the map
\[ \mathfrak{h}: \mathcal{S}_{{\bf a},\hh({\bf a})} \to H_1(M) \]
by assigning a cycle $\mathfrak{h}(\mathbf{y})$ to $\mathbf{y}=\{y_i\}_{i=1}^{2g}\in \mathcal{S}_{{\bf a},\hh({\bf a})}$ as follows: Suppose
$y_i \in a_i \cap \hh(a_{\sigma(i)})$ for some $\sigma \in
\mathfrak{S}_{2g}$.  On $[0,1]\times S$, we consider the union of
the following oriented arcs:
\begin{itemize}
\item $[0,1]\times \{y_i\}$, $i=1,\dots,2g$, where the orientation
is given by $\bdry_t$;
\item $\{0\}\times c_i$, $i=1,\dots,2g$,
where $c_i$ is a subarc of $\hh(a_i)$ which goes from
$\hh(y_{\sigma(i)})$ to $y_i$.
\end{itemize} With the identification
$(x,1)\sim (\hh(x),0)$, the arcs glue to give a cycle in $N$ which
represents $\mathfrak{h}(\mathbf{y})$.

\begin{prop}\label{s-and-h}
Let $\xi$ be the contact structure supported by the open book
decomposition $(S,\hh)$ of $M$, and let $\mathfrak{s}_{\xi}$ be the
canonical ${\rm Spin}^c$-structure determined by $\xi$. Then for any
$\mathbf{y} \in \mathcal{S}_{{\bf a},\hh({\bf a})}$ we have
\[ s_z(\mathbf{y}) = \mathfrak{s}_{\xi} + PD(\mathfrak{h}(\mathbf{y})). \]
\end{prop}

\begin{proof}
The equality holds for any $2g$-tuple $\mathbf{x}_0$ which represents the contact class.  In fact, $s_z(\mathbf{x}_0)= \mathfrak{s}_{\xi}$ by the definition of the contact class and $\mathfrak{h}(\mathbf{x}_0)=0$ since the cycle representing it is parallel to $\partial N$. Hence, in order to prove the proposition, it suffices to prove that
\[ \mathfrak{h}(\mathbf{y}) - \mathfrak{h}(\mathbf{y}') = \epsilon (\mathbf{y}, \mathbf{y}') \]
for all $\mathbf{y}, \mathbf{y}' \in \mathcal{S}_{{\bf a},\hh({\bf a})}$. One can check that $\mathfrak{h}(\mathbf{y}) - \mathfrak{h}(\mathbf{y}')$ is homologous to the union $\delta$ of the following types of arcs:
\begin{itemize}
\item $[0,1]\times \{y_i\}$ with orientation $\bdry_t$;
\item $[0,1]\times\{y'_i\}$ with orientation $-\bdry_t$;
\item subarcs of $\{1\}\times a_i$ connecting from ${\bf y}$ to ${\bf y'}$; and
\item subarcs of $\{0\}\times \hh(a_i)$ connecting from $\mathbf{y}'$ to $\mathbf{y}$.
\end{itemize}
By homotoping $\delta$ to a page $S$, we see that $[\delta]=\epsilon(\mathbf{y},\mathbf{y}')$ with respect to the Heegaard diagram $(\Sigma, \boldsymbol{\beta}, \boldsymbol{\alpha}, z)$ given in Section~\ref{subsubsection: Heegaard diagram compatible with S h}.
\end{proof}

\subsection{Twisted coefficients in Heegaard Floer homology}

In this subsection we review the definition of Heegaard Floer homology with twisted coefficients, originally defined in \cite[Section~8]{OSz2}, and prove a twisted coefficient analog of Theorem~\ref{t:hf}.  We describe the construction for $\widehat{HF}$; the construction for $HF^+$ --- which will be used in \cite{CGH-III} --- can be treated in a similar manner.

Fix a Spin$^c$-structure $\mathfrak{s}$ and a $k$-tuple of intersection points $\mathbf{y}_0$ such that $s_z(\mathbf{y}_0)= \mathfrak{s}$. A {\em complete set of paths for $\mathfrak{s}$ based at ${\bf y}_0$} is the choice, for each $k$-tuple of intersection points $\mathbf{y}$ such that $s_z(\mathbf{y}) = \mathfrak{s}$, of a surface $C_{\mathbf y}$
which is the projection to $[0,1] \times \Sigma$ of a surface representing an element of $H_2(X,\mathbf{y}, \mathbf{y}_0)$.\footnote{We are identifying relative homology classes in $\check{X}$ with relative homology classes in $[0,1] \times \Sigma$ as necessary.}

A complete set of paths determines maps
$$\mathfrak{A}: H_2(X,\mathbf{y}, \mathbf{y}') \to H_2([0,1] \times \Sigma, \{ 0 \} \times \boldsymbol{\beta} \cup \{ 1 \} \times \boldsymbol{\alpha}) \simeq H_2(M)$$
for all $\mathbf{y}$ and $\mathbf{y}'$ such that $s_z(\mathbf{y}) = s_z(\mathbf{y}')= \mathfrak{s}$ by $\mathfrak{A}(A) = [C_{\mathbf{y}'} \cup A \cup - C_{\mathbf{y}}]$. This map is compatible with the action of $H_2(M)$ on $H_2(X,\mathbf{y}, \mathbf{y}')$ and with the concatenation of chains with matching ends.

We define
$$\widehat{\underline{CF}}(\Sigma, \boldsymbol{\alpha}, \boldsymbol{\beta} ,z,\mathfrak{s}) = \widehat{CF}(\Sigma, \boldsymbol{\alpha}, \boldsymbol{\beta},z,\mathfrak{s}) \otimes_{\F} \F[H_2(M; \Z)]$$
as an $ \F[H_2(M; \Z)]$-module, with differential
$$\widehat{\partial} \mathbf{y} = \sum_{s_z(\mathbf{y}')=\mathfrak{s}} \ \sum_{A \in H_2(X,\mathbf{y}, \mathbf{y}')} \# \left(\widehat{\mathcal M}^{X,I=1}(\mathbf{y}, \mathbf{y}', A)/ \R \right ) e^{\mathfrak{A}(A)}  \mathbf{y}'.$$
The homology of this complex is the Heegaard Floer homology with twisted coefficients
$\widehat{\underline{HF}}(M, \mathfrak{s})$.

Consider the special Heegaard diagram constructed in Section~\ref{subsubsection:
Heegaard diagram compatible with S h}. For every Spin$^c$-structure $\mathfrak{s} \in
{\rm Spin}^c(M)$ we define  the complex
$$\underline{\widehat{CF}}(S, \mathbf{a}, \hh(\mathbf{a}), \mathfrak{s}) = \widehat{CF}(S, \mathbf{a}, \hh(\mathbf{a}), \mathfrak{s}) \otimes_{\F} \F[H_2(M; \Z)]$$
with the differential induced by the differential on $\underline{\widehat{CF}}(\Sigma, \boldsymbol{\beta}, \boldsymbol{\alpha}, z, \mathfrak{s})$.

\begin{thm}\label{t:hf twisted}
$H_*(\underline{\widehat{CF}}(S, \mathbf{a}, \hh(\mathbf{a}), \mathfrak{s}))$ is isomorphic to $\underline{\widehat{HF}}(\Sigma, \boldsymbol{\beta}, \boldsymbol{\alpha}, z, \mathfrak{s})$ as \newline $\F[H_2(M; \Z)]$-modules.
\end{thm}

\begin{proof}
Fix a distinguished $2g$-tuple of generators $\mathbf{y}_0$ such that $s_z(\mathbf{y}_0)= \mathfrak{s}$. We choose a complete set of paths $C_{\mathbf{y}}$ with the following property: if $\mathbf{y} = \widetilde{\mathbf{y}} \cup \{ x_i \}$, $\mathbf{y}' = \widetilde{\mathbf{y}} \cup \{ x_i' \}$ and $\mathbf{y}'' =  \widetilde{\mathbf{y}} \cup \{ x_i''\}$, then $C_{\mathbf{y}} = C_{\mathbf{y}''} \cup D_i$ and $C_{\mathbf{y}'} = C_{\mathbf{y}''} \cup D_i'$, where $D_i$ and $D_i'$ are the surfaces corresponding to the thin strips connecting $x_i''$ to $x_i$ and $x_i'$, respectively; see Figure~\ref{darkside}. With this choice of complete set of paths, we have $\mathfrak{A}(D_i)=\mathfrak{A}(D_i')=0$ for all $i$, so the proof of Theorem~\ref{t:hf} goes through unchanged.
\end{proof}

\section{Moduli spaces of multisections}
\label{section: moduli spaces of multisections}

The goal of this section is to introduce the moduli spaces which will be used to define the chain maps
$$\Phi:\widehat{CF}(S,\mathbf{a},\hh(\mathbf{a}))\to
PFC_{2g}(N,\alpha_0,\omega)$$ and
$$\Psi: PFC_{2g}(N,\alpha_0,\omega)\to
\widehat{CF}(S,\mathbf{a},\hh(\mathbf{a})).$$
The definition of these chain maps can be viewed as a melding of ideas
of Seidel~\cite{Se1,Se2} and Donaldson-Smith~\cite{DS}.

Let $\mathbf{y}=\{y_1,\dots,y_{2g}\}$ be a generator of
$\widehat{CF}(S,\mathbf{a},\hh(\mathbf{a}))$, where $y_i\in a_{i}\cap
\hh(a_{\sigma(i)})$. Intuitively, $\mathbf{y}$ is mapped to an element
of $PFC_{2g}(N)$ through an intermediary called a {\em broken closed
string}
$\gamma_{\mathbf{y}}$.  It is a union of closed curves in
$N=S\times[0,1]/\sim$, obtained by taking the union of
$y_i\times[0,1]$, $i=1,\dots,2g$, and $c_{i}\times\{0\}$,
$i=1,\dots,2g$, where $c_i$ is a subarc of $\hh(a_i)$ which connects
$\hh(y_{\sigma(i)})$ to $y_i$. Note that there is a unique homotopy
class of arcs from $\hh(y_{\sigma(i)})$ to $y_i$, since $\hh(a_i)$ is an
arc (and not a closed curve). The arcs $y_i\times[0,1]$,
$c_i\times\{0\}$, $i=1,\dots,2g$, glue up to give a union of closed
curves since $(\hh(y_{\sigma(i)}),0)\sim (y_{\sigma(i)},1)$.

\subsection{Symplectic cobordisms}
\label{subsection: symplectic cobordisms}

We recall the stable Hamiltonian structure $(\alpha_0,\omega)$ on $N$ from Section~\ref{section: periodic Floer homology}. 
The mapping torus $N$ is given by:
$$N=(S\times[0,2])/\sim,  ~~(x,2)\sim(\hh(x),0)$$
where $\hh$ is an area-preserving map of $S$ with zero flux. The projection
$S\times[0,2] \to [0,2]$ defines a projection $\pi_{S^1}: N \to S^1$, where we identify
$S^1 \cong [0,2]/ 0 \sim 2$. Here we made a slight modification to the definition of $N$: the interval $[0,1]$ in Section~\ref{subsection: first return map} is now replaced by $[0,2]$.

The area form $\omega$ on $S$, the form $dt$ on $[0,2]$, and the vector field $\partial_t$ on $S \times [0,2]$ give well-defined objects on $N$ which we still denote by $\omega$, $dt$, and $\partial_t$.
\nom[1y$\omega$2]{$\omega$}{Area form on $S$ invariant under $\hh$; also viewed as a $2$-form on $W_+$, $W_-$}
By Equation~\eqref{eqn: defn of alpha s} the form $\alpha_0$ is a positive multiple of $dt$. For simplicity we assume that $\alpha_0=dt$. Then the stable Hamiltonian vector field is $R_0=\bdry_t$ and, by construction, the monodromy $\hh$ is the first return map of $R_0$.  In view of Lemma \ref{lemma: change from h to h_0} we may assume that:
\begin{itemize}
\item all orbits of $R_0$ in $int(N)$ satisfying $\mathcal{F}\leq 2g$ are nondegenerate; and
\item $R_0$ is negative Morse-Bott along $\bdry N$.
\end{itemize}

\begin{rmk}
Indeed, the stable Hamiltonian vector field $R_0$ on $N$ has the same first return map as a Reeb vector field $R_\tau$, $\tau>0$, by construction, and we could have taken $R_\tau$ to be Morse-Bott nondegenerate.
\end{rmk}

 Let us write $W=\R\times [0,1]\times S$ and $W'=\R\times N$. Let
$$\Omega=ds\wedge dt +\omega,\quad \Omega'=ds\wedge dt+\omega$$
be the symplectic forms on $W$ and $W'$, where $s$ is the $\R$-coordinate.
\nom[W ]{$W$ and $W'$}{$\R\times[0,1]\times S$ and $\R \times N$}

\nom[1yy$\omega$2]{$\Omega$ and $\Omega'$}{Symplectic forms on $W$ and $W'$ given by $ds\wedge dt +\omega$}

In this section we introduce the symplectic cobordisms $(W_+,\Omega_+)$ and $(W_-,\Omega_-)$, as well as their ``compactifications'' $(\overline{W}_+,\overline\Omega_+)$ and $(\overline{W}_-,\overline\Omega_-)$. The cobordism $(W_+,\Omega_+)$ interpolates from the stable Hamiltonian structure $([0,1]\times S, (dt,\omega))$ at the positive end to the stable Hamiltonian structure $(N,(\alpha_0,\omega))$ at the negative end, whereas the cobordism $(W_-,\Omega_-)$ goes from $(N,(\alpha_0,\omega))$ at the positive end to $([0,1]\times S, (dt,\omega))$ at the negative end.

\subsubsection{The symplectic cobordisms $(W_+,\Omega_+)$ and $(W_-,\Omega_-)$}
\label{acorn}

Consider the infinite cylinder $\R \times S^1\simeq \R\times (\R/2\Z)$ with coordinates $(s,t)$. Let $\pi_{S^1}: N \to S^1$ be the fibration $(x,t)\mapsto t$ and let
$$\pi_{B'}: \R \times N = W' \to B'=\R\times S^1$$
be the extension $(s,x,t)\mapsto (s,\pi_{S^1}(x,t))$. Let us write $N_s = \pi^{-1}_{B'} (\{ s\} \times S^1)$ for $s\in\R$.
\nom[B]{$B'$}{Alternate notation for $\R \times S^1$}
\nom[1p$\pi$1]{$\pi_{B'} : W' \to B'$}{Symplectic fibration with fiber $S$ used in the definition of $\widehat{ECC}$}
We define $W_+ =\pi^{-1}_{B'} (B_+)$, where $B_+=(\R\times (\R/2\Z)) -B_+^c$ and $B_+^c$ is the subset $[2,\infty)\times[1,2]\subset \R\times(\R/2\Z)$ with the corners rounded.  See the left-hand side of Figure~\ref{figure: bases}. We write
$$\pi_{B_+}: W_+\to B_+$$
for the restriction of $\pi_{B'}$.
\nom[b]{$B_+$, $B_-$}{Bases of the symplectic fibrations $\pi_{B_+}: W_+\to B_+$ and $\overline{\pi}_{B_-}: \overline{W}_-\to B_-$}
Note that the boundary of $W_+$ can be decomposed into two parts that meet along a codimension two corner: the {\it vertical boundary} $\partial_v W_+ =\pi^{-1}_{B_+} ( \partial B_+)$, and the {\it horizontal boundary} $\partial_h W_+$, which is equal to the union of the boundaries of the fibers.
\nom[1p$\pi$2]{$\pi_{B_+}: W_+\to B_+$}{Symplectic fibration with fiber $S$ used in the definition of $\Phi$}
\nom[W]{$W_+$}{Symplectic fibration with fiber $S$ used in the definition of $\Phi$}
\nom[b]{$\bdry_h W_+$, $\bdry_v W_+$}{Horizontal and vertical boundaries of $W_+$}

\begin{figure}[ht]
\begin{overpic}[height=1.5in]{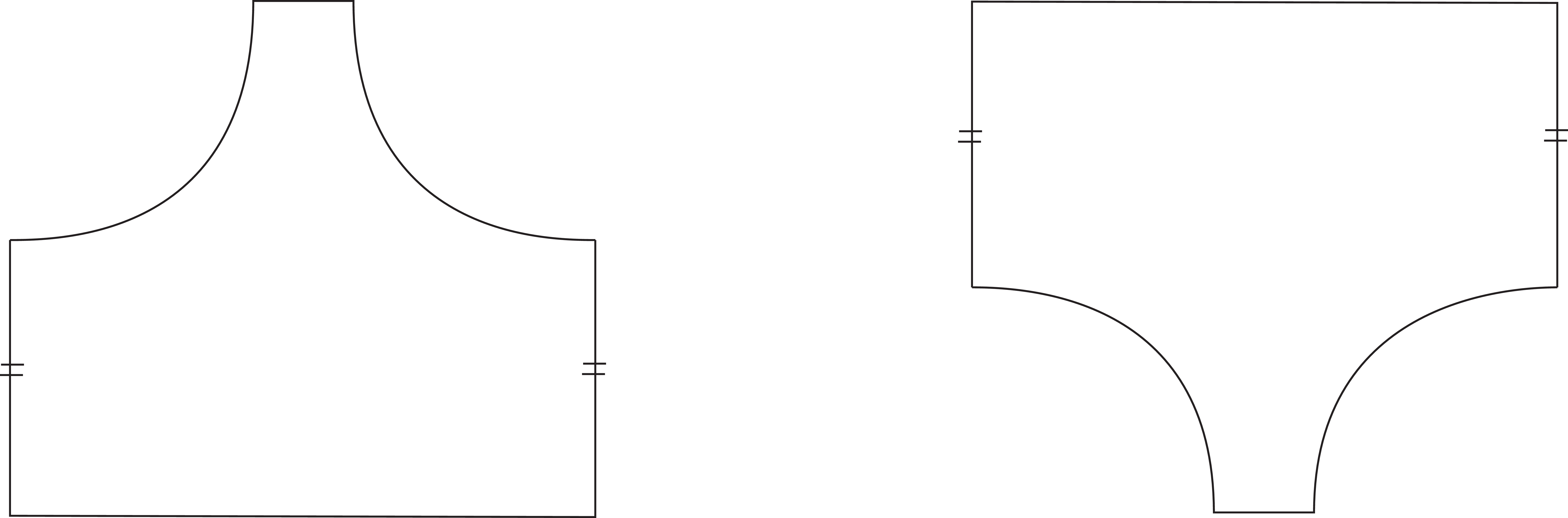}
\put(60.59,28){$\times$} \put(98,28){$\times$}
\end{overpic}
\caption{The bases $B_+$ and $B_-$.  The sides are identified. Both $B_+$ and $B_-$ are biholomorphic to a disk with an interior puncture and a boundary puncture. Here the location of $\overline{\frak m}^b$ on $B_-$ is indicated by $\times$.}
\label{figure: bases}
\end{figure}

Similarly, we define $W_-=\pi^{-1}_{B'}(B_-)$, where $B_-=(\R\times (\R/2\Z))- B_-^c$ and $B_-^c$ is $(-\infty,-2]\times[1,2]$ with the corners rounded. The projection $$\pi_{B_-}: W_-\to B_-,$$
the vertical boundary $\bdry_v W_-$, and the horizontal boundary $\bdry_h W_-$ are defined analogously.
\nom[W]{$W_-$}{Symplectic fibration with fiber $S$; $\overline{W}_-$ is used in the definition of $\Psi$}

The symplectic form $\Omega_+$ (resp.\ $\Omega_-$) is the restriction of
$$ds\wedge dt+\omega=ds\wedge dt+d_S\beta_t$$
to $W_+$ (resp.\ $W_-$).  By this we mean the following: On $\R\times S\times[0,2]$, we take the symplectic form $ds\wedge dt+\omega$.  Then the symplectic form glues under the identification $(s,x,2)\sim (s,\hh(x),0)$.
\nom[1yy$\omega$2]{$\Omega_+$ and $\Omega_-$}{Symplectic forms on $W_+$ and $W_-$}

We also write $cl(B_+)$, $cl(B_-)$ for the compactifications of $B_+$, $B_-$,  obtained by adjoining the points at infinity
\nom[cl]{$cl(B_+)$ and $cl(B_-)$}{Compactifications of $B_+$ and $B_-$}
$\mathfrak{p}_+$ corresponding to $s=+ \infty$, and $\mathfrak{p}_-$ corresponding to $s=-\infty$. Therefore $cl(B_+)$ and $cl(B_-)$ are isomorphic to the closed unit disk with one marked point on the interior and one marked point on the boundary.

\subsubsection{The extended cobordisms $(\overline{W}_+,\overline{\Omega}_+)$ and $(\overline{W}_-,\overline{\Omega}_-)$}
\label{subsubsection: overline W pm}

We now extend $(W_+,\Omega_+)$ to $(\overline{W}_+, \overline{\Omega}_+)$, which corresponds to capping off each fiber $S$ by a disk; the definition of $(\overline{W}_-,\overline{\Omega}_-)$ is analogous.

We first define the capped-off surface $\overline{S}$:
\nom[SS3]{$\overline{S}$}{Capped-off surface $S \cup D^2$}
Let $D^2=\{\rho\leq 1\}$ be a disk with polar coordinates $(\rho,\phi)$.
\nom[D]{$D^2$, $D^2_\delta$}{Unit disk $\overline{S}-int(S)$ with polar coordinates $(\rho,\phi)$, from Section~\ref{section: moduli spaces of multisections} onwards; $D^2_\delta=\{\rho\leq \delta\}$ for $\delta>0$ small}
\nom[1q$\rho$]{$(\rho,\phi)$}{Polar coordinates on $D^2$}
We write $z_\infty$
\nom[z]{$z_\infty$}{Center of the disk $D^2=\overline{S}-int(S)$}
for the origin $\rho=0$.  Let $(y,\theta)$ be the coordinates on a neighborhood $\nu(\bdry S)\simeq [-\varepsilon,\varepsilon]\times \R/\Z$ of $\bdry S=\{0\}\times \R/\Z$ (inside a slight extension of $S$), as before.  Then $\overline{S}= (S\sqcup D^2)/\sim$, where $(y,\theta)\in \nu(\bdry S)$ is identified with $({1\over y+1} ,-2\pi \theta)\in D^2$.

For every integer $m > 2g$
\nom[m]{$m$}{Positive integer $\gg 2g$; $\overline{\hh}$ and $\overline{\bf a}$ depend on $m$ starting from Section~\ref{section: moduli spaces of multisections}}
we define
$$\overline{\hh}_m:\overline{S}\stackrel\sim\to \overline{S}$$
as a smooth extension of $\hh:S\stackrel\sim\to S$, depending on $m$, such that
\nom[h6 ]{$\overline{\hh}=\overline{\hh}_m$}{Extension of $\hh$ to $\overline{S}$ which depends on $m$}
$\overline{\hh}_m|_{D^2}$ is the diffeomorphism of $D^2$ given by 
$$(\rho,\phi)\mapsto(\rho,\phi+\nu_m(\rho)),$$
where $\nu_m:[0,1]\to \R$ is a smooth function which satisfies the following:
\begin{itemize}
\item $\nu_m(\rho)={2\pi\over m}$ for $\rho\leq {1\over 2}$;
\item $\nu_m(\rho)$ is increasing for ${1\over 2}<\rho< {3\over 4}$;
\item $\nu_m(\rho)$ is decreasing and independent of $m$ for ${3\over 4}<\rho<1$;
\item $\nu_m({3\over 4})\ll {2\pi\over 2g}$ and $\nu_m(1)=0$;
\item $\nu_\infty:=\lim_{m\to \infty} \nu_m$ exists in the $C^k$-topology for $k\gg 0$.
\end{itemize}

In particular, $\nu_\infty(\rho)=0$ for $\rho\leq{1\over 2}$. Taking the limit $m\to \infty$ becomes important starting from Section~\ref{subsection: asymptotic eigenfunctions}, but until then we just need $m\gg 0$ and we simply write $\overline{\hh}=\overline{\hh}_m$ and $\nu=\nu_m$.

We then define the mapping torus
$$\overline{N}_m=(\overline{S}\times[0,2])/(x,2)\sim (\overline{\hh}(x),0).$$
When we do not need to specify $m$, we will simplify the notation by writing
$\overline{N}_m = \overline{N}$.
\nom[N2 ]{$\overline{N}=\overline{N}_m$}{Mapping torus of $\overline{\hh}: \overline{S} \to \overline{S}$}
Note that, although $\overline{N}_m$ depends on $m$ and $\nu_m$, all the $\overline{N}_m$ are diffeomorphic. The closed manifold $\overline{N}$ is obtained from $M$ by a $0$-surgery along the binding of the open book.

Let $\overline\omega$ be an area form on $\overline{S}$ which extends $\omega$ and equals $\rho d\rho\wedge d\phi$ on $D^2$.
\nom[1y$\omega$2]{$\overline{\omega}$}{Area form on $\overline{S}$ which restricts to $\omega$ on $S$ and is invariant under $\overline{\hh}$; also viewed as a $2$-form on $\overline{W}_+$, $\overline{W}_-$}
We then extend the stable Hamiltonian structure $(\alpha_0,\omega)$ on $N$ to the stable Hamiltonian structure $(\overline{\alpha}_0, \overline\omega)$ with Hamiltonian vector field $\overline{R}_0$ on $\overline{N}$, where we still have $\overline{\alpha}_0=dt$ and $\overline{R}_0= \partial_t$. Moreover, the orbit $$\delta_0=\{z_\infty\}\times[0,2] /\sim$$
is the only simple closed orbit of $\overline{R}_0$ on $\overline{N}-N$
\nom[1d$\delta$]{$\delta_0$}{Closed orbit in $\overline{N}$ over $z_\infty$}
which intersects $\overline{S}$ at most $2g$ times; it is called $\delta_0$ since it lies on the level set $\rho=0$.  The $2$-plane field of the stable Hamiltonian structure is $\ker \overline\alpha_0=T\overline{S}$.

We now define $\overline{W}_+$ and $\overline{W}_-$. First define extensions $\overline{W}=\R\times [0,1]\times\overline{S}$ of $W=\R\times[0,1]\times S$ and $\overline{W'}=\R\times\overline{N}$ of $W'=\R\times N$. Let $\overline\pi_{S^1}: \overline{N}\to S^1$ be the fibration $(x,t)\mapsto t$ and let $$\overline\pi_{B'}: \overline{W'}=\R \times \overline{N} \to B'=\R\times S^1$$
be the extension $(s,x,t)\mapsto (s,\overline\pi_{S^1}(x,t))$. For $*=+$ or $-$, we set $\overline{W}_*=\overline\pi^{-1}_{B'}(B_*)$ and write
$$\overline\pi_{B_*}:\overline{W}_*\to B_*$$
for the restriction of
$\overline\pi^{-1}_{B'}$ to $\overline{W}_*$.  We also define
$$\overline{\pi} : \overline{W} \to B.$$
\nom[1p$\pi$3]{$\overline\pi_*$ with $*=B,B',B_+,B_-$}{Symplectic fibrations with fiber $\overline S$ extending $\pi_*$}

The symplectic forms $\overline\Omega_+$, $\overline\Omega_-$, $\overline \Omega$, $\overline {\Omega'}$ for $\overline{W}_+$, $\overline{W}_-$, $\overline{W}$, $\overline{W'}$ are obtained from $ds\wedge dt +\overline{\omega}$ by gluing and restricting as necessary.
\nom[WW]{$\overline{W}$, $\overline{W'}$, $\overline{W}_+$, $\overline{W}_-$}{Extensions of $W$, $W'$, $W_+$, $W_-$ with fibers $\overline{S}$}
\nom[1yy$\omega$3]{$\overline \Omega$, $\overline {\Omega'}$, $\overline\Omega_+$, $\overline\Omega_-$}{Extensions of $\Omega$, $\Omega'$, $\Omega_+$, $\Omega_-$}

Let us write
$$V: =\overline{N} - int(N) = (D^2 \times[0,2])/(x,2)\sim (\overline{\hh}(x),0).$$
We then identify
\nom[V ]{$V$}{$\overline{N} - int(N)$ (starting from Section \ref{subsubsection: overline W pm})}
$\varphi: V\stackrel\sim\to D^2 \times (\R/2\Z)$ via $(\rho e^{i\phi},t)\mapsto (\rho e^{i(\phi+t\nu(\rho)/2)},t).$  Note that $\varphi$ relates two coordinate systems:
\begin{itemize}
\item[(i)] the ``Reeb coordinates'' $(x,t)$ on $D^2\times[0,2]$ such that $\overline{R}_0=\bdry_t$ and $(x,2)\sim (\overline{\hh}(x),0)$, and
\item[(ii)] the ``balanced coordinates'' on $D^2\times (\R/2\Z)$ such that $\overline{R}_0=\bdry_t + {\nu(\rho)\over 2}\bdry_\phi$ and $(x,2)\sim (x,0)$.
\end{itemize}

For $*=+$ or $-$, let
$$\overline\pi_{D^2}: \overline{W}_*\cap (\R\times V) \to D^2$$
be the projection
\nom[1p$\pi$5]{$\overline\pi_{D^2}:\overline{W}_*\cap (\R\times V) \to D^2$}{Projection with respect to balanced coordinates, where $*=+$ or $-$}
of $\overline{W}_*\cap (\R\times V)$ to $V$, followed by the projection of $V$ to $D^2$ via the identification $\varphi$, i.e., with respect to the balanced coordinates.

We also choose the marked point $\overline{\frak m}=(\overline{\frak m}^b, \overline{\frak m}^f) \in  \overline{W}_-$,
\nom[m0 ]{$\overline{\frak m}= (\overline{\frak m}^b, \overline{\frak m}^f)=((0,{3\over 2}),z_\infty)$}{Marked point in $\overline{W}_-$}
where $\overline{\frak m}^b=(0,{3\over 2})\in B_-$ and $\overline{\frak m}^f= z_\infty\in \overline{S}$. The marked point will play a crucial role in the definition of
the chain map $\Psi$.

\subsection{Lagrangian boundary conditions}

Let us first write
$$L_{\bf a}=\R\times\{1\}\times {\bf a} \quad \mbox{ and } \quad L_{\hh({\bf a})}=\R\times \{0\}\times \hh({\bf a})$$
for the Lagrangians in $W$ that are used in the definition of $\widehat{CF}(S,{\bf a},\hh({\bf a}))$.
\nom[L01]{$L_{\bf a}$, $L_{\hh({\bf a})}$}{Lagrangian submanifolds of $W$ used in the definition of $\widehat{CF}(S,{\bf a},\hh({\bf a}))$}

\subsubsection{Lagrangian boundary conditions for $W_\pm$} \label{subsubsection: Lagrangian}

The symplectic fibration
$$\pi_{B_+}: (W_+,\Omega_+)\to (B_+,ds\wedge dt)$$
admits a {\it symplectic connection}, defined as the $\Omega_+$-orthogonal of the tangent
plane to the fibers. The symplectic connection is spanned by
$\partial_s$ and $\partial_t$ if we consider $\Omega=ds\wedge
dt+\omega$ on $\R\times S\times[0,2]$ before the identification
$(s,x,2)\sim (s,\hh(x),0)$. (We recall that $W_+$ is defined as a subset
of $\R\times N = \R \times (S\times[0,2])/\sim$.)

We first place a copy of the basis ${\bf a}$ on the fiber $\pi_{B_+}^{-1}
(3,1)$ and take its parallel transport along $\partial B_+$ using
the symplectic connection. The parallel transport sweeps out a
Lagrangian submanifold $L^+_{\bf a}$
\nom[L1]{$L^+_{\bf a}$, $L^+_{a_i}$}{Lagrangian submanifold of $W_+$ constructed from ${\bf a}$ and a connected component of $L^+_{\bf a}$ corresponding to the arc $a_i$}
of $(W_+,\Omega_+)$. Let $L^+_{a_i}$ be the connected component of $L^+_{\bf a}$ given by parallel transport of $a_i$. Since the symplectic connection is spanned by $\partial_s$ and $\partial_t$ on $\R\times S\times[0,2]$, over the strip
$\{ s\geq 3, t\in [0,1 ] \}$ we have:
$$L^+_{\bf a} \cap \{s\geq 3, t=0\} = \{ s\geq 3\} \times
\hh({\bf a}) \times \{0\},$$
$$L^+_{\bf a} \cap \{s\geq 3, t=1\} =\{ s\geq 3\} \times
{\bf a} \times \{ 1\}.$$

Similarly, the Lagrangian submanifold $L^-_{\bf a}$ on the vertical boundary of $(W_-,\Omega_-)$ is obtained by taking the parallel transport of a copy of ${\bf a}$
--- placed on the fiber $\pi_{B_-}^{-1}(-3,1)$ --- by the symplectic connection.

\subsubsection{Extended Lagrangian boundary conditions}
\label{coconut}

\begin{figure}
\includegraphics[width=4cm]{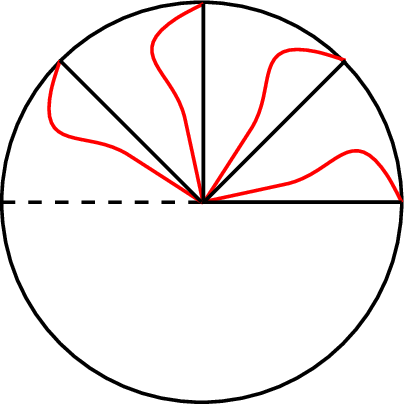}
\caption{Extended arcs in $D^2$. The black arcs are portions of the $\overline{a}_i$
and the red ones are portions of the $\overline{\hh}(\overline{a}_i)$. The dashed arc is
one of the $\vec{a}_{i,j} - \overline{a}_i$.}
\label{fig: around_z_infty}
\end{figure}

Fix $k_0 > 2g$. In what follows we assume that the basis arcs $a_i$, $i=1,\dots,2g$, depend on
an integer $m \gg 0$ (i.e., sufficiently large that the conditions below make sense) and satisfy some additional conditions. Let $E\subset \bdry D^2$ be the set of endpoints of $\cup_{i=1,\dots,2g} a_i$ and let $y_1(m),\dots,y_{4g}(m)$ be the points of $E$ in counterclockwise order. We denote by $\phi(y_i(m))$ the $\phi$-coordinate
of the endpoint $y_i(m)$. Then we assume the following:
\begin{itemize}
\item $a_i=a_i(m)$ is oriented;
\item $0<\phi(y_1(m))<{2\pi \over m}$;
\item for $1\leq j_1,j_2\leq 4g$,
$$\phi^{j_1,j_2}(m):=\phi(y_{j_1}(m))-\phi(y_{j_2}(m))$$ is an integer multiple of ${2\pi \over m}$;
\item ${\phi^{j_1,j_2}(m)\cdot m\over 2\pi}\geq k_0$ for all $j_1>j_2$, i.e., the spacing between any two $\phi(y_{j_1}(m))$ and $\phi(y_{j_2}(m))$ with $j_1>j_2$ is at least $k_0$ times $\frac{2\pi}{m}$;
\item $\lim \limits_{m\to +\infty} {\phi^{j_1,j_2}(m)\cdot m\over 2\pi}=+ \infty$ and $\lim \limits_{m\to +\infty}\phi^{j_1,j_2}(m)=0$;
\item for all quadruples $j_1>j_2$, $j_3>j_4$, $\lim \limits_{m\to +\infty} {\phi^{j_1,j_2}(m)\over \phi^{j_3,j_4}(m)}\not=1$ unless $(j_1,j_2)=(j_3,j_4)$.
\end{itemize}
Assume without loss of generality that the initial point of $a_i$ is $x_i$ and the terminal point of $a_i$ is $x_i'$. Let $\overline{a}_i$
\nom[a5 ]{$\overline{a}_i$}{Extension of $a_i$ to $\overline{S}$}
be the (oriented) extension of $a_i\subset S$ to $\overline{S}$, obtained by attaching two radial rays $\overline{a}_{i,j}=\{0\leq \rho\leq 1, \phi=\phi_{i,j}\}$, $j=0,1$, \nom[a5 ]{$\overline{a}_{i,j}$ (for $j=0,1$)}{Rays from $x_i$ and $x_i'$ to $z_\infty$ in $\overline{a}_i - a_i$}
where $\phi_{i,j}$ is a constant. Here  $\overline{a}_{i,0}$ (resp.\ $\overline{a}_{i,1}$) is the initial (resp.\ terminal) segment of $\overline{a}_i$.
We also define the extension $\vec{a}_{i,j}=\{-1< \rho\leq 1, \phi=\phi_{i,j}\}$, $j=0,1$,
\nom[a5 ]{$\vec{a}_{i,j}$ (for $j=0,1$)}{Extension of $\overline{a}_{i,j}$ past $z_\infty$}
of $\overline{a}_{i,j}=\{0\leq \rho\leq 1, \phi=\phi_{i,j}\}.$  Here $(\rho, \phi)$ with $\rho <0$ is a notation for $(- \rho, \phi + \pi)$.

We write $\overline{\mathbf{a}}=\{ \overline{a}_1,\dots,\overline{a}_{2g}\}$.
\nom[a5 ]{$\overline{\mathbf{a}}$}{Collection $\{ \overline{a}_1, \dots, \overline{a}_{2g}\}$}
Then $L^\pm_{\overline{\mathbf{a}}}$
\nom[L1]{$L^-_{\overline{\mathbf{a}}}$, $L^-_{\overline{a}_i}$}{Singular Lagrangian in $\overline{W}_-$ constructed from $\overline{{\mathbf{a}}}$ and a connected component of $L^-_{\overline{\mathbf{a}}}$ corresponding to the arc $\overline{a}_i$}
is the extension of $L^\pm_{\mathbf{a}}$, obtained by the parallel transport of a copy of $\overline{\mathbf{a}}$, placed at $\overline\pi^{-1}_{B_+}(3,1)$ or $\overline\pi^{-1}_{B_-}(-3,1)$. We similarly define  $L^\pm_{\widehat{\mathbf{a}}}$, $L^\pm_{\widehat{a}_i}$, $L^\pm_{\overline{a}_i}$, $L^\pm_{\vec{a}_{i,j}}$, and $L^\pm_{\overline{a}_i\cup\vec{a}_{i,j}}$,
\nom[L1]{$L^\pm_{\widehat{a}_i}$, $L^\pm_{\overline{a}_i}$, $L^\pm_{\vec{a}_{i,j}}$, $L^\pm_{\overline{a}_i\cup\vec{a}_{i,j}}$}{Lagrangian submanifolds of
$\overline{W}_\pm$ constructed from the arcs $\widehat{a}_i$, $\overline{a}_i$, $\vec{a}_{i,j}$ and $\overline{a}_i\cup\vec{a}_{i,j}$}
where $\widehat{a}_i= \overline{a}_i-\{z_\infty\}$
\nom[a5]{$\widehat{a}_i$}{$\overline{a}_i-\{z_\infty\}$}
and $\widehat{\mathbf{a}} = (\widehat{a}_1, \ldots , \widehat{a}_{2g})$.
\nom[a5]{$\widehat{\mathbf{a}}$}{Collection $\{ \widehat{a}_1,\dots,\widehat{a}_{2g}\}$}

\begin{defn}
A bigon (with acute angles) contained in $D^2$ and bounded by $\overline{a}_{i,j}$ and $\overline{\hh}(\overline{a}_{i,j})$ will be called a {\em thin strip}. The portion of a thin strip contained in $D_{1/2} = \{ \rho \le \frac 12 \}$ will be called a {\em thin wedge}. See Figure~\ref{fig: around_z_infty}.
\end{defn}

\subsection{Almost complex structures and moduli spaces for $W$, $\overline{W}$, $W'$, and $\overline{W'}$}
\label{subsection: acstr on W etc}

In this subsection we specify the almost complex structures and moduli spaces for
$W=\R \times [0,1] \times S$, $\overline{W} = \R \times [0,1] \times \overline{S}$, $W'=\R\times N$, and $\overline{W'}=\R\times\overline{N}$.

\begin{convention} \label{convention for gamma and bf y}
When we write ${\bf y}$ or $\boldsymbol{\gamma}$ (with possible superscripts, subscripts and other decorations), it is assumed that every intersection point in ${\bf y}$ is contained in $S$ and every orbit in $\boldsymbol{\gamma}$ is contained in $N$. (By abuse of notation, we will write ${\bf y}\subset S$ and $\boldsymbol{\gamma}\subset N$.)  In particular, ${\bf y}$ and $\boldsymbol{\gamma}$ do not contain any multiples of $z_\infty$ or $\delta_0$.
\end{convention}

\nom[y]{${\bf y}$}{A tuple of ${\bf a}\cap \hh({\bf a})$; from Section~\ref{subsection: acstr on W etc} onwards, ${\bf y}\subset S$ and does not contain multiples of $z_\infty$}
\nom[1c$\boldsymbol{\gamma}$]{$\boldsymbol{\gamma}$}{An orbit set; from Section~\ref{subsection: acstr on W etc} onwards, $\boldsymbol{\gamma}\subset N$ and does not contain multiples of $\delta_0$}

\subsubsection{Holomorphic maps to $W$ and $\overline{W}$}

Let $W=\R\times[0,1]\times S$ and $\overline{W}=\R\times[0,1]\times \overline{S}$. Also let $\Omega=ds\wedge dt+\omega$ and $\overline\Omega=ds\wedge dt+\overline{\omega}$. Then $\mathcal{J}_W$ is defined as the set of $C^\infty$-smooth $\Omega$-admissible almost complex structures on $W$.
\nom[JsetW]{$\mathcal{J}_W$}{Space of $C^\infty$-smooth $\Omega$-admissible almost complex structures on $W$}

The analogous space $\mathcal{J}_{\overline{W}}$ for $\overline{W}$ is slightly more complicated:

\begin{defn} \label{defn: almost complex structures on overline W}
Fix $\varepsilon,\delta>0$ sufficiently small and $k\gg 0$. Then $\mathcal{J}_{\overline{W}}$ is the space of $C^\infty$-smooth $\overline\Omega$-admissible almost complex structures $\overline{J}$ on $\overline{W}$ which satisfy the following:
\begin{enumerate}
\item there exists an $\overline\Omega$-admissible almost complex structure $\overline{J}_0$ which restricts to the standard complex structure on the subsurface $D^2\subset \overline{S}$ of each fiber;
\nom[J2]{$\overline{J}_0$}{Almost complex structure on $\overline{W}$ which restricts to the standard one on $D^2$ on each fiber}
\item $\overline{J}=\overline{J}_0$ on $\{\rho\leq \varepsilon\}\cup W$ and $|\overline{J}-\overline{J}_0|_k<\delta$ over $\overline{W}$.
\end{enumerate}
Here $|\cdot|_k$ is some fixed $C^k$-norm.
\end{defn}

\nom[JsetWo]{$\mathcal{J}_{\overline{W}}$}{Space of $C^\infty$-smooth $\overline\Omega$-admissible almost complex structures $\overline{J}$ on $\overline{W}$ satisfying Definition~\ref{defn: almost complex structures on overline W}}
We denote by $J$ the restriction of $\overline{J}$ to $W$.
\nom[J31]{$J$}{Almost complex structure on $W$ that is the restriction of $\overline{J}$}
\nom[J32]{$J_+$, $J_-$}{Almost complex structures on $W_+$, $W_-$ that are restrictions of $\overline{J}_+$, $\overline{J}_-$}
\nom[J33]{$\overline{J}$, $\overline{J'}$, $\overline{J}_+$, $\overline{J}_-$}{Almost complex structures on $\overline{W}$, $\overline{W'}$, $\overline{W}_+$, $\overline{W}_-$ that restrict to $J$, $J'$, $J_+$, $J_-$ on $W$, $W'$, $W_+$, $W_-$}
Let $\mathcal{S}_{{\bf a},\hh({\bf a})}$ be the set of $2g$-tuples of intersection points of ${\bf a}\cap \hh({\bf a})$.
Given ${\bf y}, {\bf y'}\in \mathcal{S}_{{\bf a},\hh({\bf a})}$ and $J\in\mathcal{J}_W$, let $\mathcal{M}_J({\bf y},{\bf y'})$ be the moduli space of HF curves from ${\bf y}$ to ${\bf y'}$ with respect to $J$ which are contained in $\R\times[0,1]\times S$.

We next discuss holomorphic curves in $\overline{W}$.
\begin{defn}
Let ${\bf y}, {\bf y'}$ be $k$-tuples of ${\bf a}\cap \hh({\bf a})$ and $\overline{J}\in\mathcal{J}_{\overline{W}}$.  Then a {\em degree $k\leq 2g$ multisection $\overline{u}$ from ${\bf y}$ to ${\bf y'}$ in $(\overline{W},\overline{J})$} is a holomorphic map $\overline{u}:\dot F\to \overline{W}$ which is a degree $k$ multisection of $\overline\pi_B: \overline{W}\to B$, satisfies the conditions of Definition~\ref{HF-curve} with $L_{\boldsymbol{\alpha}}$ and $L_{\boldsymbol{\beta}}$ replaced by $L_{\widehat{\bf a}}=\R\times \{1\}\times \widehat{\bf a}$ and $L_{\overline{\hh}(\widehat{\bf a})}=\R\times \{0\}\times\overline{\hh} (\widehat{\bf a})$ and $(k,l)$ replaced by $(2g,k)$, and is asymptotic to ${\bf y}$ and ${\bf y'}$ at the positive and negative ends.
\end{defn}

Note that we require $\bdry \dot F$ to be mapped to $L_{\widehat{\bf a}}$, $L_{\overline{\hh}(\widehat{\bf a})}$ and be disjoint from $L_{\overline{\bf a}}-L_{\widehat{\bf a}}$, $L_{\overline{\hh}(\overline{\bf a})}-L_{\overline{\hh}(\widehat{\bf a})}$.

\begin{defn}
The {\em section at infinity} $\sigma_\infty$ is the map $\sigma_{\infty}:\R \times [0,1] \to \overline{W}$ such that $\sigma_\infty (s,t) = (s,t,z_{\infty})$. By abuse of notation we will also refer to its image as the section at $\infty$.
\end{defn}
\nom[1r$\sigma$]{$\sigma_\infty$, $\sigma_\infty'$, $\sigma_\infty^+$, $\sigma_\infty^-$}{Sections at infinity of $\overline{W}$, $\overline{W'}$, $\overline{W}_+$, $\overline{W}_-$}
The section at infinity is $\overline{J}$-holomorphic for every $\overline{J} \in \mathcal{J}_{\overline{W}}$.
Let $z^\dagger_\infty \in D^2_{1/2} \subset \overline{S}$ be a point with $\rho$-coordinate less than $\frac{1}{2}$ and in the complement of all arcs $\vec{a}_{i,j}$ and $\overline{\hh}(\vec{a}_{i,j})$. The orbit $\mathcal{O}(z^\dagger_\infty)$ of $z^\dagger_\infty$ under the action of $\overline{\hh}$ consists of $m$ points, and each thin wedge in $D^2_{1/2}$ between $\overline{a}_{i,j}$ and $\overline{\hh}(\overline{a}_{i,j})$ contains exactly one point of $\mathcal{O}(z^\dagger_\infty)$.
\nom[1r$\sigma$]{$\sigma_\infty^\dagger$, $(\sigma_\infty')^\dagger$, $(\sigma_\infty^+)^\dagger$, $(\sigma_\infty^-)^\dagger$}{``Pushoff'' of $m$-fold cover of section at $\infty$ for $\overline{W}$, $\overline{W'}$, $\overline{W}_+$, $\overline{W}_-$}
Let $\sigma^\dagger_\infty=\R\times[0,1]\times\mathcal{O}(z^\dagger_\infty)$. (Note that this multisection does not have Lagrangian boundary conditions; it will only be used to impose topological constraints on the $HF$ curves.)

\begin{defn}
Given a degree $k$ multisection $\overline{u}$ of $\overline{W}$, we define $n(\overline{u}) = \langle \overline{u}, \sigma^\dagger_\infty \rangle$, where $\langle \cdot, \cdot \rangle$ denotes the algebraic intersection between the images.
\end{defn}

The intersection number $n(\overline{u})$ is a homological invariant since $z^\dagger_\infty$ was chosen so that $\mathcal{O}(z^\dagger_\infty)$ is disjoint from the Lagrangian arcs.

\begin{lemma}\label{pisapia}
The intersection number $n(\overline{u})$ satisfies the following properties:
\begin{enumerate}
\item $n(\overline{u}) \ge 0$ and $n(\overline{u}) = 0$ if and only if the image of $\overline{u}$ is disjoint from $\sigma^\dagger_\infty$;
\item if $\langle \overline{u}, \sigma_{\infty} \rangle >0$, then $n(\overline{u}) \ge m$;
\item if $n(\overline{u}) = 1$, then $\overline{u}$ projects onto a thin strip; and
\item $n(\overline{u})$ is independent of the choice of $z^\dagger_\infty$.
\end{enumerate}
\end{lemma}

\nom[n1]{$n(\overline{u})$, $n'(\overline{u})$, $n^+(\overline{u})$, $n^-(\overline{u})$}{Intersection of a curve $\overline{u}$ in $\overline{W}$, $\overline{W'}$, $\overline{W}_+$, $\overline{W}_-$ with the ``pushoff'' of the section at $\infty$}

\begin{proof}
(1) follows from the positivity of intersections of pseudo-holomorphic maps in dimension four, (2) is a consequence of the fact that holomorphic maps are open, and (3) is a consequence of the positivity of intersections and of the fact that every thin wedge contains only one point of $\mathcal{O}(z^\dagger_\infty)$. Finally, (4) follows from the fact that different choices for $\sigma^\dagger_\infty$ are connected by a path of multisections which are disjoint from the Lagrangian boundary conditions.
\end{proof}

\begin{defn}
Let ${\bf y}$ and ${\bf y'}$ be $k$-tuples of ${\bf a}\cap \hh({\bf a})$.  Then $\mathcal{M}_{\overline{J}}({\bf y},{\bf y'})$ is the moduli space of degree $k$ multisections $\overline{u}$ of $(\overline{W},\overline{J})$ from ${\bf y}$ to ${\bf y'}$.
\end{defn}

\nom[M3]{$\mathcal{M}_{\overline{J}}({\bf y},{\bf y'})$}{Moduli space of Heegaard Floer curves in $(\overline{W},\overline{J})$ from ${\bf y}$ to ${\bf y'}$}

\n
{\em Modifiers.} For any moduli space $\mathcal{M}_{\star_1}(\star_2)$ we may place modifiers $*$ as in $\mathcal{M}^*_{\star_1}(\star_2)$ to denote the subset of $\mathcal{M}_{\star_1}(\star_2)$ satisfying $*$.  Typical self-explanatory modifiers are $I=i$, $n=m$, and $\op{deg}=k$. (Note however that the degree can be inferred from $\star_2$.)
$*=s$ means ``somewhere injective''.
\nom[2s]{$*=s$}{Modifier ``somewhere injective''}

\s
The following lemma is an easy consequence of Lemma~\ref{pisapia}, in particular of points (1) and (4).

\begin{lemma}\label{pippo}
$\mathcal{M}_{\overline{J}}^{n=0}({\bf y},{\bf y'})=\mathcal{M}_J({\bf y}, {\bf y'})$ if ${\bf y}, {\bf y'}\in \mathcal{S}_{\mathbf{a}, \hh(\mathbf{a})}$.
\end{lemma}

\subsubsection{Holomorphic maps to $W'$ and $\overline{W'}$}
\label{acorn2}

Recall the stable Hamiltonian structures $(\alpha_0, \omega)$ on $N$ and $(\overline{\alpha}_0, \overline{\omega})$ on $\overline{N}$. We will denote
by $R_0$ the Hamiltonian vector field of $(\alpha_0, \omega)$ on $N$ and by
$\overline{R}_0$ the Hamiltonian vector field of $(\overline{\alpha}_0, \overline{\omega})$ on $\overline{N}$. Clearly $R_0$ is the restriction of $\overline{R}_0$ to $N$.
\nom[R]{$\overline{R}_0$, $R_0$}{Hamiltonian vector field on $\overline{N}$ and its restriction to $N$}
Since $\partial N$ is a Morse-Bott torus of $\overline{R}_0$, we briefly review Morse-Bott theory in the context of periodic Floer homology. A more extensive reference is \cite[Section~4]{CGH2}, where a completely analogous discussion is given for embedded contact homology.

While the closed orbits of $\overline{R}_0$ are nondegenerate on $int(N) \cup int(V) = \overline{N} - \partial N$, the orbits on $\partial N$ form a {\em negative Morse-Bott torus}, i.e., $\partial N$ is foliated by closed orbits whose linearized first return map at any point $p \in \partial N$ is given by a matrix
$\left ( \begin{matrix} 1 & 0 \\ a & 1 \end{matrix} \right )$
with $a<0$  with respect to an oriented basis $(v_1, v_2)$ of the tangent space to the fiber of the fibration $\overline{N} \to S^1$ passing through $p$. In this basis $v_1$
is transverse to $\partial N$ and $v_2$ is tangent to $\bdry N$.

We denote the Morse-Bott orbit family associated to $\partial N$ by $\mathcal{N}$, i.e., ${\mathcal N}$ is the quotient of $\partial N$ by the flow.  (Previously we had used $\mathcal{N}$ for the Morse-Bott family for the Reeb vector field $R_\alpha$ on $\bdry N$; for the rest of the paper $R_\alpha$ is replaced by $\overline{R}_0$.)  We choose a Morse function $\overline{g}: {\mathcal N} \to \R$ with two critical points: a maximum $h$ and a minimum $e$. After perturbing the Hamiltonian vector field by $\overline{g}$, the orbits $e$ and $h$ become nondegenerate elliptic and hyperbolic, respectively. For this reason, in constructing orbit sets, $e$ will be treated as an elliptic orbit and $h$ as a hyperbolic orbit.

We define the following set of orbits. Let $\mathcal{P}$ be the set of simple orbits of $R_0$ in $int(N)$. All orbits in ${\mathcal P}$ are nondegenerate. We also define
$$\widehat{\mathcal{P}}=\mathcal{P}\cup\{e,h\} \quad \mbox{and} \quad \overline{\mathcal{P}}=\widehat{\mathcal{P}}\cup\{\delta_0\},$$
together
\nom[P2]{$\overline{\mathcal{P}}$}{$\widehat{\mathcal{P}}\cup\{\delta_0\}$}
with the sets $\widehat{\mathcal{O}}_k$,  $\overline{\mathcal{O}}_k$ of orbit sets constructed respectively from $\widehat{\mathcal{P}}$ and $\overline{\mathcal{P}}$ which intersect $\overline{S}\times\{0\}$ exactly $k$ times.
\nom[O]{$\widehat{\mathcal{O}_k}$, $\overline{\mathcal{O}}_k$}{Orbit sets constructed from $\widehat{\mathcal{P}}$ and $\overline{\mathcal{P}}$ which intersect $\overline{S}\times\{0\}$ exactly $k$ times}

Let $\mathcal{J}_{W'}$ be the space of $C^\infty$-smooth $(\alpha_0,\omega)$-adapted almost complex structures on $W'=\R\times N$.
\nom[JsetWp]{$\mathcal{J}_{W'}$}{Space of $C^\infty$-smooth $(\alpha_0,\omega)$-adapted almost complex structures on $W'$}
A $\mathcal{J}_{W'}$-holomorphic Morse-Bott building is one of the following objects:
\begin{enumerate}
\item[(i)] a union of negative gradient flow trajectories from $h$ to $e$, or
\item[(ii)] a $\mathcal{J}_{W'}$-holomorphic map $u : F \to W'$ from an orbit set $\tilde{\boldsymbol{\gamma}}_+ \in {\mathcal P} \cup {\mathcal N}$ to an orbit set $\tilde{\boldsymbol{\gamma}}_-  \in {\mathcal P} \cup {\mathcal N}$, together with half-infinite negative gradient flow trajectories connecting $h$ to positive ends of $u$ at ${\mathcal N}- \{ e,h \}$ and negative ends of $u$ at ${\mathcal N}- \{ e,h \}$ to $e$, or
\item[(iii)] a union of (i) and (ii).
\end{enumerate}
See \cite{CGH2} for more details.
\begin{defn} \label{Arianna N}
Given $\boldsymbol{\gamma},\boldsymbol{\gamma}'\in \widehat{\mathcal{O}}_k$ and $J'\in \mathcal{J}_{W'}$, let $\mathcal{M}_{J'}(\boldsymbol{\gamma}, \boldsymbol{\gamma}')$ be the moduli space of $J'$-holomorphic Morse-Bott buildings in $W'$ from $\boldsymbol{\gamma}$ to $\boldsymbol{\gamma}'$.
\end{defn}
\nom[M4]{$\mathcal{M}_{J'}(\boldsymbol{\gamma},\boldsymbol{\gamma}')$}{Moduli space of $J'$-holomorphic curves in $W'= \R \times N$ from $\boldsymbol{\gamma}$ to $\boldsymbol{\gamma}'$}

\begin{rmk}\label{rmk: Morse-Bott for W'}
The stable Hamiltonian structure $(\alpha_0, \omega)$ can be perturbed to a  stable Hamiltonian structure $(\alpha_0', \omega')$ as described in Section~\ref{subsection: Morse-Bott}; the resulting Hamiltonian vector field has two nondegenerate orbits corresponding to $e$ and $h$ (and denoted by the same names) instead of the Morse-Bott torus ${\mathcal N}$. Any almost complex structure $J' \in {\mathcal J}_{W'}$ is then perturbed to an almost complex structure adapted to $(\alpha_0', \omega')$.

Because of the simple nature of the $J'$-holomorphic Morse-Bott buildings, and since  $\widehat{\mathcal O}_k$ for $k \le 2g$ are finite sets, for a sufficiently small
perturbation $(\alpha_0', \omega')$ of $(\alpha_0, \omega)$ and for every $\boldsymbol{\gamma}, \boldsymbol{\gamma}' \in {\mathcal O}_k$, there is a bijection between the ECH index one moduli spaces  $\mathcal{M}_{J'}^{I=1}(\boldsymbol{\gamma}, \boldsymbol{\gamma}')$ and the index  one moduli spaces of holomorphic maps from $\boldsymbol{\gamma}$ to $\boldsymbol{\gamma}'$ for the perturbation of $J'$ adapted to $(\alpha_0', \omega')$. See
\cite[Lemma~4.4.5]{CGH2} and \cite[Lemma~7.1.2]{CGH2}.

Since the Morse-Bott situation and its generic perturbation are completely equivalent,
we will go from one point of view to the other as more convenient, often without explicit mention. For this reason, {\em we will still denote by $J'$ a perturbation of $J'$ adapted to $(\alpha_0', \omega')$, and by $\mathcal{M}_{J'}(\boldsymbol{\gamma}, \boldsymbol{\gamma}')$ the moduli spaces of holomorphic maps with respect to the perturbation of $J'$.}
\end{rmk}

The definition of the analogous space $\mathcal{J}_{\overline{W'}}$ for $\overline{W'}=\R\times\overline{N}$ is slightly more complicated, but completely analogous to Definition \ref{defn: almost complex structures on overline W}. We denote $D^2_\varepsilon = \{ |z| \le \varepsilon \} \subset \C$ and, for $\varepsilon \le 1$, we define $V_\varepsilon$ the mapping torus of $\hh|_{D^2_\varepsilon}$. For $\varepsilon =1$ this coincides with the already defined solid torus $V$.

\begin{defn} \label{defn: almost complex structures on R times overline N MB version}
Fix $\varepsilon,\delta>0$ sufficiently small and $k\gg 0$. Then $\mathcal{J}_{\overline{W'}}^{MB}$ is the space of $C^\infty$-smooth $(\overline\alpha_0,\overline\omega)$-adapted almost complex structures $\overline{J'}$ on $\R\times \overline{N}$ which satisfy the following:
\begin{enumerate}
\item there exists an $(\overline\alpha_0,\overline\omega)$-adapted almost complex structure $\overline{J'_0}$ which restricts to the standard complex structure on the subsurface $D^2\subset \overline{S}$ of each fiber;
\item $\overline{J'}=\overline{J'_0}$ on $\R\times (V_\varepsilon \cup N)$ and $|\overline{J'}-\overline{J'_0}|_k<\delta$ over $\R\times \overline{N}$.
\end{enumerate}
Here $|\cdot|_k$ is some fixed $C^k$-norm.
\end{defn}

The Morse-Bott theory for $\overline{W'}$ is more subtle than that of $W'$, and therefore we will discuss it in more detail. In the following definition, all $\overline{J}'$-holomorphic maps are allowed to be nodal and to have disconnected domains. Let $\boldsymbol{\gamma}_+ = \gamma_{+,1}^{k_1} \ldots  \gamma_{+,{m_+}}^{k_{m_+}}h^a$ be an orbit set such that $\gamma_{+,i}\in \mathcal{P}\cup\{e,\delta_0\}$ and $a \in \{0,1\}$. Let $\boldsymbol{\gamma}_- = \gamma_{-,1}^{l_1} \ldots  \gamma_{-,{m_-}}^{l_{m_-}}e^b$ be an orbit set such that $\gamma_{-,i}\in \mathcal{P}\cup\{h,\delta_0\}$ and $b \ge 0$.

\begin{defn}[Morse-Bott holomorphic buildings] \label{defn: Morse-Bott buildings}
Let $\overline{J'}$ be an almost complex structure in $\mathcal{J}_{\overline{W'}}^{MB}$ and let $\boldsymbol{\gamma}_+ =  \gamma_{+,1}^{k_1} \ldots  \gamma_{+,{m_+}}^{k_{m_+}} h^a$ and $\boldsymbol{\gamma}_- = \gamma_{-,1}^{l_1} \ldots  \gamma_{-,m_-}^{l_{m_-}}e^b$ be orbit sets as above. A $\overline{J'}$-holomorphic Morse-Bott building $u$ from $\boldsymbol{\gamma}$ to $\boldsymbol{\gamma}'$ consists of
\begin{enumerate}
\item a (possibly empty) set of negative gradient flow trajectories from
$h$ to $e$ in ${\mathcal N}$, and
\item a (possibly empty) set of $\overline{J'}$-holomorphic maps $u_i: F_i \to \overline{W'}$, for $i= 1, \ldots , n$, such that the following hold:
\begin{enumerate}
\item Positive ends of $u_n$ converge to multiples of $\gamma_{+,i}$ with total multiplicity $k_i$ and to a simple orbit in ${\mathcal N} - \{e \}$ if $a\ne 0$.
\item For $i=1, \ldots, n-1$, the negative ends of $u_{i+1}$ are paired with the positive ends of $u_{i}$. For each pair, the two ends either converge to the same orbit, or they converge to covers of orbits in ${\mathcal N}- \{ e,h \}$ with the same multiplicity and there is a negative gradient flow line of $\overline{g}$ in ${\mathcal N}$ from the limit of $u_{i+1}$ to the limit of $u_{i}$. 
\item Negative ends of $u_1$ converge to multiples of $\gamma_{-,i}$ with total multiplicity $l_i$ and to multiple covers of orbits in ${\mathcal N} - \{ h \}$ with total multiplicity $b$.
\end{enumerate}
\end{enumerate}
\end{defn}

The stable Hamiltonian structure $(\overline{\alpha}_0, \overline{\omega})$ can be perturbed as described in Section~\ref{subsection: Morse-Bott}
--- in fact, $(N^+, \alpha_0, \omega)$ embeds in $(\overline{N}, \overline{\alpha}_0,
\overline{\omega})$. The perturbed stable Hamiltonian structure obtained in this way will be called a {\em nondegenerate perturbation of $(\overline{\alpha}_0,
\overline{\omega})$.} After the perturbation, the Morse-Bott torus ${\mathcal N}$ is replaced by two non-degenerate orbits corresponding to $e$ and $h$ (and still denoted by the same name) and all orbits in $\overline{N} - \partial N$ which intersect a fiber at most $2g$ times remain unchanged.

We fix a nondegenerate perturbation $(\overline{\alpha}_0',
\overline{\omega}')$ of $(\overline{\alpha}_0, \overline{\omega})$ once and for all.
We define the following space of almost complex structures on $\overline{W'}$ in analogy
to Definition~\ref{defn: almost complex structures on R times overline N MB version}.
\begin{defn} \label{defn: almost complex structures on R times overline N}
Fix $\varepsilon,\delta>0$ sufficiently small and $k\gg 0$. Then $\mathcal{J}_{\overline{W'}}$ is the space of $C^\infty$-smooth $(\overline{\alpha}_0',\overline{\omega}')$-adapted almost complex structures $\overline{J'}$ on $\R\times \overline{N}$ which satisfy the following:
\begin{enumerate}
\item there exists an $(\overline{\alpha}_0',\overline{\omega}')$-adapted almost complex structure $\overline{J'_0}$ which restricts to the standard complex structure on the subsurface $D^2\subset \overline{S}$ of each fiber;
\item $\overline{J'}=\overline{J'_0}$ on $\R\times (V_\varepsilon \cup N)$ and $|\overline{J'}-\overline{J'_0}|_k<\delta$ over $\R\times \overline{N}$;
\item every $\overline{J'}$-holomorphic map $u : F \to \overline{W'}$ connecting two orbit sets in $\overline{\mathcal{O}}_k$, for $k \le 2g$, is close to breaking to a holomorphic Morse-Bott building for some almost-complex structure in $\mathcal{J}_{\overline{W'}}^{MB}$.
\end{enumerate}
Here $|\cdot|_k$ is some fixed $C^k$-norm.
\end{defn}
 \nom[JsetWpo]{$\mathcal{J}_{\overline{W'}}$}{Space of $C^\infty$-smooth $(\overline\alpha_0,\overline\omega)$-adapted almost complex structures on $\overline{W'}$ satisfying Definition~\ref{defn: almost complex structures on R times overline N}}

\begin{rmk}
Condition (3) in Definition~\ref{defn: almost complex structures on R times overline N} will
be used to constrain the topological behavior of $\overline{J'}$-holomorphic maps: to show that a map has a certain behavior, we will show that the corresponding Morse-Bott building does, which is often simpler.
\end{rmk}

\begin{defn} \label{Arianna overline N}
Given $\delta_0^p\boldsymbol{\gamma}, \delta_0^q\boldsymbol{\gamma}'\in \overline{\mathcal{O}}_k$ (with $p, q \ge 0$ and $k \le 2g$) and $\overline{J'}\in \mathcal{J}_{\overline{W'}}$, let $\mathcal{M}_{\overline{J'}}(\delta_0^p\boldsymbol{\gamma}, \delta_0^q\boldsymbol{\gamma}')$ be the moduli space of $\overline{J'}$-holomorphic maps in $\overline{W'}$ from $\delta_0^p\boldsymbol{\gamma}$ to $\delta_0^q\boldsymbol{\gamma}'$ without fiber components.
\end{defn}
\nom[M5]{$\mathcal{M}_{\overline{J'}}(\delta_0^p\boldsymbol{\gamma},\delta_0^q\boldsymbol{\gamma}')$}{Moduli space of $(\overline{W'},\overline{J'})$-holomorphic maps from $\delta_0^p\boldsymbol{\gamma}$ to $\delta_0^q\boldsymbol{\gamma}'$ without fiber components}

Maps in $\mathcal{M}_{\overline{J'}}(\delta_0^p\boldsymbol{\gamma}, \delta_0^q\boldsymbol{\gamma}')$, with $\delta_0^p\boldsymbol{\gamma}, \delta_0^q\boldsymbol{\gamma}'\in \overline{\mathcal{O}}_k$, will be called
{\em degree $k$ PFH curve} (or {\em degree $k$ $\overline{W'}$-curve}). A degree $k$ PFH curve can be viewed as a degree $k$ holomorphic multisection of the holomorphic fibration $\R \times \overline{N} \to \R \times S^1$. A degree $2g$ PFH curve will also be called a
{\em PFH curve} (or {\em $\overline{W'}$-curve}).

\begin{defn}
The {\em section at infinity} $\sigma_\infty'$ in $\R \times \overline{N}$ is $\R\times\delta_0$.
\end{defn}

Choose a point $z^\dagger_\infty \in \overline{S}$ which is sufficiently close to $z_{\infty}$. Then there is a periodic orbit $\delta_0^\dagger$ of $\overline{R}_0$ with period $m$ which passes through $z^\dagger_\infty$. We then write $(\sigma'_\infty)^\dagger=\R\times \delta_0^\dagger$. Both $\sigma_\infty'$ and $(\sigma'_\infty)^\dagger$ are $\overline{J'}$-holomorphic sections for every $\overline{J'} \in \mathcal{J}_{\overline{W'}}$.

\begin{defn}\label{defn: n for PFH}
Given a degree $k$ PFH curve $\overline{u}$ in $\overline{W'}$, we define $n'(\overline{u}) = \langle \overline{u}, (\sigma'_\infty)^\dagger \rangle.$
\end{defn}

The proof of the following is similar to that of Lemma~\ref{pisapia}.

\begin{lemma}\label{demagistris}
The intersection number $n'(\overline{u})$ satisfies the following properties:
\begin{enumerate}
\item $n'(\overline{u}) \ge 0$ and $n'(\overline{u}) = 0$ if and only if the image of $\overline{u}$ is disjoint from $(\sigma'_\infty)^\dagger$;
\item if $\langle \overline{u}, \R \times \delta_0 \rangle >0$, then $n'(\overline{u}) >0$; and
\item $n'(\overline{u})$ is independent of the choice of $z^\dagger_\infty$.
\end{enumerate}
\end{lemma}

\begin{lemma} \label{pluto}
If $\boldsymbol{\gamma}$ and $\boldsymbol{\gamma}'$ are contained in $N$, then $\mathcal{M}^{n'=0}_{\overline{J'}}(\boldsymbol{\gamma},\boldsymbol{\gamma}') = \mathcal{M}_{J'}(\boldsymbol{\gamma},\boldsymbol{\gamma}')$.
\end{lemma}

\begin{proof}
Let $\overline{u}$ be a $\overline{J'}$-holomorphic map in $\mathcal{M}_{\overline{J'}}(\boldsymbol{\gamma},\boldsymbol{\gamma}')$ satisfying the
constraint $n'(\overline{u})=0$.  By Condition (3) in Definition~\ref{defn: almost complex structures on R times overline N}, we can assume that $\overline{u}$ is a Morse-Bott building
to the purpose of finding homological constraints on it.

Let $T_{\rho_0}\subset V$ be the torus $\{\rho=\rho_0\}$, oriented as the boundary of the solid torus $\{\rho\leq \rho_0\}$. The torus
$T_{\rho_0}$ is foliated by orbits of the Hamiltonian vector field $\overline{R}_0$ and, for a
dense set of $\rho_0$, those orbits are closed. If $\delta_{\rho_0}$ is one such orbit, then
the homology class $[\delta_{\rho_0}]$ is a rational multiple of $[\delta_0^\dagger]$ in $H_2(\overline{N} - \boldsymbol{\gamma} \cup \boldsymbol{\gamma}')$. This implies that $\langle \overline{u}, \R \times \delta_{\rho_0} \rangle =0$ for $\rho_0$ in a dense set. The positivity of intersections then implies that the image of $\overline{u}$ is contained in $N$.
\end{proof}

\subsection{Almost complex structures and moduli spaces for $W_+$, $\overline{W}_+$, $W_-$, and $\overline{W}_-$}

\subsubsection{Almost complex structures}

\begin{defn} \label{defn: admissible J for W plus}
An almost complex structure $\overline{J}_+$ on $\overline{W}_+$ is {\em admissible} if it is the restriction to $\overline{W}_+$ of an almost complex structure $\overline{J'}\in \mathcal{J}_{\overline{W'}}$ on $\overline{W'}$. If $\overline{J}_+$ agrees with $\overline{J}$ (resp.\ $\overline{J'}$) at the positive (resp.\ negative) end, then $\overline{J}_+$ is {\em compatible} with $\overline{J}$ (resp.\ $\overline{J'}$).

An {\em admissible} almost complex structure $J_+$ on $W_+$ is the restriction of an admissible almost complex structure on $\overline{W}_+$. The admissibility criteria for $J_-$ and $\overline{J}_-$ on $W_-$ and $\overline{W}_-$ are analogous.

The space of $C^\infty$-smooth admissible almost complex structures $J_{\pm}$ on $W_{\pm}$ will be denoted by $\mathcal{J}_{W_\pm}$.
 The space of $C^\infty$-smooth admissible almost complex structures $\overline{J}_{\pm}$ on $\overline{W}_{\pm}$  will be denoted by $\mathcal{J}_{\overline{W}_\pm}^{MB}$ if $\overline{J}_{\pm}$ is the restriction of $\overline{J'} \in {\mathcal J}_{\overline{W'}}^{MB}$ and by $\mathcal{J}_{\overline{W}_\pm}$ if $\overline{J}_{\pm}$ is obtained from an almost complex structure in $\mathcal{J}_{\overline{W}_\pm}^{MB}$ by modifying its positive (resp. negative) end as in Remark~\ref{rmk: Morse-Bott for W'}.
\end{defn}

\nom[Jsetz]{$\mathcal{J}_{W_+}$, $\mathcal{J}_{W_-}$, $\mathcal{J}_{\overline{W}_+}$, $\mathcal{J}_{\overline{W}_-}$}{Spaces of $C^\infty$-smooth admissible almost complex structures on $W_+$, $W_-$, $\overline{W}_+$, $\overline{W}_-$}

\begin{rmk}
If $J_+\in \mathcal{J}_{W_+}$, then the projection $\pi_{B_+}: W_+\to B_+$ is holomorphic. This is due to the fact that the fibers $\{(s,t)\}\times S$ are holomorphic and $J_+$ takes $\bdry_s$ to $\bdry_t$. The same holds for $\overline{J}_+$, $J_-$, and $\overline{J}_-$.
\end{rmk}

\subsubsection{Moduli spaces for $W_+$}
\label{subsubsection: W plus curves}

Let $(F,j)$ be a compact Riemann surface, possibly disconnected, with an $l$-tuple of punctures ${\bf p}=\{p_1 ,...,p_l\}$ in the interior of $F$ and a $k$-tuple of punctures ${\bf q} =\{q_1,...,q_k\}$ on $\partial F$, such that (i) each component of $F$ has nonempty boundary and at least one interior puncture and (ii) each component of $\bdry F$ has at least one boundary puncture. We write $\dot{F} =F-{\bf p} -{\bf q}$ and $\bdry \dot F= \bdry F- \mathbf{q}$.

\begin{defn} \label{defn: W plus curve}
Let $J_+\in \mathcal{J}_{W_+}$, $\mathbf{y}$ be a $k$-tuple of ${\bf a}\cap \hh({\bf a})$, and $\boldsymbol{\gamma}=\prod_j \gamma_j^{m_j}$ be an orbit set in $\widehat{\mathcal{O}}_k$. Then a {\em degree $k\leq 2g$ multisection of $(W_+,J_+)$ from ${\bf y}$ to $\boldsymbol{\gamma}$} is either (i) a holomorphic map
$$u: (\dot{F} ,j) \rightarrow (W_+,J_+)$$
which is a degree $k$ multisection of $\pi_{B_+}: W_+\to B_+$ and which additionally satisfies the following:
\begin{enumerate}
\item $u(\partial \dot{F} )\subset L^+_{\bf a}$ and $u$ maps each connected component of $\bdry \dot F$ to a different $L^+_{a_i}$;
\item $\displaystyle{\lim_{w\rightarrow q_i}\pi_{\R} \circ u (w) =+\infty}$ and $\displaystyle{\lim_{w\rightarrow p_i} \pi_{\R}\circ u(w) =-\infty}$;
\item $u$ converges to a strip over $[0,1]\times {\bf y}$ near ${\bf q}$;
\item $u$ converges to a cylinder over a multiple of some $\gamma_{j}$ near each puncture $p_i$ so that the total multiplicity of $\gamma_j$ over all the $p_i$'s is $m_j$;
\item the energy of $u$ given by Equation~\eqref{eqn: energy of Lipshitz curve} is finite;
\end{enumerate}
or is (ii) a Morse-Bott building 
consisting in a $J_+$ holomorphic curve as above, except for some negative ends at
Reeb orbits in $\mathcal{N}-\{e,h\}$, followed by gradient flow trajectories from those orbits to $e$.
Here $\pi_\R$ is the projection $\pi_{B_+}:W_+\to B_+\subset \R\times S^1$, followed by the projection to $\R$.

A {\em $(W_+,J_+)$-curve} is a degree $2g$ multisection of $(W_+,J_+)$.
\end{defn}

The finiteness of the Hofer energy $E(u)$ implies that $u$ is a cylinder over a Reeb chord or a closed orbit in neighborhoods of punctures $p_i$ and $q_i$. Hence (5) implies (3) and (4) for some $\mathbf{y}$ and $\boldsymbol{\gamma}$. Moreover, since the orbits are nondegenerate, the convergence is exponential by the work of Abbas~\cite{Abb} for chords and Hofer-Wysocki-Zehnder~\cite{HWZ1} for closed orbits.

\begin{rmk}
For all practical purposes, it suffices to assume that the Morse-Bott family on $\bdry N$ has been perturbed into a pair $h$, $e$ of nondegenerate orbits  as in \ref{rmk: Morse-Bott for W'} and that $J_+$ is the restriction of $J'=J_0$ which satisfies Definition~\ref{defn: almost complex structures J_tau} for the perturbed Hamiltonian vector field. Since the two points of view are completely equivalent, we will switch liberally from one to another, often without explicit mention.
\end{rmk}

Let $\check{W}_+$ be $W_+$ with the ends $\{s>3\}$ and $\{s<-1\}$ removed. We can view $\check{W}_+$ as a compactification of $W_+$, obtained by attaching $[0,1]\times S$ at $s=+\infty$ and $N$ at $s=-\infty$. We define $Z_{{\bf y},\boldsymbol{\gamma}}\subset \check{W}_+$ as the subset
$$Z_{{\bf y},\boldsymbol{\gamma}}=(L^+_{\bf a}\cap \check{W}_+) \cup (\{3\}\times[0,1]\times {\bf y}) \cup(\{-1\}\times \boldsymbol{\gamma}).$$
As in Section~\ref{subsection: HF holomorphic curves and moduli spaces}, the $W_+$-curve $u:\dot F \to W_+$ can be compactified to a continuous map
$$\check u: (\check F,\bdry \check F)\to (\check W_+,Z_{{\bf y},\boldsymbol{\gamma}}).$$

We write $\mathcal{M}_{J_+}({\bf y},\boldsymbol{\gamma})$ for the moduli space of multisections of $(W_+,J_+)$ from ${\bf y}$ to $\boldsymbol{\gamma}$.  \nom[H2]{$H_2(W_+, \mathbf{y}, \boldsymbol{\gamma})$}{Homology classes of continuous multisections for $\Phi$} We denote by $H_2(W_+, \mathbf{y}, \boldsymbol{\gamma})$ the equivalence classes of continuous degree $2g$ multisections in $W_+$ satisfying Conditions (1)--(4) of Definition~\ref{defn: W plus curve}, where two multisections are equivalent if they  represent the same element in $H_2(\check W_+,Z_{{\bf y},\boldsymbol{\gamma}})$,
and by  $\mathcal{M}_{J_+}({\bf y},\boldsymbol{\gamma},A)$ the moduli space of maps $u \in \mathcal{M}_{J_+}({\bf y},\boldsymbol{\gamma})$ which represent the class $A \in H_2(W_+, \mathbf{y}, \boldsymbol{\gamma})$. Then $$\mathcal{M}_{J_+}({\bf y},\boldsymbol{\gamma})=\coprod_{A\in H_2(W_+, {\bf y},\boldsymbol{\gamma})}\mathcal{M}_{J_+}({\bf y},\boldsymbol{\gamma},A).$$

\subsubsection{Moduli spaces for $\overline{W}_+$}

Let $\overline{J}_+ \in {\mathcal J}_{\overline{W}_+}$, $\mathbf{y}$ be a $k$-tuple of ${\bf a}\cap \hh({\bf a})$ and let $\boldsymbol{\gamma} \in\widehat{\mathcal{O}}_k$.  Then a {\em degree $k\leq 2g$ multisection of $(\overline{W}_+,\overline{J}_+)$ from ${\bf y}$ to $\boldsymbol{\gamma}$} is defined as in Definition~\ref{defn: W plus curve}, where $W_+$, $J_+$, $L_{a_i}^+$ are replaced by $\overline{W}_+$, $\overline{J}_+$, $L_{\widehat{a}_i}^+$.\footnote{In particular, the image of the multisection is disjoint from the section at infinity $\sigma_{\infty}^+$, defined in Definition~\ref{def: section at infty for W plus}.}
We write $\mathcal{M}_{\overline{J}_+}({\bf y},\boldsymbol{\gamma})$ for the moduli space of multisections of $(\overline{W}_+,\overline{J}_+)$ from ${\bf y}$ to $\boldsymbol{\gamma}$.
\nom[M7]{$\mathcal{M}_{\overline{J}_+}({\bf y},\boldsymbol{\gamma})$}{Moduli space of multisections of $(\overline{W}_+,\overline{J}_+)$ from ${\bf y}$ to $\boldsymbol{\gamma}$}

\begin{rmk}
As in Section~\ref{acorn2}, we will assume that $\overline{J}_+$ is sufficiently close to a Morse-Bott almost complex structure for which all degree $k \le 2g$ multisections of $(\overline{W}_+, \overline{J}_+)$ are close to breaking into a holomorphic building. This will allow us to use the Morse-Bott formalism to prove topological constraints on the $\overline{J}_+$-holomorphic multisections. In the context of maps to $\overline{W}_+$, the definition of a Morse-Bott building is the same as in Definition~\ref{defn: Morse-Bott buildings}, with the only difference that the topmost holomorphic map $u_1$ takes values in $\overline{W}_+$ and is positively asymptotic to chords over intersection points in $\mathbf{a} \cap \hh (\mathbf{a})$.
\end{rmk}

\begin{defn} \label{def: section at infty for W plus}
The {\em section at infinity} $\sigma_{\infty}^+$ is the restriction of $\sigma_{\infty}'=\R \times \delta_0 \subset \overline{W'}$ to $\overline{W}_+$.
\end{defn}

Let $\delta_0^\dagger$ be the closed orbit of $\overline{R}_0$ used in Definition~\ref{defn: n for PFH}. We define $(\sigma_\infty^+)^\dagger$ as the restriction of $\R \times \delta_0^\dagger$ to $\overline{W}_+$. Both $\sigma_\infty^+$ and $(\sigma_\infty^+)^\dagger$ are $\overline{J}_+$-holomorphic for every almost complex
structure $\overline{J}_+ \in \mathcal{J}_{\overline{W}_+}$.

\begin{defn}
Given a multisection $\overline{u}$ of $\overline{W}_+$, we define $n^+(\overline{u}) = \langle \overline{u}, (\sigma_\infty^+)^\dagger \rangle.$
\end{defn}

\begin{lemma}\label{properties n+}
The intersection number $n^+(\overline{u})$ satisfies the following properties:
\begin{enumerate}
\item $n^+(\overline{u}) \ge 0$ and $n^+(\overline{u}) = 0$ if and only if the image of
$\overline{u}$ is disjoint from $(\sigma^+_\infty)^\dagger$;
\item if $\langle \overline{u}, \sigma_{\infty}^+ \rangle >0$, then $n^+(\overline{u}) \ge m$; moreover $n^+(\overline{u}) = m$ if and only if there is a unique transverse intersection point between the image of $\overline{u}$ and $\sigma_{\infty}^+$; and
\item $n^+(\overline{u})$ is independent of the choice of $\delta_0^\dagger$.
\end{enumerate}
\end{lemma}

\begin{proof}
The proof is similar to that of Lemma \ref{pisapia}. The only difference is in (2), which we discuss in more detail. Consider $\overline{u}:\dot F \to \overline{W}_+$  and let $p \in \dot F$ be a point such that $\overline{u}(p) \in \sigma_{\infty}^+$. If $\pi_{D^2}$ is the projection of a neighborhood of $\overline{u}(p) \in \overline W_+$ to $D^2\subset \overline{S}$ along the symplectic connection (not to be confused with the balanced projection $\overline{\pi}_{D^2}$ from Section \ref{subsubsection: overline W pm}), then $\pi_{D^2}\circ \overline u$ is holomorphic and nonconstant, and therefore maps an open neighborhood of $p$ in $\dot F$ to an open neighborhood of $z_\infty$. This implies that $n^+(\overline{u})  \ge d\cdot m$, where $d$ is the multiplicity of the intersection between the image of $\overline{u}$ and $\sigma_{\infty}^+$.
\end{proof}

\begin{lemma} \label{lemma: restriction of overline W plus curve}
If $\overline{u}\in \mathcal{M}_{\overline{J}_+}^{n^+=0}({\bf y},\boldsymbol{\gamma})$, then $\op{Im}(\overline{u})\subset W_+$.
\end{lemma}

The proof is similar to that of Lemma~\ref{pluto} and will be omitted.

\subsubsection{Moduli spaces for $\overline{W}_-$}
\label{subsubsection: overline W minus curves}

The moduli space of holomorphic maps which is used to define the map
$\Psi$ from $\widehat{PFC}$ to $\widehat{CF}$ is of a slightly
different type from the moduli space which is used to define the map
$\Phi$ from $\widehat{CF}$ to $\widehat{PFC}$. In particular, the
target of the holomorphic maps is $\overline{W}_-$ instead of $W_-$.
The reason we need to consider more complicated holomorphic curves
instead of curves analogous to $W_+$-curves is that the naive
$W_-$-curves do not have the desired Fredholm index. See the index
calculations in Section~\ref{subsubsection: the Fredholm index for W
pm} and Remark~\ref{rmk: explanation of considering volcano} for an
explanation.

Let $(F,j)$ be a compact Riemann surface, possibly disconnected, with an $l$-tuple of punctures ${\bf p}=\{p_1 ,\dots,p_l\}$ in the interior of $F$ and a $k$-tuple of punctures ${\bf q} =\{q_1 ,\dots,q_k\}$ on $\partial F$, such that (i) each component of $F$ has nonempty boundary and has at least one interior puncture and (ii) each component of $\bdry F$ has at least one boundary puncture. We write $\dot{F}=F-{\bf p} -{\bf q}$ and $\bdry \dot F= \bdry F-{\bf q}$.

\begin{defn} \label{def: W bar minus curves}
Let $\overline{J}_-\in \mathcal{J}_{\overline{W}_-}$, $\mathbf{y}$ be a $k$-tuple of ${\bf a}\cap \hh({\bf a})$, and $\boldsymbol{\gamma}= \prod_j\gamma_j^{m_j} \in \widehat{\mathcal{O}}_k$.
Then a {\em degree $k\leq 2g$ multisection of $(\overline{W}_-,\overline{J}_-)$ from ${\bf y}$ to $\boldsymbol{\gamma}$} is a $\overline{J}_-$-holomorphic map
$$\overline{u}: (\dot{F} ,j) \rightarrow (\overline{W}_-,\overline{J}_-)$$
which is a degree $k$ multisection of $\overline\pi_{B_-}: \overline{W}_-\to B_-$ and which additionally satisfies the following:
\begin{enumerate}
\item $\overline{u}(\partial \dot{F} )\subset L^-_{\widehat{\bf a}}$ and $\overline{u}$ maps each connected component of $\bdry \dot F$ to a different $L^-_{\widehat{a}_i}$;
\item $\displaystyle{\lim_{w\rightarrow q_i}\overline\pi_{\R}\circ \overline{u}(w) =-\infty}$ and $\displaystyle{\lim_{w\rightarrow p_i} \overline\pi_{\R}\circ \overline{u}(w) =+\infty}$;
\item $\overline{u}$ converges to a strip over $[0,1]\times {\bf y}$ near ${\bf q}$;
\item $\overline{u}$ converges to a cylinder over a multiple of some $\gamma_{j}$ near each puncture $p_i$ so that the total multiplicity of $\gamma_j$ over all the $p_i$'s is $m_j$;
\item the energy of $\overline{u}$ given by Equation~\eqref{eqn: energy of Lipshitz curve} is finite.
\end{enumerate}
Here $\overline\pi_\R$ is the projection $\overline\pi_{B_-}:\overline{W}_-\to B_-\subset \R\times S^1$, followed by the projection to $\R$.
\end{defn}

\begin{rmk}
As in Section~\ref{acorn2}, we will assume that $\overline{J}_-$ is sufficiently close to a Morse-Bott almost complex structure for which all degree $k \le 2g$ multisections of $(\overline{W}_-, \overline{J}_-)$ are close to breaking into a holomorphic building. This will allow us to use the Morse-Bott formalism to prove topological constraints on the $\overline{J}_-$-holomorphic multisections. In the context of maps to $\overline{W}_-$, the definition of a Morse-Bott building is the same as in Definition~\ref{defn: Morse-Bott buildings}, with the only difference that the bottommost holomorphic map $u_n$ takes values in $\overline{W}_-$ and is negatively asymptotic to chords over intersection points in $\mathbf{a} \cap \hh (\mathbf{a})$.
\end{rmk}

\begin{defn}
The {\em section at infinity} $\sigma_{\infty}^-$ is the restriction of $\sigma_{\infty}'=\R \times \delta_0 \subset \overline{W'}$ to $\overline{W}_-$.
\end{defn}

Let $\delta_0^\dagger$ be the closed orbit of $\overline{R}_0$ used in
Definition~\ref{defn: n for PFH}. We define $(\sigma_\infty^-)^\dagger$ as the restriction
of $\R \times \delta_0^\dagger$ to $\overline{W}_-$. Both $\sigma_\infty^-$ and $(\sigma_\infty^-)^\dagger$ are $\overline{J}_-$-holomorphic for every almost complex
structure $\overline{J}_- \in \mathcal{J}_{\overline{W}_-}$.

\begin{defn}
Given a multisection $\overline{u}$ of $\overline{W}_-$, we define $n^-(\overline{u}) = \langle \overline{u}, (\sigma_\infty^-)^\dagger \rangle.$
\end{defn}

The proof of the following is similar to that of Lemma~\ref{properties n+}.

\begin{lemma}\label{properties n-}
The intersection number $n^-(\overline{u})$ satisfies the following properties:
\begin{enumerate}
\item $n^-(\overline{u}) \ge 0$ and $n^-(\overline{u}) = 0$ if and only if the image of $\overline{u}$ is disjoint from $(\sigma^-_\infty)^\dagger$;
\item if $\langle \overline{u}, \sigma_{\infty}^- \rangle >0$, then $n^-(\overline{u})\ge m$; moreover $n^-(\overline{u})= m$ if and only if there is a unique transverse intersection point between the image of $\overline{u}$ and $\sigma_\infty^-$; and
\item $n^-(\overline{u})$ is independent of the choice of $\delta_0^\dagger$.
\end{enumerate}
\end{lemma}
\begin{defn}
A {\em $(\overline{W}_-,\overline{J}_-)$-curve} is a degree $2g$ multisection
of $(\overline{W}_-,\overline{J}_-)$ satisfying $n^-(\overline{u})=m$.
\end{defn}

Let $\check{\overline{W}}_-$ be $\overline{W}_-$ with the ends $\{s>1\}$ and $\{s<-3\}$ removed. We can view
$\check{\overline{W}}_-$ as a compactification of $\overline{W}_-$, obtained by attaching $[0,1]\times \overline{S}$ at $s=-\infty$ and $\overline{N}$ at $s=+\infty$. Also let
$$Z_{\boldsymbol{\gamma},{\bf y}}= (L^-_{\widehat{\bf a}}\cap \check{\overline W})\cup (\{1\}\times
\boldsymbol{\gamma})\cup (\{-3\}\times[0,1]\times {\bf y}).$$

We write $\mathcal{M}_{\overline{J}_-} (\boldsymbol{\gamma},{\bf y})$ for the moduli space of multisections of $(\overline{W}_-, \overline{J}_-)$ from $\boldsymbol{\gamma}$ to $\mathbf{y}$.
\nom[M7]{$\mathcal{M}_{\overline{J}_-} (\boldsymbol{\gamma},{\bf y})$}{Moduli space of multisections in $(\overline{W}_-,\overline{J}_-)$ from $\boldsymbol{\gamma}$ to $\mathbf{y}$} \nom[H3]{$H_2(\overline{W}_-, \boldsymbol{\gamma}, \mathbf{y})$}{Homology classes of continuous multisections for $\Psi$}
We denote by $H_2(\overline{W}_-, \boldsymbol{\gamma}, \mathbf{y})$  the equivalence classes of continuous degree $2g$ multisections in $\overline{W}_-$ satisfying Conditions (1)--(4) of Definition~\ref{def: W bar minus curves}, where two multisections are equivalent if they represent the same element in $H_2(\check{\overline W}_-,Z_{\boldsymbol{\gamma}, {\bf y}})$.
Then
$$\mathcal{M}_{\overline{J}_-} (\boldsymbol{\gamma},{\bf y})=\coprod_{A\in H_2(\overline W_-, \boldsymbol{\gamma},{\bf y})}\mathcal{M}_{\overline{J}_-}(\boldsymbol{\gamma},{\bf y},A).$$
Also let $\mathcal{M}_{\overline{J}_-} (\boldsymbol{\gamma},{\bf y};\overline{\frak m})\subset \mathcal{M}_{\overline{J}_-} (\boldsymbol{\gamma},{\bf y})$ be the subset of curves which pass through the marked point $\overline{\frak m}$.

The following lemma is similar to Lemma~\ref{lemma: restriction of overline W plus curve}:

\begin{lemma} \label{lemma: restriction of overline W minus curve}
If $\overline{u}\in \mathcal{M}_{\overline{J}_-}^{n^-=0} (\boldsymbol{\gamma},{\bf y})$, then $\op{Im}(\overline{u})\subset W_-$.
\end{lemma}

\subsection{The Fredholm index}
\label{subsection: the Fredholm index W plus W minus}

In this subsection we compute the Fredholm index $\op{ind}_{W_+}$ of a $(W_+,J_+)$-curve from $\mathbf{y}=\{y_1,\dots,y_{2g}\}\in \mathcal{S}_{{\bf a}, \hh({\bf a})}$ to $\boldsymbol{\gamma}=\prod \gamma_j^{m_j}\in \widehat{\mathcal{O}}_{2g}$ and the Fredholm index $\op{ind}_{\overline{W}_-}$ of a $(\overline{W}_-,\overline{J}_-)$-curve from $\boldsymbol{\gamma}$ to $\mathbf{y}$.

The Fredholm indices are computed using the doubling technique of Hofer-Lizan-Sikorav \cite{HLS}, which we quickly review, referring to the original paper for the details.

\subsubsection{Doubling}

Let $(F,j)$ be a compact Riemann surface with boundary, ${\bf p}$ and ${\bf q}$ be finite sets of interior and boundary punctures of $F$, and $\dot F=F-{\bf p}-{\bf q}$. We form the {\em double} $(2\dot{F}, 2j)$ of $(\dot{F}, j)$ by gluing two copies of $\dot{F}$ with opposite complex structures $j$ and $-j$ along their boundary $\bdry \dot F=\bdry F-{\bf q}$. By the Schwarz reflection principle, the two complex structures match and the doubled surface becomes a punctured Riemann surface.

Let $E \to \dot F$ be a complex vector bundle with fiberwise complex structure $i$ and let $L \to \partial \dot{F}$ be a totally real subbundle of maximal rank. Let $\overline{E}\to \dot F$ be a complex vector bundle whose fiber $\overline{E}_p$ at $p\in\dot F$ is $E_p$ with complex structure $-i$.  We then construct the doubled complex vector bundle $2E \to 2\dot{F}$ by gluing $E\to \dot F$ and $\overline{E}\to \dot F$ along $\partial \dot{F}$ such that at each $p\in \bdry\dot{F}$ the gluing map identifies the fibers $(E_p,i)\simeq (\overline{E}_p,-i)$ via an involution which fixes $L_p$ pointwise. Let $\sigma: 2\dot F\stackrel\sim\to 2\dot F$ be the anti-holomorphic involution $\sigma$ which switches $(\dot F,j)$ and $(\dot F,-j)$ and let $\tilde\sigma: 2E\stackrel\sim \to 2E$ be the anti-$\C$-linear bundle isomorphism which projects to $\sigma$ and identifies $E_p\simeq \overline{E}_{\sigma(p)}$ by the identity map, where $p\in(\dot F,j)$.  Finally, given a linear Cauchy-Riemann type operator $D$ on $E$, we can define the doubled operator $2D$ on $2E$ with the property that $2D$ is $\tilde{\sigma}$-invariant and its restriction to $E \to \Sigma$ is $D$.

One of the results of \cite{HLS} is the following:

\begin{thm}\label{HLS}
Suppose that both $D$ and $2D$ are Fredholm operators in some suitable Sobolev spaces. Then:
\begin{itemize}
\item $\op{ind}(D) = \frac 12 \op{ind}(2D)$; and
\item if $2D$ is surjective, then $D$ is also surjective.
\end{itemize}
\end{thm}

Our situation is slightly more general than that considered by \cite{HLS}, since we are considering boundary punctures and exponential weights. The proof, however, remains largely unmodified.

\subsubsection{The $W_+$ case}
\label{subsubsection: the Fredholm index for W pm}

For the purposes of computing indices, we replace $W_+$ by $\check W_+$, as defined in Section~\ref{subsubsection: W plus curves}.  The tangent space $T\check W_+$ splits into the vertical and horizontal subbundles  via the symplectic connection.  In this section we slightly abuse notation and write $TS$ (resp. $TB_+$) for the vertical (resp. horizontal) subbundle over $\check W_+$.

We define a trivialization $\tau$ of $TS$ along $Z_{{\bf y},\boldsymbol{\gamma}}$ as follows: First define $\tau$ along $L^+_{\mathbf{a}}\cap \check W_+$ by orienting all $\mathbf{a}$-arcs arbitrarily as in Section~\ref{subsubsection: Fredholm index second version} and extending the trivialization by parallel transport along the vertical boundary $\bdry_v W_+$.  We then extend the trivialization of $TS|_{L^+_{\mathbf{a}}\cap \check W_+}$ in an arbitrary manner to a trivialization of $TS$ along $\{3\}\times[0,1]\times {\bf y}$ and along $\{-1\}\times\boldsymbol{\gamma}$.

Let $u: (\dot F,j)\to (W_+,J_+)$ be a $W_+$-curve. Suppose the negative ends of $u$ partition $m_j$ into $(m_{j1},m_{j2},\dots)$. We then write:
$$\mu_\tau(\boldsymbol{\gamma},u)=\sum_j \sum_i\mu_\tau (\gamma_{j}^{m_{ji}}),$$
where $\mu_\tau (\gamma_{j}^{m_{ji}})$ is the Conley-Zehnder index of the $m_{ji}$-fold cover of $\gamma_{j}$ with respect to $\tau$.

We now define a real rank one subbundle $\mathcal{L}_0$ of $TS$ along
$$(L^+_{\bf a}\cap \check{W}_+) \cup (\{3\}\times[0,1]\times {\bf y}).$$
Let $\mathcal{L}_0=TL^+_{\bf a}\cap TS$ on $L^+_{\bf a}\cap \check{W}_+$. In particular, $\mathcal{L}_0=T{\bf a}(y_i)$ at $\{3\}\times\{1\}\times \{y_i\}$ and $\mathcal{L}_0=T\hh({\bf a})(y_i)$ at $\{3\}\times\{0\}\times \{y_i\}$. We then extend $\mathcal{L}_0$ across $\{3\}\times[0,1]\times {\bf y}$ by rotating in the counterclockwise direction from $T\hh(\mathbf{a})$ to $T\mathbf{a}$ in $TS$ by the minimum amount possible. Assuming orthogonal intersections between $\mathbf{a}$ and $\hh({\bf a})$, the angle of rotation is ${\pi\over 2}$.

Let $\mu_\tau(y_i)$ be the Maslov index of $\mathcal{L}_0$ along $\{3\}\times[0,1]\times\{y_i\}$ with respect to $\tau$.  If  $u$ is a $W_+$-curve and $\mathcal{L}=\check{u}^*\mathcal{L}_0$, then we define $\mu_\tau({\bf y})$ to be the Maslov index of $\mathcal{L}$ with respect to $\tau$.  By the definitions of $\mathcal{L}_0$ and $\tau$, it is immediate that
$$\mu_\tau ({\bf y})=\sum_{i=1}^{2g}\mu_\tau(y_i).$$

We then have the following Fredholm index formula for $W_+$-curves:

\begin{prop}[Fredholm index formula for $W_+$-curves]
\label{prop: Fredholm index for W plus}
The Fredholm index of a $(W_+,J_+)$-curve $u: (\dot F,j)\to (W_+,J_+)$ from $\mathbf{y}$ to $\boldsymbol{\gamma}$ is given by the formula:
\begin{equation} \label{eqn: ind for W plus}
\op{ind}_{W_+}(u )=-\chi (\dot{F})-2g+\mu_\tau ({\bf y}) -\mu_\tau(\boldsymbol{\gamma},u)+2c_1(u^*TS,\tau).
\end{equation}
\end{prop}

\begin{proof}
We double the surface $\dot{F}$ along $\partial \dot{F}=\bdry F-\mathbf{q}$, where $\mathbf{q}$ is the set of boundary punctures, and double the pullback bundle $u^*TW_+$ along the real subbundle $(u|_{\bdry \dot F})^*TL^+_{\bf a}$ on $\partial \dot{F}$, to obtain the doubled surface $2 \dot{F}$ with a complex vector bundle $2u^* TW_+ \to 2\dot{F}$. The boundary punctures $q_i$ of $\dot{F}$ are doubled to give positive interior punctures $2q_i$ on $2 \dot{F}$.

Since the trivialization $\tau$ of $TS|_{Z_{{\bf y},\boldsymbol{\gamma}}}$ was chosen to be tangent to $TL^+_{\mathbf{a}}\cap TS$, its pullback to $u^*TW_+$ is compatible with the doubling operation and gives a trivialization $2\tau$ of $2u^*TS$ over a  neighborhood of the punctures of $2 \dot{F}$. Also let $\tau'$ be a partially defined trivialization of $TB_+$ which is given by $\bdry_s$ at the positive and negative ends of $B_+$.  Then $\tau'$ can be doubled to $2\tau'$ on $2u^*TB_+$ which is defined over a neighborhood of the punctures of $2 \dot{F}$.

The linearized $\overline\bdry$-operator at $u$ splits as a sum $D_u= D^0_u \oplus K_u$, where $D^0_u$ is a Cauchy-Riemann type operator on $W^{1,p}$-sections of $u^*TW_+$ with exponential weights and $K_u$ is a direct sum of finite-dimensional operators, one for each puncture of $\dot{F}$, with index $2$ for interior punctures and index $1$ for boundary punctures; see \cite[Section 2.3]{Dr} for the definition of the operator $K_u$. We define $2 D_u = 2D_u^0 \oplus K_u'$, where $K_u'$ is a finite-dimensional operator and contributes by $2$ to the index for each puncture of $2\dot{F}$. Technically speaking, $K_u'$ is not the double of $K_u$; nevertheless $\op{ind} (K_u) = \frac 12 \op{ind} (K_u')$. By applying Theorem \ref{HLS} to $D_u^0$, we obtain $\op{ind} (D_u) = \frac 12 \op{ind} (2 D_u)$.

We now apply the standard Fredholm index formula (see for example Dragnev~\cite{Dr}) for holomorphic curves in symplectizations --- with slight modifications ---
to obtain:
\begin{align} \label{eqn: first ind 2u formula}
\op{ind}(2D_u)& = -\chi (2\dot F) +\mu_{2\tau} (2{\bf y}) - \mu_{2\tau} (2\boldsymbol{\gamma},2u)\\
\notag &  \quad + 2c_1(2u^* TB_+,2\tau')+ 2c_1(2 u^*TS,2\tau).
\end{align}
Here $2{\bf y}$ and $2\boldsymbol{\gamma}$ stand for the doubles of ${\bf y}$ and $\boldsymbol\gamma$. Also $\mu_{2\tau} (2{\bf y})$ is the sum of the Conley-Zehnder indices, computed with respect to the trivialization $2 \tau$, of the paths of symplectic matrices arising from the asymptotic operators at $2 y_i$.  The last term is the relative first Chern class of the double of the pullback bundle $u^*TS$ with respect to $2\tau$. The one term that is not present in the Fredholm index formula for $J$-holomorphic curves in a symplectization
is the penultimate term $2c_1(2u^*TB_+,2\tau')$.

We then compute the following:
\begin{enumerate}
\item[(a)] $\chi (2\dot F)=2\chi (\dot{F} )-2g$;
\item[(b)] $\mu_{2\tau} (2y_i)=2 \mu_\tau (y_i)-1$;
\item[(c)] $\mu_{2\tau} (2\boldsymbol{\gamma},2u)=2\mu_\tau (\boldsymbol{\gamma},u)$;
\item[(d)] $c_1(2u^*TB_+,2\tau')=-2g$;
\item[(e)] $c_1(2u^*TS,2\tau)= 2 c_1(u^*TS,\tau)$.
\end{enumerate}
(a), (c) and (e) are straightforward. (b) will be proved in Lemma~\ref{doubling chords} below. (d)  By arguing in a manner similar to that of Claim~\ref{claim: calculation of c_1 of TF} we have $c_1(2u^*TB_+,2\tau') = 2g \chi(2B_+) = -2g$.

Summarizing, we obtain:
\begin{align*}
\op{ind}_{W_+} (u)&= (-\chi(\dot F)+g)+(\mu_\tau ({\bf y})-g) -\mu_\tau(\boldsymbol{\gamma},u) -2g+2c_1(u^*TS,\tau) \\
&= -\chi(\dot F)-2g+\mu_\tau ({\bf y} ) -\mu_\tau(\boldsymbol{\gamma},u)+2c_1(u^*TS,\tau).
\end{align*}
This completes the proof of the proposition.
\end{proof}

\begin{lemma} \label{doubling chords}
$\mu_{2\tau} (2y_i)=2 \mu_\tau (y_i)-1$.
\end{lemma}

\begin{proof}
Let $q_i \in \partial F$ be a (positive) boundary puncture and let $\mathcal{E}=[0,\infty)\times [0,1]\subset \dot F$ be a strip-like end with coordinates $(s,t)$ which parametrizes a neighborhood of $q_i$. Suppose $u$ maps the end asymptotically to $y_i$. Fix a symplectic trivialization $$\Theta: u|_{\mathcal{E}}^*TS \stackrel\sim\to \R^2 \times \mathcal{E}$$ such that $u|_{\bdry\mathcal{E}}^*(TL^+_{\mathbf a}\cap TS)$ corresponds to $(\R \oplus 0) \times \bdry\mathcal{E}$. Here $\bdry \mathcal{E}=\bdry \dot F\cap \mathcal{E}$.

Let $2\mathcal{E}=([0,\infty)\times [0,2])/(s,0)\sim (s,2)$ be the end of the doubled surface $2\dot F$, which corresponds to the interior puncture $2q_i$ and is obtained by doubling $\mathcal{E}$. The involution $\sigma$ is given by $\sigma(s,t) = (s, 2-t)$ with respect to these coordinates. Similarly, we have a symplectic trivialization
$$2\Theta:2u|_{\mathcal{E}}^*TS \stackrel\sim\to \R^2 \times 2\mathcal{E}$$
and $\tilde{\sigma}$ is given by $\tilde{\sigma}((x_1, x_2),s,t) = ((x_1,-x_2),s,2-t)$.

Let $J_0 \partial_t + S_i(t)$ be the {\em asymptotic operator} of $D^0_u$ corresponding to $q_i$ with respect to the trivialization $\Theta$. Here $J_0$ is the standard complex structure on $\R^2$ and $S_i(t) $ is a symmetric $2 \times 2$ matrix with real coefficients. For the definition of the asymptotic operator and its relation with the Fredholm theory of linearized $\overline\bdry$-operators, see \cite[Section 3]{Dr} or \cite{HT2}.\footnote{What we call $S_i$ here is written as $C^i_{2 \infty}$ in \cite{Dr}.}  The solution of the Cauchy problem
$$\dot{\Phi}(t) = J_0 S_i(t)\Phi(t), \quad \Phi(0)=0$$
is a path of symplectic matrices which represents the linearized Reeb flow along the chord $y_i$, expressed with respect to the trivialization $\tau$. In our setting, we may assume that $\Phi(t)$ is a path of unitary matrices. By identifying $(\R^2, J_0) = (\C, i)$, we write $\Phi(t)=e^{i \alpha(t)}$ for some function $\alpha: [0,1] \to \R$. Then

$$\mu_{\tau}(y_i)=\mbox{\Large $\llfloor$}  \mbox{$ \frac{\alpha(1) - \alpha(0)}{\pi} $} \mbox{\Large $\rrfloor$} +1,$$ where $\llfloor x \rrfloor$ is the greatest  integer $\leq x$.

The double of the asymptotic operator, i.e., the asymptotic operator of the doubled operator $2D^0_u$ at the interior puncture $2q_i$, can be written as $J_0 \partial_t + \widetilde{S}_i(t)$, where:
$$ \widetilde{S}_i(t) = \left \{ \begin{array}{cl} S_i(t), &  t \in [0,1]; \\

CS_i(2-t)C^{-1}, &  t \in [1,2], \end{array} \right. $$
and $C=\op{diag}(1,-1)$. The solution of the corresponding Cauchy problem is
$$ \widetilde{\Phi}(t)=  \left \{ \begin{array}{cl} \Phi(t), & t \in [0,1]; \\
C \Phi (2-t) \Phi(1)^{-1}C^{-1} \Phi(1), &  t \in [1,2]. \end{array} \right. $$
Hence can write $ \widetilde{\Phi}(t)= e^{i \tilde{\alpha}(t)}$, where $\tilde{\alpha}: [0,2] \to \R$ is given by:
$$\tilde{\alpha}(t) = \left \{ \begin{array}{cl} \alpha(t), &  t \in [0,1];
\\ - \alpha(2-t) + 2 \alpha(1), & t \in [1,2].  \end{array} \right. $$

The Conley-Zehnder index of the path $\widetilde{\Phi}$ is
$$\mu_{2 \tau}(2 y_i) = 2 ~\mbox{\Large $\llfloor$}\mbox{{ $\frac{\tilde{\alpha}(2) - \tilde{\alpha}(0)}{2 \pi} $}}\mbox{\Large $\rrfloor$} +1.$$ Since
\begin{align*}
\mbox{\Large $\llfloor$} \mbox{{ $\frac{\tilde{\alpha}(2) - \tilde{\alpha}(0)}{2 \pi} $}} \mbox{\Large $\rrfloor$}
&= \mbox{\Large $\llfloor$} \mbox{{ $\frac{(\tilde{\alpha}(2) - \tilde{\alpha}(1)) + (\tilde{\alpha}(1) - \tilde{\alpha}(0))}{2 \pi} $}} \mbox{\Large $\rrfloor$}\\
&= \mbox{\Large $\llfloor$} \mbox{{ $\frac{2(\tilde{\alpha}(1) - \tilde{\alpha}(0))}{2 \pi} $}}  \mbox{\Large $\rrfloor$} = \mbox{\Large $\llfloor$} \mbox{{ $\frac{\alpha(1) - \alpha(0)}{\pi} $}}\mbox{\Large $\rrfloor$},
\end{align*}
we obtain $\mu_{2 \tau}(2 y_i) = 2 \mu_{\tau}(y_i)-1$, as desired.
\end{proof}

\subsubsection{The $\overline{W}_-$ case}

The trivialization $\tau$ of $T\overline{S}_{\overline{W}_-}$ is defined on $Z_{\boldsymbol{\gamma},{\bf y}}$ in a manner similar to that of $W_+$, except that the positive and negative ends are reversed. The definition of the real rank one subbundle $\mathcal{L}$ of $T\overline{S}$ along $(L^-_{\widehat{\bf a}}\cap \check{\overline{W}}_-) \cup (\{-3\}\times[0,1]\times {\bf y})$ is also similar and will be omitted.

If $\overline u$ is a $\overline{W}_-$-curve from $\boldsymbol{\gamma}$ to ${\bf y}$, then the
Fredholm index $\op{ind}_{\overline{W}_-}(\overline u)$ is defined as the expected
dimension of the moduli space $\mathcal{M}_{\overline{J}_-}(\boldsymbol{\gamma},{\bf y})$ near
$\overline{u}$.

\begin{rmk}
The expected dimension of $\mathcal{M}_{\overline{J}_-}(\boldsymbol{\gamma},{\bf y};\overline{\frak m})$ is $\op{ind}_{\overline{W}_-}(\overline u)-2$.
\end{rmk}

We then have the following Fredholm index formula:

\begin{prop} [Fredholm index formula for $\overline{W}_-$-curves]
\label{prop: Fredholm index for W minus}
The Fredholm index of a $(\overline{W}_-,\overline{J}_-)$-curve $\overline{u}: (\dot{F} ,j)
\rightarrow (\overline{W}_-, \overline{J}_-)$ from $\boldsymbol{\gamma}$ to $\mathbf{y}$ is given by the formula:
\begin{equation} \label{eqn: second in W minus proposition}
\op{ind}_{\overline{W}_-}(\overline u) = -\chi (\dot F)+\mu_\tau (\boldsymbol{\gamma},\overline{u}) -\mu_\tau(\mathbf{y}) +2c_1(\overline u^*T\overline{S},\tau).
\end{equation}
\end{prop}

\begin{proof}
By a calculation similar to that of Proposition~\ref{prop: Fredholm index for W plus}, we obtain
\begin{equation*}
\op{ind}_{\overline{W}_-} (\overline{u}) = (-\chi(\dot F)+g) +\mu_\tau(\boldsymbol{\gamma},\overline{u})-(\mu_\tau ({\bf y})-g) -2g +2c_1(\overline{u}^*T\overline{S},\tau),
\end{equation*}
which simplifies to the desired result.
\end{proof}

\begin{rmk}[Reason for considering $\overline{W}_-$]
\label{rmk: explanation of considering volcano}
We will give a rough explanation of the reason for considering holomorphic curves in $\overline{W}_-$ which pass through $\overline{\frak m}$. Suppose we have a $W_+$-curve $u$ from $\mathbf{y}$ to $\boldsymbol{\gamma}$ and a $W_-$-curve $v$ from $\boldsymbol{\gamma}$ to $\mathbf{y}$ (i.e., $\op{Im}(v)\subset W_-$). Then, by taking the sum of Equations~\eqref{eqn: ind for W plus} and \eqref{eqn: second in W minus proposition}, we compute the Fredholm index of the glued curve $u\# v$, corresponding to the stacking of $W_+$ at the top ($s>0$) and $W_-$ at the bottom ($s<0$), to be:
$$\op{ind}(u\# v )=-\chi (\dot{F})+2c_1((u\# v)^*TS,\tau)-2g,$$
where $\dot F$ is obtained from gluing the domains of $u$ and $v$. The stacking gives rise to a chain map $\widehat{CF}\to\widehat{CF}$, which we expect to map $\mathbf{y}\mapsto \mathbf{y}$ via restrictions of trivial cylinders (modulo chain homotopy). This would mean $\chi(\dot F)=0$ and $c_1((u\# v)^* TS,\tau)=0$. This leaves us with a deficiency of $2g$. Introducing the point constraint at $\overline{\frak m}$, from the perspective of Fredholm indices, is basically equivalent to applying a multiple connected sum to the holomorphic curve $v$ and the fiber $\overline{S}$ which passes through $\overline{\frak m}$.  The multiple connected sum is performed at the $2g$ intersection points between $v$ and $\overline{S}$.  We effectively increase the Fredholm index by $2g+2$, obtained by adding up the following contributions:
\begin{itemize}
\item[(i)] the Fredholm index of the fiber $\overline{S}$, which is $\chi(\overline{S})=2-2g$;
\item[(ii)] $2g$ intersection points, each of which contributes $+2$ to minus the Euler characteristic.
\end{itemize}
The point constraint then cuts the expected dimension of the moduli space by $-2$, for a net gain of $2g$.
\end{rmk}

\subsection{The ECH index}
\label{subsection: ECH index W plus minus}

In this section we present the ECH index $I_{W_+}$, the relative adjunction formula, and the ECH index inequality for $W_+$. The situation for $\overline{W}_-$ is analogous, and will not be discussed explicitly, except to point out some differences.

\subsubsection{Definitions}

Let ${\bf y}=\{y_1,\dots,y_{2g}\}\in \mathcal{S}_{{\bf a}, \hh({\bf a})}$ and $\boldsymbol{\gamma}=\prod_{j=1}^l \gamma_j^{m_j}\in \widehat{\mathcal{O}}_{2g}$.  Let $\tau$ be a trivialization of $TS_{\check W_+}$ along $Z_{{\bf y},\boldsymbol{\gamma}}$, as given in
Section~\ref{subsubsection: the Fredholm index for W pm}. Using $\tau$, for each simple orbit $\gamma_j$ of $\boldsymbol{\gamma}$,
we choose an identification of a sufficiently small neighborhood
$\nu(\gamma_j)$ of $\gamma_j$ with $\gamma_j\times D^2$, where $D^2$
has polar coordinates $(r,\theta)$.

\begin{defn}[$\tau$-trivial representative]
A {\em $\tau$-trivial representative} $\check C$ of $A\in H_2(W_+,{\bf y},\boldsymbol{\gamma})$
 is an oriented immersed compact surface in
the class $A$ which satisfies the following:
\begin{enumerate}
\item $\check C$ is embedded on $\check W_+-\{s=-1\}$;
\item $\check C$ is positively transverse to the fibers
$\{(s,t)\}\times \overline{S}$ along all of $\bdry \check C$;
\item $\check C$ is $\tau$-trivial in the sense of
Definition~\ref{defn: tau-trivial representative HF} at the HF end;
\item $\check C$ is {\em $\tau$-trivial} at the ECH end, i.e., for
all sufficiently small $\varepsilon>0$, $\check C\cap
\{s=-1+\varepsilon\}$ consists of $m_j$ disjoint circles
$\{r=\varepsilon,\theta=\theta_{ji}\}$, $i=1,\dots,m_j$, in $\nu(\gamma_j)$ for all $j$. (See \cite[Definition~2.3]{Hu1}.)
\end{enumerate}
\end{defn}

Let $\bdry_+\check C=\bdry\check C\cap \{s>0\}$ and $\bdry_-\check C=\bdry \check C\cap \{s<0\}$.

\begin{defn}[Relative intersection form]
Let $A\in H_2(W_+, {\bf y},\boldsymbol{\gamma})$ be a homology class which is realized by a $\tau$-trivial representative $\check C$.
Then the relative intersection form $Q_\tau(A)$ is given by $\langle \check C, \check C'\rangle$, where $\check C'$ is a pushoff of $\check C$ which satisfies the following:
\begin{enumerate}
\item $\check C$ is pushed off in the $J_+\tau$-direction along $\bdry_+ \check C$; and
\item for small $\varepsilon>0$, $\check C'\cap \{s=-1+\varepsilon\}$ consists of $m_j$ disjoint circles $\{r=\varepsilon,\theta=\theta_{ji}+\varepsilon'\}$, $i=1,\dots,m_j$, in $\nu(\gamma_j)$ for all $j$, where $\varepsilon'>0$ is a sufficiently small constant.
\end{enumerate}
(See \cite[Definition~2.4]{Hu1}.)
\end{defn}

\subsubsection{Relative adjunction formula}

Let $u: \dot F\to W_+$ be a $W_+$-curve and let $\check u: \check F\to \check W_+$ be its compactification. Then we write $w_\tau^-(u)$ for the total writhe of the braids $u(\dot F)\cap \{s=s_0\}$, $s_0\ll 0$, with respect to $\tau$. Similarly, if $\overline{u}:\dot F\to \overline{W}_-$ is a $\overline{W}_-$-curve, then we write $w_\tau^+(\overline{u})$ for the total writhe of the braids $\overline{u}(\dot F)\cap \{s=s_0\}$, $s_0\gg 0$, with respect to $\tau$.

\begin{lemma}[Relative adjunction formula] \label{lemma: relative adjunction for W plus}
Let $u: \dot F \to W_+$ be a $W_+$-curve in the homology class 
$A \in H_2(W_+, {\bf y}, \boldsymbol{\gamma})$.  Then
\begin{equation}
\label{eqn: relative adjunction formula for W plus}
c_1(u^*TW_+,(\tau,\bdry_t)) = \chi(\dot F)-w_\tau^-(u) +Q_\tau(A) -2\delta(u),
\end{equation}
where $\bdry_t$ trivializes $TB_+$.
\end{lemma}

For a $\overline{W}_-$-curve $\overline{u}$, we replace $-w_\tau^-(u)$ by $w_\tau^+(\overline{u})$.

\begin{proof}
Suppose $u$ is immersed. If $\nu$ is the normal bundle of $u$, then we have the formula:
\begin{equation}
c_1(u^*TW_+,(\tau,\bdry_t)) = c_1(T\check F,\bdry_t) + c_1(\nu,\tau).
\end{equation}

In the general case, we combine the calculations of Lemma~\ref{lemma: c1 and Q} and \cite[Proposition~3.1]{Hu1} to obtain the equation
$$c_1(u^*TW_+,(\tau,\bdry_t)) = c_1(T\check F,\bdry_t) -w_\tau^-(u) +Q_\tau(A) -2\delta(u).$$
The analog of Claim~\ref{claim: calculation of c_1 of TF} for the present situation is
\begin{equation} \label{eqn: calc of c1 TF for W plus}
c_1(T\check F,\bdry_t)= \chi(\dot F).
\end{equation}
The difference in the shape of the base (i.e., $B$ vs.\ $B_+$) accounts for the discrepancy between Equation~\eqref{eqn: calc of c1 TF for W plus} and Claim~\ref{claim: calculation of c_1 of TF}.
\end{proof}

\begin{rmk}
Since $c_1(u^*TB_+,\bdry_t)=0$, we have
$$c_1(u^*TW_+,(\tau,\bdry_t))= c_1(u^*TS,\tau)+c_1(u^*TB_+,\bdry_t)=c_1(u^*TS,\tau).$$
\end{rmk}

\subsubsection{ECH index}

We now define the ECH indices for $W_+$ and $\overline{W}_-$.

\begin{defn}[ECH index for $W_+$]
Given a class 
$A \in H_2(W_+,{\bf y},\boldsymbol{\gamma})$
which admits a $\tau$-trivial representative $\check C$, we define
\begin{equation}
\label{eqn: definition of ECH index for W+}
I_{W_+}(A)=  c_1(T\check W_+|_A, (\tau,\bdry_t)) + Q_{\tau}(A)+ \mu_{\tau}({\bf y})-\widetilde \mu_\tau(\boldsymbol{\gamma})-2g,
\end{equation}
where $\widetilde\mu_\tau(\boldsymbol{\gamma})$ is the symmetric Conley-Zehnder index at the negative (ECH) end.
\nom[IJ]{$I_{W_+}$, $I_{\overline{W}_-}$}{ECH indices for $W_+$, $\overline{W}_-$}
\end{defn}

\begin{defn}[ECH index for $\overline W_-$]
Given a class 
$A \in H_2({\overline W}_-,\boldsymbol{\gamma},{\bf y})$
which admits a $\tau$-trivial representative $\check C$, we define
\begin{equation}
\label{eqn: definition of ECH index for W-}
I_{\overline{W}_-}(A)=  c_1(T\check{\overline W}_-|_A, (\tau,\bdry_t)) + Q_{\tau}(A)+ \widetilde \mu_\tau(\boldsymbol{\gamma})- \mu_{\tau}({\bf y}).
\end{equation}
\end{defn}

As usual, the ECH indices $I_{W_+}(A)$ and $I_{\overline{W}_-}(A)$ are independent of the choice of trivialization $\tau$.

\begin{rmk}
To obtain a finite count of $\overline{W}_-$-curves which pass through the point $\overline{\frak m}$, we count curves $\overline{u}$ with ECH index $I_{\overline{W}_-}(\overline{u})=2$.
\end{rmk}

\subsubsection{Additivity of indices}

\begin{lemma}[Additivity of indices]
\label{lemma: additivity part 2}
If $u\in \mathcal{M}_{J}(\mathbf{y},\mathbf{y'})$, $v\in \mathcal{M}_{J_+}(\mathbf{y'},\boldsymbol{\gamma})$, and $u \# v$ is a pre-glued curve, then
\begin{align}
\op{ind}_{W_+}(u\# v) &= \op{ind}_{HF}(u) + \op{ind}_{W_+}(v),\\
I_{W_+}(u\# v) &= I_{HF}(u) + I_{W_+}(v).
\end{align}
Similarly, if  $u\in \mathcal{M}_{J_+}(\mathbf{y},\boldsymbol{\gamma})$, $v\in \mathcal{M}_{J'}(\boldsymbol{\gamma},\boldsymbol{\gamma}')$, and $u\# v$ is a pre-glued curve, then
\begin{align}
\op{ind}_{W_+}(u\# v) &= \op{ind}_{W_+}(u) + \op{ind}_{ECH}(v),\\
I_{W_+}(u\# v) &= I_{W_+}(u) + I_{ECH}(v).
\end{align}
\end{lemma}

\begin{proof}
The additivity for $\op{ind}$ is well-known and the additivity for $I$ is immediate from the definitions.
\end{proof}

\subsubsection{Index inequality}

Although we will not define it here, given an integer $m_k>0$ and a
simple orbit $\gamma_k$, we can define the {\em incoming partition}
$P^{in}_{\gamma_k}(m_k)$ and the {\em outgoing partition}
$P^{out}_{\gamma_k}(m_k)$ as in \cite[Definition~4.14]{Hu2}.

We have the following index inequality, which is analogous to
~\cite[Theorem~1.7]{Hu1} (also see~\cite[Theorem~4.15]{Hu2} which is
applicable to symplectic cobordisms). Note that a $W_+$-curve is
automatically simply-covered.

\begin{thm}[Index inequality]
\label{thm: index inequality for W+ and W-} Let $u$ be a $W_+$-curve
from ${\bf y}$ to $\boldsymbol{\gamma}=\prod_{k=1}^l\gamma_k^{m_k}$.  If the
negative ends of $u$ partition $m_k$ into $(m_{k1},m_{k2},\dots)$,
then
$$\op{ind}_{W_+}(u)+2\delta(u)\leq I_{W_+}(u),$$
and equality holds only if
$P^{in}_{\gamma_k}(m_k)=(m_{k1},m_{k2},\dots)$ for all $k$.
Similarly, if $\overline{u}$ is a $\overline{W}_-$-curve from
$\boldsymbol{\gamma}=\prod_{k=1}^l\gamma_k^{m_k}$ to ${\bf y}$ and the positive
ends of $u$ partition $m_k$ into $(m_{k1},m_{k2},\dots)$, then
$$\op{ind}_{W_-}(\overline{u})+2\delta(\overline{u})\leq
I_{W_-}(\overline{u}),$$ and equality holds only if
$P^{out}_{\gamma_k}(m_k)=(m_{k1},m_{k2},\dots)$ for all $k$.
\end{thm}

Here $\delta(u)$ is the signed count of interior singularities of
$u$, where each singularity contributes positively to $\delta(u)$.
There are no boundary singularities because $\pi_{B_+} \circ u$ has no
branch point at the boundary.

The following are immediate: (i) if $\op{ind}_{W_+}(u)=I_{W_+}(u)$,
then $u$ is embedded; and (ii) if $I_{W_+}(u)=0$ or $1$, then $u$ is
embedded.

\begin{proof}
If we plug Equation~\eqref{eqn: ind for W plus}, Equation~\eqref{eqn: definition of ECH index for W+}, and the relative adjunction formula \eqref{eqn: relative adjunction formula for W plus} into $I_{W_+}(u) - \op{ind}_{W_+}(u)$, we obtain:
$$I_{W_+}(u) - \op{ind}_{W_+}(u) = \mu_{\tau}(\boldsymbol{\gamma}, u) + w^-_{\tau}(u) - \widetilde{\mu}_{\tau}(\boldsymbol{\gamma}) +2 \delta(u).$$
The statement for $W_+$ then follows from the writhe inequality
$$w^-_{\tau}(u) \ge \widetilde{\mu}_{\tau}(\boldsymbol{\gamma}) - \mu_{\tau}(\boldsymbol{\gamma}, u),$$
where equality holds if and only if the partition $(m_{k1},m_{k2},\dots)$ of the negative end of $u$ coincides with the incoming partition $P^{in}_{\gamma_k}(m_k)$. (See \cite[Lemma~4.20]{Hu2}.) The proof for $\overline{W}_-$ is similar.
\end{proof}

\subsection{Holomorphic curves with ends at $z_\infty$}
\label{subsection: modified indices at z infty}

In this subsection we explain how to extend the definitions of the Fredholm and ECH indices to holomorphic curves which have ends at multiples of $z_\infty$. The novelty is that the Lagrangian boundary condition is singular at $z_{\infty}$ and that the chord over $z_{\infty}$ can be used more than once.  We will treat in detail the case of a curve $\overline{u}:\dot F\to \overline{W}$ which is a multisection of $\overline{W}\to\R\times[0,1]$; multisections of $\overline{W}_+$ and $\overline{W}_-$ can be treated similarly.

\subsubsection{Data at $z_\infty^p$}

We define the {\em data $\overrightarrow{\mathcal{D}}$ at $z_\infty^p$} as a $p$-tuple of
matchings
\nom[D]{$\overrightarrow{\mathcal{D}}$, $\mathcal{D}^{to}$, $\mathcal{D}^{from}$, $\mathcal{D}$}{Data at $z_\infty^p$ with initial points $\mathcal{D}^{from}$ and terminal points $\mathcal{D}^{to}$. $\mathcal{D}=(\mathcal{D}^{to}, \mathcal{D}^{from})$}
$$\{(i_1',j_1')\to (i_1,j_1), \dots,(i_p',j_p')\to (i_p,j_p)\},$$
where $i_\ell,i_\ell'\in \{1,\dots,2g\}$, $j_\ell,j_\ell'\in\{0,1\}$ for $\ell=1,\dots,p$ and
$i_\ell\not=i_{\ell'}$, $i_\ell'\not=i_{\ell'}'$ for $\ell\not=\ell'$. To the data $\overrightarrow{\mathcal{D}}$
we associate its set of {\em initial points} $\mathcal{D}^{from}= \{ (i_1', j_1'), \ldots ,
(i_p', j_p') \}$ and its set of {\em terminal points} $\mathcal{D}^{to}= \{ (i_1, j_1), \ldots ,
(i_p, j_p) \}$, and define $\mathcal{D} = ( \mathcal{D}^{to}, \mathcal{D}^{from} )$.

We write ${\bf z}=\{z_\infty^p(\overrightarrow{\mathcal{D}})\}\cup {\bf y}$ for a tuple of points of
$\overline{\bf a}\cap \overline{\hh}(\overline{\bf a})$, where ${\bf y}\subset S$,
$z_\infty$ has multiplicity $p$, $\overrightarrow{\mathcal{D}}$ is the data at $z_\infty^p$, and each
arc of $\{\overline{a}_i,\overline{\hh}(\overline{a}_i)\}_{i=1}^{2g}$ is used at most once.
Here $z_\infty^p(\overrightarrow{\mathcal{D}})$ is viewed as a collection of chords $z_\infty$ from
$\overline{\hh}(\overline{a}_{i_\ell',j_\ell'})$ to $\overline{a}_{i_\ell,j_\ell}$, where $\ell=1,\dots,p$.
\nom[z]{${\bf z}= \{z_\infty^p(\overrightarrow{\mathcal{D}})\}\cup {\bf y}$}{A tuple of $\overline{\bf a}\cap \overline{\hh}(\overline{\bf a})$, where $z_\infty$ has multiplicity $p$ and ${\bf y}\subset S$}

\subsubsection{Multisections}\label{subsubsection: multisections}

In this subsection we define multisections $\overline{u}: (\dot F,j)\to (\overline{W},\overline{J})$ with irreducible components which branched cover $\sigma_\infty$.  The notation is complicated by the fact that there may be branch points of $\overline{\pi}_{B}\circ \overline{u}$ along $\bdry B$.

For the rest of Section \ref{subsection: modified indices at z infty}, $(F,j)$ will be a compact {\em nodal} Riemann surface, possibly disconnected, with two sets of punctures ${\bf q}^+=\{q_1^+,\dots,q_\ell^+ \}$ and ${\bf q}^-=\{q_1^-,\dots,q_\ell^-\}$ on $\bdry F$ and a set ${\bf r}=\{r_1,\dots,r_l\}$ of nodes on $\bdry F$, such that
\begin{itemize}
\item[(i)] the nodes ${\bf r}$ are disjoint from ${\bf q}^+$ and ${\bf q}^-$;
\item[(ii)] each connected component of $F$ has nonempty boundary;
\item[(iii)] on each oriented loop of $\bdry F$ whose orientation agrees with that of $\bdry F$, there is at least one puncture from each of ${\bf q}^+$ and ${\bf q}^-$; and
\item[(iv)] the punctures on $\mathbf{q}^+$ and $\mathbf{q}^-$ alternate along each oriented loop of $\bdry F$.
\end{itemize}
Here, by a {\em loop or path on $\bdry F$} we mean a loop or path on $\partial F$ which is a concatenation of subarcs of $\bdry F$ with endpoints on ${\bf q}^+\cup{\bf q}^-\cup {\bf r}$.  An oriented loop (or path) is a loop (or path) formed from subarcs of $\partial F$ which all point in the same direction, when oriented as the boundary of $F$.

We write $\dot F=F-{\bf q}^+-{\bf q}^-$. Given a holomorphic multisection $\overline{u}: (\dot F,j)\to (\overline{W},\overline{J})$, we decompose $F = F' \sqcup F''$ such that
\begin{itemize}
\item $\overline{u}' = \overline{u}|_{\dot F'}$ is a possibly disconnected branched cover of $\sigma_\infty$ and
\item $\overline{u}'' = \overline{u}|_{\dot F''}$  is the union of irreducible components which do not branch cover $\sigma_\infty$.
\end{itemize}
\nom[u]{$\overline{u}=\overline{u}'\cup\overline{u}''$}{Decomposition of a holomorphic multisection $\overline{u}$ in $\overline{W}$, $\overline{W'}$, $\overline{W}_+$, or $\overline{W}_-$ into a possibly disconnected branched cover $\overline{u}'$ of the section at $\infty$ and the rest $\overline{u}''$}
By abuse of notation, we write
\begin{equation}\label{decomposition of overline u}
\overline{u} = \overline{u}' \cup \overline{u}''.
\end{equation}
We additionally assume the following:
\begin{itemize}
\item[(v)] all the nodes of $\bdry F$ lie on $\partial F'$.
\end{itemize}
Observe that the nodes of $\bdry F'$ correspond to branch points of $\overline\pi_B\circ \overline{u}$ along $\bdry B$.   If we view $\dot F'$ as a branched cover of $B$ with some branch points along $\bdry B$, then let $\dot F'_{ext}$ be the branched cover of $B$ obtained by pushing the boundary branch points into $int(B)$. 
The cases of $\overline{W}_+$ and $\overline{W}_-$ are analogous.

The following is the local model near a branch point along the boundary: Consider the map $\C \to\C$, $z \mapsto z^n$. If we write $z = re^{i \theta}$, then the preimage of the upper half plane is a union of sectors $\theta \in \left [ \frac{2k \pi}{n}, \frac{(2k+1) \pi}{n} \right ]$ for $k=0, \ldots, n-1$.

\begin{defn}[Data $\mathcal{C}$] \label{defn: data C}
A {\em maximal path $d$ of $\bdry \dot F'$} is an oriented connected path from ${\bf q}^+$ to ${\bf q}^-$ along $\bdry\dot F'$, where the orientation either agrees with that of $\bdry F$ on all the subarcs of $d$ or is opposite that of $\bdry F$ on all the subarcs of $d$.
A set $\{d_1,\dots,d_r\}$ is a {\em decomposition of $\bdry \dot F'$ into maximal paths}, if $\bdry \dot F'=\cup_{i=1}^r d_i$, $d_i$ is a maximal path of $\bdry \dot F'$, and $d_i$, $d_j$, $i\not=j$, intersect only at points of ${\bf q}^+\cup{\bf q}^-\cup {\bf r}$.  The {\em canonical decomposition $\Delta=\{d_1,\dots,d_r\}$ of $\bdry\dot F'$} is the decomposition which corresponds to $\bdry \dot F'_{ext}$.  A {\em data} $\mathcal{C}$ is a map $$\Delta\to \{\overline{a}_{i,j},\overline{\hh}(\overline{a}_{i,j})\}_{i,j}.$$
\end{defn}

\nom[C]{$\mathcal{C}$}{Data $\mathcal{C}$ as in Definition~\ref{defn: data C}}
In words, $\overline{u}'$ can be viewed as mapping $\Delta$ to the set of Lagrangians $L_{\overline{a}_{i,j}}$ or $L_{\overline{\hh}(\overline{a}_{i,j})}$. 

We then make the following definition which agrees with Definition~\ref{HF-curve} when ${\bf z}_+ ={\bf y}_+$ and ${\bf z}_-={\bf y}_-$:

\begin{defn}  \label{overline W curve, version 2}
Let ${\bf z}_+=\{z_\infty^{p_+}(\overrightarrow{\mathcal{D}}_+)\}\cup {\bf y}_+$ and ${\bf z}_-=\{z_\infty^{p_-}(\overrightarrow{\mathcal{D}}_-)\}\cup {\bf y}_-$ be $k$-tuples of points of $\overline{a}\cap \overline{\hh}(\overline{a})$, where $z_\infty$ is counted with multiplicities $p_+$, $p_-$. A {\em degree $k$ multisection of $(\overline{W},\overline{J})$ from ${\bf z}_+$ to ${\bf z}_-$} is a pair $(\overline{u},\mathcal{C})$ consisting of a holomorphic map
$$\overline{u} : (\dot F,j) \to (\overline{W},\overline{J}),$$
where $\overline{u}$ is a degree $k$ multisection of $\overline\pi_B: \overline{W}\to B$ and there are decompositions $\dot F = \dot F' \sqcup \dot F''$, $\overline{u}=  \overline{u}' \cup \overline{u}''$ as above, and the data $\mathcal{C}$ for $\overline{u}'$, and which additionally satisfies the following:
\begin{enumerate}
\item $\overline{u}''(\partial \dot{F}'' )\subset L_{\widehat{\bf a}} \cup L_{\overline{\hh}(\widehat{\bf a})}$;
\item there is a canonical decomposition $\Delta$ of $\bdry \dot F'$ such that $\overline{u}$ maps each connected component of $\bdry\dot F''$ and each irreducible component of $\Delta$ to a different $L_{\overline{a}_i}$ or $L_{\overline{\hh}(\overline{a}_i)}$  (here we are using $\mathcal{C}$ to assign some $L_{\overline{a}_i}$ or $L_{\overline{\hh}(\overline{a}_i)}$ to each irreducible component of $\Delta$);
\item $\displaystyle{\lim_{w\rightarrow q_i^+}\overline\pi_{\R}\circ \overline{u}(w) =+\infty}$ and $\displaystyle{\lim_{w\rightarrow q_i^-} \overline\pi_{\R}\circ \overline{u}(w) =-\infty}$;
\item $\overline{u}$ converges to a trivial strip over $[0,1]\times {\bf z}_+$ near ${\bf q}^+$ and to a trivial strip over $[0,1]\times {\bf z}_-$ near ${\bf q}^-$;
\item the positive and negative ends of $\overline{u}$ which limit to $z_\infty$ are described by $\overrightarrow{\mathcal{D}}_+$ and $\overrightarrow{\mathcal{D}}_-$;
\item the energy of $\overline{u}$ is finite.
\end{enumerate}
\end{defn}

\begin{rmk}
It will always be assumed that a multisection $\overline{u}$ of $\overline{W}$ from $\{z_\infty^{p_+}\}\cup {\bf y}_+$ to $\{z_\infty^{p_-}\}\cup {\bf y}_-$ comes with data $\mathcal{C}$. Indeed, in this paper, such a curve $\overline{u}$ only appears as the SFT limit of curves $\overline{u}_i$ without components which branch cover $\sigma_\infty$ and hence naturally inherits $\mathcal{C}$ and $\overrightarrow{\mathcal{D}}_\pm$.
\end{rmk}

\subsubsection{Multivalued trivializations}
\label{subsubsubsection: multivalued trivialization}

We define a class of trivializations $\tau=\tau_{\overrightarrow{\mathcal{D}}_+,\overrightarrow{\mathcal{D}}_-}$ which will be used in the definition of the Fredholm index of a map $\overline{u}$ from $\{z_\infty^{p_+}(\overrightarrow{\mathcal{D}}_+)\} \cup {\bf y}_+$ to $\{z_\infty^{p_-}(\overrightarrow{\mathcal{D}}_-)\}\cup {\bf y}_-$. The trivialization $\tau=\tau_{\overrightarrow{\mathcal{D}}_+, \overrightarrow{\mathcal{D}}_-}$ is {\em multivalued} near $z_\infty$ and depends on $\overrightarrow{\mathcal{D}}_+$ and $\overrightarrow{\mathcal{D}}_-$.

Let $\tau'=\bdry_\rho$ be a trivialization of $T\overline{S}$ on $D^2_\delta-\{z_\infty\}=\{0<\rho\leq \delta\}\subset \overline{S}$, where $\delta>0$ is small. Let $\check{\overline{W}}=[-1,1]\times[0,1]\times \overline{S}$. We extend $\tau'$ arbitrarily to a trivialization of $T\overline{S}$ along $\overline{\bf a}$ and pull it back to $\tau_{\overrightarrow{\mathcal{D}}_+,\overrightarrow{\mathcal{D}}_-}$ on $T\overline{S}_{\check{\overline{W}}}$ along ${[-1,1]\times\{1\}\times \overline{\bf a}}$; similarly we extend $\tau'$ arbitrarily to a trivialization of $T\overline{S}$ along $\overline{\hh}(\overline{\bf a})$ and pull it back to $\tau_{\overrightarrow{\mathcal{D}}_+,\overrightarrow{\mathcal{D}}_-}$ on $T\overline{S}_{\check{\overline{W}}}$ along ${[-1,1]\times\{0\}\times \overline{\hh} (\overline{\bf a})}$.\footnote{If we want consistency with $\tau$ from Section~\ref{subsubsection: Fredholm index second version}, we assume that $\tau'$ corresponds to a nonsingular tangent vector field along $\{1\}\times {\bf a}$ and $\{0\}\times\hh ({\bf a)}$.} Then we extend $\tau_{\overrightarrow{\mathcal{D}}_+,\overrightarrow{\mathcal{D}}_-}$ along $(\{ 1 \} \times [0,1] \times \mathbf{y_+}) \cup (\{ -1 \} \times [0,1] \times \mathbf{y_-})$, while still denoting it by $\tau_{\overrightarrow{\mathcal{D}}_+,\overrightarrow{\mathcal{D}}_-}$. Finally, for each $(i_{\pm,\ell}',j_{\pm,\ell}')\to (i_{\pm,\ell},j_{\pm,\ell})$ in $\overrightarrow{\mathcal{D}}_\pm$, we choose an extension $\tau_{\overrightarrow{\mathcal{D}}_+,\overrightarrow{\mathcal{D}}_-}$ to $\{\pm 1\}\times [0,1]\times \{z_\infty\}$ in an arbitrary manner.

 A multivalued trivialization $\tau=\tau_{\overrightarrow{\mathcal{D}}_+,\overrightarrow{\mathcal{D}}_-}$ as constructed in the previous paragraph is said to be {\em compatible with $(\overrightarrow{\mathcal{D}}_+,\overrightarrow{\mathcal{D}}_-)$.} We also say that $\tau$ is {\em compatible with $(\mathcal{D}_+,\mathcal{D}_-)$} if it is compatible with some $(\overrightarrow{\mathcal{D}}_+,\overrightarrow{\mathcal{D}}_-)$ whose initial/terminal point sets are $(\mathcal{D}_+,\mathcal{D}_-)$.

\begin{rmk}
Note that the extensions to $\overline{\bf a}$ and to $\overline{\hh}(\overline{\bf a})$ might conflict, but it does not matter here. In the cobordism cases (i.e., for $\overline{W}_+$ and $\overline{W}_-$), when  $\overline{\bf a}$ and $\overline{\hh}(\overline{\bf a})$ are connected by the Lagrangian $L^{\pm}_{\overline{\bf a}}$, we extend the trivialization $\tau'$ to $\overline{\bf a}$ in an arbitrary manner and then sweep it around using the symplectic connection along the boundary of the cobordism.
\end{rmk}

\subsubsection{Groomed multivalued trivializations}

Let $\tau$ be a multivalued trivialization. Let
$$A_\varepsilon=[0,1]\times \bdry D^2_\varepsilon\subset [0,1]\times \overline{S},$$
where $\varepsilon>0$ is small. We use coordinates $(t,\phi)$ on $$A_\varepsilon\simeq ([0,1]\times [0,2\pi])/(t,0)\sim(t,2\pi).$$

The branches of the trivialization $\tau$ along $[0,1] \times \{ z_{\infty} \}$ can be used to push off $[0,1] \times \{ z_{\infty} \}$ into a collection of arcs $c^\pm_\ell$, $\ell=1,\dots,p_\pm$, in $A_\varepsilon$, subject to some conditions:  If we denote the initial point of $c^\pm_\ell$ by $p^\pm_{\ell,0}$ and the terminal point by $p^\pm_{\ell,1}$, then $p^\pm_{\ell,0}$ lies on $\{0\}\times\overline{\hh}(\overline{a}_{i'_{\pm, \ell}, j'_{\pm, \ell}})$ and $p^\pm_{\ell,1}$ lies on $\{1\}\times \overline{a}_{i_{\pm, \ell}, j_{\pm, \ell}}$.  We also assume that each $c_\ell^\pm$ is linear with respect to the identification of the universal cover of $A_\varepsilon$ with $[0,1]\times \R$.

We then write
$$P^{\pm}_0=\{p^\pm_{\ell,0}\}_{\ell=1}^{p_\pm} \quad \mbox{ and } \quad P^{\pm}_1=\{p^\pm_{\ell,1}\}_{\ell=1}^{p_\pm}.$$

\begin{defn} \label{defn: groomings}
The multivalued trivialization $\tau$ is {\em groomed} if the arcs $\{c^*_\ell\}_{\ell=1}^{p_*}$ are pairwise disjoint for both $*=+$ and $*=-$. The collections $\mathfrak{c}^+= \{c^+_\ell\}_{\ell=1}^{p_+}$ and $\mathfrak{c}^- = \{c^-_\ell\}_{\ell=1}^{p_-}$ are the {\em groomings} at the positive and negative ends.
\end{defn}

\nom[c]{$\mathfrak{c}^+$, $\mathfrak{c}^-$}{Groomings at the positive and negative ends as in Definition~\ref{defn: groomings}}

Note that every groomed multivalued trivialization induces data $\overrightarrow{\mathcal{D}}_\pm$, but not every $\overrightarrow{ \mathcal{D}}_\pm$ is compatible with a groomed multivalued trivialization.

\subsubsection{Index formulas} \label{lerici2}

We first give the Fredholm index of a multisection.

\begin{prop}
Let $\overline{u}:\dot F\to \overline{W}$ be a degree $k$ multisection with data $\overrightarrow{\mathcal{D}}_{\pm}$ at $z_{\infty}$, and let $\tau = \tau_{\overrightarrow{\mathcal{D}} _+, \overrightarrow{\mathcal{D}}_-}$. Then the Fredholm index of $\overline{u}$ is given by the formula:
\begin{equation} \label{Fredholm index revisited}
\op{ind}(\overline{u})=-\chi(F)+k+ \mu_\tau(\overline{u})+2c_1(\overline{u}^*T \overline{S},\tau).
\end{equation}
\end{prop}

The proof is a computation in the pullback bundle $\overline{u}^*T \overline{W}$, which is completely analogous to that yielding Equation~\eqref{eqn: Fredholm index for HF, version 2}.

Next we discuss the ECH index of a multisection $(\overline{u},\mathcal{C})$ with ends
$${\bf z}_+=\{z_\infty^{p_+}(\overrightarrow{\mathcal{D}}_+)\} \cup {\bf y}_+,\quad {\bf z}_-=\{z_\infty^{p_-}(\overrightarrow{\mathcal{D}}_-)\}\cup {\bf y}_-.$$
Let $Z_{{\bf z}_+,{\bf z}_-}$ be the subset of $\check{\overline{W}}$ given by
\begin{align*}
Z_{{\bf z}_+,{\bf z}_-} &= \check{L}_{\widehat{\bf a}}\cup \check{L}_{\overline{\hh}(\widehat{\bf a})}\cup (\{1\}\times[0,1]\times {\bf z}_+) \cup ( \{-1\}\times[0,1]\times {\bf z}_- ),
\end{align*}
where we are viewing ${\bf z}_\pm$ as sets of points (i.e., we are forgetting the data $\overrightarrow{\mathcal{D}}_\pm$),
$$\check{L}_{\widehat{\bf a}}=[-1,1]\times\{1\}\times \widehat{\bf a}, \quad \mbox{and} \quad \check{L}_{\overline{\hh}(\widehat{\bf a})}=[-1,1]\times\{0\}\times \overline{\hh}(\widehat{\bf a}).$$
We then define $H_2(\overline{W}, {\bf z}_+,{\bf z}_-) \subset H_2(\check{\overline{W}},Z_{{\bf z}_+,{\bf z}_-})$ in a manner analogous to that of $H_2(X,{\bf y}, {\bf y}')$ in Section~\ref{subsection: HF holomorphic curves and moduli spaces}. Observe that $H_2(\overline{W},{\bf z}_+,{\bf z}_-)$ and $Z_{{\bf z}_+,{\bf z}_-}$ depend on the endpoints $\mathcal{D}_\pm$ but not on the matchings in $\overrightarrow {\mathcal{D}}_\pm$.

\s\n
{\em $\tau$-trivial representative.}  Let $\tau$ be a groomed multivalued trivialization which is compatible with $(\mathcal{D}_+,\mathcal{D}_-)$ and let ${\frak c}^\pm=\{c_\ell^\pm\}$ be the groomings of $\tau$. 
{\em The matchings defined by ${\frak c}^\pm$ do not need to coincide with the matchings in $\overrightarrow {\mathcal{D}}_\pm$; only their endpoints do.}
The construction of a $\tau$-trivial representative $\check C\subset \check{\overline{W}}$ of $[\overline{u}]\in H_2(\overline{W},{\bf z}_+,{\bf z}_-)$ is as follows:

\s\n
{\em Step 1.} Replace $\overline{u}':\dot F'\to \overline{W}$ by $\overline{u}'_{ext}:\dot F'_{ext}\to \overline{W}$ so that there are no nodes along $\bdry \dot F_{ext}$ and each component of $\bdry \dot F_{ext}$ is mapped to some $L_{\overline{a}_{i,j}}$ or $L_{\overline{\hh}(\overline{a}_{i,j})}$ in a manner consistent with the data $\mathcal{C}:\Delta\to \{\overline{a}_{i,j}, \overline{\hh}(\overline{a}_{i,j})\}_{i,j}$.

\s\n
{\em Step 2.} Compactify the ends of $\overline{u}_{ext}=\overline{u}'_{ext}\cup\overline{u}''$ to $\check{\overline{u}}_{ext}$ as in Section~\ref{subsection: HF holomorphic curves and moduli spaces}.

\s\n
{\em Step 3.} Perturb $\check{\overline{u}}_{ext}$ so that the resulting representative $\check C$ is immersed, $\bdry \check C\subset Z_{{\bf z}_+,{\bf z}_-,\tau}$, each component of $\bdry \check C\cap \{t=0,1\}$ is mapped to $\check L_{\widehat{a}_{i,j}}$ or $\check L_{\overline{\hh}(\widehat{a}_{i,j})}$ as specified by $\mathcal{C}$, and $\check C$ satisfies
$$\pi(\check{C}|_{s=\pm 1})\cap ([0,1]\times int(D^2))={\frak c}^\pm\subset A_\varepsilon$$
and the conditions in Definition \ref{defn: tau-trivial representative HF} at all the other ends.  Here
$$\pi: [-1,1]\times[0,1]\times\overline{S}\to [0,1]\times\overline{S}$$
is the projection onto the second and third factors. Then resolve the self-intersections to make $\check C$ embedded.

\s\n {\em The quadratic form $Q_\tau$.} Let $\check{C}$ be a $\tau$-trivial representative of $\overline{u}$, where $\tau$ is viewed as a nonsingular vector field. Let $c^\pm_\ell(\delta)$ be the $\phi=\delta$ translate of $c^\pm_\ell$, where $\delta>0$ is small. We then take a pushoff $\check{C}'$ of $\check{C}$ such that $\bdry \check{C}$ is pushed in the direction of $J\tau+Y$, where $Y$ is $C^0$-small, and
$$\pi(\check{C}'|_{s=\pm 1})\cap ([0,1]\times int(D^2))=\cup_{\ell=1}^{p_\pm} c^\pm_\ell(\delta)\subset A_\varepsilon.$$
Then $Q_\tau([u])=\langle \check{C},\check{C}'\rangle$ as in Definition~\ref{defn of Q}.

\begin{defn}[ECH index]
Given $A \in H_2(\overline{W}, {\bf z}_+,{\bf z}_-)$ and a groomed multivalued trivialization $\tau$ compatible with $(\mathcal{D}_+,\mathcal{D}_-)$, we define
\begin{equation} \label{ECH index revisited}
I_{\tau}(A) =  Q_\tau(A) + \widetilde{\mu}_\tau(\partial A) + c_1(T \overline{S}|_{A} ,\tau),
\end{equation}
where the analog $\widetilde{\mu}_\tau(\partial A)$ of the symmetric Conley-Zehnder index will be given in Section~\ref{lerici}.
\end{defn}

Sometimes we will write $I_{\tau}(u)$ to mean $I_{\tau}(A)$, where $A$ is the relative homology class defined by $u$.

\subsubsection{Definition of $\widetilde\mu_\tau(\partial A)$} \label{lerici}

We first define the Conley-Zehnder index $\mu_{\tau}(\partial A)$. The groomed multivalued trivialization $\tau$ determines the matchings $\mathcal{D}^{from}_\pm \to \mathcal{D}^{to}_\pm$.  Pick a cycle $\zeta$ which represents $\partial A$ and respects the matchings along $\{ \pm 1 \} \times [0,1] \times \{z_{\infty}\}$. The pullback bundle $\zeta^*T \overline{S}$ is trivialized by the pullback of $\tau$. We now define a multivalued real rank one subbundle $\mathcal{L}_0$ of $T\overline{S}$ along $Z_{\overline{\bf a}, \overline{\hh}(\overline{\bf a})}$ by setting $\mathcal{L}_0=T\check{L}_{\overline{\bf a}}\cap T\overline{S}$ on $\check{L}_{\overline{\bf a}}$ and $T\check{L}_{\overline{\hh}(\overline{\bf a})}\cap T\overline{S}$ on $\check{L}_{\overline{\hh}(\overline{\bf a})}$ and extending $\mathcal{L}_0$ across $\{\pm 1\}\times[0,1]\times {\bf z}_\pm$ by rotating in the counterclockwise direction from $T\overline{\hh}(\overline{\mathbf{a}})$ to $T\overline{\mathbf{a}}$ in $T\overline{S}$ by the minimum amount possible. We then define $\mu_{\tau}(\partial A)$ as the Maslov index of $\mathcal{L}=\zeta^*\mathcal{L}_0$ with respect to $\tau$.

Now we define the corrections to add to $\mu_{\tau}(\partial A)$ to obtain the symmetric Conley-Zehnder index $\widetilde{\mu}_{\tau}(\partial A)$. Let $\tau$ be a groomed multivalued trivialization which is compatible with $(\mathcal{D}_+, \mathcal{D}_-)$ and which induces the groomings $\mathfrak{c}^+= \{c^+_\ell\}_{\ell=1}^{p_+}$ and $\mathfrak{c}^-= \{c^-_\ell\}_{\ell=1}^{p_-}$.

\begin{defn}
Given a grooming ${\frak c}=\{c_\ell\}_{\ell=1}^p$ from $P_0$ to $P_1$, its {\em winding number} is given by:
\begin{equation}\label{eq: winding number}
w(\mathfrak{c})=\sum_\ell w(c_\ell), \quad w(c_\ell)= \langle c_\ell, \{\phi=\pi\}\rangle,
\end{equation}
where $\langle\cdot,\cdot\rangle$ is the algebraic intersection number in $A_\varepsilon$, the arcs $c_\ell$ and $\{\phi=\pi\}$ are oriented from $t=0$ to $t=1$, and $A_\varepsilon$ is oriented by $(\bdry_\phi,\bdry_t)$.
\end{defn}

\begin{rmk}
A grooming $\mathfrak{c}$ is determined by its endpoints $P_0$ and $P_1$ and its winding number $w(\mathfrak{c})$.
\end{rmk}

Given $P_0$, $P_1$, and $q\in\Z$, choose a grooming $\mathfrak{c}=\{c_\ell\}_{\ell=1}^p=\mathfrak{c}(P_0,P_1,q)$ from $P_0$ to $P_1$ such that $w(\mathfrak{c})=q$. Let $c^\flat_{\ell}$ be the linear arc in $A_\varepsilon$ which is disjoint from $\{\phi=\pi\}$ and has the same endpoints as $c_{\ell}$. (Note that the collection $\{c^\flat_{\ell}\}$ is usually not groomed.)

\begin{defn}
$\alpha_{(P_0,P_1,q)}$ (or simply $\alpha_q$ if $P_0, P_1$ are understood) is the number of arcs in $\{c^\flat_{\ell}\}$ whose $\phi$-coordinate {\em decreases} as $t$ increases (in the universal cover).
\end{defn}

If $P_0, P_1$ are fixed, then the number $\alpha_q$ only depends on $q$ (modulo $p$), through the bijection $P_0 \stackrel\sim\to P_1$.

Let $q_\pm=w(\mathfrak{c}^\pm)$ and $P_0^\pm, P_1^\pm$ be the endpoints of $\mathfrak{c}^\pm$.

\begin{defn}
The {\em discrepancies} $d^\pm$ at the positive and negative ends are:
\begin{equation}
d^\pm = - (\alpha_{(P_0^\pm,P_1^\pm,q_\pm)}-\alpha_{(P_0^\pm,P_1^\pm,0)}) - q_\pm(p_\pm-1).
\end{equation}
\end{defn}

\begin{rmk}
Observe that $d^\pm=0$ if the grooming ${\frak c}^\pm$ satisfies $w({\frak c}^\pm)=q_\pm=0$; in other words, we are viewing $q_\pm=0$ as a reference point and comparing the discrepancy that arises from choosing another grooming ${\frak c}^\pm$.
\end{rmk}

\begin{defn}
The {\em symmetric Conley-Zehnder index} $\widetilde\mu_\tau(\bdry A)$ is given by:
\begin{equation}
\widetilde\mu_\tau(\partial A)= \mu_\tau(\partial A) + d^+ - d^-.
\end{equation}
\end{defn}

The discrepancy term $d^+-d^-$ was (somewhat artificially) added to make Lemma \ref{lemma: HF index of sections at infinity} hold.

\begin{rmk} \label{rmk: p zero and p one}
Let us view $P_0^*$, $P_1^*$ as points on $(-\pi,\pi)$, where $*=\pm$. In the special case where the points of $P^*_0$ and $P^*_1$ alternate along $(-\pi,\pi)$, we can write:
\begin{equation} \label{eqn: discrepancy}
d^* = \left \{
\begin{array}{ll} -p_* (q_*-\llfloor {q_*\over p_*}\rrfloor), & \mbox{ if } \min_{x\in P_0} x < \min_{x\in P_1} x;\\
-p_*(q_*-\llceil {q_*\over p_*}\rrceil), & \mbox{ if } \min_{x\in P_0} x > \min_{x\in P_1} x. \end{array} \right.
\end{equation}
Here $\llfloor x\rrfloor$ is the greatest integer $\leq x$ and $\llceil x\rrceil$ is the smallest integer $\geq x$.
\end{rmk}

\subsubsection{ECH indices of branched covers of sections at infinity}

\begin{lemma} \label{lemma: HF index of sections at infinity}
If $A \in H_2(\overline{W},\{z_\infty^p(\overrightarrow{\mathcal{D}}_+), \{z_\infty^p(\overrightarrow{\mathcal{D}}_-)\})$ is the relative homology class of a $p$-fold branched cover of $\sigma_\infty$ with possibly empty branch locus, then $I_{\tau}(A)=0$ for all groomed multivalued trivializations $\tau$ compatible with $(\mathcal{D}_+,\mathcal{D}_-)$.
\end{lemma}

\begin{proof}
Let $\mathfrak{c}^+= \{c^+_\ell\}$ and $\mathfrak{c}^-= \{c^-_\ell\}$ be the groomings of $\tau$ and let $q_\pm=w(\mathfrak{c}^\pm)$. Here $P_i=P_i^+=P_i^-$, $i=0,1$. Let us also assume that $\tau'$ which appears in the definition of $\tau_{\overrightarrow{\mathcal{D}}_+,\overrightarrow{\mathcal{D}}_-}$ in Section~\ref{subsubsubsection: multivalued trivialization} satisfies $\tau'=\bdry_\rho$ on $D^2_{2\varepsilon}-\{z_\infty\}$.

In order to compute $Q_{\tau}(A)$, we choose a $\tau$-trivial representative $\check{C}$ of $A$ such that:
\begin{itemize}
\item $\check{C} \cap ([0, 1] \times [0,1] \times \overline{S}) =  [0, 1] \times \mathfrak{c}^+$;
\item $\check{C} \cap ((-1, 0) \times [0,1] \times \overline{S}) \subset (-1,0) \times [0,1] \times int(D^2_{\varepsilon})$;
\item $\check{C} \cap (\{-1\} \times [0,1] \times \overline{S}) =  \{-1\} \times \mathfrak{c}^-$;
\end{itemize}
and a representative $\check{C}'$ such that:
\begin{itemize}
\item $\check{C}' \cap (\{ 1 \} \times [0,1] \times \overline{S}) =  \{ 1 \} \times \mathfrak{c}^+(\delta)$;
\item $\check{C}' \cap ((0,1) \times [0,1] \times \overline{S}) \subset (0,1) \times [0,1] \times int(D^2_{\varepsilon})$;
\item $\check{C}' \cap ([-1,0] \times [0,1] \times \overline{S}) =  [-1,0] \times \mathfrak{c}^-(\delta)$.
\end{itemize}
Recall that $\mathfrak{c}^\pm \subset A_\varepsilon= [0,1]\times \bdry D^2_\varepsilon$.
Since all the intersections between $\check{C}$ and $\check{C}'$ are contained in the level $s =0$,
$$Q_{\tau}(A) =  \mathfrak{c}^+\cdot \mathfrak{c}^-(\delta) = p(q_+ - q_-).$$

Next we claim that
$$\mu_{\tau}(\partial A)= (\alpha_{q_+} - 2 q_+)- (\alpha_{q_-} - 2 q_-).$$
Indeed, given the end which corresponds to the strand $c^\pm_\ell$, the Maslov index of the end is given by $-2w(c^\pm_\ell)$ if the endpoints of $c^\pm_\ell$ satisfy $0<p^\pm_{\ell,0} < p^\pm_{\ell,1}<2\pi$, and is given by $1-2w(c^\pm_\ell)$ if $0<p^\pm_{\ell,1} < p^\pm_{\ell,0}<2\pi$. On the other hand, the number of strands for which $p^\pm_{\ell,0} > p^\pm_{\ell,1}$ holds is exactly $\alpha_{q_{\pm}}$.
Since the Maslov indices of the portions of $\bdry A$ that map to $[-1,1]\times \{1\}\times \overline{\bf a}$ and $[-1,1]\times \{0\}\times \overline{\hh}(\overline{\bf a})$ are zero,  the claim follows.

Hence,
\begin{align*}
\widetilde{\mu}_{\tau}(\bdry A) ={} &  (\alpha_{q_+} - 2 q_+)-(\alpha_{q_+} - \alpha_0) - q_+ (p-1) \\
&  - (\alpha_{q_-} - 2 q_-) + (\alpha_{q_-} - \alpha_0) + q_- (p-1) \\
={} &  -(q_+ -q_-)(p+1).
\end{align*}

We also have $c_1(T\overline{S}|_{A}, \tau) = q_+ - q_-$.
Putting everything together, we obtain:
$$I_{\tau}(A) = p(q_+ - q_-) - (q_+ -q_-)(p+1) + (q_+ - q_-) =0.$$
This proves the lemma.
\end{proof}

We also state the following lemmas without proof:

\begin{lemma}\label{lemma: ECH index of sections at infinity}
If $\overline{u}: \dot F \to \overline{W}_-$ is a degree $p\leq 2g$ multisection which branch covers
$\sigma_\infty^-$ with possibly empty branch locus, then $I(\overline{u})=0$.
\end{lemma}

\begin{lemma}\label{lemma: ECH index of thin wedges}
If $\overline{u}: \dot F \to \overline{W}_--int(W_-)$ is a degree $p\leq 2g$ multisection with positive ends at $\delta_0$ with total multiplicity $p$ and negative ends at a $p$-element subset of $\{x_1, \ldots, x_{2g}$, $x_1', \ldots, x_{2g}' \}$, then $I(\overline{u})=p$.
\end{lemma}

\subsubsection{Additivity of indices and independence of the trivialization}

The Fredholm and the ECH index are additive with respect to concatenation. The proofs of the following lemmas are straightforward.

\begin{lemma}
Let  $\overline{u}_1$ be a multisection with positive end
$\{z_\infty^p(\overrightarrow{\mathcal{D}})\}\cup {\bf y}$, and let $\overline{u}_2$
be a multisection with negative end $\{z_\infty^p(\overrightarrow{\mathcal{D}})\}
\cup{\bf y}$. If $\overline{u}_1 \# \overline{u}_2$ is the multisection
obtained by gluing $\overline{u}_1$ and $\overline{u}_2$ along their common end, then
$$\op{ind}(\overline{u}_1\# \overline{u}_2)=\op{ind}(\overline{u}_1)+
\op{ind}(\overline{u}_2).$$
\end{lemma}

\begin{lemma}
Let $\tau_2$ and $\tau_1$ be groomed multivalued trivializations compatible with $(\mathcal{D}_2, \mathcal{D}_1)$ and $(\mathcal{D}_1, \mathcal{D}_0)$, respectively, and let $\tau$ be obtained by concatenating $\tau_2$ and $\tau_1$. Given relative homology classes
$$A_2 \in H_2(\overline{W}, \{z_\infty^{p_2}(\overrightarrow{\mathcal{D}}_2)\}\cup {\bf y}_2, \{z_\infty^{p_1}(\overrightarrow{\mathcal{D}}_1)\}\cup {\bf y}_1),$$
$$A_1 \in H_2(\overline{W}, \{z_\infty^{p_1}(\overrightarrow{\mathcal{D}}_1)\}\cup {\bf y}_1, \{z_\infty^{p_0}(\overrightarrow{\mathcal{D}}_0)\}\cup {\bf y}_0),$$
we can form the concatenation
$$A_2 \# A_1 \in  H_2(\overline{W},\{z_\infty^{p_2}(\overrightarrow{\mathcal{D}}_2)\}\cup {\bf y}_2,\{z_\infty^{p_0}(\overrightarrow{\mathcal{D}}_0)\}\cup {\bf y}_0).$$ Then we have:
$$I_{\tau}(A_2 \# A_1)=I_{\tau_2}(A_2)+I_{\tau_1}(A_1).$$
\end{lemma}

In view of the following lemma, we can suppress $\tau $ from $I_\tau$.

\begin{lemma}
\label{lemma: groomed independence}
$I_\tau(A)$ is independent of the choice of groomed multivalued trivialization $\tau$.
\end{lemma}

\begin{proof}
Let $\tau$ and $\tau'$ be two groomed multivalued trivializations adapted to the same data $(\mathcal{D}_+, \mathcal{D}_-)$. It suffices to consider the particular cases when $\tau$ and $\tau'$ differ only either at some $y_i$ or at $z_{\infty}$. In the first case, the argument is analogous to the proof of Lemma \ref{lemma: change of trivialization}. In the second case, we can glue branched covers of $\sigma_{\infty}$ to switch groomings. Then the statement follows from Lemma \ref{lemma: HF index of sections at infinity} and the additivity of the ECH index.
\end{proof}

\subsubsection{The ECH index inequality}

Let $\overline{u}$ be a degree $k$ multisection of $\overline{W}$ from ${\bf z}_+=\{z_\infty^{p_+}(\overrightarrow{\mathcal{D}}_+)\}\cup {\bf y}_+$ to ${\bf z}_-=\{z_\infty^{p_-}(\overrightarrow{\mathcal{D}}_-)\}\cup {\bf y}_-$ such that $\overline{u}=\overline{u}''$. Let $c_\ell^\pm$, $\ell=1,\dots,p_\pm$ be the intersections of $A_\varepsilon$ with the $\pi$-projections of the $\pm$ ends of $\overline{u}$ that limit to $z_\infty$. Here $\varepsilon>0$ is small and depends on $\overline{u}$, and the map $\pi$ projects out the $s$-direction. Let $\mathfrak{c}^\pm = \{ c_\ell^\pm \}_{\ell=1,\dots,p_\pm}$ and let $P_0^\pm$ (resp.\ $P_1^\pm$) be the set of the initial (resp.\ terminal) points of the arcs $c_\ell^\pm$.

\begin{lemma}\label{index inequality for z infinity case}
If ${\frak c}^\pm$ are groomings and $A \in H_2(\overline{W},{\bf z}_+,{\bf z}_-)$ is the relative homology class of $\overline{u}$, then
\begin{equation} \label{eq z infty}
I(A) \geq \op{ind}(\overline{u})+ (d^+-d^-).
\end{equation}
In particular, if the points of $P_0^*$ and $P_1^*$ alternate along $(0,2\pi)$ for both $*=+$ and $-$, then $I(A)\geq \op{ind}(\overline{u})$.
\end{lemma}

\begin{proof}
Let $\tau$ be a groomed multivalued trivialization which is compatible with $(\mathcal{D}_+,\mathcal{D}_-)$ and ${\frak c}^\pm$. If we use $\tau$ to compute both the Fredholm and ECH indices, then the proof of the relative adjunction formula (Lemma \ref{lemma: adjunction formula for X}) goes through unmodified. Hence Equation~\eqref{Fredholm index revisited} and Lemma~\ref{lemma: adjunction formula for X} imply:
\begin{align*}
\op{ind}(\overline{u})&= -\chi(F)+k+ \mu_\tau(\overline{u})+2c_1(\overline{u}^*T \overline{S},\tau) \\
\label{second eq} &= Q_{\tau}(A) + \mu_\tau(\overline{u})+ c_1(\overline{u}^*T \overline{S},\tau) - 2 \delta(\overline{u}).
\end{align*}
Since $\tau$ is compatible with ${\frak c}^\pm$, we have $\mu_\tau(\overline{u}) = \mu_{\tau}(\partial A)$. Comparing with the definition of $I(\overline{u})$, we have:
\begin{align}
I(\overline{u}) &= \op{ind}(\overline{u})+ (\widetilde{\mu}_\tau(\partial A)- \mu_\tau(\partial A))+2\delta(\overline{u})\\
\notag &= \op{ind}(\overline{u})+ (d^+-d^-)+2\delta(\overline{u}).
\end{align}
(Note that $\overline{u}$ has no boundary singularities because $\overline{\pi}_B \circ \overline{u}$ has no branch point at the boundary: in fact
$\overline{u}$ has boundary values in $L_{\hat{\mathbf{a}}} \cup L_{\overline{\hh}(\hat{\mathbf{a}})}$, and therefore the images of different boundary components of $\dot F$ are mutually disjoint.)
Equation~\eqref{eq z infty} follows from observing that $\delta(\overline{u}) \ge 0$ since $\overline{u}$ is $J$-holomorphic.

Next suppose that the points of $P_0^*$ and $P_1^*$ alternate along $(0,2\pi)$ for both $*=+$ and $-$. We claim that $d^+\geq 0$.  Let $\pi:\R\times[0,1]\times \overline{S}\to [0,1]\times \overline{S}$ be the projection onto the second and third factors.  By the positivity of intersections, $\pi(\overline{u})$ is positively transverse to the Hamiltonian vector field $\partial_t$; hence $q_+\leq 0$. By Equation~\eqref{eqn: discrepancy}, if $q_+\leq 0$, then $d^+\geq 0$ (in both cases).  $d^-\leq 0$ is proved similarly.
\end{proof}

\subsubsection{ECH index calculation} \label{example of I calc}

In this subsection we compute the ECH index of a multisection of $\overline{W}$ which is the union of a branched cover of the section at infinity $\sigma_{\infty}$ and a curve which limits to a multiple of $z_\infty$. This calculation will be used in Section~II.\ref{P2-snowman}. For simplicity we consider the cases of
positive ends and negative ends at $z_\infty$ separately. The general statement can easily be deduced from the two separate cases.

Let $\overline{u}=\overline{u}'\cup\overline{u}''$ be a degree $p_1+p_2+l$ multisection of $\overline{W}$, where $\deg\overline{u}'=p_1$, $\overline{u}''$ is a multisection from ${\bf y}$ to $\{z_\infty^{p_2} (\overrightarrow{\mathcal D}_2)\} \cup {\bf y}'$, and $l$ is the cardinality of ${\bf y}'$. Suppose that the data $\overrightarrow{\mathcal{D}}_1$ and $\overrightarrow{\mathcal{D}}_2$ for $\overline{u}'$ and $\overline{u}''$ at the negative ends $z_\infty^{p_1}$ and $z_\infty^{p_2}$ satisfy:
\begin{enumerate}
\item[(D$_1$)] $\overrightarrow{\mathcal{D}}_1=\{(i_1,j_1)\to (i_1,j_1),\dots, (i_{p_1},j_{p_1})\to (i_{p_1},j_{p_1})\}$;
\item[(D$_2$)] $\overrightarrow{\mathcal{D}}_2=\{(i'_1,j'_1)\to (i''_1,j''_1),\dots,(i'_{p_2},j'_{p_2})\to (i''_{p_2},j''_{p_2})\}$;
\item[(D$_3$)] $\mathcal{D}_1^{from}=\mathcal{D}_1^{to}$ and the sets $\mathcal{D}_2^{from}$ and $\mathcal{D}_2^{to}$ are disjoint from $\mathcal{D}_1^{from}=\mathcal{D}_1^{to}$.
\end{enumerate}
Let $\mathcal{E}_{-,\ell}$ be the end of $\overline{u}''$ at $z_{\infty}$ corresponding to $(i_\ell', j_\ell') \to (i_\ell'', j_\ell'')$.

Suppose there exist groomings ${\frak c}_1^-\subset A_{\varepsilon/2}=\bdry D^2_{\varepsilon/2}\times[0,1]$\footnote{Here we are writing $A_{\varepsilon/2}=\bdry D^2_{\varepsilon/2}\times[0,1]$ instead of $[0,1]\times \bdry D^2_{\varepsilon/2}$ to indicate the orientation of $A_{\varepsilon/2}$.} and ${\frak c}_2^-\subset A_{\varepsilon}=\bdry D^2_{\varepsilon}\times[0,1]$ which correspond to $\overrightarrow{\mathcal{D}}_1$ and $\overrightarrow{\mathcal{D}}_2$ and satisfy the following:
\begin{itemize}
\item[(G$_1$)] ${\frak c}_1^-$ has winding number $q_1:=w({\frak c}_1^-) =0$;
\item[(G$_2$)] ${\frak c}_2^-=\pi(\cup_{i=1}^{p_2} \mathcal{E}_{-,\ell})\cap A_\varepsilon$ and $q_2:=w({\frak c}_2^-)=0$ or $1$;
\item[(G$_3$)] the points of $P_0^2$ and $P_1^2$ alternate along $(0,2\pi)$, i.e., the projection of ${\frak c}^-_2$ to $\bdry D^2_\varepsilon$ is injective.
\end{itemize}
Here $P_0^i$ (resp.\ $P_1^i$) is the set of initial (resp.\ terminal) points of ${\frak c}_i^-$.

Let $\pi$ be the projection of $\check{\overline{W}}=[-1,1]\times[0,1]\times\overline{S}$ to $[0,1]\times\overline{S}$ and let $\check C'$ and $\check C''$ be representatives of $\overline{u}'$ and $\overline{u}''$ in $\check{\overline{W}}$ such that $\pi(\check C'|_{s=-1})= {\frak c}_1^-$ and $\pi(\check C''|_{s=-1})= {\frak c}_2^-$.

Let $\pi_\rho:[{\varepsilon\over 2},\varepsilon]\times \bdry D^2 \times[0,1]\to \bdry D^2\times[0,1]$ be the projection along the $\rho$-direction.  Let ${\frak w}$ be the signed number of crossings of $\pi_\rho({\frak c}_1^-\cup {\frak c}_2^-)$, i.e., the writhe. Observe that all the crossings of $\pi_\rho({\frak c}_1^-\cup {\frak c}_2^-)$ are positive as a consequence of (G$1$) and (G$_2$), and $p_1\geq {\frak w}$ as a consequence of (G$_3$).

\begin{lemma}  \label{calc of almost sum}
Let $\overline{u}=\overline{u}'\cup\overline{u}''$ be a degree $p_1+p_2+l$ multisection of $\overline{W}$, where $\deg\overline{u}'=p_1$, $\overline{u}''$ is a multisection from ${\bf y}$ to $\{z_\infty^{p_2} (\overrightarrow{\mathcal D}_2)\} \cup {\bf y}'$, and $l$ is the cardinality of ${\bf y}'$. If the data $\overrightarrow{\mathcal{D}}_1$, $\overrightarrow{\mathcal{D}}_2$ at the negative end satisfy (D$_1$)--(D$_3$) and (G$_1$)--(G$_3$), then
\begin{align} \label{almost additive}
I(\overline{u}) &\geq I(\overline{u}')+I(\overline{u}'') + \left\{
\begin{array}{cl}
2{\frak w}, & \mbox{if $q_2=0$};\\
p_1+{\frak w}, & \mbox{if $q_2=1$}.
\end{array}
\right.
\end{align}
If $\overline{u}'\cap\overline{u}''=\varnothing$ in addition, then equality holds.
\end{lemma}

\begin{proof}
It suffices to prove the equality, assuming $\overline{u}'\cap\overline{u}''=\varnothing$, since the extra intersections contribute positively towards the ECH index. The representatives $\check C'$ and $\check C''$ can be taken to be disjoint; in fact we can assume that $\check{C}''$ is disjoint from the ${\varepsilon\over 2}$-neighborhood $\R \times [0,1] \times D^2_{\varepsilon / 2}$ of $\sigma_{\infty}$.

If $q_2=0$, then we resolve the positive crossings of $\pi_\rho({\frak c}_1^-\cup {\frak c}_2^-)$. This is equivalent to appending a union of disks $\check D\subset [-2,-1]\times[0,1]\times \overline{S}$ to $\check{C}'\cup\check{C}''$ such that $\check D|_{s=-1}=(\check C'\cup \check C'')|_{s=-1}$ and $\pi(\check D|_{s=-2})$\footnote{Here we are taking $\pi$ to be the projection of the appropriate space to $[0,1]\times\overline{S}$.} is a grooming on $A_\varepsilon$ which satisfies $w(\pi(\check D|_{s=-2}))=0$. A quick calculation shows that $\check D$ contributes ${\frak w}$ to $Q$, $0$ to $c_1$, and ${\frak w}$ to $\mu$. Since the discrepancy for $\pi(\check D|_{s=-2})$ is zero, Equation~\eqref{almost additive} follows for $q_2=0$.

If $q_2=1$, then we first switch the crossings of $\pi_\rho({\frak c}_1^-\cup {\frak c}_2^-)$ by appending a union of disks $\check D_{-1}\subset [-2,-1]\times[0,1]\times \overline{S}$ such that $\pi(\check D_{-1}|_{s=-2})$ consists of ${\frak c}_1^-$ on $A_\varepsilon$ and ${\frak c}_2^-$ on $A_{\varepsilon/2}$. This contributes $2{\frak w}$ to $Q$ and $0$ to $c_1$ and $\mu$.  Next append $\check D_{-2}\subset [-3,-2]\times[0,1]\times \overline{S}$ so that $\pi(\check D_{-2}|_{s=-3})$ consists of ${\frak c}_1^-$ on $A_\varepsilon$ and ${\frak d}_2^-$ on $A_{\varepsilon/2}$, where ${\frak d}_2^-$ is groomed and $w({\frak d}_2^-) =0$. The surface $\check D_{-2}$ is similar to that used in the proof of Lemma~\ref{lemma: HF index of sections at infinity} and has zero ECH index. Finally, we resolve the positive crossings of $\pi_\rho({\frak c}_1^-\cup {\frak d}_2^-)$ by appending $\check D_{-3}\subset [-4,-3]\times[0,1]\times \overline{S}$.  Since there are $p_1-{\frak w}$ crossings, $\check{D}_{-3}$ contributes $p_1-{\frak w}$ to $Q$, $0$ to $c_1$, and $0$ to $\mu$.  Equation~\eqref{almost additive} then follows for $q_2=1$.
\end{proof}

Next we consider the variant where the multiple of $z_\infty$ is at the positive end. Let $\overline{u}=\overline{u}'\cup\overline{u}''$ be a degree $p_1+p_2+l$ multisection of $\overline{W}$, where $\deg \overline{u}'=p_1$, $\overline{u}''$ is a multisection from $\{z_\infty^{p_2}(\overrightarrow{\mathcal D}_2)\}\cup {\bf y}'$ to ${\bf y}$, and $l$ is the cardinality of ${\bf y}'$.  We use the same notation as above, with $-$ replaced by $+$.

We make the following assumptions:
\begin{itemize}
\item[(G$_1'$)] ${\frak c}_1^+$ has winding number $q_1:=w({\frak c}_1^+) =0$;
\item[(G$_2'$)] ${\frak c}_2^+=\pi(\cup_{i=1}^{p_2} \mathcal{E}_{+,\ell})\cap A_\varepsilon$ and $q_2:=w({\frak c}_2^+)=0$ or $-1$; and
\item[(G$_3'$)] the projection of ${\frak c}_2^+$ to $\bdry D^2_\varepsilon$ is injective except on $\kappa\geq 0$ short intervals of $\bdry D^2_\varepsilon$ which correspond to thin sectors between $\overline{a}_{i,j}$ and $\overline{\hh}(\overline{a}_{i,j})$.
\end{itemize}
Let ${\frak w}$ be the signed number of crossings of $\pi_\rho({\frak c}_1^+\cup {\frak c}_2^+)$. In this case, all the crossings of $\pi_\rho({\frak c}_1^+\cup {\frak c}_2^+)$ are negative.  The analog of Lemma~\ref{calc of almost sum}, stated without proof, is:

\begin{lemma} \label{calc of almost sum 2}
Let $\overline{u}=\overline{u}'\cup\overline{u}''$ be a degree $p_1+p_2+l$ multisection of $\overline{W}$, where $\deg \overline{u}'=p_1$, $\overline{u}''$ is a multisection from $\{z_\infty^{p_2}(\overrightarrow{\mathcal D}_2)\}\cup {\bf y}'$ to ${\bf y}$, and $l$ is the cardinality of ${\bf y}'$. If the data $\overrightarrow{\mathcal{D}}_1$, $\overrightarrow{\mathcal{D}}_2$ at the positive end satisfy (D$_1$)--(D$_3$) and (G$_1'$)--(G$_3'$), then
\begin{align} \label{almost additive 2}
I(\overline{u}) &\geq  I(\overline{u}')+I(\overline{u}'') + \left\{
\begin{array}{cl}
-{\frak w}, & \mbox{if $q_2=0$};\\
-2{\frak w}, & \mbox{if $q_2=-1$}.
\end{array}
\right.
\end{align}
If $\overline{u}'\cap\overline{u}''=\varnothing$ in addition, then equality holds.
\end{lemma}

We also compute the discrepancy $d_+$ of ${\frak c}_2^+$ when $q_2=-1$. Since $\alpha_{-1}=p_2-1$ and $\alpha_0=\kappa$,
\begin{align}
d_+ & = -(\alpha_{-1}-\alpha_0)-q_2 (p_2-1) \\
\notag &= -((p_2-1)-\kappa) -(-1) (p_2-1) = \kappa.
\end{align}
By Lemma~\ref{index inequality for z infinity case},
\begin{equation} \label{kappa}
I(\overline{u}'') \geq \op{ind}(\overline{u}'') +\kappa.
\end{equation}

\subsubsection{Extended moduli spaces} \label{subsubsection: more moduli spaces}
We now describe the extended moduli spaces which involve multiples of $z_\infty$ at the ends.  Details will be given for $\overline{W}$; the $\overline{W}_+$ and $\overline{W}_-$ cases are analogous.

Let $\mathcal{M}=\mathcal{M}_{\overline{J}}({\bf z}_+,{\bf z}_-)$ be the moduli space of multisections of $(\overline{W},\overline{J})$ from ${\bf z}_+=\{z_\infty^{p_+}(\overrightarrow{\mathcal{D}}_+)\}\cup {\bf y}_+$ to ${\bf z}_-=\{z_\infty^{p_-}(\overrightarrow{\mathcal{D}}_-)\}\cup {\bf y}_-$.  A map $\overline{u} \in \mathcal{M}$ has boundary conditions on $L_{\overline{\mathbf{a}}} \cup L_{\overline{\hh}(\overline{\mathbf{a}})}$. Let $\dagger$ be the modifier ``$\overline{u}'=\varnothing$''. We now describe an enlargement of the moduli space $\mathcal{M}^\dagger$.
\nom[2dagger]{$*=\dagger$}{Modifier ``$\overline{u}'=\varnothing$''}

Recall that $\vec{a}_{i,j}\subset D^2$ is of the form $\{-1<\rho\leq 1,\phi=\phi_{i,j}\}$ for some constant $\phi_{i,j}$.

\begin{defn}\label{defn: extended curves}
An {\em extended $\overline{W}$-curve} $\overline{u}$ from ${\bf z}_+$ to ${\bf z}_-$ is a multisection of $(\overline{W},\overline{J})$ which satisfies the conditions of Definition~\ref{overline W curve, version 2} with $\dot F'=\varnothing$ and (1) and (2) replaced by the following:
\begin{enumerate}
\item[(1$'$)] There exist either positive or negative ends (or both) $\mathcal{E}_{+,i}$ and $\mathcal{E}_{-,i}$ of $\dot F$ that limit to $z_\infty$ such that:
\begin{itemize}
\item $\overline{u}(\bdry \dot F -\cup_i\mathcal{E}_{+,i}-\cup_i\mathcal{E}_{-,i})\subset L_{\widehat{\bf a}}\cup L_{\overline{\hh}(\widehat{\bf a})}$;
\item $\overline{u}$ maps one component of $\bdry\dot F \cap\mathcal{E}_{+,i}$ (resp. $\bdry \dot F\cap\mathcal{E}_{-,i}$) to $L_{\vec{a}_{i_\ell,j_\ell}}$ and the other to $L_{\overline{\hh}(\vec{a}_{i_\ell',j_\ell'})}$, where $(i_\ell',j_\ell')\to(i_\ell,j_\ell)\in \overrightarrow{\mathcal{D}}_+$ (resp.\ $\in \overrightarrow{\mathcal{D}}_-$).
\end{itemize}
\item[(2$'$)] $\overline{u}$ maps each connected component of $\bdry\dot F-\cup_i\mathcal{E}_{+,i}-\cup_i\mathcal{E}_{-,i}$ to a different $L_{\widehat{a}_i}$ or $L_{\overline{\hh}(\widehat{a}_i)}$.
\end{enumerate}
The moduli space of extended $\overline{W}$-curves from ${\bf z}_+$ to ${\bf z}_-$ is denoted by $\mathcal{M}^{\dagger,ext}_{\overline{J}}({\bf z}_+,{\bf z}_-)$.
\nom[2ext]{$*=ext$}{Modifier ``extended moduli space''; cf.\ Definition~\ref{defn: extended curves}}
\end{defn}

The extended moduli spaces $\mathcal{M}^{\dagger,ext}_{\overline{J}_+}({\bf z}_+, \delta_0^r \boldsymbol{\gamma}')$ and $\mathcal{M}^{\dagger,ext}_{\overline{J}_-}(\delta_0^r \boldsymbol{\gamma}',{\bf z}_-)$ are defined similarly.

\s\n
{\em Terminology.} A {\em sector ${\frak S}$ of $D^2_{\rho_0}$, $0<\rho_0\leq 1$, from $\phi_0$ to $\phi_1$} is the map
$$[0,\rho_0]\times[\phi_0,\phi_1]\to D^2_{\rho_0},\quad (\rho,\phi)\mapsto \rho e^{i\phi}.$$
By abuse of notation, we also refer to the image of ${\frak S}$ as a sector.

If $R,R'\subset D^2_{\rho_0}$ are distinct radial rays and $I\subset \R$ is an interval, then ${\frak S}(R,R';I)$ is a sector of $D^2_{\rho_0}$ from $\phi_0$ to $\phi_1$, where $R$ and $R'$ can be written as $\{\phi=\phi_0,\rho\geq 0\}$ and $R'=\{\phi=\phi_1,\rho\geq 0\}$ and $[\phi_0,\phi_1]\subset I$.  If we do not specify $I$, then we write ${\frak S}(R,R')$ and assume that $\phi_1-\phi_0$ is the smallest positive angle.

A sector is {\em large} if it has angle $\phi_1-\phi_0>\pi$ and {\em small} if it has angle $0<\phi_1-\phi_0<\pi$.
\nom[Sfrak]{${\frak S}(R,R')$}{Counterclockwise sector in $D^2_{\rho_0}$ from radial ray $R$ to $R'$}

\s
Let $\mathcal{M}^*=\mathcal{M}^*({\bf z}_+,{\bf z}_-)$, where $*$ denotes any set of modifiers, and let
$$\pi_{D^2}: D^2\times[0,1]\to D^2$$
be the projection onto the first factor. If $\overline{u}\in \mathcal{M}^{\dagger}$, then  $(i_\ell',j_\ell')\to (i_\ell,j_\ell)$ in $\overrightarrow{\mathcal{D}}_+$ corresponds to an end $\mathcal{E}_{+,\ell}$ of $\overline{u}$ such that the restriction of $\pi_{D^2}\circ\overline{u}(\mathcal{E}_{+,\ell})$ to $D^2_\varepsilon$ for $\varepsilon>0$ small is a sector ${\frak S}={\frak S} (\overline{a}_{i_\ell,j_\ell}, \overline{\hh}(\overline{a}_{i_\ell',j_\ell'}))$. If ${\frak S}$ is large, then $\pi_{D^2}\circ \overline{u}(\mathcal{E}_{+,\ell})$ has  a slit of length $0$ by the definition of $\overline{u}\in \mathcal{M}^\dagger$ (i.e., $\pi_{D^2}\circ\overline{u}$ maps no boundary point of $\mathcal{E}_{+,i}$ to $\vec{a}_{i_\ell,j_\ell} - \overline{a}_{i_\ell,j_\ell}$ or $\overline{\hh}(\vec{a}_{i_\ell',j_\ell'}-\overline{a}_{i_\ell',j_\ell'})$).  The
ends of $\overline{u}$ can cover at most one large sector because $n(\overline{u})=m$,
and therefore $\overline{u}$ is an extended curve. The neighborhood of $\overline{u}$ in $\mathcal{M}^\dagger$ is generically a codimension $\geq 1$ submanifold of an extended moduli space $\mathcal{M}^{\dagger,ext}$.

\subsection{Transversality}

We first discuss the regularity of almost complex structures on $\overline{W}$, $\overline{W'}$, $\overline{W}_+$ and $\overline{W}_-$.

\subsubsection{Transversality for $\overline{W}$ and $\overline{W'}$}

\begin{defn}
The almost complex structure $J\in \mathcal{J}_W$ is {\em regular} if $\mathcal{M}_J({\bf y},{\bf y'})$ is transversely cut out for all tuples ${\bf y}$ and ${\bf y'}$ in ${\bf a}\cap \hh({\bf a})$. The almost complex structure $\overline{J}\in \mathcal{J}_{\overline{W}}$ is {\em regular} if $\mathcal{M}^{\dagger,ext}_{\overline{J}}({\bf z},{\bf z'})$ is transversely cut out for all tuples ${\bf z}=\{z_\infty^p(\mathcal{D})\}\cup {\bf y}$ and ${\bf z'}=\{z_\infty^q(\mathcal{D'})\}\cup {\bf y'}$.
\end{defn}

Note that $\mathcal{M}_J({\bf y},{\bf y'})=\mathcal{M}^s_J({\bf y},{\bf y'})$ and $\mathcal{M}^{\dagger,ext}_{\overline{J}}({\bf z},{\bf z'})= \mathcal{M}^{\dagger,ext,s}_{\overline{J}}({\bf z},{\bf z'})$  because the positive ends have boundaries on distinct Lagrangian submanifolds.
This implies that a generic $J\in \mathcal{J}_W$ is regular; see Lemma~\ref{lemma: HF regularity}. The same proof also gives:

\begin{lemma} \label{lemma: HF regularity for overline J}
A generic $\overline{J}\in \mathcal{J}_{\overline{W}}$ is regular.
\end{lemma}

We write $\mathcal{J}_W^{reg}\subset \mathcal{J}_W$ for the subset of all regular $J$ and  $\mathcal{J}_{\overline{W}}^{reg}\subset\mathcal{J}_{\overline{W}}$ for the subset of all regular $\overline{J}$.

\begin{defn}
The almost complex structure $J'\in \mathcal{J}_{W'}$ is {\em regular} if $\mathcal{M}_{J'}^s(\boldsymbol{\gamma},\boldsymbol{\gamma}')$ is transversely cut out (in the Morse-Bott sense) for all $\boldsymbol{\gamma}$ and $\boldsymbol{\gamma}'$ in $\widehat{\mathcal{O}}_k$ and all $k\leq 2g$. The almost complex structure $\overline{J'}\in \mathcal{J}_{\overline{W'}}$ is {\em regular} if $\mathcal{M}_{\overline{J'}}^s (\delta^p\boldsymbol{\gamma},\delta^q\boldsymbol{\gamma}')$ is transversely cut out for all $\delta^p\boldsymbol{\gamma}$, $\delta^p\boldsymbol{\gamma}'\in \overline{\mathcal{O}}_k$ and for all $k\leq 2g$.
\end{defn}

Recall that a generic $J'\in \mathcal{J}_{W'}$ is regular by Lemma~\ref{lemma: PFH transversality}.  The same proof also gives:

\begin{lemma}
A generic $\overline{J'}\in \mathcal{J}_{\overline{W'}}$ is regular.
\end{lemma}

We write $\mathcal{J}_{W'}^{reg}\subset \mathcal{J}_{W'}$ for the subset of all regular $J'$ and $\mathcal{J}_{\overline{W'}}^{reg}\subset \mathcal{J}_{\overline{W'}}$ for the subset of all regular $\overline{J'}$.

\subsubsection{Transversality for $W_+$, $\overline{W}_+$ and $\overline{W}_-$}

\begin{defn} \label{regular pear}
The almost complex structure $J_+\in \mathcal{J}_{W_+}$ is {\em regular} if the following hold:
\begin{enumerate}
\item all moduli spaces $\mathcal{M}_{J_+}(\mathbf{y},\boldsymbol{\gamma})$ with ${\bf y}$ a $k$-tuple of ${\bf a}\cap \hh({\bf a})$ and $\boldsymbol{\gamma}\in \widehat{\mathcal{O}}_k$, for all $k\leq 2g$, are transversely cut out (in the Morse-Bott sense in the case of a Morse-Bott building); and
\item the restrictions $J$ and $J'$ of $J_+$ to the positive and negative ends belong to $\mathcal{J}_W^{reg}$ and $\mathcal{J}_{W'}^{reg}$, respectively.
\end{enumerate}
\end{defn}

Here every $u\in \mathcal{M}_{J_+}(\mathbf{y},\boldsymbol{\gamma})$ is somewhere injective due to the presence of the HF end. In other words, $\mathcal{M}_{J_+}(\mathbf{y},\boldsymbol{\gamma})=\mathcal{M}_{J_+}^s(\mathbf{y},\boldsymbol{\gamma})$.

\begin{defn}\label{regular apple}
The almost complex structure $\overline{J}_+\in \mathcal{J}_{\overline{W}_+}$ is {\em regular} if the following hold:
\begin{enumerate}
\item all moduli spaces $\mathcal{M}^{\dagger,ext}_{\overline{J}_+}({\bf z},\delta_0^r\boldsymbol{\gamma}')$ with ${\bf z}=\{z_\infty^p(\mathcal{D})\}\cup {\bf y}$, $\boldsymbol{\gamma}'\in \widehat{\mathcal{O}}_k$ for some $k \le 2g$, and ${\bf y}$ a tuple of ${\bf a}\cap \hh({\bf a})$, are transversely cut out; and
\item the restrictions $\overline{J}$ and $\overline{J'}$ of $\overline{J}_+$ to the positive and negative ends belong to  $\mathcal{J}_{\overline{W}}^{reg}$ and $\mathcal{J}_{\overline{W'}}^{reg}$, respectively.
\end{enumerate}
\end{defn}

Note that $\mathcal{M}^{\dagger,ext}_{\overline{J}_+}({\bf z},\delta_0^r\boldsymbol{\gamma}')=\mathcal{M}^{\dagger,ext,s}_{\overline{J}_+}({\bf z},\delta_0^r\boldsymbol{\gamma}')$. The regularity of $\overline{J}_-\in \mathcal{J}_{\overline{W}_-}$ is defined similarly.

We write $\mathcal{J}_{W_+}^{reg}$ for the space of regular $J_+\in\mathcal{J}_{W_+}$ and $\mathcal{J}_{\overline{W}_\pm}^{reg}$ for the space of regular $\overline{J}_\pm\in\mathcal{J}_{\overline{W}_\pm}$. Without loss of generality we may assume that the regular $J_+$ of interest are the restrictions of regular $\overline{J}_+$.

\begin{rmk}
The vertical fibers $\{(s,t)\}\times S$ and $\{(s,t)\}\times \overline S$ are holomorphic, but are not transversely cut out.
\end{rmk}

\begin{prop}
\label{prop: J+ and J- regular}
A generic admissible $J_+$ (resp.\ $\overline{J}_\pm$) is regular.
\end{prop}

\begin{proof}
We first treat the $W_+$ case. The proposition follows from a standard transversality argument along the lines of \cite[Theorem 3.1.5]{MS}, with some modifications. The necessary modifications for almost complex structures $J'\in \mathcal{J}_{W'}$, defined on $\R\times N$, were described in \cite[Lemma 9.12(b)]{Hu1}, and our situation is almost identical since $J_+\in \mathcal{J}_{W_+}$ is the restriction to $W_+$ of some $J'\in \mathcal{J}_{W'}$.

The key observation is that each irreducible component of a $W_+$-curve $u:\dot F\to W_+$ is somewhere injective, since each $[0,1]\times \{y_i\}$, $y_i\in \mathbf{y}$, is used exactly once as a positive asymptotic limit.  Let $\pi_N: W_+\to N$ be the restriction of the projection $\pi_N:\R\times N\to N$ onto the second factor. We then claim that there is a dense open set of points $p\in \dot F$ which are {\em $\pi_N$-injective}, i.e.,
\begin{enumerate}
\item[(i)] $d(\pi_N\circ u)(p)$ has rank $2$; and
\item[(ii)] $(\pi_N\circ u)(p)=(\pi_N\circ u)(q)$ implies $p=q$.
\end{enumerate}
 First note that if the claim does not hold, then there exist open sets $U, U'\subset \dot F$ such that $u(U)$ and $u(U')$ differ by some translation $T_{s_0}$ by $s_0$ in the $s$-direction. By repeated application of $T_{s_0}$ or $T_{-s_0}$, there is an infinite sequence $U, U'', U''', \dots \subset \dot F$ that have the same $\pi_N\circ u$ image.  This contradicts the finite energy condition, and hence proves the claim.
The perturbations to $J_+$ can then be carried out in a neighborhood of a $\pi_N$-injective point $p\in \dot F$ as in \cite[Lemma 9.12(b)]{Hu1}.

The regularity of the almost complex structures $J$ and $J'$ at the ends was already treated, i.e., $\mathcal{J}_W^{reg}\subset \mathcal{J}_W$ and $\mathcal{J}_{W'}^{reg}\subset \mathcal{J}_{W'}$ are dense by Lemmas~\ref{lemma: PFH transversality} and \ref{lemma: HF regularity}.

In the $\overline{W}_+$ case, the perturbations of $\overline{J}_+$ are allowed on the subset $U=\overline{W}_+\cap ((\R\times\overline{N})-\{\rho\leq \varepsilon\})$ for some small $\varepsilon>0$.  We simply observe that all the curves $\overline{u}$ in the moduli spaces $\mathcal{M}^{\dagger,ext}_{\overline{J}_+}({\bf z},\delta_0^r\boldsymbol{\gamma}')$ in Definition~\ref{regular apple} pass through $U$, and pick a $\pi_N$-injective point $p\in\dot F$ such that $\overline{u}(p)\in U$. The $\overline{W}_-$ case is similar.
\end{proof}

\subsubsection{Some automatic transversality results}

We collect some automatic transversality results.

\begin{lemma} \label{lemma: regularity of curve at infinity}
The curve $\sigma_\infty^-\subset \overline{W}_-$, viewed as having Lagrangian boundary $L^-_{\vec{a}_{i,j}}$, is a regular holomorphic curve with $\op{ind}(\sigma_\infty^-)=0$.
\end{lemma}

\begin{proof}
We first calculate the Fredholm index of the holomorphic embedding $\overline{u}: \dot F \to \overline{W}_-$ with image $\sigma_\infty^-\subset \overline{W}_-$. Here $\dot F$ is a disk with a boundary puncture and an interior puncture. Let $\tau$ be a groomed trivialization whose grooming ${\frak c}$ corresponds to the matching $(i,j)\to (i,j)$ and satisfies $w({\frak c})=0$. Then we compute that $-\chi(\dot F)=0$, $\mu_\tau(\delta_0,\overline{u})=1$, $\mu_\tau(z_\infty)=1$, $c_1(\overline{u}^*T\overline{S},\tau)=0$. Hence $\op{ind}(\overline{u})=0$ by Equation~\eqref{eqn: second in W minus proposition}, which is still valid in the current situation.

We now use the doubling technique from Theorem~\ref{HLS}. The double of $\dot F$ --- a sphere with three punctures --- is denoted by $2\dot F$ and the double of $\overline{u}$ is denoted by $2\overline{u}$. The index of the doubled operator $2D_{\overline{u}}$ is $\op{ind}(2\overline{u})=2\op{ind}(\overline{u})=0$ and $D_{\overline{u}}$ is surjective if and only if $2D_{\overline{u}}$ is surjective.

Now, by Wendl's automatic transversality theorem~\cite[Theorem 1]{We3}, $2D_{\overline{u}}$ is surjective if
\begin{equation} \label{eqn: Wendl automatic transversality}
\op{ind}(2\overline{u})\geq 2g + \#\Gamma_0-1,
\end{equation}
where $g$ is the genus of $2\dot F$ and $\#\Gamma_0$ is the count of punctures with even Conley-Zehnder index. In the present situation, $g=0$ and $\#\Gamma_0=0$, so Equation~\eqref{eqn: Wendl automatic transversality} becomes $\op{ind}(2\overline{u})\geq -1$, which is satisfied.
\end{proof}

The following is easier, and is stated without proof:

\begin{lemma}\label{lemma: transversality thin strips}
Let $\overline{J}$ and $\overline{J}_0$ be almost complex structures as in Definition \ref{defn: almost complex structures on overline W}.
If $\overline{J}$ is sufficiently close to $\overline{J}_0$, then $\mathcal{M}^{\dagger,n^*=1}_{\overline{J}}(\{z_\infty\}\cup {\bf y}, \{y_0\}\cup{\bf y})/\R$ is transversely cut out and consists of a unique curve which is represented by a thin strip in $D^2$ from $z_\infty$ to $y_0=x_i$ or $x_i'$.
\end{lemma}

\subsubsection{Marked points and transversality} \label{subsubsection: marked points and transversality}

In the definition of the $\Psi$-map in Section~\ref{section: chain map psi}, we consider multisections of $\overline{W}_-$ which pass through the marked point $\overline{\frak m}=((0,{3\over 2}),z_\infty)$. The marked point $\overline{\frak m}$, however, is {\em nongeneric}.  In order to ensure the regularity of such moduli spaces with respect to $\overline{\frak m}$, we need to enlarge the class of $\overline{J}_-\in \mathcal{J}_{\overline{W}_-}^{reg}$ to the class of $\overline{J}_-^\Diamond$, which we now define.

\begin{defn} \label{defn: diamond}
Let $\varepsilon>0$ and let $U\not\ni \overline{\frak m}$ be an open set of $\overline{W}_-$.  Then an almost complex structure $\overline{J}_-^\Diamond$ on $\overline{W}_-$ is {\em $(\varepsilon,U)$-close} to $\overline{J}_-$ if:
\begin{itemize}
\item $\overline{J}_-^\Diamond=\overline{J}_-$ on $\overline{W}_--U$;
\item $\overline{J}_-^\Diamond$ is $\varepsilon$-close to $\overline{J}_-$ on $U$; and
\item $\nabla \overline{J}_-^\Diamond$ is $\varepsilon$-close to $\nabla \overline{J}_-$ on $U$.
\end{itemize}
\end{defn}

\nom[J4]{$\overline{J}_-^\Diamond=\overline{J}_-^\Diamond(\varepsilon,U)$}{Almost complex structure on $\overline{W}$ that is $(\varepsilon,U)$-close to $\overline{J}_-$ as in Definition~\ref{defn: diamond} and Convention~\ref{convention}}
\nom[1e$\epsilon$]{$(\varepsilon,U)$ or $(\varepsilon,\delta,p)$}{Perturbation data for $\overline{J}_-^\Diamond$; cf.\ Convention~\ref{convention}}

Here the $\varepsilon$-closeness is measured with respect to a metric $g$ on $\overline{W}_-$ which is the restriction of an $s$-invariant metric on $\R\times \overline{N}$ and $\nabla$ is the Levi-Civita connection of $g$.

\begin{convention} \label{convention}
Unless stated otherwise:
\begin{itemize}
\item $U=U_{p,2\delta}=\overline{\pi}^{-1}_{B_-}(B_\delta(p))-\{\rho\leq \delta\}$ is an open neighborhood of $K_{p,2\delta}= \overline{\pi}^{-1}_{B_-}(p)-\{\rho< 2\delta\},$ where $\delta>0$ is arbitrarily small, $\overline{\frak m}^b\not=p\in B_-$, and $B_\delta(p)\subset B_-$ is an open ball of radius $\delta$ about $p$.
\item $\overline{J}_-^\Diamond$ is $(\varepsilon,U)$-close to $\overline{J}_-\in \mathcal{J}_{\overline{W}_-}^{reg}$, where $\varepsilon>0$ is arbitrarily small.
\end{itemize}
Observe that $U$ is disjoint from the section at infinity. When we want to emphasize $(\varepsilon,U)$ or $(\varepsilon,\delta,p)$, we write $\overline{J}_-^\Diamond(\varepsilon,U)$ or $\overline{J}_-^\Diamond(\varepsilon,\delta,p)$.
\nom[K]{$K_{p,2\delta}$}{$\overline{\pi}^{-1}_{B_-}(p)-\{\rho< 2\delta\}$ with $\delta>0$ small}
\end{convention}

\begin{defn} \label{defn: almost multisection}
A {\em degree $k$ almost multisection of $(\overline{W}_-,\overline{J}_-^\Diamond)$ from $\delta_0^r\boldsymbol{\gamma}$ to ${\bf z'}=\{z_\infty^q(\mathcal{D}')\}\cup {\bf y'}$} is a pair $(\overline{u},\mathcal{C})$ which is defined in the same way as a degree $k$ multisection of $(\overline{W}_-,\overline{J}_-)$, except that $\overline{u}$ is a degree $k$ multisection of
$$\overline\pi_{B_-}: \overline{W}_- - \overline\pi_{B_-}^{-1}(B_\delta(p))\to B_- - B_\delta(p).$$
\end{defn}

Let $\mathcal{M}_{\overline{J}_-^\Diamond} (\delta_0^r\boldsymbol{\gamma},{\bf z'})$ be the moduli space of almost multisections of  $(\overline{W}_-,\overline{J}_-^\Diamond)$ from $\delta_0^r\boldsymbol{\gamma}$ to ${\bf z'}$. The regularity of $\overline{J}_-$ and the closeness of $\overline{J}_-^\Diamond$ to $\overline{J}_-$ imply:
\begin{enumerate}
\item[(i)] $\mathcal{M}_{\overline{J}_-^\Diamond}^\dagger (\delta_0^r\boldsymbol{\gamma},{\bf z'})$ is regular;
\item[(ii)] $\mathcal{M}_{\overline{J}_-^\Diamond}^\dagger (\delta_0^r\boldsymbol{\gamma},{\bf z'})$ is close to $\mathcal{M}_{\overline{J}_-}^\dagger (\delta_0^r\boldsymbol{\gamma},{\bf z'})$; and
\item[(iii)] all the boundary strata of $\mathcal{M}_{\overline{J}_-^\Diamond}^\dagger (\delta_0^r\boldsymbol{\gamma},{\bf z'})$ are close to the corresponding boundary strata of $\mathcal{M}_{\overline{J}_-}^\dagger (\delta_0^r\boldsymbol{\gamma},{\bf z'})$.
\end{enumerate}
Note that we can still refer to $n^*(\overline{u})$ since it is a homological quantity.

Let $K\not\ni \overline{\frak m}$ be a compact set of $\overline{W}_-$. We define the modifier $K$ to mean that $\overline{u}$ passes through $K$.
\nom[2k]{$*=K$}{Modifier ``passes through the compact set $K$''}
\nom[M8]{$\mathcal{M}^*_{\overline{J}_-^\Diamond}(\star_1,\star_2;\overline{\frak m})$}{Moduli space of $\overline{J}_-^\Diamond$-holomorphic multisections that pass through $\overline{\frak m}$}
\begin{defn}
The almost complex structure $\overline{J}_-^\Diamond$ on $\overline{W}_-$ is {\em $K$-regular with respect to $\overline{\frak m}$} if all the moduli spaces $\mathcal{M}^{\dagger,ext,K}_{\overline{J}_-^\Diamond} (\delta_0^r\boldsymbol{\gamma},\mathbf{z'};\overline{\frak m})$ are transversely cut out.
\end{defn}

\begin{lemma} \label{lemma: regularity of W minus diamond}
A generic $\overline{J}_-^\Diamond$ is $K_{p,2\delta}$-regular with respect to $\overline{\frak m}$.
\end{lemma}

\begin{proof}
The proof is similar to that of \cite[Theorem~3.1.7]{MS}, with modifications as in Proposition~\ref{prop: J+ and J- regular}.
\end{proof}

\section{The chain map from $\widehat{HF}$ to $PFH$}
\label{section: chain map phi}

\subsection{Compactness for $W_+$-curves}

In this subsection we treat the compactness of holomorphic curves in $W_+$ which will be used to establish the chain map $\Phi$ in Section~\ref{subsection: defn of chain maps}.

Suppose $J_+\in \mathcal{J}_{W_+}$ and $J$, $J'$ are the restrictions of $J_+$ to the positive and negative ends. Let $\mathbf{y}=\{y_1,\dots,y_{2g}\}\in \mathcal{S}_{{\bf a},\hh({\bf a})}$ and $\boldsymbol{\gamma}=\prod_{k=1}^l \gamma_k^{m_k}\in \widehat{\mathcal{O}}_{2g}$.

\s {\em In this section, we may pass to a subsequence of a sequence of holomorphic curves without specific mention.}

\subsubsection{Euler characteristic bounds}

We first state a preliminary lemma:

\begin{lemma} \label{lemma: bound on F for W plus}
Let $u_i: (\dot F_i,j_i)\to (W_+,J_+)$, $i\in \N$, be a sequence of $W_+$-curves from $\mathbf{y}$ to $\boldsymbol{\gamma}$ with index $I_{W_+}(u_i) =n$ for some integer $n$. Then there is a subsequence such that all the $\dot F_i$ are diffeomorphic to a fixed $\dot F$.
\end{lemma}

\begin{proof}
The proof is given in two steps.

\s\n {\bf Step 1} ($\omega$-area bounds). We define the {\em $\omega$-area} of $u_i$ as
\nom[E2]{$E_\omega(u)$}{$E_\omega$-area of the holomorphic curve $u$}
$$E_\omega(u_i)= \int_{\dot F_i} u_i^* \omega,$$
where $\omega$ is the $2$-form as in Section~\ref{subsection: symplectic cobordisms}.
The boundedness of $E_\omega(u_i)$ is a consequence of the vanishing of the flux $F_{\hh}$ of $\hh$ (cf.\ Section~\ref{subsection: flux}). View the broken closed
 string $\gamma_{\bf y}$ corresponding to ${\bf y}$ as a collection of curves in
$$(L_{\bf a}^+\cap \check W_+)\cup (\{3\}\times[0,1]\times{\bf y})\subset \check{W}_+.$$
Then $\gamma_{\bf y}$ is uniquely determined up to a homotopy which is supported on $L_{\bf a}^+\cap \check W_+$. Let $u_i: \dot F_i\to W_+$, $i=1,2$, be two $W_+$-curves from $\mathbf{y}$ to $\boldsymbol{\gamma}$ and let $\check u_i: \check F_i \to \check W_+$ be their compactifications. Then $\check u_i(\bdry_+\check F_i)$ is homotopic to $\gamma_{\bf y}$, $i=1,2$, where $\bdry_+ \check F_i$ is the union of boundary components of $\check F_i$ which map to the positive ($s>0$) part of $\check W_+$. Hence $\check u_1-\check u_2$ can be viewed as a closed surface $Z\in H_2(W_+)\simeq H_2(\check{W}_+)\simeq H_2(N)$. Since the flux $F_\hh$ vanishes, the integral of $\omega$ over $Z$ vanishes, and therefore the $\omega$-area of a $W_+$-curve $u$ only depends on ${\bf y}$ and $\boldsymbol{\gamma}$.

\s\n {\bf Step 2} (Genus bounds).  The $\omega$-area bound on the sequence $\{u_i\}$ and the fact that all $u_i$ have the same asymptotics imply an energy bound on the $u_i$ (for the energy defined as in Equation~\eqref{eqn: energy of Lipshitz curve}). The energy bound implies a local $\Omega_+$-area bound, and therefore we can apply the Gromov-Taubes compactness
theorem~\cite[Proposition~3.3]{T3}, which is a local result and carries over to the symplectic cobordism $(W_+,\Omega_+)$ without difficulty. As explained in \cite[Lemma~9.8]{Hu1}, the Gromov-Taubes compactness theorem implies the weak convergence of $u_i$ as currents to a holomorphic building $u_\infty$. In particular, we may assume that 
$[u_i]\in H_2(W_+,{\bf y},\boldsymbol{\gamma})$  is fixed for all $i$.

We now use the fact that $[u_i]$ is fixed to bound the genus of $\dot F_i$. The relative adjunction formula (Lemma~\ref{lemma: relative adjunction for W plus}) gives:
$$c_1(\check u_i^*TW_+,(\tau,\bdry_t)) = \chi(\dot F_i) -w_\tau^-(u_i) +Q_\tau(u_i) - 2 \delta(u_i).$$
In view of the writhe bound
$$ w_\tau^-(u_i)\geq \widetilde\mu_\tau(\boldsymbol{\gamma})-\mu_\tau^-(u_i)$$
from \cite[Lemma~4.20]{Hu2} and the nonnegativity of $\delta(u_i)$, we obtain:
\begin{equation}
\chi(\dot F_i)\geq c_1(\check u_i^*TW_+,(\tau,\bdry_t))+\widetilde\mu_\tau(\boldsymbol{\gamma})-\mu_\tau^-(u_i) -Q_\tau(u_i).
\end{equation}
This bounds $\chi(\dot F_i)$ from below, since all the terms on the
right-hand side either depend on the homology class of $u_i$ or the
data of the ends. Hence we may assume that all the $\dot F_i$ are
diffeomorphic to a fixed $\dot F$.
\end{proof}

\subsubsection{SFT compactness}
\begin{defn} \label{defn of holom building}
A {\em holomorphic $W_+$-building}
$$u_\infty=v_{-b}\cup\dots \cup v_a,\quad a,b\in \Z^{\geq 0}$$
consists of the following data:
\begin{enumerate}
\item[(B1)]  For each $j=-b,\dots,a$, a compact {\em nodal} Riemann surface $G_j$ (possibly with boundary), a nodal set $\mathfrak{n}_j$, disjoint sets of interior punctures  ${\bf p}_j^+$ and ${\bf p}_j^-$ if $j \le 0$, and disjoint sets of boundary punctures ${\bf q}_j^+$ and ${\bf q}_j^-$ if $j \ge 0$.  The nodes may be interior or boundary nodes and are disjoint from ${\bf p}_j^\pm$, ${\bf q}_j^\pm$, and the nodal set $\mathfrak{n}_j$ may be empty.
\item[(B2)] For each $j=-b,\dots,a$, a holomorphic map $v_j: \dot G_j\to W_j$, where $W_j=W$ for $0<j\leq a$, $W_0=W_+$ for $j=0$, $W_j=W'$ for $-b\leq j<0$, and $\dot G_j=G_j-{\bf p}_j^+-{\bf p}_j^--{\bf q}_j^+-{\bf q}_j^-$ for all $j$. (Here some sets of punctures may be empty.)
\item[(B3)] For each $j=0,\dots,a$, $\bdry \dot G_j$ is mapped to the appropriate Lagrangian submanifold $L_{\bf a}\sqcup L_{\hh({\bf a})}$ or $L^+_{\bf a}$.
\item[(B4)] For each $j=-b,\dots,a$, $v_j$ converges to a strip over a ``Reeb chord'' at the positive (resp.\ negative) end near each boundary puncture of ${\bf q}_j^+$ (resp.\ ${\bf q}_j^-$) and to a cylinder over a closed orbit at the positive (resp.\ negative) end near each interior puncture of ${\bf p}_j^+$ (resp.\ ${\bf p}_j^-$).
\item[(B5)] for each $j=-b,\dots,a-1$, there is an identification between ${\bf p}_j^+$ and ${\bf p}_{j+1}^-$ and an identification between ${\bf q}_j^+$ and ${\bf q}_{j+1}^-$ such that the pairs that are identified are asymptotic to the same Reeb chord or closed orbit.
\item[(B6)] No level $v_j$ is a union of trivial cylinders or trivial strips.
\end{enumerate}
The levels $v_j$ are arranged in order from lowest to highest.
\end{defn}

\begin{prop} \label{prop: SFT compactness for W plus}
Let $u_i: (\dot F_i,j_i)\to (W_+,J_+)$, $i\in \N$, be a sequence of $W_+$-curves from $\mathbf{y}$ to $\boldsymbol{\gamma}$ with index $I_{W_+}(u_i) =n$ for some integer $n$. Then there is a subsequence which converges in the sense of SFT to a level $a+b+1$ holomorphic $W_+$-building $u_\infty=v_{-b}\cup\dots \cup v_a$.
\end{prop}

Here ``convergence in the SFT sense'' means convergence with respect to the topology described in \cite{BEHWZ}.

\begin{proof}
By Lemma~\ref{lemma: bound on F for W plus}, we may assume that $\dot F_i=\dot F$
as smooth surfaces.  We can then apply the SFT compactness theorem from \cite{BEHWZ}; note that SFT compactness in the presence of Lagrangian boundary conditions is sketched in \cite[Section 11.3]{BEHWZ} and \cite[Theorem 3.20]{Abbook}.  More details can be extracted from the proof of the less standard compactness theorem for $\overline{W}_-$-curves; see Proposition~\ref{prop: SFT compactness for W minus}.
\end{proof}

The limiting curve $u_\infty$ can be written as a level $a+b+1$ holomorphic building $v_{-b}\cup\dots \cup v_{a}$, where each $v_j$ is not necessarily irreducible and may have nodes. Here (i) $a,b$ are nonnegative integers, (ii) $v_j$ is a holomorphic map to $W_j$, and (iii) the levels $v_j$ are ordered from the negative end to the positive end as $j$ increases.  As usual, if the level $v_j$ is just a union of trivial cylinders, then it will be elided.

\subsubsection{Main theorem}

Suppose $J_+\in \mathcal{J}_{W_+}^{reg}$ and $J$, $J'$ are the restrictions of $J_+$ to the positive and negative ends.

The following is the main theorem of this subsection:

\begin{thm} \label{thm: compactness for W plus}
Let $u_i: (\dot F_i,j_i)\to (W_+,J_+)$, $i\in \N$, be a sequence of $W_+$-curves from
$\mathbf{y}$ to $\boldsymbol{\gamma}$. If $I_{W_+}(u_i)=1$ for all $i$, then a subsequence of $u_i$ converges
in the sense of SFT to one of the following:
\begin{enumerate}
\item an $I_{W_+}=1$ curve;
\item a building with two levels consisting of an $I_{HF}=1$ curve and an $I_{W_+}=0$ curve; or
\item a building with multiple levels consisting of an $I_{W_+}=0$ curve, an $I_{ECH}=1$ curve, and possible $I_{ECH}=0$ connectors in between.
\end{enumerate}
Similarly, if $I_{W_+}(u_i)=0$ for all $i$, then a subsequence of $u_i$ converges to an $I_{W_+}=0$ curve.
\end{thm}

We write ``an $I_{\#}=i$ curve'' as shorthand for ``a $\#$-curve with ECH index $I_{\#}=i$''.

We postpone the proof of the main theorem until after a more detailed discussion of the structure of the SFT limit.
Let $u_\infty$ be the SFT limit of the sequence $\{u_i\}$, given by Proposition~\ref{prop: SFT compactness for W plus}.
A {\em ghost component} of $u_\infty$ is an irreducible component of $u_\infty$ which
maps to a point. By the SFT compactness theorem, the domain of a ghost component
is necessarily a stable Riemann surface. We recall that a Riemann surface is stable if the following holds for each component $F$ with
$k_{int}$ interior marked points and $k_{bdr}$ boundary marked points:
\begin{itemize}
\item $- \chi(F) + k_{int} \ge 1$ when $\partial F = \varnothing$, or
\item $- 2 \chi(F) + 2 k_{int}+k_{bdr} \ge 1$ if $\partial F \neq \varnothing$.
\end{itemize}

Let us write $u_\infty=v_{-b}\cup\dots \cup v_a\cup u_\infty^g$, where $v_j$ has no ghost
components and $u_\infty^g$ is the union of ghost components. We will also write
$u_\infty^{ng}=v_{-b}\cup\dots \cup v_a$.

\begin{rmk}
The ECH index of a curve depends only on its relative homology class and therefore ghost components do not contribute to it. Hence, by the additivity of ECH indices (Lemma~\ref{lemma: additivity part 2}), if $u_i$ is a sequence of $J_+$-holomorphic maps with constant ECH index, then:
\begin{equation} \label{eqn: sum of ECH indices}
\sum_{j=1}^a I_{HF}(v_j) + I_{W_+}(v_0)+ \sum_{j=1}^b I_{ECH}(v_{-j})=I_{W_+}(u_i).
\end{equation}
\end{rmk}

\begin{lemma}
\label{lemma: levels v sub j W plus}
Each level $v_j$, $j=-b,\dots,a$, is a degree $2g$ multisection of $\pi_j:W_j\to B_j$ with
no branch points along $\bdry B_j$.
\end{lemma}

\begin{proof}
Since $u_i$ is a degree $2g$ multisection of $\pi_{B_+}: W_+\to B_+$ for all $i$, it follows that, with the exception of finitely many $p\in B_j$, every level $v_j$ intersects a fiber $\pi^{-1}_j(p)$ exactly $2g$ times.

We show that on any level $v_j$ there are no irreducible components which lie in a fiber $\pi_j^{-1}(p)$. Arguing by contradiction, suppose $\widetilde v: \widetilde F\to W_j$ is an irreducible component which maps to a fiber $\pi_j^{-1}(p)$. If $p\in int(B_j)$, then $\widetilde v$ is a holomorphic map from a closed Riemann surface $\widetilde F$ to $\pi_j^{-1}(p)$. Since $\pi_j^{-1}(p)$ is a Riemann surface with nonempty boundary, $\widetilde v$ must be constant. On the other hand, if $p\in \bdry B_j$, it is also possible that $\widetilde F$ is a compact Riemann surface with nonempty boundary and $\widetilde v(\bdry \widetilde F)\subset \mathbf{a}$. However, since $S-\mathbf{a}$ is connected and nontrivially intersects $\bdry S$, $\widetilde{v}$ must also be constant.  Since ghost components are excluded from $v_j$ by definition, we have a contradiction.

Finally, if $j\geq 0$, then we claim that $\pi_j\circ v_j$ has no branch points along $\bdry B_j$.  This is due to the fact that $v_0$ uses each component of $L_{\mathbf{a}}^+$ exactly once and $v_j$, $j>0$, uses each component of $\R\times\{1\}\times \mathbf{a}$ and each component of $\R\times \{0\}\times \hh(\mathbf{a})$ exactly once.
\end{proof}

\begin{lemma}\label{lemma: all simply covered}
Let $u_{\infty}$ be the SFT limit of a sequence of $J_+$-holomorphic multisections
$u_i$ with constant ECH index. If $J_+$ is regular, then:
\begin{itemize}
\item $I_{HF}(v_j)>0$ for $j>0$,
\item $I_{W_+}(v_0) \ge 0$, and
\item $I_{ECH}(v_j)\geq 0$ for $j<0$.
\end{itemize}
Moreover, all the $v_j$, $j\geq 0$, are somewhere injective and satisfy $\op{ind}(v_j) \ge 0$. If $I_{W_+}(u_i)\leq 1$ in addition, then $v_j$, $j<0$, is somewhere injective and satisfies $\op{ind}(v_j) \ge 0$.
\end{lemma}

\begin{proof}
Since $v_j$, $j\geq 0$, is a degree $2g$ multisection by Lemma~\ref{lemma: levels v sub j W plus} and uses each connected component of the boundary Lagrangian exactly once, it is somewhere injective. Also since $J$ and $J_+$ are regular, it follows that the curves $v_j$, $j\geq 0$, are regular. Hence $\op{ind}_W(v_j)\geq 0$ for $j>0$ and $\op{ind}_{W_+}(v_0)\geq 0$. Moreover, since $\op{ind}_W(v_j)=0$, $j>0$, if and only if $v_j$ is a union of trivial strips, we may assume that $\op{ind}_W(v_j)>0$ for all $j>0$. By the index inequality (Theorems~\ref{thm: index inequality for HF} and \ref{thm: index inequality for W+ and W-}) we have $I_{HF}(v_j)> 0$ for $j>0$ and $I_{W_+}(v_0)\geq 0$. On the other hand, if $j<0$, then $I_{ECH}(v_j)\geq 0$ by \cite[Proposition~7.15(a)]{HT1}.

If $I_{W_+}(u_i)\leq 1$ in addition, then $I_{ECH}(v_j)\leq 1$ by the additivity of the ECH index. Hence \cite[Lemma 9.5]{Hu1} implies that $v_j$ is somewhere injective and $\op{ind}(v_j) \geq 0$ since $J'$ is regular.
\end{proof}

\begin{lemma} \label{lemma: no ghosts}
Let $u_{\infty}$ be the SFT limit of a sequence of $J_+$-holomorphic multisections
$u_i$ with $I_{W_+} (u_i)\le 1$ for all $i$.
If $J_+$ is regular, then $u_\infty^g=\varnothing$.
\end{lemma}

\begin{proof}
Let $\op{ind}(u_\infty^g)$ and $\op{ind}(u_\infty^{ng})$ be the sum of the Fredholm
indices of the irreducible components of $u_\infty^g$ and $u_\infty^{ng}$, respectively.
Also let $F^{g}$ and $F^{ng}$ be the domains of $u_\infty^{g}$ and $u_\infty^{ng}$,
respectively. Then $\op{ind}(u_\infty^g)=-\chi(F^g)$. If $u_\infty^{g}$ and $u_\infty^{ng}$ are attached along $k_{int}$ interior
nodes and $k_{bdr}$ boundary nodes, then
\begin{align}
\op{ind}(u_i) &=\op{ind}(u_\infty^{ng}) + \op{ind}(u_\infty^g)+2k_{int}+k_{bdr}\\
\label{eqn: sum of Fredholm indices}
&=\op{ind}(u_\infty^{ng}) -\chi(F^g) +2k_{int}+k_{bdr}\leq 1.
\end{align}
Next, if $F^g\not=\varnothing$, then the stability of the Riemann surface and the observation that each component of $\bdry F^g$ has a marked point (which in turn follows from Definition~\ref{defn: W plus curve}(1)) imply that
\begin{equation} \label{eq10}
-\chi(F^g) +2k_{int}+k_{bdr}\geq 2.
\end{equation}
Equations~\eqref{eqn: sum of Fredholm indices} and \eqref{eq10} together imply
that $\op{ind}(u_\infty^{ng})\leq -1$. Hence we have $\op{ind}(v_j)\leq -1$ for some
$v_j$, which is a contradiction of Lemma \ref{lemma: all simply covered} for a regular
$J_+$. The contradiction came from assuming that $u_\infty^g \neq \varnothing$.
\end{proof}

We now finish the proof of Theorem \ref{thm: compactness for W plus}.

\begin{proof}[Proof of Theorem \ref{thm: compactness for W plus}.]
The proof is based on a classification of the types of allowable buildings $u_\infty=v_{-b}\cup\dots\cup v_{a}\cup u^g_\infty$. We consider the situation of $I_{W_+}(u_i)=1$, leaving the easier $I_{W_+}(u_i)=0$ case to the reader.

By Lemmas \ref{lemma: levels v sub j W plus} and \ref{lemma: no ghosts}, the limit
$u_{\infty}$ consists of a building of degree $2g$ multisections $v_j$. Moreover, by
Lemma \ref{lemma: all simply covered}, $I_{W_j}(v_j) \ge 0$ for all $j$. The additivity of
ECH indices gives three possibilities
for the limit:
\begin{enumerate}
\item $u_\infty=v_0$, where $v_0$ is a multisection of  $W_+$ and $I_{W_+}(v_0)=1$;
\item $u_\infty=v_0\cup v_1$, where $v_0$ is a multisection of $W_+$ with
$I_{W_+}(v_0)=0$, and $v_1$ is a multisection of  $W$ with $I_{HF}(v_1)=1$;
\item $u_\infty= v_{-b}\cup\dots\cup v_{0},$ where $v_{-b},\dots,v_{-1}$ are multisections
of $W'$, $v_0$ is a multisection of $W_+$, all but one of $v_{-b},\dots,v_0$ have $I=0$,
and the remaining level has $I=1$.
\end{enumerate}

What is left to prove is that $I_{ECH}(v_{-b})=1$ in Case (3).
This follows from considerations of incoming partitions as in \cite[Proof of Lemma~7.23]{HT1}. Let $m_k$ be the multiplicity of the elliptic orbit $\gamma_k$ in the orbit set $\boldsymbol{\gamma}$ and let $\theta_k$ be the rotation of the first return map of $\gamma_k$. By \cite[Definition~1.8]{HT1}, there is a partial order $\geq_{\theta_k}$ on the set of partitions of $m_k$, where
$$(a_1,\dots,a_{l_1})\geq_{\theta_k} (b_1,\dots,b_{l_2})$$
if there is a Fredholm index zero branched cover of $\R\times \gamma_k$ with positive ends which partition $m_k$ into $(a_1,\dots,a_{l_1})$ and negative ends which partition $m_k$ into $(b_1,\dots, b_{l_2})$. By \cite[Lemma~7.5]{HT1}, the incoming partition $P^{\op{out}}_{\gamma_k}(m_k)$ --- the partition which corresponds to $\gamma_k^{m_k}$ at the negative end of $v_{-b}$ --- is maximal. Hence $I_{ECH}(v_{-b})=0$ implies that $v_{-b}$ is a trivial cylinder, which we have already eliminated.
\end{proof}

\subsection{Definition of $\Phi$}
\label{subsection: defn of chain maps}

Suppose $J_+\in \mathcal{J}_{W_+}^{reg}$ and $J$, $J'$ are the restrictions of $J_+$ to the positive and negative ends. In this subsection we define the chain map
\begin{equation}
\Phi_{J_+}:\widehat{CF}(S,\mathbf{a},\hh(\mathbf{a}),J)\to PFC_{2g}(N,J').
\end{equation}
We will usually suppress $J$, $J'$, and $J_+$ from the notation.

We first define an approximation of $\Phi$:

\begin{defn}\label{defn: Phi'}
Let $\widehat{CF'} (S,\mathbf{a},\hh(\mathbf{a}))$ be the chain complex generated by $\mathcal{S}_{{\bf a},\hh({\bf a})}$, before quotienting by the equivalence relation $\sim$ given in Section~\ref{subsubsection: variant CF of S}. We define the map
$$\Phi': \widehat{CF'} (S,\mathbf{a},\hh(\mathbf{a}))\to PFC_{2g}(N),$$
$$\Phi'(\mathbf{y}) = \sum_{\boldsymbol{\gamma}\in \widehat{\mathcal{O}}_{2g}}\langle\Phi'(\mathbf{y}),\boldsymbol{\gamma}\rangle \cdot  \boldsymbol{\gamma},$$
where $\langle \Phi'(\mathbf{y}),\boldsymbol{\gamma}\rangle$ is the mod $2$ count of $\mathcal{M}_{J_+}^{I=0}({\bf y},\boldsymbol{\gamma})$. The count is meaningful since $\mathcal{M}_{J_+}^{I=0}({\bf y},\boldsymbol{\gamma})$ is compact by Theorem~\ref{thm: compactness for W plus}.
\end{defn}

\nom[1u$\phi$]{$\Phi$}{Chain map from $\widehat{CF}(S,\mathbf{a},\hh(\mathbf{a}))$ to $PFC_{2g}(N)$}
\nom[1u$\phi$]{$\Phi'$}{Map from $\widehat{CF'}(S,\mathbf{a},\hh(\mathbf{a}))$ to $PFC_{2g}(N)$}

\begin{prop}\label{prop: Phi' chain map}
The map
$$\Phi': \widehat{CF'} (S,\mathbf{a},\hh(\mathbf{a}))\to PFC_{2g}(N),$$
is a chain map.
\end{prop}

\begin{proof}
By Theorem~\ref{thm: compactness for W plus} and Lemma~\ref{lemma: restriction of overline W plus curve}, $\bdry \mathcal{M}_{J_+}^{I=1}({\bf y},\boldsymbol{\gamma}) = A\sqcup B$, where
\begin{align*}
A & = \coprod_{{\bf y'}\in \mathcal{S}_{{\bf a}, \hh({\bf a})}}\left( \left(\mathcal{M}^{I=1}_J({\bf y},{\bf y'})/\R\right)
\times \mathcal{M}^{I=0}_{J_+}({\bf y'},\boldsymbol{\gamma}) \right),\\
B & = \ \ \coprod_{\boldsymbol{\gamma}\in \widehat{\mathcal{O}}_{2g}} \ \left( \mathcal{M}^{I=0}_{J_+}({\bf y'},\boldsymbol{\gamma}')
\times\left( \mathcal{M}^{I=1}_{J'}(\boldsymbol{\gamma}',\boldsymbol{\gamma})/\R\right) \right).
\end{align*}
Here we have omitted the connector components for simplicity.

We examine the corresponding gluings of the holomorphic buildings. The first gluing is that of $(v_1,v_0)\in A$. This type of gluing was treated by Lipshitz; see Propositions~A.1 and A.2 in ~\cite[Appendix~A]{Li}. Observe that there are no multiply-covered curves to glue, since each Reeb chord of a $2g$-tuple is used exactly once.

The second type of gluing is that of $(v_0,v_{-b})\in B$, with $I_{ECH}=0$ connectors $v_{-1},\dots,v_{-b+1}$ in between. The curve $v_0$ is simply-covered since it has an HF end and the curve $v_{-b}$ is simply-covered since $I_{ECH}=1$ (see \cite[Proposition~7.15]{HT1}). This type of gluing was treated carefully in \cite{HT1,HT2}. Although the setting there was the gluing for $\bdry^2=0$, in fact most of the work goes towards
properly counting $I_{ECH}=0$ connectors, i.e., branched covers of trivial cylinders. Their treatment of gluing/counting the $I_{ECH}=0$ connectors carries over with little modification to our case. See Section~\ref{subsection: gluing for Phi} for more details.
\end{proof}

Let $\delta_x=([0,2]\times\{x\})/(2,x)\sim (0,x)$, where $x\in \bdry S$; it is a Reeb orbit in the negative Morse-Bott family $\mathcal{N}$ which foliates $\bdry N$. Also let $(\R\times \delta_x)^+$ be the restriction of $\R\times \delta_x$ to $W_+$.

\begin{lemma}
\label{lemma: trivial cylinder when x i involved}
A $W_+$-curve $u$ which has $x_i$ or $x_i'$ at the positive end must have $(\R\times \delta_{x_i})^+$ or $(\R\times \delta_{x_i'})^+$ as an irreducible component.
\end{lemma}

Recall that the intersection points $x_i, x_i'$ are components of the Heegaard Floer contact class of $\xi_{(S,\hh)}$.

\begin{proof}
Let $v:\dot F\to W_+$ be the irreducible component of $u$ which has $x_i$ at the positive end. The component $v$ may {\em a priori} have other positive ends besides $x_i$. We will show that $v=(\R\times\delta_{x_i})^+$. Let $\pi_S:\{s\geq 3\}\cap W_+ \to S$ be defined by identifying $\{s\geq 3\}\cap W_+=[3,\infty)\times[0,1]\times S$ and projecting to the third factor.

Let $(r,\theta)$ be polar coordinates on a small neighborhood $\nu(x_i)\subset S$ of $x_i$ so that $\hh(a_i)=\{\theta=-{\pi\over 4}\}$, $a_i=\{\theta=0\}$, and $\nu(x_i)=\{-{\pi\over 2}\leq \theta \leq {\pi\over 2}, r>0\}$.  By Condition (1) in Definition~\ref{defn: almost complex structures on overline W}, the composition $\pi_S\circ v$ is holomorphic when restricted to the end that limits to $x_i$.
If $v$ is not the restriction of a trivial cylinder, then $\pi_S\circ v$ must map a neighborhood of the puncture of $\dot F$ corresponding to $x_i$ to a sector $\{0\leq \theta\leq k\pi-{\pi\over 4}, r>0\}$ or $\{\pi \leq \theta \leq (k+1)\pi-{\pi\over 4},r>0\}$, where $k\geq 1$. Since such a sector cannot be contained in $S$, the map $\pi_S\circ v|_{s\geq C}$ must map identically to $x_i$, where $C\gg 0$.  By the unique continuation property, $v=(\R\times \delta_{x_i})^+$.
\end{proof}

\begin{thm} \label{thm: Phi is a chain map}
The map $\Phi'$ descends to a chain map
$$\Phi: \widehat{CF}(S,{\bf a},\hh({\bf a}))\to PFC_{2g}(N)$$
which maps the Heegaard Floer contact invariant for $\xi_{(S,\hh)}$ to the ECH contact invariant for $\xi_{(S,\hh)}$.
\end{thm}

\begin{proof}
Let ${\bf y} =\{y_1 ,\dots , y_{2g} \}\in \mathcal{S}_{{\bf a},\hh({\bf a})}$. Assume without loss of generality that $y_1=x_1$. Then ${\bf y}$ is equivalent to ${\bf y'}=\{x_1',y_2,\dots,y_{2g}\}$. Let $u$ be an $I_{W_+}=0$ Morse-Bott building from $\mathbf{y}$ to a generator $\boldsymbol{\gamma}$ of $PFC_{2g}(N)$. Since $(\R\times \delta_{x_1})^+$ is an irreducible component of $u$ by Lemma~\ref{lemma: trivial cylinder when x i involved} and is automatically transverse in the Morse-Bott sense, $\boldsymbol{\gamma}$ can be written as $e\boldsymbol{\gamma}'$, where $e$ is the elliptic orbit of the negative Morse-Bott family $\mathcal{N}$. Now, by replacing
$(\R\times\delta_{x_1})^+$ by $(\R\times \delta_{x_1'})^+$ and the augmenting gradient trajectory from $\delta_{x_1}$ to $e$ by the augmenting trajectory from $\delta_{x_1'}$ to $e$, we obtain an $I_{W_+}=0$ Morse-Bott building from $\mathbf{y'}$ to $\boldsymbol{\gamma}=e\boldsymbol{\gamma}'$. Hence there is a one-to-one correspondence between $W_+$-curves from ${\bf y}$ to $\boldsymbol{\gamma}$ and $W_+$-curves from ${\bf y'}$ to $\boldsymbol{\gamma}$. Since $\Phi' (\mathbf{y})= \Phi' (\mathbf{y'})$, the map $\Phi'$ descends to $\widehat{CF}(S,\mathbf{a},\hh(\mathbf{a}))$.

The Heegaard Floer contact invariant is given by the equivalence class of ${\bf x}=\{x_1,\dots,x_{2g}\}$ and the ECH contact invariant is given by $e^{2g}$.  By Lemma~\ref{lemma: trivial cylinder when x i involved} and some ECH index considerations, the only $I_{W_+}=0$ Morse-Bott building from ${\bf x}$ consists of $(\R\times \delta_{x_i})^+$, $i=1,\dots,2g$, augmented by connecting gradient trajectories from $\delta_{x_i}$ to $e$ in the Morse-Bott family $\mathcal{N}$. Hence $\Phi'({\bf x})=e^{2g}$. Then $\Phi$ maps the equivalence class of ${\bf x}$ to $e^{2g}$.
\end{proof}

\subsection{${\rm Spin}^c$-structures}

Let ${\mathcal S}_{{\bf a},\hh({\bf a})}$ be the set of $2g$-tuples of ${\bf a}\cap \hh({\bf a})$ and let $\overline{\mathcal{S}}_{{\bf a},\hh({\bf a})}={\mathcal S}_{{\bf a},\hh({\bf a})}/\sim$, where $\{x_i\}\cup {\bf y'}\sim \{x_i'\}\cup{\bf y'}$ for all $(2g-1)$-tuples ${\bf y'}$.

We define a map
$$\mathfrak{h}'_+: {\mathcal S}_{{\bf a},\hh({\bf a})}\to H_1(W_+, \partial_hW_+)$$
as follows: Given $\mathbf{y} = \{y_1, \ldots, y_{2g}\} \in {\mathcal S}_{{\bf a},\hh({\bf a})}$, where $y_i \in a_i \cap \hh(a_{\sigma(i)})$ for some permutation $\sigma\in\mathfrak{S}_{2g}$, we define
$\mathfrak{h}'_+(\mathbf{y})$ as the homology class of the broken closed string obtained by concatenating the following oriented arcs:
\begin{itemize}
\item for each $i \in \{ 1, \ldots, 2g \}$, the arc $\{ 3 \} \times [0,1] \times \{ y_i \}$, with orientation given by $\partial_t$;
\item for each $i \in \{ 1, \ldots, 2g \}$, an arc from $\{ 3 \} \times \{ 1 \} \times \{ y_i \}$ to $\{ 3 \} \times \{ 0 \} \times \{ y_{\sigma^{-1}(i)} \}$ contained in $L_{a_i}^+$.
\end{itemize}
The homology class $\mathfrak{h}'_+(\mathbf{y})$ is well-defined since the Lagrangians $L_{a_i}^+$ are simply-connected.

The map $\mathfrak{h}'_+$ descends to a map
$$\mathfrak{h}_+: \overline{\mathcal{S}}_{{\bf a},\hh({\bf a})}\to H_1(W_+, \partial_hW_+),$$
since the components $x_i$ and $x_i'$ of the contact class are both converted into broken closed strings which are nullhomologous in $H_1(W_+,\bdry_h W_+)$.

The following lemma is an immediate consequence of the definition of $W_+$-curves.

\begin{lemma}\label{W_+ obstruction}
Let $\boldsymbol{\gamma}$ be an orbit set at the negative end of $W_+$ with total homology class $[\boldsymbol{\gamma}] \in H_1(W_+, \partial_hW_+)$. Then ${\mathcal M}_{J_+}(\mathbf{y}, \boldsymbol{\gamma}) \ne \varnothing$ implies that $\mathfrak{h}_+(\mathbf{y}) = [\boldsymbol{\gamma}]$.
\end{lemma}

There are natural isomorphisms
$$H_1(W_+, \partial_hW_+) \cong H_1(N, \partial N) \cong H_1(M).$$
With respect to these isomorphisms, the total homology class in $H_1(W_+, \partial_hW_+)$ of an orbit set $\boldsymbol{\gamma}$ at the negative end of $W_+$ corresponds to the usual total homology class of $\boldsymbol{\gamma}$ in $H_1(M)$. Moreover, $\mathfrak{h}_+(\mathbf{y})=\mathfrak{h}({\bf y})$, where $\mathfrak{h}$ is as defined in Section \ref{spin-c revisited}.

Combining Proposition~\ref{s-and-h} and Lemma~\ref{W_+ obstruction}, we obtain the following theorem:

\begin{thm}
The chain map $\Phi$ respects the splitting according to ${\rm
Spin}^c$-struc\-tures, i.e., $\Phi$ is the direct sum of maps
$$\Phi_A: \widehat{CF}(S,\mathbf{a},\hh(\mathbf{a}), \mathfrak{s}_{\xi}
+PD (A)) \to PFC_{2g}(N, A),$$
with $A \in H_1(M; \Z)$.
\end{thm}

\subsection{Twisted coefficients} \label{subsection: twisted coefficients phi map}

Let $\underline{PFH}_{2g}(N, A)$ be the twisted coefficient version of $PFH_{2g}(N, A)$, defined as in Section~\ref{subsection: twisted coefficients ECH}. For any homology class
$A \in H_1(M)$ we can define a map
$$\underline{\Phi}_A: \underline{\widehat{CF}}(S,\mathbf{a},\hh(\mathbf{a}), \mathfrak{s}_{\xi} +PD (A)) \to \underline{PFC}_{2g}(N, A),$$
given in the following paragraphs. Choose a $2g$-tuple of intersection points $\mathbf{y}_0$ such that
$s_z(\mathbf{y}_0)= \mathfrak{s}_{\xi}+ PD(A)$ and a complete set of paths
$\{ C_{\mathbf{y}} \}$ for $\mathfrak{s}_{\xi}+ PD(A)$ based at $\mathbf{y}_0$.

Let $\pi_N: W_+ \to N$ be the restriction of the projection $\R \times N \to N$, $(s, x)
\mapsto x$, and let $\Gamma \subset N$ be the projection by $\pi_N$ of a broken closed string associated to $\mathbf{y}_0$. By Lemma \ref{W_+ obstruction}, $[\Gamma]=A$.
We choose a complete set of paths $\{ C_{\boldsymbol{\gamma}} \}$ for $A$ based at $\Gamma$.

The projection $\pi_N$ associates a well-defined homology class in $H_2(N) \cong H_2(M)$ to any $2$-chain in $\check{W}_+$ whose boundary consists of a broken closed string corresponding to $\mathbf{y}_0$ on one side and $\Gamma$ on the other side. Then we can define maps
$$\mathfrak{A}_+: H_2(W_+, \mathbf{y}, \boldsymbol{\gamma}) \to H_2(M)$$
for all $\mathbf{y}$ such that $s_z(\mathbf{y_0})= \mathfrak{s}_{\xi}+ PD(A)$ and
$\boldsymbol{\gamma}$ such that $[\boldsymbol{\gamma}]=A$ by
$$\mathfrak{A}_+(C) = (\pi_N)_*[C_{\boldsymbol{\gamma}} \cup C \cup - C_{\mathbf{y}}].$$

\begin{defn}\label{defn: Phi twisted}
We define the $\F[H_2(M; \Z)]$-linear map
$$\underline{\Phi}_A: \underline{\widehat{CF}}(S,\mathbf{a},\hh(\mathbf{a}),
 \mathfrak{s}_{\xi} +PD (A)) \to \underline{PFC}_{2g}(N, A)$$
by
$$\underline{\Phi}_A(\mathbf{y}) = \sum_{\boldsymbol{\gamma} \in \widehat{\mathcal O}_{2g}}
\sum_{C \in H_2(W_+, \mathbf{y}, \boldsymbol{\gamma})} \# {\mathcal M}^{I=0}_{J_+}(\mathbf{y}, \boldsymbol{\gamma}, C) e^{\mathfrak{A}_+(C)} \boldsymbol{\gamma}.$$
\end{defn}

\begin{thm} \label{thm: Phi twisted is a chain map}
The map $\underline{\Phi}_A$ is a chain map.
\end{thm}
\begin{proof}
The proof is the same as that of Theorem~\ref{thm: Phi is a chain map}, plus some bookkeeping of the homology classes of the holomorphic maps involved.
\end{proof}

\subsection{Gluing} \label{subsection: gluing for Phi}

We explain here how to glue the pair $(v_0,v_{-1})$ consisting of a $W_+$-curve $v_0$ with $I(v_0)=0$ and an ECH curve $v_{-1}$ with $I(v_{-1})=1$, by inserting a branched cover of a trivial cylinder (a {\em connector}). The connector is needed because, while the negative ends of $v_0$ and the positive ends of $v_{-1}$
converge to the same orbit sets, they usually do not match; in fact the ends of $v_0$ must satisfy the outgoing partition, while the ends of $v_{-1}$ must satisfy the incoming partition.  The same situation appears in the definition of ECH and the procedure of gluing two $I_{ECH}=1$ curves in the symplectization, as explained in Hutchings-Taubes~ \cite{HT1,HT2}, applies with very little modification.

\subsubsection{Close to breaking}

We first make precise what we mean by a holomorphic curve which is ``close to breaking''.  We rephrase the definition for $W'$ (which appears in \cite{HT1}), leaving the analogous definitions for $W$, $W_+$, and $W_-$ to the reader.

Choose an $s$-invariant Riemannian metric $g$ on $W'$ and denote the induced distance by $d_g$.

\begin{defn} \label{defn: kappa nu close curves}
Given $\kappa,\nu>0$, two curves $u_i: (F_i,j_i)\to W'$, $i=1,2$, are {\em $(\kappa,\nu)$-close} if there exists a $(1+\nu)$-quasiconformal homeomorphism $\phi: F_1\stackrel\sim\to F_2$ with respect to $(j_1,j_2)$ such that
$$\sup_{x\in F_1} d_g( u_1(x), u_2\circ \phi(x)) \leq \kappa.$$
\end{defn}

Given a $J$-holomorphic curve $u$ in $\R \times Y$, by an abuse of notation we will denote the restriction of $u$ to the preimage of $[C, + \infty) \times M$ by $u|_{s\leq C}$. (Self-explanatory variations of this notations will also be used).
The following is similar to \cite[Definition 1.10]{HT1}, although the phrasing is slightly different.

\begin{defn}\label{defn: kappa nu close buildings}
Let $\kappa>0$ and $\nu\geq 0$. A curve $u$ in $W'$ is {\em $(\kappa,\nu)$-close} to an $(a+b+1)$-level building $u_\infty = v_{-b}\cup\dots\cup v_a$ if there exist $R^{(-2b+1)},\dots, R^{(2a)} >{1\over \kappa}$ such that, after a suitable translation of $u$ in the $s$-direction which we still call $u$, each of the pairs below is $(\kappa,\nu)$-close:
\begin{itemize}
\item ${u}|_{-R^{(0)}\leq s\leq R^{(1)}}$ and ${v}_0|_{-R^{(0)}\leq s\leq R^{(1)}}$;
\item ${u}|_{R^{(1)}\leq s\leq R^{(1)}+R^{(2)}}$ and the restriction of a collection of branched covers of trivial cylinders to $R^{(1)}\leq s \leq R^{(1)}+R^{(2)}$;
\item ${u}|_{R^{(1)}+R^{(2)}\leq s\leq R^{(1)}+2R^{(2)}+R^{(3)}}$ and an $s\mapsto s+R^{(1)}+2R^{(2)}$ translate of ${v}_1|_{-R^{(2)} \leq s\leq R^{(3)}}$;
\item ${u}|_{R^{(1)}+2R^{(2)}+R^{(3)}\leq s\leq R^{(1)}+2R^{(2)}+R^{(3)}+R^{(4)}}$ and the restriction of a collection of branched covers of trivial cylinders to $R^{(1)}+2R^{(2)}+R^{(3)}\leq s\leq R^{(1)}+2R^{(2)}+R^{(3)}+R^{(4)}$;
\end{itemize}
and so on.
\end{defn}

Note that, in the case of ${u}$ in ${W}_+$ or ${W}_-$, we do not need to (and indeed we cannot) translate ${u}$ in the $s$-direction.

\subsubsection{Review of \cite{HT1,HT2}} \label{subsubsection: reviw of HT}

We now summarize the Hutchings-Taubes proof of $\bdry^2=0$ and discuss the small changes that need to be carried out.

Let $(Y,\xi =\ker \lambda )$ be a closed $3$-manifold with a nondegenerate contact form $\lambda$ and corresponding Reeb vector field $R=R_\lambda$. Let $(\R \times Y ,d(e^s \lambda ))$ be the symplectization of $Y$, where $s$ denotes the $\R$-coordinate, and let $J$ be an adapted almost complex structure on $\R\times Y$.

Let $\alpha=(\alpha_1,\dots,\alpha_k)$ and $\beta=(\beta_1,\dots,\beta_l)$ be ordered sets of Reeb orbits of $R$. {\em This notation is confined to Section \ref{subsection: gluing for Phi} and therefore should not create confusion with the similar notation denoting the Heegaard Floer multicurve.} Here the Reeb orbits may not be embedded and may also be repeated. Then let $\mathcal{M}_J (\alpha,\beta)$ be the moduli space of finite energy $J$-holomorphic curves in $(\R\times Y,J)$ from $\alpha$ to $\beta$. Let $\{\gamma_1,\gamma_2,\dots\}$ be the list of simple orbits of $R$. For each simple orbit $\gamma_i$, we tally the total multiplicity $m_i(\alpha)$ of $\gamma_i$ in $\alpha$. In this way we can assign an orbit set $\boldsymbol{\gamma}(\alpha)=\prod_i \gamma_i^{m_i(\alpha)}$ to $\alpha$.

We want to glue the pair $(u_+,u_-)$, where $u_+ \in \mathcal{M}_J^{I=1} (\alpha_+, \beta_+ )$ and $u_- \in \mathcal{M}_J^{I=1} (\beta_-,\alpha_-)$, provided $\boldsymbol{\gamma}(\beta_+)=\boldsymbol{\gamma}(\beta_-)$.  Since $I(u_+)=I(u_-)=1$, the curves $u_+$ and $u_-$ satisfy the {\it partition condition} at $\beta_+$ and $\beta_-$ for a generic $J$ (cf.\ Definition 7.11 and Proposition 7.14 of \cite{HT1}). In particular, for each simple orbit $\gamma_i$, the total multiplicity $m_i(\beta_+)=m_i(\beta_-)$ completely determines the number of ends of $u_+$ or $u_-$ going to a cover of $\gamma_i$, together with their individual multiplicities. For $u_+$ (resp.\ $u_-$), this is encoded by the {\em incoming partition} (resp.\ {\em outgoing partition}) and is denoted by $P^{\op{in}}_{\gamma_i}(m_i(\beta_+))$ (resp.\ $P^{\op{out}}_{\gamma_i}(m_i(\beta_-))$).  If $P^{\op{in}}_{\gamma_i}(m_i(\beta_+))=(a_1,\dots,a_r)$, then $u_+$ has $r$ negative ends which go to a cover of $\gamma_i$ with covering multiplicities $a_1,\dots,a_r$. A similar statement holds for the positive ends of $u_-$.

Since $P^{\op{in}}_{\gamma_i}(m_i(\beta_+))\not=P^{\op{out}}_{\gamma_i}(m_i(\beta_-))$ in general, we need to insert branched covers of cylinders $\R\times \gamma_i$ in order to be able to glue $u_+$ to $u_-$.  Such branched covers $\pi:\Sigma\to \R\times \gamma_i$ must have $\beta_+$ at the positive end and $\beta_-$ at the negative end. A Fredholm index count (\cite[Lemma 1.7]{HT1}) implies that $\Sigma$ must have genus zero;
moreover, the partition condition is equivalent to saying that the Fredholm index of the composition  $u_\Sigma = \iota \circ \pi :\Sigma\to \R\times Y$ is zero, where $\iota:\R\times \gamma_i \to \R\times Y$ is the inclusion of the trivial cylinder $\R\times \gamma_i$.

Given $u_+ \in \mathcal{M}_J^{I=1} (\alpha_+ ,\beta_+ )$ and $u_- \in \mathcal{M}_J^{I=1} (\beta_-,\alpha_-)$ with $\boldsymbol{\gamma}(\beta_+)=\boldsymbol{\gamma}(\beta_-)$,
we define ${\mathcal G}_{\kappa, \nu}(u_+, u_-)$ as the set of curves in $\mathcal{M}_J^{I=2}(\alpha_+, \alpha_-)$ which are $(\kappa, \nu)$-close to breaking to $(u_+, u_-)$. This is a rephrasing of \cite[Definition 1.10]{HT1}. We define $G(u_+, u_-)$ as the modulo $2$ count of the boundary points of ${\mathcal G}_{\kappa, \nu}(u_+, u_-)/ \R$ for $\kappa$ and $\nu$ sufficiently small.

The main result of \cite{HT1, HT2} implies the following.
\begin{thm} \label{thm: ECH gluing}
Let $u_+ \in \mathcal{M}_J^{I=1} (\alpha_+ ,\beta_+ )$ and $u_- \in \mathcal{M}_J^{I=1} (\beta_-,\alpha_-)$ be $J$-holomorphic curves with $\boldsymbol{\gamma}(\beta_+) = \boldsymbol{\gamma}(\beta_-)$. Then $G(u_+, u_-)=1$.
\end{thm}
It is clear that Theorem \ref{thm: ECH gluing} is sufficient to define ECH; the original Hutchings-Taubes gluing theorem is more general, but for simplicity we have stated only the relevant part.

\s
We now give the steps of the Hutchings-Taubes gluing construction. For simplicity we assume that $\boldsymbol{\gamma}(\beta_+)=\boldsymbol{\gamma}(\beta_-)=\gamma^{m}$ and $\gamma$ is elliptic.  The case of more than one orbit $\gamma_i$ only differs in notation and the case of hyperbolic orbits is simpler since it can be treated with standard techniques; see \cite[Section 1.5]{HT1}.

\s\n {\em Step 1: Form a preglued curve}  (cf.\ \cite[Section 5.2]{HT2}). First we fix some notation. Let $\mathcal{M}$ be the moduli space of connected, genus zero branched covers $\pi: \Sigma\to \R\times \gamma$ which have positive and negative ends determined by $P^{\op{in}}_{\gamma}(m)$ and $P^{\op{out}}_{\gamma}(m)$. Abusing the notation, we will often write only $\Sigma$.
The preglued curve is constructed as follows:

\s\n
(1) Choose gluing constants $0<h<1$ and $r\gg 1/h$.

\s\n
(2) With $r,h$ fixed, choose the ``gluing parameters'' which consist of $\Sigma \in \mathcal{M}$ and real numbers $T_+ , T_- \geq 5r$.

\s\n
(3) Given $\pi:\Sigma\to \R\times \gamma$, let ${\frak b}\subset \R\times \gamma$ be the set of branch points of $\pi$. Then let $s_+=\max_{b\in {\frak b}} s(b) +1$ and $s_-=\min_{b\in {\frak b}} s(b)-1$. Let $\Sigma'$ be the preimage of  $[s_--T_-, s_+ +T_+] \times \gamma$. The preimages of $[ s_+, s_+ +T_+] \times \gamma$ are cylinders of ``height'' $T_+$  and the preimages of $[s_--T_-,s_-]$ are cylinders of ``height'' $T_-$.

\s\n
(4) Fix $\kappa>0$ sufficiently small.  We choose a representative $u_+$ in $[u_+]\in \mathcal{M}_J^{I=1} (\alpha_+ ,\beta_+ )/\R$ such that each component of $u_+|_{s\leq 0}$ is $(\kappa,0)$-close to a cylinder over some multiple cover of $\gamma$; here the multiplicities are given by $P^{\op{in}}_{\gamma}(m)$. Similarly, choose $u_-$ so that each component of $u_-|_{s\geq 0}$ is $(\kappa,0)$-close to a cylinder over some multiple cover of $\gamma$.

\s\n
(5) Let $u_{+T}$ be the $(s_+ +T_+)$-translate of $u_+$ in the $\R$-direction and let $u_{+T}' = u_{+T}|_{s\geq s_+ + T_+}$. Similarly, let $u_{-T}$ be the $(s_--T_-)$-translate of $u_-$ in the $\R$-direction and let $u_{-T}'= u_{-T}|_{s\leq s_-  -T_-}$. The domain $C_*$ of the preglued curve is $C'_{+T} \cup \Sigma' \cup C'_{-T}$, where $C'_{\pm T}$ are domains of $u'_{\pm T}$, modulo the identifications along their boundary components.  (To do the identifications correctly, we need asymptotic markers at the ends.)

\s\n
(6)  The preglued map $u_*$ is defined explicitly (see \cite[Equations (5.5) and (5.6)]{HT2}) via a cutoff functions which allow us to interpolate between $u_{\pm T}$ and $\Sigma$ in the preimages of the regions $s_+ \leq s\leq s_+ +T_+$ and $s_--T_-\leq s \leq s_-$. These cutoff functions involve the constants $h$ and $r$.

\s\n {\em Step 2: Deform the preglued curve} (cf.\ \cite[Section 5.3]{HT2}).
We choose ``exponential maps'' $e_-,e,e_+$ which are obtained by flowing in the directions ``normal'' to $u_-,u_\Sigma,u_+$ in $\R\times Y$.
The exponential maps can be glued to give the exponential map $e_*$ which maps to a small tubular neighborhood of $u_*$ in $\R\times Y$.  Also choose the cutoff function $\widetilde\beta_+: C_*\to [0,1]$\footnote{The notation in \cite{HT2} is $\beta_+$. Here write $\widetilde\beta_+$ to distinguish it from the sets of orbits in Section~\ref{subsubsection: reviw of HT}.} which equals $1$ on $C'_{+T}$, interpolates between $0$ and $1$ on the cylinders $s_+ \leq s\leq s_+ +T_+$ in $\Sigma'$ and equals $0$ elsewhere. Similarly define $\widetilde\beta_\Sigma$ and $\widetilde\beta_-$.

Let $(\psi_+,\psi_\Sigma,\psi_-)$ be a triple, where $\psi_\pm$ is a section of the normal bundle of $u_{\pm T}$ and $\psi_\Sigma$ is a section of the normal bundle of $u_\Sigma$ (which is trivial, so that $\psi_\Sigma$ can be identified with a function $\Sigma\to \C$). The deformation of $u_*$ with respect to $(\psi_+,\psi_\Sigma,\psi_-)$ is a map $C_*\to \R\times Y$ given by
$$x\mapsto e_* (x, \widetilde\beta_- \psi_- +\widetilde\beta_\Sigma \psi_\Sigma + \widetilde\beta_+ \psi_+ ).$$

\s\n {\em Step 3:} We now consider the equation for the deformation to be $J$-holomorphic.  This equation has the form:
\begin{equation} \label{eqn: for deformation to be J holom}
\widetilde\beta_- \Theta_- (\psi_- ,\psi_\Sigma )+\widetilde\beta_\Sigma \Theta_\Sigma (\psi_- ,\psi_\Sigma ,\psi_+ ) +\widetilde\beta_+ \Theta_+ (\psi_\Sigma ,\psi_+ )=0.
\end{equation}
See Equations (5.11), (5.12) and (5.13) of \cite{HT2} for explicit expressions of $\Theta_-$, $\Theta_+$ and $\Theta_\Sigma$.

The strategy is to solve three equations separately:
\begin{equation} \label{theta minus}
\Theta_- (\psi_-, \psi_\Sigma )=0 \quad\mbox{ on all of $u_{-T}$};
\end{equation}
\begin{equation} \label{theta plus}
\Theta_+ (\psi_\Sigma ,\psi_+ )=0 \quad\mbox{ on all of $u_{+T}$};
\end{equation}
\begin{equation} \label{theta zero}
\Theta_\Sigma (\psi_- ,\psi_\Sigma ,\psi_+ )= 0 \quad \mbox{ on all of $\Sigma$}.
\end{equation}

In \cite[Proposition 5.6]{HT2} it is shown that for sufficiently small $\psi_\Sigma$ (in a suitable Banach space) there exist maps $\psi_\pm$ such that $\psi_\pm=\psi_\pm(\psi_\Sigma)$ solves Equations~\eqref{theta minus} and \eqref{theta plus}.

We can then view Equation~\eqref{theta zero} as an equation in the variable $\psi_\Sigma$ on all of $\Sigma$:
\begin{equation} \label{theta zero new}
\Theta_\Sigma (\psi_- (\psi_\Sigma ), \psi_\Sigma ,\psi_+ (\psi_\Sigma ))=0.
\end{equation}
Equation~\eqref{theta zero new} can be written as
\begin{equation}\label{theta zero again}
D_\Sigma\psi_\Sigma+ \mathcal{F}_\Sigma(\psi_\Sigma)=0,
\end{equation}
where $D_\Sigma$ is the normal component of the linearized Cauchy-Riemann operator at $u_{\Sigma}$ and $\mathcal{F}_\Sigma$ is a nonlinear term which is small in some suitable sense.

\s\n {\em Step 4:} If $N$ denotes the normal bundle to $\R \times \gamma$ in $\R \times M$, then
$$D_\Sigma : L^2_1(\pi^* N) \to L^2(\pi^*N \otimes T^{0,1} \Sigma).$$
The operator $D_\Sigma$ is injective and moreover $\dim \op{coker} (D_\Sigma)= \dim{\mathcal M}$; see \cite[Lemma 2.14]{HT1} and \cite[Lemma 2.15]{HT1}. The explicit form of $D_\Sigma$ is given in \cite[Definition 2.3]{HT1}.

We introduce the orthogonal projection
$$\Pi : L^2(\pi^*N \otimes T^{0,1} \Sigma) \to \ker (D_\Sigma^*) \cong \op{coker}(D_\Sigma)$$
and decompose Equation~\eqref{theta zero again} into two equations (cf.\ Equations~(5.37) and (5.38) in \cite{HT2}):
\begin{equation}\label{eq one}
D_\Sigma\psi_\Sigma+(1-\Pi)\mathcal{F}_{\Sigma}(\psi_\Sigma)=0,
\end{equation}
\begin{equation} \label{eq two}
\Pi \mathcal{F}_{\Sigma}(\psi_\Sigma)=0
\end{equation}
Equation~\eqref{eq one} has a unique solution by \cite[Proposition 5.7]{HT2}. Hence the problem reduces to solving Equation~\eqref{eq two}.

This is shown to be equivalent to finding the zeros of a section of the associated obstruction bundle, defined in \cite[Section 2]{HT1}. Briefly, the obstruction bundle \footnote{In \cite{HT2} the obstruction bundle is denoted by ${\mathcal O}$. However, in \cite{HT1}, $\mathcal{O}$ denotes a related, but different, bundle and this causes confusion. We reserve the notation $\mathcal{O}$ for the latter bundle.} $\mathcal{O}' \to [5r,\infty)^2\times\mathcal{M}$ has fiber
$$\mathcal{O}'_{(T_+,T_-,\Sigma)}=\op{Hom}(\op{Coker}(D_\Sigma),\R),$$
and the section ${\frak s}: [5r,\infty)^2\times\mathcal{M}\to \mathcal{O}'$ is given by:
$${\frak s}(T_+,T_-,\Sigma)(\sigma)= \langle\sigma,\mathcal{F}_{\Sigma}(\psi_\Sigma)\rangle,$$
where $\sigma\in \op{Coker}(D_\Sigma)$ and $\psi_\Sigma$ is the solution to Equation~\eqref{eq one} corresponding to the gluing parameters $(T_+,T_-,\Sigma)$.

\s\n {\em Step 5:} Recall that $\mathcal{G}_{\kappa,\nu}(u_+,u_-)$ is the set of $J$-holomorphic maps of index $I=2$ which are $(\kappa,\nu)$-close to breaking into $(u_+,u_-)$ (see~\cite[Section 7.1]{HT2}).  Let $\mathcal{U}_{\kappa,\nu}\subset [5r,\infty)^2\times \mathcal{M}$ be the set of $(T_+,T_-,\Sigma)$ such that the corresponding pre-glued curve $u_*(T_+,T_-,\Sigma)$ (as given in Step 1) is $(\kappa,\nu)$-close to breaking.
It remains to show that $\mathcal{G}_{\kappa,\nu}(u_+,u_-)$ is homeomorphic to ${\frak s}^{-1}(0)\cap \mathcal{U}_{\kappa,\nu}$ for $r>0$ sufficiently large and $\kappa,\nu$ sufficiently small.  This is the content of \cite[Theorem 7.3]{HT2}.

\s\n {\em Step 6:} We define the linearized section ${\frak s}_0$ which depends  only  on the asymptotic eigenfunctions of the negative ends of $u_+$ and the
positive ends of $u_-$ 
but not on the actual choices of $u_+$ and $u_-$. The explicit formula for
$\mathfrak{s}_0$ is given in \cite[Definition 8.1]{HT1}.
We consider the linear deformation ${\frak s}_t = t  {\frak s} + (1-t)  {\frak s}_0$.

\s\n  {\em Step 7:}
Consider the action of $\R$ on $[5r, + \infty)^2 \times {\mathcal M}$ that fixes the $[5r, + \infty)$ factors and translates the $s$-coordinate on ${\mathcal M}$. This action extends to the obstruction bundle and the sections $\mathfrak{s}_t$ are invariant under this action. Hence the obstruction bundle can be viewed as a bundle over $[5r, + \infty)^2 \times {\mathcal M}/ \R$ which we still denote by ${\mathcal O}'$ by abuse of notation, and each $\mathfrak{s}_t$ can be viewed as a section of it.

Given $R \ge 10r$, we define the ``slice'' ${\mathcal V}_R$ which consists of triples
$$(T_-, T_+, \Sigma) \in [5r, + \infty)^2 \times {\mathcal M}, \quad T_+ + s_+ - s_- + T_- = R.$$
The action of $\R$ on $[5r, + \infty)^2 \times {\mathcal M}$ preserves ${\mathcal V}_R$ and the boundary of $\mathcal{G}_{\kappa,\nu}(u_+,u_-)$ is in bijection with $\mathfrak{s}^{-1}(0) \cap {\mathcal V}_R/ \R$; in other words
$$G(u_+, u_-) = \# \left ( \mathfrak{s}^{-1}(0) \cap {\mathcal V}_R/ \R \right ).$$

It turns out that the section $\mathfrak{s}_0$ has the same number of zeroes as $\mathfrak{s}$ on ${\mathcal V}_R/ \R$.
The key point to check is that, during the deformation $({\frak s}_t )_{t\in [0,1] }$ from ${\frak s}_1={\frak s}$ to ${\frak s}_0$, the zeros of ${\frak s}_t$ do not cross the boundary of  ${\mathcal V}_R/ \R$. This is guaranteed by \cite[Proposition 8.2]{HT2}.

\s\n {\em Step 8:} Define
$${\mathcal M}_R = \{ \Sigma \in {\mathcal M} | -\tfrac{R}{2} + 5r \le s_-, s_+ \le \tfrac{R}{2} -5r \}.$$
Then ${\mathcal V}_R / \R\simeq {\mathcal M}_R$ and the bundle ${\mathcal O} \to {\mathcal M}_R$ with fiber ${\mathcal O}_\Sigma = \op{Hom}(\op{Coker}(D_\Sigma),\R)$ is isomorphic to the bundle ${\mathcal O}' \to {\mathcal V}_R / \R$. The explicit formula for $\mathfrak{s}_0$, regarded as a section
of ${\mathcal O} \to {\mathcal M}_R$, is given in \cite[Definition 3.2]{HT1}. See \cite[Remark 8.5]{HT2} for more details about the correspondence between ${\mathcal O}'$ and ${\mathcal O}$. Finally, the combinatorial formula for the algebraic count of zeros of $\mathfrak{s}_0$ (as a section of ${\mathcal O}$) is given by \cite[Theorem 1.13]{HT1}.

\subsubsection{The $\Phi$-map}

We now turn to gluing the pair $(v_0,v_{-1})$, where $v_0$ is a $W_+$-curve with $I(v_0)=0$ and $v_{-1}$ is an ECH curve with $I(v_{-1})=1$. Suppose the negative end of $v_0$ is given by $\beta_+$, the positive end of $v_{-1}$ is given by $\beta_-$, and $\boldsymbol{\gamma}(\beta_+)=\boldsymbol{\gamma}(\beta_-)=\gamma^{m}$
for some elliptic orbit $\gamma$. Again,  the case of more than one orbit $\gamma_i$ only differs in notation and the case of hyperbolic orbits can be treated by standard techniques. In our case, there are a few things to check:

\s\n
(1) The curves $v_0$ and $v_{-1}$ must satisfy the partition conditions at their negative and positive ends, respectively.  This is a consequence of the ECH index inequality, i.e., Theorem~\ref{thm: index inequality for W+ and W-}.

\s\n
(2) Since the $W_+$-curve $v_0$ is not $s$-translation invariant, we pick $s_0$ so that each component of $v_0|_{s\leq s_0}$ is $(\kappa,0)$-close to a cylinder over some multiple cover of $\gamma$. (We may still assume that $v_{-1}$ satisfies the condition that $v_{-1}|_{s\geq 0}$ is $(\kappa,0)$-close to a cylinder.)  Given $(T_+,T_-,\Sigma)\in [5r,+\infty)^2\times \mathcal{M}$, we take
\begin{itemize}
\item $v_0|_{s\geq s_0}$;
\item $\Sigma'$ shifted by $s_0-(s_++T_+)$; and
\item $v_{-1}|_{s\leq 0}$ shifted by $(s_--T_-)+s_0-(s_++T_+)$;
\end{itemize}
and preglue. The preglued curve only depends on $(T_+, T_-,[\Sigma])$, where $[\Sigma] \in {\mathcal M}/\R$.

\s\n
(3) \cite[Proposition 5.6]{HT2} allows us to solve Equations~\eqref{theta minus} and \eqref{theta plus} in terms of $\psi_\Sigma$. The inputs for \cite[Proposition 5.6]{HT2} are \cite[Lemmas 5.3 and 5.4]{HT2}, which are consequences of the fact that the linearized $\overline\bdry$-operators corresponding to $v_0$ and $v_{-1}$ are Fredholm and surjective; this also holds in our case with Lagrangian boundary conditions.  Hence Step 3 extends easily to our setting,  yielding a section $\mathfrak{s}$ of the obstruction bundle ${\mathcal O}' \to [5r, + \infty)^2 \times {\mathcal M}/ \R$.
The linearized section $\mathfrak{s}_0$ and the ``slice'' $\mathcal{V}_R$ are defined similarly and, after identifying ${\mathcal O}'\to \mathcal{V}_R/\R$ with the bundle ${\mathcal O} \to {\mathcal M}_R$, the remaining steps carry over without change.

\subsection{The variant $\widetilde\Phi$} \label{subsection: variant widetide psi}

In this subsection we define a complex $\widetilde{CF}(S,{\bf a},\hh({\bf a}))$ which is
closely related to $\widehat{CF}(S,{\bf a},\hh({\bf a}))$ and a variant
$$\widetilde \Phi: \widetilde{CF}(S,{\bf a},\hh({\bf a})) \to PFC_{2g}(N)$$
of $\Phi$.  This will be useful in the proof of Theorem~II.\ref{P2-thm: chain homotopy part i}. As a vector space, $\widetilde{CF}(S,{\bf a},\hh({\bf a}))$ will
also be used in Section \ref{section: chain map psi}.
\nom[1u$\phi$]{$\widetilde\Phi$}{Variant of $\Phi$ from $\widetilde{CF}(S,\mathbf{a},\hh(\mathbf{a}))$ to $PFC_{2g}(N)$}
\nom[CF1]{$\widetilde{CF}(S,\mathbf{a},\hh(\mathbf{a}))$}{Variant of $\widehat{CF}(S,\mathbf{a},\hh(\mathbf{a}))$ given by Definition~\ref{defn: variant chain cx}}

\subsubsection{The chain complex $\widetilde{CF}(S,\mathbf{a},\hh(\mathbf{a}))$}

Let $\overline{J}\in \mathcal{J}_{\overline{W}}^{reg}$. We recall some notation introduced in Section \ref{subsubsection: variant CF of S}: Let $\mathcal{I} \subset \{1, \ldots ,2g \}$ be a subset, $\mathcal{I}^c$ its complement, and $\mathfrak{S}_{\mathcal{I}^c}$ the group of permutations of $\mathcal{I}^c$.

\begin{defn} \label{defn: variant chain cx}
We define the chain complex $(\widetilde{CF}(S,\mathbf{a},\hh(\mathbf{a}),\overline{J}),
\widetilde\bdry)$ generated by the $2g$-tuples $\{z_{\infty, i}\}_{i \in \mathcal{I}}\cup
{\bf y'},$ where $\mathcal{I} \subset \{1, \ldots ,2g \}$, ${\bf y'}=\{y'_i \}_{i
\in \mathcal{I}^c}$, and the following hold:
\begin{itemize}
\item $z_{\infty,i}$ is viewed as an intersection point of $\overline{a}_i$ and $\overline{\hh}(\overline{a}_i)$ and
\item $y'_i\in a_i \cap \hh(a_{\sigma(i)})$ for some $\sigma\in {\frak S}_{\mathcal{I}^c}$.
\end{itemize}
\nom[Sbar]{$\mathcal{S}_{\overline{\bf a}, \overline{\hh}(\overline{\bf a})}$}{$2g$-tuples of intersection points $\{z_{\infty, i}\}_{i \in \mathcal{I}}\cup {\bf y'}$ of $\overline{\bf a}$ and $\overline{\hh}(\overline{\bf a})$}
We denote the set of intersection points $\{z_{\infty, i}\}_{i \in \mathcal{I}}\cup {\bf y'}$ as above by $\mathcal{S}_{\overline{\mathbf{a}}, \overline{\hh}(\overline{\mathbf{a}})}$.

The datum $z_{\infty, i}$ is equivalent to ``$z_\infty$ with either of the two matchings $(i,0)\to (i,0)$ or $(i,1)\to (i,1)$''
and therefore, using the notation of Section~\ref{subsection: modified indices at z infty}, $\{z_{\infty, i}\}_{i \in \mathcal{I}}\cup {\bf y'}$ is equivalent to the equivalence class of elements $z_{\infty}^p (\overrightarrow{\mathcal D}) \cup {\bf y'}$, where ${\mathcal I} = \{i_1, \ldots , i_p \}$ and $\overrightarrow{\mathcal D}=
\{ (i_1, j_1) \to (i_1, j_1), \dots, (i_1, j_1) \to (i_p, j_p) \}$ for any choice of $(j_1, \ldots,
j_p) \in \{0, 1 \}^p$.

The differential $\widetilde{\bdry}$ counts $I=1$ multisections $\overline{u}$ in
$\overline{W}$ with $n(\overline{u})\leq 1$, which satisfy one extra condition, i.e., if
we write $\overline{u}=\overline{u}'\cup\overline{u}''$ (according to the notation
introduced in Section \ref{subsubsection: multisections}), then $\overline{u}'$ has
empty branch locus. The homology of $(\widetilde{CF}(S,\mathbf{a},\hh(\mathbf{a})),\widetilde\bdry)$ is denoted by $\widetilde{HF}(S,\mathbf{a},\hh(\mathbf{a}))$.
\end{defn}

In the differential $\widetilde{\bdry}$, with the exception of trivial strips, we are counting the following curves:
\begin{enumerate}
\item thin strips from $z_{\infty,i}$ to either $x_{i}$ or $x'_{i}$; and
\item $I=1$ curves whose projections to $\overline{S}$ have image in $S$.
\end{enumerate}

\begin{lemma}
There is an isomorphism of chain complexes:
$$\kappa:(\widehat{CF}(\Sigma,\boldsymbol\beta,\boldsymbol\alpha,z),\widehat\bdry)\stackrel\sim\to (\widetilde{CF}(S,\mathbf{a},\hh(\mathbf{a})),\widetilde\bdry),$$
$$\{ x''_i \}_{i \in \mathcal{I}}\cup {\bf y'} \mapsto \{ z_{\infty,i}\}_{i \in \mathcal{I}}
\cup {\bf y'},$$
where $(\Sigma,\boldsymbol\beta,\boldsymbol\alpha,z)$ is as given in Section~\ref{subsubsection: Heegaard diagram compatible with S h}.
\end{lemma}

\begin{proof}
This is immediate from the definitions.
\end{proof}

Recall from Section~\ref{subsubsection: variant CF of S} that
$$(\widehat{CF}(S,{\bf a},\hh({\bf a})),\widehat{\bdry'})= (\widehat{CF'}(S,{\bf a},\hh({\bf a})),\bdry')/\sim$$ is the $E^1$-term of the spectral sequence for $(\widehat{CF}(\Sigma,\boldsymbol\beta,\boldsymbol\alpha,z),\widehat\bdry)$ (viewed as a double complex) in Theorem~\ref{t:hf}. (Here we are writing $\widehat{\bdry'}$ for the differential of $\widehat{CF}(S,{\bf a},\hh({\bf a}))$ to distinguish it from the differential $\widehat{\bdry}$ of $\widehat{CF}(\Sigma,\boldsymbol\beta,\boldsymbol\alpha,z)$.) 

In this paragraph ${\bf y}$ and ${\bf y}'_i$ denote linear combinations of tuples.
By tracing the zigzags in the double complex we obtain the isomorphism
$$\nu:\widehat{HF}(S,{\bf a},\hh({\bf a}))\stackrel\sim\to
\widehat{HF}(\Sigma,\boldsymbol\beta,\boldsymbol\alpha,z),$$
which is defined as follows:  Let ${\bf y}\in \widehat{CF'}(S,{\bf a},\hh({\bf a}))$ be a cycle in $\widehat{CF}(S,{\bf a},\hh({\bf a}))$, i.e.,
$$\bdry'{\bf y}=\sum_{i}(\{x_i\}\cup {\bf y}'_i+\{x_i'\}\cup {\bf y}'_i).$$
Then $\nu$ maps the equivalence class $[{\bf y}]$ to the equivalence class of
$${\bf y} + \sum_i \{x_i''\}\cup {\bf y}'_i+\mbox{ h.o.},$$
where ``h.o.'' means terms with more $x_i''$ components. Composing with the map
induced by $\kappa$ in homology, we obtain an isomorphism
$$\tilde{\nu} = \kappa_* \circ \nu: \widehat{HF}(S,{\bf a},\hh({\bf a}))\stackrel\sim\to \widetilde{HF}(S,\mathbf{a},\hh(\mathbf{a})).$$

\subsubsection{The map $\widetilde\Phi$}
 Let
$$\mathcal{M}_{\overline{J}_+} (\{z_{\infty,i}\}_{i \in \mathcal{I}}\cup {\bf y'},\boldsymbol{\gamma})=\coprod_{\overrightarrow{\mathcal{D}}}\mathcal{M}_{\overline{J}_+}(\{z_\infty^{\# \mathcal{I}}(\overrightarrow{\mathcal{D}})\}\cup {\bf y'},\boldsymbol{\gamma}),$$
where the moduli spaces on the right-hand side are defined in a manner analogous to Definition~\ref{overline W curve, version 2} and $\overrightarrow{\mathcal{D}}$ ranges over all matchings $\{(i,j_i)\to (i,j_i)\}_{i\in\mathcal{I}, j_i\in\{0,1\}}$.

\begin{defn}
Let $\overline{J}_+\in \mathcal{J}_{\overline{W}_+}^{reg}$ which restricts to $\overline{J} \in \mathcal{J}_{\overline{W}}^{reg}$ and $\overline{J}'
\in \mathcal{J}_{\overline{W'}}^{reg}$. We define the map $\widetilde\Phi : \widetilde{CF}(S, \mathbf{a}, \hh(\mathbf{a})) \to PFC_{2g}(N)$ as follows:
$$\langle \widetilde\Phi(\{z_{\infty,i}\}_{i \in \mathcal{I}}\cup {\bf y'}),\boldsymbol{\gamma}\rangle=
\#\mathcal{M}^{I=0,n^*\leq | \mathcal{I} |}_{\overline{J}_+} (\{z_{\infty,i}\}_{i \in \mathcal{I}}
\cup {\bf y'},\boldsymbol{\gamma}).$$
\end{defn}

In our analysis of $\widetilde\Phi$ we will use balanced coordinates
 (cf.\ Section~\ref{subsubsection: overline W pm}) for $\overline{N}-int(N)=
D^2\times (\R/2\Z)$.  The Morse-Bott family $\mathcal{N}$ can be identified with
$\bdry D^2$ and we write $\gamma_\phi$ for the orbit in $\mathcal{N}$ corresponding
to $e^{i\phi}$. So far in this paper $h\in \mathcal{N}$ was an arbitrary point.

\begin{convention} \label{convention for h}
From now on we specialize $h$ so that $h=\gamma_{\phi_h}$ is generic and $\phi_h$ is close to $-{2\pi \over m}$, where $m$ is as defined in Section~\ref{coconut}. In particular, the radial ray corresponding to $\phi_h$ does not lie on the thin wedges from $\overline{a}_i$ to $\overline{\hh}(\overline{a}_i)$ for all $i$. There are no restrictions on the orbit $e$ except that $e\not=h$.
\end{convention}

\begin{lemma} \label{lemma: value of widetilde Phi}
$\widetilde\Phi(\{z_{\infty,i}\}_{i \in \mathcal{I}}\cup {\bf y'})=0$ if $\mathcal{I} \neq
\varnothing$ and $\widetilde\Phi({\bf y'})= \Phi(\mathbf{y})$ if $\mathcal{I} = \varnothing$.
\end{lemma}

\begin{proof}
Let us fix $\{z_{\infty,i}\}_{i \in \mathcal{I}} \cup {\bf y'}$ with $|\mathcal{I}| \geq 1$ and write
$$\mathcal{M}^0(\boldsymbol{\gamma}):=\mathcal{M}_{\overline{J}_+}^{I=0,n^*\leq |\mathcal{I}|}
(\{z_{\infty,i}\}_{i \in \mathcal{I}}\cup {\bf y'}, \boldsymbol{\gamma}),$$
where $\boldsymbol\gamma\in\widehat{\mathcal{O}}_{2g}$.

We claim that $\mathcal{M}^0(\boldsymbol{\gamma})=\varnothing$. Arguing by contradiction, suppose there exists $\overline{u}\in \mathcal{M}^0(\boldsymbol{\gamma})$. By our assumptions on $\overline{J}_+$, $\overline{u}$ is close to breaking into a Morse-Bott building. We will show that such a Morse-Bott building cannot exist using the positivity of intersections, thus establishing the contradiction. For the rest of the proof, $\overline{u}$ will denote the holomorphic part of the Morse-Bott building instead of the original curve.

We consider the projection $\pi_{\overline{N}} : \overline{W}_+ \to \overline{N}$, obtained by restricting the projection $\overline{W'}= \R \times \overline{N} \to \overline{N}$ to $\overline{W}_+$. Let $T_\rho$ be the boundary of the neighborhood $V_\rho = S^1 \times D^2_\rho$ of the orbit $\delta_0$. We choose $\rho >1$, but arbitrarily close to $1$, so that $T_\rho$ is contained in $N$ and parallel to $\partial N$. Then $T_\rho$ is a pre-Lagrangian torus and we can assume without loss of generality that it is foliated by closed orbits.

We identify $H_1(T_\rho) \cong \Z^2$ such that --- writing vectors as rows --- $(1,0)$ corresponds to the homology class of the meridian (i.e., the closed curve which bounds a disk in $V_\rho$ intersecting $\delta_0$ once positively) and $(0,1)$ corresponds to the class of the orbits in $\partial N$ (i.e., those which are perturbed into $e$ and $h$) under the obvious identification $H_1(T_\rho) \cong H_1(\partial N)$. Then the Hamiltonian orbits on $T_\rho$ represent a homology class $(-p, q)$ with $p,q \ge 1$.

By abuse of notation, we do not distinguish $\overline{u}$ from its image, e.g., we write $\pi_{\overline{N}}(\overline{u})$ to denote the image of the composition $\pi_{\overline{N}} \circ \overline{u}$. The intersection  $\pi_{\overline{N}} (\overline{u}) \cap T_\rho$, oriented as the boundary of $\pi_{\overline{N}}(\overline{u}) \cap (\overline{N} - V_\rho)$, represents a class  $C \in H_1(T_\rho, \pi_{\overline{N}}(L_{\overline{\bf a}}^+) \cap T_\rho$). We recall that $\pi_{\overline{N}}(L_{\overline{\bf a}}^+) \cap T_\rho$ consists of $2g$ pairwise disjoint segments parallel to the Hamiltonian flow. The condition $n^*(\overline{u}) \le |{\mathcal I}|$ implies that $C$ consists of at most $|{\mathcal I}|$ arcs; moreover each arc has both endpoints on the same connected component of $\pi_{\overline{N}}(L_{\overline{\bf a}}^+) \cap T_\rho$ and represents the image of the class $(0,1) \in H_1(T_\rho)$ under the map $H_1(T_\rho)\stackrel\sim \to H_1(T_\rho, \pi_{\overline{N}}(L_{\overline{\bf a}}^+))$. This can be seen by considering the relation of the asymptotic eigenfunctions of $\overline{u}$ at $z_{\infty,i}$ with $n^*(\overline{u})$ and the possible ends of $\overline{u}$ in $V_\rho$.

The algebraic intersection between a Hamiltonian orbit of $T_\rho$ and the image of $\overline{u}$ is the same as the algebraic intersection in $T_\rho$ of a Hamiltonian orbit with $C$, which is negative. This contradicts the positivity of intersections, unless $\pi_{\overline{N}} (\overline{u})$ is disjoint from $T_\rho$. Since this holds for $\rho >1$ and arbitrarily close to $1$, there is a decomposition
$\overline{u}=\overline{u}''\cup\overline{u}'''$, where $I(\overline{u}'')=
I(\overline{u}''')=0$, $\overline{u}''$ has image in $\overline{W}_+-int(W_+)$, and
$\overline{u}'''$ has image in $W_+$.

Then $\overline{u}''$ is a curve from $z_\infty$
to $h$ since $I(\overline{u}'')=0$. The proof of the nonexistence of $\overline{u}''$ is
modeled on the proof of \cite[Lemma~8.4.8]{CGH2}.
By \cite[Section 4.2]{We1},
$\R\times (\overline{N}-N-\delta_0)$ is foliated by finite energy cylinders $Z_{s_0,\phi_0}$,
$(s_0,\phi_0)\in \R\times(\R/2\pi\Z)$, from $\delta_0$ to $\gamma_{\phi_0}$ such that:
\begin{itemize}
\item the image of $Z_{s_0,\phi_0}$ under the projection $\pi_{\overline{N}}: \R\times
\overline{N}\to \overline{N}$ is the open annulus $\{\phi=\phi_0, 0<\rho<1\}$;
\item $Z_{s_0+s_1,\phi_0}$ is the $s_1$-translate of $Z_{s_0,\phi_0}$.
\end{itemize}
We then set $Z^+_{s,\phi}=Z_{s,\phi}\cap \overline{W}_+$ and examine the
intersections $Z^+_{s,\phi}\cap \overline{u}''$. Observe that
$K=\pi_{\overline{N}}(Z^+_{s,\phi})\cap \pi_{\overline{N}}(\overline{u}'')\not=\varnothing$
for a suitable choice of $\phi$ which is close to but not equal to $\phi_h$; this is possible
by the positioning of $h$ given by Convention~\ref{convention for h}. Hence
$Z^+_{s_0,\phi}\cap \overline{u}''\not =\varnothing$ for some $s_0$. On the other hand,
since $K$ is compact, $Z^+_{s_0+s_1,\phi}\cap \overline{u}'' =\varnothing$ for a
sufficiently large $s_1$.  Finally, since $Z^+_{s,\phi}$ and $\overline{u}''$ have no
boundary intersections and no intersections near their ends for all $s\in \R$, we have
a contradiction.
\end{proof}

\begin{thm} \label{thm: phi tilde is a chain map}
$\widetilde\Phi$ is a chain map.
\end{thm}

\begin{proof}
Since $\Phi$ is a chain map by Theorem~\ref{thm: Phi is a chain map}, it suffices to
verify that
$$\bdry_{PFH}\circ \widetilde\Phi(\{z_{\infty,i}\}_{i \in \mathcal{I}}\cup {\bf y'}) =
\widetilde\Phi\circ \widetilde\bdry(\{z_{\infty,i}\}_{\mathcal{I}}\cup {\bf y'}),$$
whenever $\mathcal{I}\not=\varnothing$. On the one hand, $\widetilde\Phi (\{z_{\infty,i}\}_{\mathcal{I}}\cup {\bf y'})=0$ by
Lemma~\ref{lemma: value of widetilde Phi}. On the other hand, if $|\mathcal{I}|=1$,
then
\begin{align*}
\widetilde\Phi\circ \widetilde\bdry (\{z_{\infty,i_1}\}\cup {\bf y'}) &
= \widetilde\Phi(\{x_{i_1}\}\cup {\bf y'}+\{x'_{i_1}\}\cup {\bf y'})=0;
\end{align*}
and if $|\mathcal{I}|>1$, then each term of
$\widetilde\bdry (\{z_{\infty,i}\}_{i \in \mathcal{I}}\cup {\bf y'})$ contains some copy of
$z_\infty$, and $\widetilde\Phi$ maps the term to zero by Lemma~\ref{lemma: value of
widetilde Phi}. This proves the theorem.
\end{proof}

A corollary of Lemma~\ref{lemma: value of widetilde Phi} is the following:

\begin{cor} \label{tigger}
$\Phi=\widetilde\Phi\circ\kappa\circ \nu$ on the level of homology.
\end{cor}

\section{The chain map from $PFH$ to $\widehat{HF}$}
\label{section: chain map psi}

\subsection{Definition of $\Psi$} \label{subsection: defn of psi}
Fix the integer $m\gg 0$ which appears in the definition of the monodromy map $\overline{\hh}=\overline{\hh}_m$ from Section~\ref{subsubsection: overline W pm}. The condition $m\gg 0$ will be useful when applying limiting arguments in Sections~\ref{subsection: rescaling} and \ref{subsection: theorem complement}.

Suppose $\overline{J}_-\in \mathcal{J}_{\overline{W}_-}^{reg}$, $\overline{J'}$ and $\overline{J}$ are restrictions of $\overline{J}_-$ to the positive and negative ends, and $\overline{J}_-^\Diamond$ is $K_{p,2\delta}$-regular with respect to $\overline{\frak m}= ((0,{3\over 2}),z_\infty)$. We assume that the constants $\varepsilon,\delta>0$ that go into the definition of $\overline{J}_-^\Diamond$ (cf.\ Section~\ref{subsubsection: marked points and transversality}) are arbitrarily small.

\begin{defn}[Definitions of $\Psi'$ and $\Psi$] \label{defn: Psi map} $\mbox{}$
\begin{enumerate}
\item We define the map
$$\Psi'=\Psi'_{\overline{J}_-^\Diamond}(m,\overline{\frak m}):PFC_{2g}(N)\to \widehat{CF'}(S,{\bf a}, \hh({\bf a}))$$
$$\boldsymbol{\gamma}\mapsto \sum_{{\bf y}\in \mathcal{S}_{{\bf a},\hh({\bf a})}} \langle \Psi'(\boldsymbol{\gamma}),{\bf y}\rangle \cdot {\bf y},$$
where $\langle \Psi'(\boldsymbol{\gamma}),{\bf y}\rangle$ is the mod $2$ count of $\mathcal{M}^{I=2, n^*=m}_{\overline{J}_-^\Diamond}(\boldsymbol{\gamma},{\bf y};\overline{\frak m})$.

\item We define the map $\Psi=\Psi_{\overline{J}_-^\Diamond}(m,\overline{\frak m})$ as the composition
$$PFC_{2g}(N)\stackrel{\Psi'}\to \widehat{CF'}(S,\mathbf{a},\hh(\mathbf{a}))\stackrel{q}\to \widehat{CF}(S,\mathbf{a},\hh(\mathbf{a})),$$
where $q$ is the quotient map of chain complexes.
\end{enumerate}
\end{defn}
\nom[1x$\psi$]{$\Psi=\Psi_{\overline{J}_-^\Diamond}(m,\overline{\frak m})$}{Chain map from $PFC_{2g}(N)$ to $\widehat{CF}(S,\mathbf{a},\hh(\mathbf{a}))$}
\nom[1x$\psi$]{$\Psi'=\Psi'_{\overline{J}_-^\Diamond}(m,\overline{\frak m})$}{Map from $PFC_{2g}(N)$ to $\widehat{CF'}(S,\mathbf{a},\hh(\mathbf{a}))$}

The count $\langle \Psi'(\boldsymbol{\gamma}),{\bf y}\rangle$ is meaningful because of the following theorem:

\begin{thm}
$\mathcal{M}^{I=2, n^*=m}_{\overline{J}_-^\Diamond}(\boldsymbol{\gamma},{\bf y};\overline{\frak m})$ is compact.
\end{thm}

\begin{proof}
This follows from Theorems~\ref{thm: compactness for W minus I=2, first version} and \ref{thm: complement}(i) and Corollary~\ref{alternative}.
\end{proof}

Now we define the maps $\widetilde{U}_{m-1}$, $\widetilde\Psi_0$, and $\widetilde\bdry_1$ which measure the failure of $\Psi'$ to be a chain map. The subscripts in $\widetilde{U}_{m-1}$, $\widetilde\Psi_0$, and $\widetilde\bdry_1$ indicate that we are counting curves $\overline{u}$ satisfying $n^*(\overline{u})=m-1$, $0$, and $1$.

First we define $\widetilde{U}_{m-1} : ECC_{2g}(N) \to ECC_{2g}(\overline{N})$: Let $\boldsymbol{\gamma}\in \widehat{\mathcal{O}}_{2g}$ and $\boldsymbol{\gamma}'\in \widehat{\mathcal{O}}_{2g-p}$, $p=0,\dots,2g$. Then
$$\langle\widetilde{U}_{m-1} (\boldsymbol{\gamma}),\delta_0^p\boldsymbol{\gamma}'\rangle = \left\{
\begin{array}{cl}
A, & \mbox{ if } p=1;\\
0, & \mbox{ if } p\not=1,
\end{array}
\right.$$
where $A$ is the count of $I=2$ multisections $\overline{u}$ of $\overline{W'}$ from $\boldsymbol{\gamma}$ to $\delta_0\boldsymbol{\gamma}'$ which satisfy $n^*(\overline{u})=m-1$ and a certain asymptotic condition near $\delta_0$ (the precise definition will be given in Section~\ref{subsubsection: melon d'eau}).

Next we define $\widetilde\Psi_0 : ECC_{2g}(\overline{N}) \to \widetilde{CF}(S,{\bf a},\hh({\bf a}))$: Let $\boldsymbol{\gamma}'\in \widehat{\mathcal{O}}_{2g-p}$, $p=0,\dots,2g$, and $\{z_{\infty, i}\}_{i \in \mathcal{I}}\cup {\bf y'} \in {\mathcal S}_{\overline{\bf a}, \overline{\hh}(\overline{\bf a})}$.
Then
$$\langle\widetilde\Psi_0(\delta_0^p\boldsymbol{\gamma}'),\{z_{\infty,i}\}_{\mathcal{I}}
\cup {\bf y'}\rangle=\left\{
\begin{array}{cl}
B, & \mbox{ if } p=1 \mbox{ and } |\mathcal{I}|=1; \\
0, & \mbox{ if } p\not=1 \mbox{ or } |\mathcal{I}| \not =1,
\end{array} \right.$$
where $B$ is the count of degree $2g-1$, $I=0$ almost multisections $\overline{u}$ of $\overline{W}_-$ from $\boldsymbol{\gamma}'$ to ${\bf y'}$ which satisfy $n^*(\overline{u})=0$, together with the section at infinity $\sigma_\infty^-$ from $\delta_0$ to $z_\infty$.

Finally we define $\widetilde\bdry_1 : \widetilde{CF}(S,{\bf a},\hh({\bf a})) \to \widehat{CF}(S,{\bf a},\hh({\bf a}))$. If $\{z_{\infty,i}\}_{\mathcal{I}} \cup {\bf y'} \in {\mathcal S}_{\overline{\bf a}, \overline{\hh}(\overline{\bf a})}$,
then we define
\begin{equation} \label{eqn: bdry sub 1}
\widetilde\bdry_1(\{z_{\infty,i}\}_{\mathcal{I}} \cup {\bf y'}) = \left\{
\begin{array}{cl} \{x_i\}\cup{\bf y'} + \{x_i'\}\cup {\bf y'}, & \mbox{ if } \mathcal{I}= \{ i \}; \\ 0, & \mbox{ if } |\mathcal{I}| \ne 1.
\end{array} \right.
\end{equation}
We can see $\widetilde\bdry_1$ as the composition of two maps: a map from $\widetilde{CF}(S,{\bf a},\hh({\bf a}))$ to itself induced by the count of $n^*(\overline{u})=1$ curves in $\overline{W}$ (which are necessarily thin wedges), followed by the projection $\widetilde{CF}(S,{\bf a},\hh({\bf a})) \to \widehat{CF}(S,{\bf a},\hh({\bf a}))$ sending every $2g$-tuple of chords involving $z_\infty$ to zero.

Let $\bdry'_{HF}$ be the differential for $\widehat{CF'}(S,{\bf a},\hh({\bf a}))$ and let $\bdry_{PFH}$ be the differential for $PFC_{2g}(N)$.
The proof of the following theorem will occupy the rest of Section~\ref{section: chain map psi}.

\begin{thm} \label{thm: almost chain map for psi}
If $m\gg 0$, then
\begin{equation}
\label{eqn: chain map for Psi prime}
\bdry'_{HF}\circ \Psi' +\Psi'\circ \bdry_{PFH} = \widetilde\bdry_1\circ\widetilde\Psi_0\circ \widetilde{U}_{m-1}.
\end{equation}
\end{thm}

Assuming Theorem~\ref{thm: almost chain map for psi} for the moment, we have the following:

\begin{cor} \label{cor: psi is a chain map}
The map $\Psi $ is a chain map.
\end{cor}

\begin{proof}
Similar to the proof of Theorem~\ref{thm: Phi is a chain map} and based on
Equation~\eqref{eqn: chain map for Psi prime}.  There is one major difference:
{\em $\Psi'$ is not a chain map}.  However, since 
the image of $\widetilde{\partial}_i$ is contained in $\ker q$,
by composing Equation~\eqref{eqn: chain map for Psi prime}
with $q$ we obtain:
\begin{equation}
\widehat\bdry_{HF}\circ \Psi+\Psi\circ \bdry_{PFH} =0,
\end{equation}
where $\widehat\bdry_{HF}$ is the differential for $\widehat{CF}(S,{\bf a},\hh({\bf a}))$.
\end{proof}

We can also define a map
\begin{equation} \label{eqn: Psi twisted}
\underline{\Psi} : \underline{PFH}_{2g}(N) \to \underline{\widehat{CF}}(S,{\bf a},\hh({\bf a}))
\end{equation}
which is $\F[H_2(M; \Z)]$-linear. As for $\underline{\Phi}$ in Section~\ref{subsection: twisted coefficients phi map}, the key point is the definition of maps
$$\mathfrak{A}_- : H_2(\overline{W}_-, \boldsymbol{\gamma}, {\bf y}) \to H_2(M).$$
The main difference is that $H_2(\overline{W}_-) \cong H_2(\overline{N}) \cong H_2(M) \oplus \langle \overline{S} \rangle$. Hence we define
$$\mathfrak{A}_-(C) = p((\pi_{\overline{N}})_*[C_{\bf y} \cup C \cup -C_{\boldsymbol{\gamma}}]),$$
where $\pi_{\overline{N}}$ is the projection $\overline{W}_- \to \overline{N}$ and $p : H_2(\overline{N}) \to H_2(M)$ is the projection given by $p = id - \langle \delta_0, \cdot \rangle[\overline{S}]$.

\subsection{Outline of proof of Theorem~\ref{thm: almost chain map for psi}}
\label{subsection: outline of proof}

In this subsection we outline the proof of Theorem~\ref{thm: almost chain map for psi}.

Let $\overline{J}_-\in \mathcal{J}_{\overline{W}_-}^{reg}$, with restrictions $\overline{J'}$ and $\overline{J}$ to the positive and negative ends, and let $\overline{J}_-^\Diamond$ be $K_{p,2\delta}$-regular with respect to $\overline{\frak m}$. Suppose that the constants $\varepsilon,\delta>0$ that go into the definition of $\overline{J}_-^\Diamond$ are arbitrarily small.

We abbreviate:
$$\mathcal{M}^i:=\mathcal{M}^{I=i,n^*=m}_{\overline{J}_-^\Diamond}(\boldsymbol{\gamma},{\bf y}), \quad \mathcal{M}^i_{\overline{\frak m}}:=\mathcal{M}^{I=i,n^*=m}_{\overline{J}_-^\Diamond}(\boldsymbol{\gamma},{\bf y};\overline{\frak m}).$$
Let $\overline{\mathcal{M}^i_{\overline{\frak m}}}$ be the SFT compactification of $\mathcal{M}^i_{\overline{\frak m}}$, obtained by adding holomorphic
$\overline{W}_-$-buildings as in Definition~\ref{defn of holom building 2}, and let $\bdry \mathcal{M}^i_{\overline{\frak m}}= \overline{\mathcal{M}^i_{\overline{\frak m}}}-\mathcal{M}^i_{\overline{\frak m}}.$

\s\n {\bf Step 1.} By SFT compactness for $\overline{W}_-$-curves (Proposition~\ref{prop: SFT compactness for W minus}), a sequence $\overline{u}_i\in\mathcal{M}_{\overline{\frak m}}^3$ admits a subsequence which converges to a holomorphic $\overline{W}_-$-building
$$\overline u_\infty=\overline{v}_{-b}\cup\dots \cup \overline{v}_a.$$
As before, we write $\overline{v}_j=\overline{v}'_j\cup \overline{v}''_j$, where $\overline{v}'_j$ branch covers the section at infinity $\sigma_\infty^*$ and $\overline{v}''_j$ is the union of irreducible components which do not branch cover $\sigma_\infty^*$. Here $*=\varnothing$, $'$, or $-$.
\nom[u]{$\overline{u}_\infty=\overline{v}_{-b}\cup\dots\cup \overline{v}_a$}{Limit of $\overline{W}_-$-curves in Section~\ref{section: chain map psi}}

There are two cases: $\overline{v}'_0=\varnothing$ or $\overline{v}'_0\not=\varnothing$. The latter case is harder, and will be treated first.  The former one will be treated in Step 3. By analyzing the two constraints $n^-(\overline{u}_i)=m$ and $I(\overline{u}_i)=3$, we obtain Theorem~\ref{thm: compactness for W minus I=3, first version}, which gives a preliminary list of possibilities for $\overline{u}_\infty$ when $\overline{v}'_0\not=\varnothing$. Many of the possibilities in Theorem~\ref{thm: compactness for W minus I=3, first version} actually do not occur.   Theorem~\ref{thm: complement} eliminates Cases (2)--(6) from the list. This is done by a finer analysis of the behavior of $\overline{u}_i$ in the vicinity of the component $\sigma_\infty^-$ and is similar in spirit to the {\em layer structures} of Ionel-Parker~\cite[Section 7]{IP1}.

Summarizing, we have:

\begin{lemma} \label{lemma alt}
If $\overline{u}_\infty\in \bdry \mathcal{M}_{\overline{\frak m}}^3$ and $\overline{v}'_0\not=\varnothing$, then $\overline{u}_\infty\in A_1$, where
\begin{align*}
A_1 ={} & \coprod_{\delta_0\boldsymbol{\gamma}',\{z_\infty\}\cup {\bf y'}} \left(\mathcal{M}^{I=2,n^*=m-1}_{\overline{J'}}(\boldsymbol{\gamma},\delta_0\boldsymbol{\gamma}')\times \ \mathcal{M}^{I=0,n^*=0}_{\overline{J}_-^\Diamond} (\delta_0\boldsymbol{\gamma}',\{z_\infty\}\cup {\bf y}')\right.\\
&\qquad \qquad \left.\times \left(\mathcal{M}^{I=1,n^*=1}_{\overline{J}}(\{z_\infty\}\cup{\bf y}',{\bf y})/\R\right)\right),
\end{align*}
if ${\bf y}=\{x_i\}\times {\bf y'}$ or $\{x_i'\}\times{\bf y}'$ and $A_1=\varnothing$ otherwise.
Here we have omitted the potential contributions of connector components for simplicity.
\end{lemma}

 Case (1) in Theorem~\ref{thm: compactness for W minus I=3, first version} corresponds to $A_1$.

\s\n {\bf Step 2.}
We now glue the triples $(\overline{v}_1,\overline{v}_0,\overline{v}_{-1})$ in $A_1$, subject to the constraint $\overline{\frak m}$. This gluing accounts for the term $\widetilde\bdry_1\circ \widetilde\Psi_0\circ \widetilde{U}_{m-1}$ and is a bit involved. Let us abbreviate
\begin{align*}
&\mathcal{M}_1: =\mathcal{M}^{I=2,n^*=m-1}_{\overline{J'}}(\boldsymbol{\gamma},\delta_0\boldsymbol{\gamma}');\\
&\mathcal{M}_1':= \mathcal{M}^{I=2,n^*=m-1,f_{\delta_0}}_{\overline{J'}}(\boldsymbol{\gamma},\delta_0\boldsymbol{\gamma}');\\
&\mathcal{M}_0 := \mathcal{M}^{I=0,n^*=0}_{\overline{J}_-^\Diamond} (\delta_0\boldsymbol{\gamma}',\{z_\infty\}\cup {\bf y'});\\
&\mathcal{M}_{-1} := \mathcal{M}^{I=1,n^*=1}_{\overline{J}}(\{z_\infty\}\cup{\bf y'},{\bf y}),
\end{align*}
where ${\bf y}= \{x_i\}\cup {\bf y''}$ or $\{x_i'\}\cup{\bf y''}$. Here $f_{\delta_0}$ is a nonzero normalized asymptotic eigenfunction of $\delta_0$ at the negative end which, used as a modifier, stands for ``the normalized asymptotic eigenfunction at the negative end $\delta_0$ is $f_{\delta_0}$.''  See Section~\ref{subsection: asymptotic eigenfunctions} for more details on asymptotic eigenfunctions.
\nom[2f]{$*=f_{\delta_0}$}{Modifier ``the normalized asymptotic eigenfunction at the negative end $\delta_0$ is $f_{\delta_0}$''}

We first observe that $\overline{v}_0=\overline{v}_0'\cup \overline{v}_0''\in \mathcal{M}_0$ is regular: $\overline{v}_0'$ is regular by Lemma~\ref{lemma: regularity of curve at infinity} and $\overline{v}_0''$ is regular since $\overline{J}_-\in \mathcal{J}_{\overline{W}_-}^{reg}$ and $\overline{J}_-^\Diamond$ is $(\varepsilon,U)$-close to $\overline{J}_-$. The moduli spaces $\mathcal{M}_1$ and $\mathcal{M}_{-1}$ are also regular since $\overline{J}_-$ is regular. Hence we can glue triples $([\overline{v}_1],\overline{v}_0,[\overline{v}_{-1}])$ in $(\mathcal{M}_1/\R)\times \mathcal{M}_0\times(\mathcal{M}_{-1}/\R)$.  More precisely, consider the gluing parameter space
\begin{equation}
{\frak P}:=\coprod_{\delta_0\boldsymbol{\gamma}',\{z_\infty\}\cup {\bf y'}} {\frak P}_{\delta_0\boldsymbol{\gamma}',\{z_\infty\}\cup {\bf y'}},
\end{equation}
where
\begin{equation} \label{gluing parameter space}
{\frak P}_{\delta_0\boldsymbol{\gamma}',\{z_\infty\}\cup {\bf y'}}= (5r,\infty)^2 \times (\mathcal{M}_1/\R)\times \mathcal{M}_0\times (\mathcal{M}_{-1}/\R),
\end{equation}
$0<h<1$ and $r\gg 1/h$ are gluing constants, and $\mathcal{M}_1,\mathcal{M}_0,\mathcal{M}_{-1}$ correspond to the pair $(\delta_0\boldsymbol{\gamma}',\{z_\infty\}\cup {\bf y'})$. We may assume that all the multiplicities of $\boldsymbol{\gamma}'$ are $1$, since the Hutchings-Taubes gluing of branched covers is essentially independent of the present gluing problem.

We recall the extended moduli space
$$\mathcal{M}^{ext}:= \mathcal{M}_{\overline{J}_-}^{I=3,n^*=m,ext}(\boldsymbol{\gamma},{\bf y})$$
from Definition~\ref{defn: extended curves}, where ${\bf y}= \{x_i\}\cup {\bf y''}$ or $\{x_i'\}\cup{\bf y''}$. Here the modifier $ext$ means that $\overline{u}:(\dot F,j)\to (\overline{W}_-,\overline{J}_-)$ is a multisection which maps all the connected components of $\bdry \dot F$ but one to a different $L^-_{\overline{a}_i}$ and the last connected component to some $L^-_{\overline{a}_i\cup\vec{a}_{i,j}}$.

There is a gluing map
$$G: {\frak P}\to \mathcal{M}^{ext},\quad {\frak d}=(T_\pm,\overline{v}_1,\overline{v}_0,\overline{v}_{-1})\mapsto\overline{u}({\frak d}),$$
which is a diffeomorphism onto its image for $r\gg 0$. Here $T_\pm$ is shorthand for the pair $T_+,T_-$. The map $G$ is defined in a manner similar to that of Section~\ref{subsubsection: reviw of HT} and is described in more detail in Section~\ref{subsection: proof of gluing}. Let $r_0\gg 0$, let ${\frak P}_{(r_0)}\subset {\frak P}$ be the subset $\{T_+ \geq r_0\}$ and let
$$\mathcal{M}_{\overline{\frak m}}^{3,(r_0)}:=\mathcal{M}_{\overline{\frak m}}^3- G({\frak P}_{(r_0)})$$
be the truncated moduli space. For generic $r_0\gg 0$,
$$\bdry_0\mathcal{M}_{\overline{\frak m}}^{3,(r_0)}:=G(\bdry{\frak P}_{(r_0)})\cap\mathcal{M}_{\overline{\frak m}}^{3}$$
is a transverse intersection.

The following theorem is proved in Section~\ref{subsection: gluing for psi}:

\begin{thm} \label{thm: transversality of ev map}
For generic $r_0\gg 0$,
\begin{equation} \label{two degrees}
\#(\bdry_0 \mathcal{M}_{\overline{\frak m}}^{3,(r_0)}\cap G({\frak P}_{\delta_0\boldsymbol{\gamma}',\{z_\infty\}\cup {\bf y'}}))\equiv \# (\mathcal{M}'_1/\R)\cdot \# \mathcal{M}_0\cdot \# (\mathcal{M}_{-1}/\R) \mbox{ mod $2$}.
\end{equation}
\end{thm}
Hence the contributions from $A_1$ account for the term $\widetilde\bdry_1\circ \widetilde\Psi_0\circ \widetilde{U}_{m-1}$.

\s\n {\bf Step 3.} Let $\bdry_1 \mathcal{M}_{\overline{\frak m}}^{3,(r_0)}= \bdry \mathcal{M}_{\overline{\frak m}}^{3,(r_0)} - \bdry_0\mathcal{M}_{\overline{\frak m}}^{3,(r_0)}$.  The following lemma is proved in Section~\ref{proof of lemma}.

\begin{lemma} \label{claim 1}
$\bdry_1 \mathcal{M}_{\overline{\frak m}}^{3,(r_0)}\subset A_2\sqcup A_3$, where
\begin{align*}
A_2 ={} &  \coprod_{{\bf y''}\in \mathcal{S}_{{\bf a},\hh({\bf a})}} \left(\mathcal{M}^{I=2,n^*=m}_{\overline{J}_-^\Diamond}(\boldsymbol{\gamma},{\bf y''};\overline{\frak m})\times \left(\mathcal{M}^{I=1,n^*=0}_{\overline{J}}({\bf y''},{\bf y})/\R\right)\right); \\
A_3 ={} & \quad\coprod_{\boldsymbol{\gamma}'\in \widehat{\mathcal{O}}_{2g}}\quad\left(\left(\mathcal{M}^{I=1,n^*=0}_{\overline{J'}}(\boldsymbol{\gamma},\boldsymbol{\gamma}')/\R\right)\times \mathcal{M}^{I=2,n^*=m}_{\overline{J}_-^\Diamond}(\boldsymbol{\gamma}',{\bf y};\overline{\frak m})\right).
\end{align*}
Here we have omitted the potential contributions of connector components for simplicity.
\end{lemma}

On the other hand, gluing the pairs in $A_2$ (resp.\ $A_3$) using the usual (resp.\ Hutchings-Taubes) gluing theorem implies that $A_2\cup A_3\subset \bdry_1 \mathcal{M}_{\overline{\frak m}}^{3,(r_0)}$ and accounts for the term $\bdry'_{HF}\circ \Psi'$  (resp.\ $\Psi'\circ \bdry_{PFH}$).  This proves Theorem~\ref{thm: almost chain map for psi}.

\s\n {\em Organization of Section~\ref{section: chain map psi}.} In Sections~\ref{subsection: SFT compactness for W minus} we describe the necessary modifications to apply the proof of SFT compactness to the nonstandard setting of $\overline{W}_-$-curves. The signed intersection numbers $n^*(\overline{u})$ are analyzed in Sections~\ref{subsection: intersection numbers} and ~\ref{some restrictions, first version}. Theorem~\ref{thm: compactness for W minus I=3, first version} is then proved in Section~\ref{subsection: compactness theorem, first version}.  We then discuss asymptotic eigenfunctions in Section~\ref{subsection: asymptotic eigenfunctions}, rescaling in Section~\ref{subsection: rescaling}, and the ``involution lemmas'' in Section~\ref{subsection: involutions}, on our way to proving Theorem~\ref{thm: complement} in Section~\ref{subsection: theorem complement}. Lemma~\ref{claim 1} is proved in Section~\ref{proof of lemma}. Finally, the gluing is discussed in Section~\ref{subsection: gluing for psi}.

\subsection{SFT compactness}\label{subsection: SFT compactness for W minus}

\subsubsection{Statement of the compactness theorem}
\begin{defn} \label{defn of holom building 2}
A {\em holomorphic $\overline W_-$-building}
$$\overline{u}_\infty=\overline{v}_{-b}\cup\dots \cup \overline{v}_a,\quad a,b\in \Z^{\geq 0},$$
consists of the following data:
\begin{enumerate}
\item[(B1$'$)]  For each $j=-b,\dots,a$, a compact {\em nodal} Riemann surface $G_j$ (possibly with boundary), a nodal set $\mathfrak{n}_j$, disjoint sets of interior punctures ${\bf p}_j^+$ and ${\bf p}_j^-$ if $j \ge 0$, and disjoint sets of boundary punctures ${\bf q}_j^+$ and ${\bf q}_j^-$ if $j \le 0$. 
The nodes may be interior or boundary nodes and are disjoint from ${\bf p}_j^\pm$, ${\bf q}_j^\pm$, and the nodal set $\mathfrak{n}_j$ may be empty.
\item[(B2$'$)] For each $j=-b,\dots,a$,  a holomorphic map $\overline{v}_j: \dot G_j\to \overline{W}_j$, where $\overline{W}_j=\overline{W'}$ for $0<j\leq a$,    $\overline{W}_0=\overline{W}_-$ for $j=0$, $\overline{W}_j= \overline{W}$ for $-b\leq j<0$, and $\dot G_j=G_j-{\bf p}_j^+-{\bf p}_j^--{\bf q}_j^+-{\bf q}_j^-$ for all $j$. (Here some sets of punctures may be empty.)
\item[(B3$'$)] For each  $j=-b,\dots,0$, $\bdry \dot G_j$ is mapped to the appropriate Lagrangian submanifold $L_{\overline{\bf a}}\sqcup L_{\overline{\hh}(\overline{{\bf a}})}$ or $L^-_{\overline{\bf a}}$.
\item[(B4$'$)] For each $j=-b,\dots,a$, $\overline{v}_j$ converges to a trivial strip over a ``Reeb chord'' at the positive (resp.\ negative) end near each boundary puncture of ${\bf q}_j^+$ (resp.\ ${\bf q}_j^-$) and to a trivial cylinder over a closed orbit at the positive (resp.\ negative) end near each interior puncture of ${\bf p}_j^+$ (resp.\ ${\bf p}_j^-$).
\item[(B5$'$)] For each $j=-b,\dots,a-1$, there is an identification between ${\bf p}_j^+$ and ${\bf p}_{j+1}^-$ and an identification between ${\bf q}_j^+$ and ${\bf q}_{j+1}^-$ such that the pairs that are identified are asymptotic to the same Reeb chord or closed orbit.
\item[(B6$'$)] No level $\overline{v}_j$ is a union of trivial cylinders or trivial strips.
\end{enumerate}
The levels $\overline{v}_j$ are arranged in order from lowest to highest.
\end{defn}

\begin{prop} \label{prop: SFT compactness for W minus}
Let $\overline{J}_-\in \mathcal{J}_{\overline W_-}$ and let $\overline{u}_i: (\dot F_i,j_i)\to (\overline{W}_-,\overline{J}_-)$, $i\in \N$, be a sequence of $\overline{W}_-$-curves from $\boldsymbol{\gamma}$ to ${\bf y}$. Then there is a subsequence which converges in the sense of SFT to a level $a+b+1$ holomorphic $\overline{W}_-$-building
$$\overline u_\infty=\overline{v}_{-b}\cup\dots \cup \overline{v}_a.$$
The same also holds when $\{\overline{u}_i\}$ is a sequence of $(\overline{W}_-,\overline{J}_-^\Diamond)$-curves.
\end{prop}

\subsubsection{Preliminary lemmas}

Here we prove some lemmas which will be used in the proof of Proposition~\ref{prop: SFT compactness for W minus}. First we recall some terminology. Let $\overline{W}_*$ denote one of $\overline{W}$, $\overline{W'}$, $\overline{W}_\pm$ and $\overline{\omega}$ the $2$-form on $\overline{W}_*$ defined in Section~\ref{subsubsection: overline W pm}. We recall that the $\overline{\omega}$-area of a map $w : \dot F \to \overline{W}_*$ is defined as
$$E_{\overline{\omega}}(w) = \int_{\dot F} w^* \overline{\omega}.$$

We also introduce the {\em energy} $E(w)$ in the spirit of Definition~\ref{defn: energy of Lipshitz curve}. More precisely, let $\mathcal{C}$ be the set of nondecreasing smooth functions $\phi :\R \rightarrow [0,1]$; for all $\phi \in {\mathcal C}$ we define the $1$-form
$\lambda_\phi = \overline{\pi}_{B_*}^* (\phi(s)dt)$, where as usual $B_*$ can stand for
$B$, $B'$ or $B_\pm$. Then the energy of a map $w : \dot F \to \overline{W}_*$ is
\begin{equation} \label{ene}
E(w)=\int_{\dot F}  w^* \overline{\omega} + \sup_{\phi \in \mathcal{C}} \int_{\dot F} w^* d \lambda_\phi= E_{\overline{\omega}}(w) + \sup_{\phi \in \mathcal{C}} \int_{\dot F} w^* d \lambda_\phi.
\end{equation}

If $w$ is a holomorphic map, then $E_{\overline{\omega}}(w) \ge 0$ and
moreover a finite energy holomorphic cylinder $w:[-R,R]\times S^1\to \overline{W}_*$ (or strip
$w:[-R,R]\times [0,1]\to \overline{W}_*$) with $E_{\overline{\omega}}(w) =0$ is either a constant map or is part of a trivial cylinder (or strip) over a Hamiltonian orbit (or chord). We recall also that the
{\em action} of a Hamiltonian orbit or chord is the integral, over that orbit or chord, of the $1$-form $\overline{\alpha}_0$ defined in Section~\ref{subsubsection: overline W pm}. By Equation~\eqref{ene}, the energy of $w$ is determined by its $\overline{\omega}$-area and the actions of the orbits and chords at its ends.

The following lemma is a by-now standard observation originally due to Hofer.

\begin{lemma}\label{lemma: simple L-infty bound}
 Let $F$ be a fixed compact Riemann surface, possibly with boundary, and let $w_n : F \to \overline{W}_*$ be a sequence of holomorphic maps such that:
\begin{itemize}
\item $E(w_n) \le C_0$ for some constant $C_0$, and
\item $\lim \limits_{n \to + \infty} E_{\overline{\omega}}(w_n) =0$.
\end{itemize}
Then there is a constant $C$ such that $\| \nabla w_n \|_{C^0} \le C$ for all $n$.
\end{lemma}

\begin{proof}[Sketch of proof]
The lemma follows from the standard bubbling argument: if there exists a sequence of points $p_n \in F$ such that $\lim \limits_{n \to + \infty} |\nabla w_n(p_n)| = + \infty$, then, by the usual rescaling argument (we may apply a translation in the target if the sequence $w_n$ is not bounded), we have bubbling of a {\em nonconstant}
holomorphic plane or half-plane $\tilde{w}$.

The function $f  = \overline{\pi}_* \circ \tilde{w}$ is holomorphic and has finite energy. Hence it can be extended to a holomorphic function $\tilde{f} : S^2 \to B_*$ or
$\tilde{f} : (\D, \partial \D) \to (B_*, \partial B_*)$. For topological reasons $\tilde{f}$ must be
constant and hence the image of $\tilde{w}$ must be contained in a fiber. However $E_{\overline{\omega}}(\tilde{w})=0$, and therefore $\tilde{w}$ is constant. This is a contradiction.
\end{proof}

The following ``Long Cylinder'' and ``Long Strip'' Lemmas, morally speaking, say that a cylinder or strip which is sufficiently long and has sufficiently small $\overline{\omega}$-area is either almost constant or close to a piece of a trivial cylinder or strip (cf.~\cite[Proposition 5.7]{BEHWZ} and \cite{HWZ3} for related results).

\begin{lemma}[Long Cylinder Lemma, Version 1]\label{long cylinder lemma v1}
Let $\gamma_+$ and $\gamma_-$ be closed Hamiltonian orbits in $\overline{N}$.  Fix $E_0$ larger than the action of both $\gamma_+$ and $\gamma_-$, and $\eta_0>0$ such that every closed Hamiltonian orbit $\gamma$ of action at most $E_0$ is at distance at least $2 \eta_0$ from $\gamma_+$, if $\gamma \ne \gamma_+$, and from $\gamma_-$, if $\gamma \ne \gamma_-$. Then for all $\eta \in (0, \eta_0)$ there
exist $c_0 >0$, $R_0 >0$, $\eta' >0$ and $\delta >0$ such that, for all $c \ge c_0$ and all $R>R_0$, the following holds: If
$$w : [-R-c, R+c] \times S^1 \to \overline{W'}$$
is a $\overline{J}'$-holomorphic map which satisfies:
\begin{enumerate}
\item $w|_{[-R-c, -R] \times S^1}$ and $w|_{[R, R+c] \times S^1}$ are $\eta'$-close in the $C^1$-metric to portions of trivial cylinders over $\gamma_-$ and $\gamma_+$, respectively, and
\item $E_{\overline{\omega}}(w) < \delta$,
\end{enumerate}
then $\gamma_- = \gamma_+$ and $w$ is $\eta$-close in the $C^1$-metric to a portion of the trivial cylinder over $\gamma_-$.
\end{lemma}

\begin{proof}
First we claim that, if the hypotheses of the lemma hold, then the image of $w$ is contained in an $\eta$-neighborhood of the cylinder over $\gamma_-$.  This implies that $\gamma_-= \gamma_+$ because $\gamma_-$ is the only closed Hamiltonian orbit of action less than $E_0$ which is contained in an $\eta$-neighborhood of $\gamma_-$.

Arguing by contradiction, assume there exist  $\eta < \eta_0$,
sequences $R_n$, $c_n$, $\eta_n'$, $\delta_n$, $p_n \in  [-R, R] \times S^1$ and $\overline{J}'$-holomorphic maps $w_n : [-R_n-c_n, R_n+c_n] \times S^1 \to \overline{W'}$ such that:
\begin{itemize}
\item $R_n \to + \infty$, $c_n \to + \infty$, $\eta_n' \to 0$ and $\delta_n \to 0$,
\item $w_n$ satisfy (1) and (2) with $R_n$, $c_n$, $\eta'_n$ and $\delta_n$ in place of $R$, $c$, $\eta'$  and $\delta$, and
\item $w_n(p_n)$ is at a distance of $\eta$ from the trivial cylinder over $\gamma_-$. (The proof is the same for $\gamma_+$ instead of $\gamma_-$.)
\end{itemize}
Conditions (1) and (2) imply that $E_{\overline{\omega}}(w_n) \to 0$ and $E(w_n) < E_0$ for $n$ sufficiently large. Hence by Lemma~\ref{lemma: simple L-infty bound} there is a constant $C$ such that $\| \nabla w_n\|_{C^0}\leq C$ for all $n$.  Then, after translating the maps $w_n$ so that the points $w_n(p_n) \in \overline{W'} = \R \times \overline{N}$ all have the same $\R$-coordinate and taking a subsequence, the maps $w_n$ converge in $C^\infty_{loc}$ to a $\overline{J}'$-holomorphic map $w_\infty : \R \times S^1 \to \overline{W'}$ and the points $p_n$ converge to a point  $p_\infty =(s_\infty, t_\infty) \in \R \times S^1$ such that $E(w_\infty) < E_0$, $E_{\overline{\omega}}(w_\infty)=0$, and $w_\infty(p_\infty)$ is at $\eta$-distance from the trivial cylinder over $\gamma_-$.

Since there is no closed Hamiltonian orbit with action less than $E_0$ at $\eta$-distance from $\gamma_-$, the map $w_\infty$ is constant.
However this is a contradiction because it implies that, for $n$ sufficiently large,
the loop $t \mapsto w_n(s_\infty,t)$ is contractible, while the loop $t \mapsto w_n(-R_n-c_n, t)$, being close to $\gamma_-$, is not. (We made use here of the fact that $N$ has no contractible Hamiltonian orbits.)  This proves the claim.

To pass from the claim to ``$w$ is $\eta$-close in the $C^0$-metric to a trivial cylinder over $\gamma_-$'', we consider the projection $\overline{\pi}_{B'} : \overline{W'} \to \R \times S^1$. The composition $\overline{\pi}_{B'}\circ w:  [-R-c, R+c] \times S^1 \to \R \times S^1$ is holomorphic. Let $k$ be the
 multiplicity of $\gamma_-$ and identify $\R \times S^1$ with $\C^*$. Then (1) and the maximum principle imply that the derivatives of $f(s,t) = e^{-k(s+it)}(\overline{\pi}_{B'} \circ w)(s,t)$ are small (depending only on $\eta'$). Therefore, after possibly reducing $\eta'$ in the statement of the theorem, $w$ is close to a trivial cylinder.
This concludes the proof of the lemma.
\end{proof}

A similar lemma holds for holomorphic strips. Its proof is left to the reader.

\begin{lemma}[Long Strip Lemma, Version 1]\label{long strip lemma v1}
Let $y_+$ and $y_-$ be intersection points in $\overline{a}_i \cap \overline{\hh}(\overline{a}_{i'})$.  Then for all $\eta \in (0, \eta_0)$ there exist $c_0 >0$, $R_0 >0$, $\eta' >0$ and $\delta>0$ such that, for all $c \ge c_0$ and all $R>R_0$, the following holds: If
$$w : [-R-c, R+c] \times [0,1] \to \overline{W}$$
is a $\overline{J}$-holomorphic map which satisfies:
\begin{enumerate}
\item $w$ maps $[-R,R]\times \{1\}$ to $L_{\overline{a}_i}$ and $[-R,R]\times\{0\}$ to $L_{\overline a_{i'}}$;
\item $w|_{[-R-c, -R] \times S^1}$ and $w|_{[R, R+c] \times S^1}$ are $\eta'$-close in the $C^1$-metric to portions of trivial strips over $y_-$ and $y_+$, respectively, and
\item $E_{\overline{\omega}}(w)< \delta$,
\end{enumerate}
then $y_- = y_+$ and $w$ is $\eta$-close in the $C^1$-metric to a portion of the trivial strip over $y_-$.
\end{lemma}

Now we consider long cylinders and strips whose extremities are close to constant maps. First we develop a bound on their $\overline{\omega}$-area in terms of the
behavior of the extremities, which is a consequence of the topology of $\overline{W}_-$.
\begin{lemma}\label{uppsala}
For every $\delta>0$ there is $\eta'>0$ such that, for every $R>0$, if either
\begin{itemize}
\item $w: [-R, R] \times S^1 \to \overline{W}_-$ is a $\overline{J}_-$-holomorphic map, or
\item $w : [-R, R] \times [0,1] \to \overline{W}_-$ is a  $\overline{J}_-$-holomorphic map mapping $[-R,R] \times \{0,1 \}$ to $L^-_{\overline{\mathbf{a}}}$,
\end{itemize}
and $w|_{\{ \pm R \} \times S^1}$ (or $w|_{\{ \pm R \} \times [0,1]}$) are $\eta'$-close to constant maps in the $C^1$-metric, then $E_{\overline{\omega}}(w) < \delta$.
\end{lemma}

\begin{proof}
We consider strips first since cylinders are easier. We have two cases: either the boundary of $[- R,  R] \times [0,1]$ is mapped to the same component $L^-_{\overline{a}_i}$ or the two boundary components are mapped to different components $L^-_{\overline{a}_i}$ and $L^-_{\overline{a}_j}$.

We embed $[- R, R] \times [0,1]$ holomorphically in $D^2=\{|z|\leq 1~|~ z\in \C\}$ as the complement of an open neighborhood of $\{1, -1 \}$ and extend
$w$ to a map $\widetilde{w} : D^2 \to \overline{W}_-$ with boundary on  $L^-_{\overline{a}_i}$ or $L^-_{\overline{a}_i} \cup L^-_{\overline{a}_j}$, depending on the situation. We can make the extension so that, in the neighborhood of $\{ 1, -1\}$, $\widetilde{w}$ remains $\eta'$-close in the $C^1$-metric to constant maps, 
and therefore there is a positive constant $C$ depending only on $\overline{\omega}$ such that
\begin{equation} \label{area bound}
|E_{\overline{\omega}}(w) - E_{\overline{\omega}}(\widetilde{w})| \le C \eta'.
\end{equation}

In the first case, $\widetilde{w}$ defines an element in $\pi_2(\overline{W}_-, L^-_{\overline{a}_i})$, and, since the latter group is trivial, $E_{\overline{\omega}}(\widetilde{w})=0$ and then we can chose $\eta'$ so that Equation~\eqref{area bound} implies that $E_{\overline{\omega}}(w) < \delta$.

The second case is similar, but the map $\widetilde{w}$ defines an element in the relative homotopy group $\pi_2(\overline{W}_-, L^-_{\overline{a}_i} \cap  L^-_{\overline{a}_j})$. If $g(\overline S)>1$, then the group is trivial and we can conclude as before.

If $\overline{S}$ is a torus, the first observation is that $\widetilde{w}$ is homotopic to a map (denoted by the same symbol) $\widetilde{w} : (D^2, \partial D^2) \to (\overline{S}, \overline{a}_i \cup \overline{a}_j)$ such that the preimage of an $\eta'$-neighborhood of $z_\infty$ consists of the two fixed neighborhoods of $+1$ and $-1$, respectively. The homotopy is obtained by taking the universal cover $\widetilde{B}_-$ of $B_-$, identifying the pullback of the fibration $\overline{W}_- \to B_-$ with $\widetilde{B}_- \times \overline{S}$, lifting $\widetilde{w}$ to  $\widetilde{B}_- \times \overline{S}$, and composing the lift with the projection onto
$\overline{S}$.

The universal cover of $\overline{S}$ is diffeomorphic to $\R^2$ and we can assume that the preimage of $\overline{a}_i \cup \overline{a}_j$ is a grid with vertices in the integer points of $\R^2$. If we lift $\widetilde{w}$ to the universal cover of $\overline{S}$, we obtain a map from $D^2$ to $\R^2$ with boundary on the grid.
If $\widetilde{w}$ represents a non-trivial element in $\pi_2(\overline{S}, \overline{a}_i \cup \overline{a}_j)$, then its lift maps $\partial D^2$ surjectively to
some square of the grid, and therefore the preimage of an $\eta'$-neighborhood of the grid points consists of at least four disjoint open sets in $D^2$, which is a contradiction.

The case of cylinders is proved in the same way, using the fact that $\pi_2(\overline{W}_-)=0$.
\end{proof}

\begin{rmk}
Since we are free to chose the genus of the open book decomposition in Theorem~\ref{thm: main}, the simpler case of $g(\overline{S })>1$ would have sufficed for the main result of the paper. However we also treated the case of $g(\overline S)=1$ for completeness.
\end{rmk}

\begin{lemma}[Long Cylinder Lemma, Version 2]\label{long cylinder lemma v2}
For all $\eta >0$ there exists $\eta' >0$
such that, for all $R>0$, if
$$w : [-R, R] \times S^1 \to \overline{W}_-$$
is a $\overline{J}_-$-holomorphic such that 
$w|_{\{ -R \} \times S^1}$ and $w|_{\{ R \} \times S^1}$ are $\eta'$-close in the $C^1$-metric to constant maps to points $p_-$ and $p_+$, respectively, 
then $w$ is $\eta$-close to a constant map in the $C^1$-metric.
\end{lemma}

\begin{proof}
First we prove that, for $\eta'$ sufficiently small (and at least smaller than $\eta/2$), the map $\overline{\pi}_{B_-} \circ w$ is $\eta$-close to a constant map. We embed  $[-R, R] \times S^1$ holomorphically in $S^2$.  If $\eta'$ is small enough, we can extend $w$ to a smooth map $\widetilde{w} : S^2 \to \overline{W}_-$ and moreover we can assume that the extension occurs inside $\eta'$-neighborhoods of $p_+$ and $p_-$.

The composition $\overline{\pi}_{B_-} \circ w : [-R, R] \times S^1 \to B_-$ is holomorphic and its extension $\overline{\pi}_{B_-} \circ \widetilde{w} : S^2 \to B_-$ has degree zero.  This implies that the image of $\overline{\pi}_{B_-} \circ w$ is contained in the union of the balls of radius $\eta'$ centered at $\overline{\pi}_{B_-}(p_+)$ and $\overline{\pi}_{B_-}(p_-)$. In fact, assume by contradiction that this
is not the case. Then there is a regular value $q \in B_- - \overline{\pi}_{B_-}(p_+) -\overline{\pi}_{B_-}(p_-)$ of $\overline{\pi}_{B_-} \circ \widetilde{w}$ such that $(\overline{\pi}_{B_-} \circ \widetilde{w})^{-1}(q)$ is contained in an open subset of $S^2$ on which $\overline{\pi}_{B_-} \circ \widetilde{w}$ coincides with $\overline{\pi}_{B_-} \circ w$, and therefore is holomorphic.  Hence have a contradiction because all preimages of $q$ contribute positively to the degree of $\overline{\pi}_{B_-} \circ \widetilde{w}$. The balls of radius $\eta'$ centered at $\overline{\pi}_{B_-}(p_+)$ and $\overline{\pi}_{B_-}(p_-)$ have a nonempty intersection because $[-R, R]\times S^1$ is connected, and therefore the image of $\overline{\pi}_{B_-} \circ w$ is contained in a ball or radius $2 \eta' < \eta$ around some point $y \in B_-$.

Let $\overline{S}_y=\overline{\pi}_{B_-}^{-1}(y)$ be the fiber over $y$ and $N_\eta(\overline{S}_y)$ the $\eta$-neighborhood of $\overline{S}_y$.
We denote $p:N(\overline{S}_y)\to\overline{S}_y$ the map obtained by projecting out the $s$- and $t$-directions. Since  the image of $w$ is contained in $N_\eta(\overline{S}_y)$, the map $p \circ \widetilde{w} : S^2 \to \overline{S}_y$
has degree $0$ if $\eta'$ is small enough and $E_{\overline{\omega}} \le \delta$ for some $\delta$ which depends only on $\overline{\omega}$.\footnote{In fact, $p \circ \widetilde{w}$ has degree zero because $\overline{S} \ne S^2$. We use this unnecessary $\overline{\omega}$-area argument so that the same proof can be applied to Lemma~\ref{long strip lemma v2}.} The bound on $E_{\overline{\omega}}$ follows from Lemma~\ref{uppsala}.

The map $p$ is not holomorphic, but its differential $dp(x)$ at any point $x\in N_\eta(\overline{S}_y)$ is the composition of the $\C$-linear projection
$T_x\overline{W}_- \to T_x \overline{S}_{\overline{\pi}_{B_-}(x)}$ with the differential of the orientation-preserving diffeomorphism $\overline{S}_{\overline{\pi}_{B_-}(x)} \to \overline{S}_y$ induced by parallel transport. This implies that all the preimages of a regular value of $p \circ w$ come with positive sign for the purpose of computing the index. Hence we can then repeat for $p \circ w$ the same argument we used for $\overline{\pi}_{B_-} \circ w$ and this proves the lemma.
\end{proof}

A similar lemma holds for holomorphic strips. Its proof is left to the reader.

\begin{lemma}[Long Strip Lemma, Version 2]\label{long strip lemma v2}
For all $\eta >0$ there exists $\eta' >0$ 
such that for all $R>0$ the following holds: If
$$w: [-R,R]\times [0,1]\to \overline{W}_-$$
is a $\overline{J}_-$-holomorphic map which satisfies:
\begin{enumerate}
\item $w$ maps $[-R,R]\times \{1\}$ to $L^-_{\overline{a}_i}$ and $[-R,R]\times\{0\}$ to $L^-_{\overline{a}_{i'}}$,
\item $w|_{\{R\}\times [0,1]}$ and $w|_{\{-R\}\times [0,1]}$ are $\eta'$-close in the $C^1$-norm to constant maps to $p_+$ and $p_-$, respectively, 
\end{enumerate}
then $w$ is $\eta$-close to a constant map in the $C^1$-metric.
\end{lemma}

\subsubsection{Proof of the compactness theorem}

We follow the lines of the proof of Proposition \ref{prop: SFT compactness for W plus}. First we state the analog of Lemma \ref{lemma: bound on F for W plus} and sketch its proof.

\begin{lemma} \label{lemma: bound on F for W bar minus}
Let $\overline{u}_i: (\dot F_i,j_i)\to (\overline{W}_-, \overline{J}_-)$, $i\in \N$, be a sequence of $\overline{W}_-$-curves from $\boldsymbol{\gamma}$ to $\mathbf{y}$ with index $I_{\overline{W}_-}(\overline{u}_i) =n$ for some integer $n$. Then there is a subsequence such that all the $\dot F_i$ are diffeomorphic to a fixed $\dot F$.
\end{lemma}

\begin{proof}
The $\overline{\omega}$-area bound for the $\overline{W}_-$ case is similar to that of the $W_+$ case. We take the difference of two $\overline{W}_-$-curves $\overline u_1$ and $\overline u_2$ from $\boldsymbol{\gamma}$ to ${\bf y}$ to obtain $Z\in H_2(\overline{W}_-)$. Since they both intersect $\sigma_\infty^-$ once, $Z$ can be represented by a surface in $W_-$, and hence can be viewed as a class in $H_2(N)$. The zero flux condition implies that the $\overline{\omega}$-area of $Z$ is zero.

We then use the Gromov-Taubes compactness theorem as before to extract a subsequence $\overline u_i$ which converges in the sense of currents to a holomorphic building. After passing to a subsequence, we may assume that the homology class $[\overline{u}_i]\in H_2(\check{\overline{W}}_-,Z_{\boldsymbol{\gamma},{\bf y}})$ is fixed for all $i$. Once the homology class $[\overline{u}_i]$ is fixed, we use the relative adjunction formula (Lemma~\ref{lemma: relative adjunction for W plus}) in the same way as in Lemma~\ref{lemma: bound on F for W plus} to obtain a bound on $\chi(\dot F_i)$.
\end{proof}

The geometric setup for the $\overline{W}_-$-curves deviates from the standard
SFT setup because the Lagrangian submanifold $L_{\overline{\bf a}}$ is a union of finitely many immersed non-compact Lagrangians with boundary, with the property that each is
an embedding when restricted to the interior, and the boundary components all lie
in a properly embedded noncompact one-manifold along which the various branches of $L_{\overline{\bf a}}$ have a clean intersection. The noncompactness of the
singular set means that the resulting ``Legendrian'' boundary condition at the negative end is singular. For this reason we will sketch the proof of Proposition~\ref{prop: SFT compactness for W minus}.

\begin{proof}[Proof of Proposition~\ref{prop: SFT compactness for W minus}]
By Lemma~\ref{lemma: bound on F for W bar minus} we may assume that $\dot F_i = \dot F$ as smooth surfaces. We recall the main steps of the proof of SFT compactness.
\begin{enumerate}
\item gradient bound,
\item degeneration of the conformal structure,
\item convergence on the thick part, and
\item convergence on the thin part.
\end{enumerate}

\s\n
{\em Gradient bound.}
Let $h$ be a Riemannian metric on $\overline{W}_-$ which is compatible with $\overline{J}_-$ and is cylindrical at both ends.  By \cite[Lemma~10.7]{BEHWZ}, after passing to a subsequence we can add finite sets of marked points $Z_i$ to $\dot F$, all of the same cardinality, so that the following holds.  Let $g_i$ be a complete hyperbolic metric on $\dot F - Z_i$ which is compatible with $j_i$ and which has geodesic boundary and cusps at the boundary/interior punctures. Then we have a uniform bound
\begin{equation} \label{eq: gradient bound}
\rho_i \| \nabla \overline{u}_i\|_{C^0} \le C,
\end{equation}
the so-called  ``gradient bound'', where $\rho_i$ is the injectivity radius of  $(\dot F -Z_i, g_i)$ doubled along the boundary, and the norm $\| \nabla \overline{u}_i\|_{C^0}$ is computed with respect to the metrics $g_i$ and $h$.

\s\n {\em Degeneration of the conformal structure.}
 If $g_i$ degenerates as $i\to \infty$, then after passing to a subsequence there is a finite
collection of homotopy classes of properly embedded arcs and closed
curves on $\dot F- Z_i$, whose geodesic representatives are mutually
disjoint, and which are pinched as $i\to \infty$, i.e., the lengths
of the geodesic representatives go to zero.  Let $\dot F_{\infty}$
be the nodal surface obtained by pinching all these geodesics. The complex
structures $j_i$ converge in the $C^{\infty}_{loc}$ topology outside the
pinching geodesics to a complex structure $j_\infty$ on $\dot F_{\infty}$. Then
in the limit we obtain a nodal punctured Riemann surface with boundary
$(\dot F_{\infty}, j_{\infty})$. Let $F_\infty'$ be $\dot F_\infty$ with the nodes removed.

For any $\varepsilon >0$ we define $\op{Thick}_\varepsilon(\dot F_i, g_i)$ as the set
of points $x \in \dot  F_i$ with injectivity radius $\rho_i(x) \ge \varepsilon$ and
$\op{Thin}_\varepsilon(\dot F_i, g_i)$ as the closure of the set of points $x \in \dot  F_i$ with injectivity radius $\rho_i(x) < \varepsilon$. If $i$ is sufficiently large and
$\varepsilon$ is sufficiently small, then $\op{Thin}_\varepsilon(\dot F_i, g_i)$ is a disjoint union of cusps and tubular neighborhoods of pinching geodesics. In this case,
let $\upsilon$ be any pinching geodesic; we denote  $\op{Thin}_\varepsilon(\dot F_i, g_i, \upsilon)$ the connected component of $\op{Thin}_\varepsilon(\dot F_i, g_i)$ containing $\upsilon$.

\s\n
{\em Convergence on the thick part.}
Fix $\varepsilon >0$. Then the gradient bound (Equation~\eqref{eq: gradient
bound}) implies that the sequence $\overline{u}_i$ converges  in the
$C^\infty_{loc}$ topology on the $\varepsilon$-thick part, after possibly passing to a subsequence and translating in the $s$-direction.  In fact an elliptic bootstrapping argument propagates the gradient bound to bounds on the derivatives of any order and the Arzel\`a-Ascoli theorem produces a convergent subsequence. Now we take a sequence $\varepsilon_i \to 0$ and use a diagonal argument to find a limit
finite energy holomorphic map $\overline{u}_\infty'$ defined on $F_{\infty}'$ (i.e., $\dot F_\infty$ minus the double points). Depending on the connected component of $F_\infty'$, the restriction of $\overline{u}_\infty'$ can take values in $\overline{W}_-$, $\overline{W}$ or $\overline{W'}$. We observe that, although the maps $\overline{u}_i$ have boundary conditions in $L^-_{\widehat{\mathbf{a}}}$, the levels of the limit $\overline{u}_\infty'$ can
have boundary conditions in $L^-_{\overline{\mathbf{a}}}$ or in $L_{\overline{\mathbf{a}}} \cup L_{\overline{\hh}(\overline{\mathbf{a}})}$, depending on the connected component of $F_\infty'$.

We analyze the convergence of $\overline{u}_\infty'$ near the punctures in detail because the singularity in $L^-_{\overline{\mathbf{a}}}$ is an unusual situation. We describe only the case of components mapping to $\overline{W}_-$; the behavior of components mapping to $\overline{W}$ or $\overline{W'}$ is analogous. The first step is to replace  $\overline{\mathbf{a}}$ with $\vec{\mathbf{a}}$ (this is temporary notation for the union of $\overline{a}_i$ which have been extended by adding the $\vec{a}_{i,j}$). This replacement has no practical effect on the curve $\overline{u}_\infty'$, but makes the geometric setting slightly less exotic: in fact the Lagrangian $L^-_{\vec{\mathbf{a}}}$ has a {\em clean self-intersection} along $\partial B_- \times \{z_\infty\}$. We recall that a clean intersection is an intersection between two submanifolds which are nontransverse in a uniform way. Clean intersections in Lagrangian Floer homology --- which represent the Morse-Bott case in that theory --- have been treated in \cite{Po} and \cite{Schm}.

There are four possibilities, depending whether the puncture is in the interior or on the boundary of $F_\infty'$ and whether $\overline{u}_\infty'$ is bounded or unbounded near the puncture. Interior punctures are treated in the standard way: if $\overline{u}_\infty'$ is bounded near an interior puncture, then the singularity is removable (i.e., $\overline{u}_\infty'$ can be extended as a holomorphic map across the puncture) by \cite[Theorem~4.1.2]{MS}, and if $\overline{u}_\infty'$ is unbounded, then it is asymptotic to a closed
orbit in the positive end of $\overline{W}_-$ by \cite[Theorem~1.2]{HWZ1}. (Note that this orbit can either be an orbit in $N$ or a multiple cover of the extra orbit $\delta_0$ over the point $z_\infty$.)

Next we consider boundary punctures of $\overline{u}'_\infty$.
We restrict attention to the neighborhood $Z_- = (-\infty, 0] \times [0,1]\subset \dot F_\infty$ of a boundary puncture and prove the following lemma.

\begin{lemma}\label{lemma: boundary puncture}
Let $w : Z_- \to \overline{W}_-$ be a finite energy $\overline{J}_-$-holomorphic map which maps $(-\infty,0]\times\{0,1\}$ to $L_{\widehat{\mathbf{a}}}^-$. Then there exists either
\begin{enumerate}
\item a point $p \in L_{\overline{\mathbf{a}}}^-$ such that $\lim_{s \to -\infty}w(s,t) = p$ (uniformly in $t$), or
\item an intersection point $y \in \overline{\mathbf{a}} \cap\overline{\hh}(\overline{\mathbf{a}})$ such that $w$ is asymptotic to the trivial strip over $y$ for $s \to -\infty$.
\end{enumerate}
\end{lemma}

The removal of boundary singularities in the presence of clean intersections between Lagrangian submanifolds has been treated in \cite[Section~3.7]{Frau} and \cite[Section~3]{Schm}. For completeness we reprove it in a weak form (i.e., without reproving exponential convergence estimates) which suffices for our purposes, using ideas from \cite{Ho1}.

\begin{proof}
Consider the holomorphic map $f  = \overline{\pi}_{B_-} \circ w : Z_- \to B_-$. Since $f$ has finite energy, it either converges to a point $q \in \partial B_-$ (uniformly in $t$) as $s \to -\infty$ or for any $C\gg 0$ there exists $s_0\gg 0$ such that $f|_{\{s\geq s_0\}}$ has image in the negative end $(- \infty, -C] \times [0,1] \subset B_-$. Note that, in the latter case, the two components of $\partial Z_-|_{\{s\geq s_0\}}$ are mapped to the two boundary components of the end.

Fix a sequence $s_n \to +\infty$ and define $w_n : Z_- \to \overline{W}_-$ as $w_n(s,t) = w(s-s_n,t)$.  The finiteness of the energy $E(w)$ implies the finiteness of the $\overline{\omega}$-area $E_{\overline{\omega}}(w)$. Therefore the sequence $w_n$
satisfies
\begin{itemize}
\item $E(w_n) < E(w)$, and
\item $\lim \limits_{n \to \infty} E_{\overline{\omega}}(w_n) =0$. 
\end{itemize}

\s\n
(1) If $f$ converges to a point $q \in \partial B_-$, then the image of $w$ is contained in a compact subset of $\overline{W}_-$.
Then by Lemma \ref{lemma: simple L-infty bound} there is a constant $C$ such that
$\| \nabla w_n \|_{C^0} \le C$ for all $n$ and therefore we can extract a subsequence
$w_{k_n}$ which converges in $C^{\infty}_{loc}$ to a $\overline{J}_-$-holomorphic map
$w_{\infty} : Z_- \to \overline{W}_-$. 
The $\overline{\Omega}_-$-area
$$E_{\overline{\Omega}_-}(w)= \int_{Z_-} w^* \overline{\Omega}_-$$
is finite because the energy of $w$ is finite and the image of $w$ is contained in a
compact subset of $\overline{W}_-$. This implies that
$$E_{\overline{\Omega}_-}(w_\infty) = \lim_{n \to \infty} E_{\overline{\Omega}_-}(w_{k_n})=0$$
and thus $w_{\infty}$ is constant.  Hence, for any sequence $s_n \to + \infty$, there is a subsequence $s_{k_n}$ and a point $p \in L_{\overline{\mathbf{a}}}^-$ such that
$$\lim_{n \to \infty}w(s_{k_n},t) = p$$
and the limit is uniform in $t$. Moreover, $\nabla w(s_{k_n},t)$ also converges to $0$ uniformly in $t$.

A different sequence $s_n' \to + \infty$ might a priori have a subsequence $s_{l_n}'$ such that $w(s_{l_n},t) $ converges to a different $p' \in L_{\hat{\mathbf{a}}}^-$. We are then in Case (1) if we prove that $p=p'$.  Up to extracting further subsequences we can assume that $s_{k_n} < s_{l_n}'$. We fix $\varepsilon$ which is strictly smaller than the distance between $p$ and $p'$. Then the Long Strip Lemma, Version 2 (Lemma~\ref{long strip lemma v2}) implies that, for $n$ sufficiently large, $w|_{[s_{k_n}, s_{l_n}'] \times [0,1]}$ is $\varepsilon$-close to a constant map, which is a contradiction.

\s\n
(2) If $f = \overline{\pi}_{B_-} \circ w$ has image in the negative end of $B_-$ for $s$ sufficiently large, the sequence $w_n$ converges, after a translation of the image, to a trivial strip in $\overline{W}$ over some intersection point $y \in \overline{\mathbf{a}} \cap \overline{\hh}(\overline{\mathbf{a}})$. The uniqueness of the limit is proved as in the previous case using the Long Strip Lemma, Version 1 (Lemma~\ref{long strip lemma v1}).
\end{proof}

\s\n
{\em Convergence on the thin part.}
The map $\overline{u}_\infty'$ defined above is not the SFT limit of the sequence $\overline{u}_i$ in general. In fact the $\overline \omega$-area of $\overline{u}_\infty'$ can be smaller than the $\overline \omega$-area of $\overline{u}_i$ and, when this happens, we have to look for the missing area in the thin part.

We analyze the behavior of the limit near a node of $\dot F_\infty$ and leave the analogous case of a cusp to the reader. Let $\upsilon$ be the pinching geodesics corresponding to a node $n$. We  slightly change perspectives by identifying the neighborhoods $\op{Thin}_\varepsilon(\dot F_i, g_i, \upsilon)$ with standard Euclidean cylinders $[-R_i^\varepsilon,R_i^\varepsilon] \times S^1$ or strips $[-R_i^\varepsilon,R_i^\varepsilon] \times S^1$.

Suppose that $n$ is an interior node. For $i$ large and $\varepsilon$ small, $u_i(\{ \pm R^\varepsilon_i \} \times S^1)$ is close to closed orbits $\gamma_\pm$ or points $p_\pm$ by the asymptotic properties of $u_\infty'$. The first observation is that $u_i([- R^\varepsilon_i, R^\varepsilon_i] \times S^1)$ cannot be close to a closed orbit on one side and to a point on the other side because
there is no homotopically trivial closed Hamiltonian orbit in $N$. Similarly, if $n$ is a boundary node, then $u_i([- R^\varepsilon_i, R^\varepsilon_i] \times [0,1])$ cannot be close to a chord on one side and to a point on the other side because all chords define nontrivial classes in $H_1(\overline{W}_-, L^-_{\overline{\mathbf{a}}})$.

\begin{lemma} \label{vallecchia}
Let $n \in \dot F_\infty$ be a node and let $e_\pm : [0, + \infty) \times S^1 \to F'_\infty$ or $e_\pm : [0, + \infty) \times [0,1]\to F'_\infty$ be the two ends corresponding to $n$. If
$$\lim \limits_{s \to + \infty} (\overline{u}_\infty' \circ e_\pm)(s,t) = p_\pm\in \overline{W}_-,$$
then $p_+=p_-$.
\end{lemma}

\begin{proof}
Suppose that $n$ is an interior node; the proof for a boundary node is completely analogous.  Arguing by contradiction, suppose that $p_- \ne p_+$.
We choose $\eta < \frac 12 d(p_-, p_+)$, where $d(p_-, p_+)$ is the distance between $p_-$ and $p_+$. Let $\eta'>0$ be the constant which depends on $\eta$ and appears in  the Long Cylinder Lemma, Version 2 and let $\overline{s}$ be a constant such that, for all $s> \overline{s}$, the distance from $p_\pm$ to $ (\overline{u}_\infty' \circ e_\pm)(s,t)$ is less than $\eta'/2$.
Let $\upsilon$ be the pinching geodesic that gives rise to $n$ and let $\op{Thin}_\varepsilon(F_\infty', g_\infty, \upsilon)$ be the union of the two connected components of $\op{Thin}_\varepsilon(F_\infty', g_\infty)$ corresponding to the node $n$.  Then there is a constant $\overline{\varepsilon} >0$ such that, for all $\varepsilon < \overline{\varepsilon}$, the portions of the two ends of $F_\infty'$ corresponding to $n$ which are parametrized by $[0, \overline{s}] \times S^1$ are contained in $\op{Thick}_\varepsilon(F_\infty', g_\infty, \upsilon)$. We then identify $\op{Thin}_\varepsilon(\dot F_i, g_i, \upsilon)\cong [-R^\varepsilon_i, R^\varepsilon_i] \times S^1$ with its standard complex structure. Then the $C^{\infty}_{loc}$-convergence of $\overline{u}_i$ to $\overline{u}_\infty'$ implies that, for $\varepsilon < \overline{\varepsilon}$ and $i\gg 0$, $\overline{u}_i(\{ \pm R^\varepsilon_i \} \times S^1)$ are contained in $\eta'$-neighborhoods of $p_\pm$.  The Long Cylinder Lemma, Version 2 then implies that $p_+=p_-$, a contradiction.
\end{proof}

\begin{lemma}\label{holomorphic sausage}
Let $n \in \dot F_\infty$ be a node arising from pinching a geodesic in the homotopy class $\upsilon$.  Assume that the two ends of $F_\infty'$ corresponding to $n$ are
asymptotic to closed orbits $\gamma_+$ and $\gamma_-$ or to chords over intersection points in $\overline a_i \cap \overline{\hh}(\overline a_j)$. Then, for all $\varepsilon$ sufficiently small, the restrictions  $\overline{u}_i|_{\op{Thin}_\varepsilon(\dot F_i, g_i, \upsilon)}$ converge (in the SFT sense), as $i\to\infty$, to a stack of holomorphic cylinders or strips whose ends are asymptotic to closed orbits or chords, and whose successive components have matching ends, as in \cite[Figure~14]{BEHWZ}.
\end{lemma}

\begin{proof}
We consider the case of interior nodes first. For  $\varepsilon>0$ sufficiently small and $i$ sufficiently large (depending on $\varepsilon$), the maps $\overline{u}_i|_{\op{Thin}_\varepsilon(\dot F_i, g_i, \upsilon)}$ define a homology class $A \in H_2(\overline{N}, \gamma_-, \gamma_+)$. There are two cases:
\begin{enumerate}
\item $E_{\overline{\omega}}(A)=0$, or
\item $E_{\overline{\omega}}(A)>0$.
\end{enumerate}
We will identify $\op{Thin}_\varepsilon(\dot F_i, g_i, \upsilon)\cong [-R^\varepsilon_i, R^\varepsilon_i] \times S^1$.

\s\n
{\em Case (1).} We verify that $\overline{u}_i|_{\op{Thin}_\varepsilon(\dot F_i, g_i, \upsilon)}$ satisfies the conditions of Lemma~\ref{long cylinder lemma v1} for $i\gg 0$: Fix $\eta$ arbitrarily small, and consider $c_0$, $R_0$, $\eta'$ and $\delta$ as in Lemma~\ref{long cylinder lemma v1}. Let $\varepsilon>0$ be sufficiently small such that  $\op{Thin}_\varepsilon(F_\infty', g_\infty, \upsilon)$ is mapped by $\overline{u}_\infty'$ to $\frac{\eta'}{2}$-neighborhoods of trivial cylinders over $\gamma_-$ and $\gamma_+$.  Then there exists $c>c_0$ such that for all $i\gg 0$,
\begin{itemize}
\item $R^\varepsilon_{i}-c > R_0$,
\item $\overline{u}_{i}|_{[-R^\varepsilon_{i}, - R^\varepsilon_{i}+c] \times S^1}$ is $\eta'$-close to a portion of trivial cylinder over $\gamma_-$,
\item $\overline{u}_{i}|_{[R^\varepsilon_{i} -c, R^\varepsilon_{i}] \times S^1}$ is $\eta'$-close to a portion of trivial cylinder over $\gamma_+$, and
\item $E_{\overline{\omega}}(\overline{u}_{i}|_{[-R^\varepsilon_{i},  R^\varepsilon_{i}] \times S^1}) < \delta$.
\end{itemize}
This follows from observing that every compact subset $K \subset \op{Thin}_\varepsilon(\dot F_\infty', g_\infty, \upsilon)$ can be identified with some compact subset $K_i \subset \op{Thin}_\varepsilon(\dot F_i', g_i, \upsilon)$ provided $i\gg 0$, and the sequence $\overline{u}_i|_{K_i}$ converges uniformly to $\overline{u}_{\infty}'|_{K}$ with respect to all derivatives.

We then apply Lemma~\ref{long cylinder lemma v1} to $\overline{u}_{i}|_{[-R^\varepsilon_{i}, R^\varepsilon_{i}] \times S^1}$ and conclude that $\gamma_+=\gamma_-$, the parametrizations of $\gamma_-$ and $\gamma_+$ induced by the two ends of $\dot F'$ at the two sides of the node $n$ match, and $\overline{u}_{i}|_{[-R^\varepsilon_{i}, R^\varepsilon_{i}] \times S^1}$ is $C^1$-close to a portion of the trivial cylinder over $\gamma_+=\gamma_-$.

\s\n {\em Case (2).}
Fix a constant $h>0$ such that, for every $A \in H_2(\overline{N}, \gamma', \gamma'')$ with $\gamma'$ and $\gamma''$ of action less than $E = E(\overline{u}_i)$, we have
$\int_A \overline{\omega} > h$. Such a constant exists by \cite[Lemma 10.9]{BEHWZ}.

Fix $\varepsilon>0$ sufficiently small.  Since $\overline{u}_i([-R^\varepsilon_i, R^\varepsilon_i] \times S^1)$ is contained in the cylindrical end of $\overline{W}_-$, we may assume that the maps $\overline{u}_i$ take values in $\overline{W'}$. Using the argument from the proof of Lemma~\ref{long cylinder lemma v1}, if
the image of $\overline{u}_{i}|_{[-R^\varepsilon_{i}, R^\varepsilon_{i}] \times S^1}$ is contained in an $\eta$-neighborhood of the cylinder over $\gamma_-$, then $\overline{u}_{i}|_{[-R^\varepsilon_{i}, R^\varepsilon_{i}] \times S^1}$ is $\eta$-close in the $C^0$-metric to a trivial cylinder, which contradicts (2).  Hence there is a sequence of points $p_i\in [-R^\varepsilon_{i}, R^\varepsilon_{i}] \times S^1$ such that $\overline{u}_i(p_i)$ limits to a point $q_\infty$ which is a distance $>\eta$ from cylinders over $\gamma_\pm$, after suitable translations of $\overline{u}_i$ in the target.

Next consider the translations $\overline{u}_i^1(s,t)=\overline{u}_i(s-s(p_i),t)$.  The maps $\overline{u}_i^1$ converge in $C^\infty_{loc}$ to a $\overline{J'}$-holomorphic map $\overline{u}^1_\infty : \R \times S^1 \to \overline{W'}$. By construction, $E_{\overline{\omega}}(\overline{u}^1_\infty)>0$. We repeat this process on $\overline{u}_i$ defined on $[-R^\varepsilon_{i}, R^\varepsilon_{i}] \times S^1$ minus a suitable annulus $[K_i',K_i'']\times S^1$ until we obtain $\overline{J'}$-holomorphic maps $\overline{u}_\infty^1, \ldots, \overline{u}_\infty^l$ such that $E_{\overline{\omega}}(\overline{u}_\infty^1) + \ldots + E_{\overline{\omega}}(\overline{u}_\infty^l)=E_{\overline{\omega}}(A)$. An argument similar to that of Case (1) and based on Lemma~\ref{long cylinder lemma v1} shows that the ends of $\overline{u}_\infty', \overline{u}_\infty^1, \ldots, \overline{u}_\infty^l$ match up with those of their adjacent curves.  (Here we may need to reorder the $\overline{u}_\infty^j$ so they are adjacent.)

The case of boundary nodes is completely analogous: the only difference is that we have to replace
the Long Cylinder Lemma by the Long Strip Lemma.
\end{proof}

This completes the proof of Proposition~\ref{prop: SFT compactness for W minus}.
\end{proof}

\begin{rmk}
Proposition~\ref{prop: SFT compactness for W minus} is only a preliminary result.  We are still left with the task of analyzing the limit more precisely.
\end{rmk}

\subsection{Intersection numbers}
\label{subsection: intersection numbers}

In order to analyze the SFT limit $\overline{u}_\infty$, we use the intersection numbers
$n^*(\overline{u}_i)$ and $n^*(\overline{v}_j)$ to constrain the behavior of
holomorphic maps which are asymptotic to $\delta_0$ or to $[0,1] \times \{z_{\infty}\}$.

We briefly recall the definition of the intersection numbers $n^*(\overline{u})$ given in Section \ref{section: moduli spaces of multisections}.
Let $\rho_0>0$ be sufficiently small. Consider the torus $T_{\rho_0}=\{\rho=\rho_0\}
\subset \overline{N}$, oriented as the boundary of $\{\rho\leq \rho_0\}$. We take an
oriented identification $T_{\rho_0}\simeq \R^2/\Z^2$ such that the meridian has slope
$0$ and the closed orbits of $\overline{R}_0$ on $T_{\rho_0}$ have slope $m$.
We pick a closed orbit $\delta_0^{\dagger} \subset T_{\rho_0}$ and consider the parallel
sections $(\sigma_{\infty}^*)^{\dagger}$ determined by $\delta_0^{\dagger}$. We assume
we have chosen $\delta_0^{\dagger}$ so that $(\sigma_{\infty}^*)^{\dagger}$ is disjoint
from the relevant Lagrangian submanifold. Then $n^*(\overline{u})= \langle \overline{u},(\sigma_{\infty}^*)^{\dagger}\rangle$.

Since $n^-(\overline{u}_i)=m\gg 2g$ for an $\overline{W}_-$-curve $\overline{u}_i$ by Lemma~\ref{properties n-}, we have
\begin{equation} \label{eqn: m}
\sum_{j=-b}^a n^*(\overline{v}_j)=m.
\end{equation}

For each level $\overline{v}_j$ of $\overline{u}_\infty$ we have a decomposition
$$\overline{v}_j = \overline{v}_j' \cup \overline{v}_j''$$
as in Equation~\eqref{decomposition of overline u}, where $\overline{v}_j'$ is the
union of the irreducible components of $\overline{v}_j$ which branch cover the section at infinity $\sigma_\infty^*$ and $\overline{v}_j''$ is the union of all other irreducible components.

\begin{lemma}[Intersection with $\overline v_j$, $j>0$] \label{intersezione 1}
Suppose $j>0$. Then the following hold for $\rho_0>0$ sufficiently small:
\begin{enumerate}
\item If $\overline{v}''_j$ has a positive end ${\mathcal E}_+$ which converges to
$\delta_0^{p}$, then
\begin{equation} \label{eqn: m prime}
\langle \mathcal{E}_+, (\sigma_{\infty}')^{\dagger}\rangle \geq p.
\end{equation}
\item If $\overline{v}''_j$ has a negative end ${\mathcal E}_-$ which converges to
$\delta_0^{p}$, then
\begin{equation} \label{eqn: m minus m prime}
\langle \mathcal{E}_-, (\sigma_{\infty}')^{\dagger} \rangle \geq m-p.
\end{equation}
\end{enumerate}
If $\overline{v}''_j$ has multiple ends at covers of $\delta_0$, then their contributions
to $n'(\overline v_j)$ are summed.
\end{lemma}

\begin{proof}
Let ${\mathcal E}_+$ be a positive end which converges to $\delta_0^{p}$. Let $\pi_{\overline N}:\R\times \overline N \to \overline{N}$ be the projection onto the second factor.  Provided $\rho_0$ is sufficiently small, $\pi_{\overline N}(\mathcal{E}_+)\cap T_{\rho_0}$ determines a homology class $(q,p)\in H_1(T_{\rho_0})\simeq \R^2/\Z^2$. One can easily check that
$$\langle \mathcal{E}_+, (\sigma_{\infty}')^{\dagger}\rangle =  \det \left( \begin{matrix} 1 & q \\ m & p \end{matrix} \right )= p-qm.$$
Since $\langle \mathcal{E}_+, (\sigma_{\infty}')^{\dagger}\rangle >0$ by the positivity of
intersections in dimension four and $m\gg p$, we must have $q\leq 0$. We then obtain:
$$\langle \mathcal{E}_+,(\sigma_{\infty}')^{\dagger} \rangle  \ge p.$$

Let ${\mathcal E}_-$ be a negative end which converges to $\delta_0^{p}$.  As above,
$\pi_{\overline N}(\mathcal{E}_-) \cap T_{\rho_0}$ determines a homology class $(q,-p)\in
H_1(T_{\rho_0})\simeq \R^2/\Z^2$ such that
$$\langle \mathcal{E}_-, (\sigma_{\infty}')^{\dagger}\rangle = det \left( \begin{matrix} 1 & q \\ m & - p \end{matrix} \right )= -p-qm.$$
Since $\langle \mathcal{E}_-, (\sigma_{\infty}')^{\dagger}\rangle>0$ by the positivity of
intersections in dimension four and $m\gg p$, we must have $q\leq -1$.  We then obtain:
$$\langle \mathcal{E}_-, (\sigma_{\infty}')^{\dagger}\rangle \geq m-p.$$

Finally, if $\overline{v}''_j$ has multiple ends at covers of $\delta_0$, then the total
intersection of $\overline{v}''_j$ with $(\sigma_{\infty}')^{\dagger}$ is bounded below by
the sum of the contributions of each end.
\end{proof}

Next we consider $\overline{v}_j$ when $j<0$.  Let $\pi_{\overline S}: \R\times[0,1]\times \overline{S}\to \overline{S}$ be the projection along the Hamiltonian vector field $\bdry_t$. Also let ${\mathcal R}_{\phi_0}$ be the radial ray $\{\phi=\phi_0\}\subset D^2_\varepsilon=\{\rho\leq \varepsilon\}$, where $\varepsilon>0$ is small.  Under the projection $\pi_{\overline S}$, each positive end $\mathcal{E}_+$ of $\overline{v}''_j$ which limits to $[0,1]\times \{z_\infty\}$ maps to a sector $\mathfrak{S}(\mathcal{E}_+)$ of $D^2_\varepsilon=\{\rho\leq\varepsilon\}$ going in the counterclockwise direction from ${\mathcal R}_{\phi_1}\subset \overline{a}_{i_1}$ to ${\mathcal R}_{\phi_2}\subset \overline{\hh}(\overline{a}_{i_2})$. The sector $\mathfrak{S}(\mathcal{E}_+)$ is a {\em thin wedge} if $i_1=i_2$ and the angle is ${2\pi\over m}$. Similarly, each negative end $\mathcal{E}_-$ of $\overline{v}''_j$ which limits to $[0,1]\times\{z_\infty\}$ maps to a sector $\mathfrak{S}(\mathcal{E}_-)$ going counterclockwise from ${\mathcal R}_{\phi_1}\subset \overline{\hh}(\overline{a}_{i_1})$ to ${\mathcal R}_{\phi_2}\subset \overline{a}_{i_2}$.

\begin{lemma}[Intersection with $\overline v_j$, $j<0$] \label{intersezione 2}
Suppose $j<0$. Then the following hold for $\rho_0>0$ sufficiently small:
\begin{enumerate}
\item If $\overline{v}''_j$ has a positive end ${\mathcal E}_+$ which converges to
$[0,1]\times\{z_\infty\}$, then
\begin{equation} \label{two}
\langle \mathcal{E}_+, \sigma_{\infty}^{\dagger} \rangle \geq 1,
\end{equation}
and the relevant sector is a thin wedge if and only if equality holds. Moreover, if the
sector is not a thin wedge, then
$$\langle \mathcal{E}_+, \sigma_{\infty}^{\dagger} \rangle > 2g.$$
\item If $\overline{v}''_j$ has a negative end ${\mathcal E}_-$ which converges to
$[0,1]\times\{z_\infty\}$, then
\begin{equation} \label{one}
\langle \mathcal{E}_-, \sigma_{\infty}^{\dagger} \rangle >2g.
\end{equation}
\end{enumerate}
\end{lemma}

\begin{proof}
Suppose $\rho_0<\varepsilon$. The restriction of $\delta_{\rho_0}$ to $[0,1]\times \overline{S}$ consists of $m$ Reeb arcs $[0,1]\times \{\phi_l\}$ on $\{\rho=\rho_0\}$, where $0\leq l<m$ and $\phi_l= \phi_0+l(2\pi/m)$. Let $\mathcal{E}_+$ (resp.\ $\mathcal{E}_-$) be a positive (resp.\ negative) end of $\overline{v}''_j$ which converges to $[0,1]\times\{z_\infty\}$ and let $\mathfrak{S}(\mathcal{E}_\pm)$ be the corresponding sector in $D^2_\varepsilon$.

By the assumptions on the $\overline{a}_i$ given in Section~\ref{coconut}, the angles of the thin wedges are ${2\pi\over m}$ and the other sectors of $D^2_\varepsilon-\cup_i \overline{a}_i-\cup_i \overline{\hh}(\overline{a}_i)$ have angles greater than ${2\pi (2g)\over m}$. This implies that:
\begin{itemize}
\item[(i)]  $\langle\mathcal{E}_+,\sigma_{\infty}^{\dagger}\rangle \geq 1$;
\item[(ii)] $\langle \mathcal{E}_+, \sigma_{\infty}^{\dagger} \rangle=1$ if and only if $\mathfrak{S}(\mathcal{E}_+)$ is a thin wedge; and
\item[(iii)] $\langle\mathcal{E}_-, \sigma_{\infty}^{\dagger} \rangle >2g$.
\end{itemize}
The lemma follows.
\end{proof}

Similarly, we have the following lemma, whose proof is the same as those of
Lemmas~\ref{intersezione 1} and \ref{intersezione 2} and will be omitted.

\begin{lemma}[Intersection with $\overline v_0$] \label{intersezione 3}
The following hold for $\rho_0>0$ sufficiently small:
\begin{enumerate}
\item If $\overline{v}_0''$ has a positive end ${\mathcal E}_+$ which converges to
$\delta^p$, then
\begin{equation} \label{four}
\langle \mathcal{E}_+, (\sigma^-_{\infty})^{\dagger} \rangle \geq p.
\end{equation}
\item If $\overline{v}_0''$ has a negative end ${\mathcal E}_-$ which converges to
$[0,1]\times\{z_\infty\}$, then
\begin{equation} \label{three}
\langle \mathcal{E}_-, (\sigma^-_{\infty})^{\dagger} \rangle >2g.
\end{equation}
\end{enumerate}
\end{lemma}

Consider $\overline{v}_j'':\dot F_j''\to \overline{W}_j$ for $j\leq 0$.
A point $p\in\bdry F''_j$ will be called a {\em boundary point at  $z_\infty$} if
$\overline{v}_j''(p) \in L_{\overline{\bf a}^-} - L_{\hat{\bf a}^-}$ (if $j=0$), or
$(L_{\overline{\hh}(\overline{\bf a})} - L_{\overline{\hh}(\hat{\bf a})}) \cup
(L_{\overline{\bf a}} - L_{\hat{\bf a}}) = \R \times \{0,1 \} \times \{z_\infty \}$ (if $j<0$).

\begin{lemma} \label{lemma: cherimoya2}
Consider $\overline{v}_j'':\dot F_j''\to \overline{W}_j$, $j\leq 0$. If $p\in\bdry F''_j$ is a  boundary point at $z_\infty$, then, for any sufficiently small neighborhood $\nu(p)\subset F_j''$ of $p$, there exists $\rho_0>0$ such that
$\langle \overline{v}_j''(\nu(p)), (\sigma_{\infty}^*)^{\dagger} \rangle \geq k_0-1 > 2g$.
\end{lemma}

Here the constant $k_0$ is as given in Section~\ref{coconut}.

\begin{proof}
We will prove the case $j<0$, leaving $j=0$ for the reader. Let $p$ be a boundary point at $z_\infty$ of $\overline{v}''_j$ which maps to a point on $\R\times \{1\}\times \{z_\infty\}$ (without loss of generality). We consider the projection $\pi_{\overline{S}}: \R\times [0,1] \times \overline{S}\to \overline{S}$. By Definition \ref{defn: almost complex structures on overline W}, the projection $\pi_{\overline{S}}$ is holomorphic when restricted to $\pi^{-1}_{\overline{S}}(D^2_\varepsilon)$, where $\varepsilon>0$ is sufficiently small. Hence $\pi_{\overline{S}}\circ \overline{v}''_j$ is holomorphic when restricted to $(\pi_{\overline{S}}\circ\overline{v}''_j)^{-1}(D^2_\varepsilon)$ and some sector of $D^2_\varepsilon-\overline{\bf a}$ must be contained in $\op{Im}(\pi_{\overline{S}}\circ \overline{v}''_j)$ because holomorphic maps are open. This in turn shows that $$\langle\overline{v}''_j(\nu(p)),  (\sigma_{\infty}^*)^{\dagger}\rangle\geq k_0-1 > 2g$$ by assumption, provided $\rho_0$ is sufficiently small.
\end{proof}

\subsection{Some restrictions on $\overline{u}_\infty$}
\label{some restrictions, first version}

We now present some lemmas in preparation for Theorems~\ref{thm: compactness for W minus I=3, first version} and \ref{thm: compactness for W minus I=2, first version}. Lemmas~\ref{no closed components}--\ref{lemma: preliminary restrictions part 2} give restrictions on $\overline{u}_\infty$, which arise from intersection number calculations from Section~\ref{subsection: intersection numbers}, and Lemma~\ref{claim in proof} gives a lower bound on the ECH index of the levels $\overline{v}_j$.

In what follows, ``component'' is shorthand for ``irreducible component''. We first discuss the ``fiber components'', i.e., components $\widetilde{v}: F\to \overline{W}_j$ of $\overline{v}_j$ which map to fibers of $\overline{W}_j$. There are three types of fiber components:
\begin{itemize}
\item[(i)] ghosts, i.e., $\widetilde{v}$ is constant;
\item[(ii)] ``closed fiber components'', i.e., $F$ is closed and $\widetilde{v}$ is nonconstant;
\item[(iii)] ``boundary fiber components'', i.e., $\bdry F\not=\varnothing$ and $\widetilde{v}$ is nonconstant.
\end{itemize}
If $\widetilde{v}$ is a closed fiber component, then $\widetilde{v}$ is a branched cover of a fiber of $\overline{W}_j$. Next let $\widetilde{v}:F\to \overline{W}_j$ be a boundary fiber component. If $j<0$ and $\widetilde{v}$ maps to $\overline{\pi}^{-1}_{B}(p)$, $p\in \bdry B$, then $\widetilde{v}(F)\supset \overline{\pi}^{-1}_B(p)-A$, where $A= L_{\overline{\bf a}}$ or $L_{\overline{\hh}(\overline{\bf a})}$. Similarly, if $j=0$ and $\widetilde{v}$ maps to $\overline{\pi}_{B_-}^{-1}(p)$, $p\in \bdry B_-$, then $\widetilde{v}(F)\supset \overline{\pi}^{-1}_{B_-}(p)-L^-_{\overline{\bf a}}$.

\begin{lemma}\label{no closed components}
The only possible non-ghost fiber component of $\overline u_\infty$ is a closed component $\overline\pi_{B_-}^{-1}(p)$ which passes through $\overline{\frak m}$. There is at most one such component.
\end{lemma}

\begin{proof}
Let $\widetilde{v}$ be a non-ghost fiber component.
Then $\widetilde{v}$ is a $d$-fold branched cover of a fiber if it is a closed fiber
component and a $d$-fold branched cover of a fiber cut along $\overline{\bf a}$
(or $\overline{\hh}(\overline{\bf a})$) if it is a boundary fiber component.
In either case, $n^*(\widetilde{v}) = d\cdot m$. This implies that $d=1$.

We now prove that $\widetilde{v}$ is a closed fiber component passing through $\overline{\frak m}$.
Arguing by contradiction, if $\overline{\frak m} \not \in \op{Im}(\widetilde{v})$, then
there is a component $\hat{v}$ of $\overline{u}_{\infty}$ such that
$\overline{\frak m} \in \op{Im}(\hat{v})$.
If $\hat{v}$ is a cover of $\sigma^-_{\infty}$, then there must be some component of
$\overline{v}''_j$ for $j>0$, which has a negative end at some $\delta_0^p$. Therefore
$n^-(\overline{v}''_j) >0$ by Lemma~\ref{intersezione 1}(2). If $\hat{v}$ is not a
cover of $\sigma^-_{\infty}$, then $\hat{v}$ has a nonzero intersection with
$\sigma_\infty^-$ and hence $n^-(\hat{v})>0$.  In either case we have
$n^*(\overline{u}_{\infty})> m$, which is a contradiction. This proves that
$\overline{\frak m} \in \op{Im}(\widetilde{v})$.

Finally boundary fiber components are eliminated because they project to a point in
$\partial B_-$, and therefore cannot pass through $\overline{\mathfrak{m}}$.
\end{proof}

We recall the notation $\overline{v}_j = \overline{v}_j' \cup \overline{v}_j''$, where
$\overline{v}_j'$ denotes the union of all components of $\overline{v}_j$ which cover
the section at infinity $\sigma_{\infty}^*$ and $\overline{v}_j''$ denotes the union of all
other components of $\overline{v}_j$. The covering degree of  $\overline{v}_j'$ will
always be denoted by $p_j$. We also define $\overline{v}_j^\sharp$ to be the
union of the components of $\overline{v}_j''$ which are asymptotic to a multiple of
$\delta_0$ or $z_\infty$ at either end and are not trivial cylinders or strips, $\overline{v}_j^\flat$ to be the union of the
remaining non-fiber components of $\overline{v}_j''$, and $\overline{v}_j^f$ to be the union of the fiber components of $\overline{v}_j''$.
\nom[v]{$\overline{v}_j^\sharp$}{Union of irreducible components of $\overline{v}_j''$ asymptotic to a multiple of $\delta_0$ or $z_\infty$ at either end}
\nom[v]{$\overline{v}_j^\flat$}{Union of non-fiber components of $\overline{v}_j''$ not asymptotic to any multiple of $\delta_0$ or $z_\infty$}

\begin{lemma}
\label{lemma: preliminary restrictions part 1}
If $\overline{v}'_0=\varnothing$, then, with the exception of ghost components:
\begin{enumerate}
\item $\overline{v}'_j=\varnothing$ and $\overline{v}_j^\sharp=\varnothing$ for all $j$;
\item no $\overline{v}_j''$, $j\leq 0$, has a boundary point at $z_\infty$;
\item every level $\overline v_j$, $j\not=0$, has image inside $W'$ or $W$; and
\item $\overline v_0$ is a $\overline{W}_-$-curve or a {\em degenerate $\overline{W}_-$-curve}, i.e., $\overline{v}_0$ is the union of a $W_-$-curve and a fiber $\overline\pi_{B_-}^{-1}(p)$ which passes through $\overline{\frak m}$.
\end{enumerate}
\end{lemma}

\begin{proof}
If $\overline{v}'_0=\varnothing$, then $\overline v_0=\overline{v}_0''$. Since $\overline{u}_i$ passes through $\overline{\frak m}$ for all $i$, the level $\overline{v}_0:\dot F_0\to \overline{W}_-$ must also pass through $\overline{\frak m}$. By Lemma~\ref{properties n-} and the proof of Lemma~\ref{properties n+}, $n^-(\overline v_0) \ge m$ because $\overline{\frak m} \in \sigma_{\infty}^-$. Equation~\eqref{eqn: m} and the nonnegativity of $n^*$ then imply that $n^-(\overline v_0) = m$ and $n^*(\overline{v}_j) =0$ for $j \neq 0$.

(1) If $\overline{v}'_j \not=\varnothing$ or $\overline{v}_j^\sharp\not=\varnothing$, then at least one of Equations~\eqref{eqn: m minus m prime}--\eqref{four} applies and $\sum_{j\not=0} n^*(\overline{v}_j^\sharp)> 0$.

(2) This follows from $n^-(\overline v_0) = m$ and Lemma~\ref{lemma: cherimoya2}.

(3) Since $n^*(\overline v_j)=0$ for $j\not=0$, (1), combined with Lemma~\ref{pluto}, implies that $\op{Im}(\overline{v}_j)\subset W'$ if $j>0$, whereas (1) and (2), combined with Lemma~\ref{pippo}, implies that $\op{Im}(\overline{v}_j)\subset W$ if $j<0$.

(4) Since $n^-(\overline v_0)=m$, $\overline{v}_0$ intersects $\sigma_\infty^-$ only at $\overline{\frak m}$ and the intersection is transverse by Lemma \ref{properties n-}(2).   By (1) and (2) and Lemmas~\ref{no closed components} and \ref{lemma: restriction of overline W minus curve}, if $\overline{v}_0$ is not a $\overline{W}_-$-curve, then it must be a degenerate $\overline{W}_-$-curve.

This completes the proof of the lemma.
\end{proof}

\begin{lemma}
\label{lemma: preliminary restrictions part 2}
If $\overline{v}'_0\not=\varnothing$, then, with the exception of ghost components:
\begin{enumerate}
\item there is only one negative end of $\overline{v}_j^\sharp$, $j>0$, (say $\overline{v}_{a'}^\sharp$) which is asymptotic to a multiple of $\delta_0$;
\item $\overline{u}_\infty$ has no fiber components;
\item no $\overline{v}_j''$, $j\leq 0$, has a boundary point at $z_\infty$;
\item $\overline v_j^\flat$ has image inside $W$, $W_-$, or $W'$;
\item $\overline v_j^\sharp$, $j<0$, is a union of thin strips from $z_\infty$ to some $x_i$ or $x_i'$;
\item $\overline{v}_0^\sharp$ has image inside $\overline{W}_--int(W_-)$ and has $\delta_0^{r_0}$, for some $r_0>0$, at the positive end and some subset of $\{x_1,\dots,x_{2g},x'_1,\dots,x'_{2g}\}$ of cardinality $r_0$ at the negative end;
\item $\overline{v}_j^\sharp$, $0<j\leq a'$, has image inside $\overline{W'}-int(W')$ and has $\delta_0^{r_j}$, for some $r_j>0$, at the positive end and $h^{r_j'}e^{r_j''}$ with
 $r_j'+r_j''=r_j$, at the negative end.
\end{enumerate}
\end{lemma}

\begin{proof}
(1) By Lemma~\ref{intersezione 1}, each negative end of $\overline{v}_j^\sharp$, for
$j>0$, which is asymptotic to a multiple of $\delta_0$ contributes at least $m-2g$ to
$\sum_{j=1}^a n^*(\overline{v}_j)$, where $m\gg 2g$, because the total multiplicity of
$\delta_0$ is $\leq 2g$.  If there are at least two such negative ends of
$\overline{v}_j^\sharp$, then the total contribution to $\sum_{j=1}^a
n^*(\overline{v}_j)$ is at least $2(m-2g)>m$, which is a contradiction.  Suppose $\overline{v}_{a'}^\sharp$, $0<a'\leq a$, has a negative end at a multiple at $\delta_0$.

(2) By Lemma~\ref{no closed components}, a fiber component of $\overline{u}_\infty$ is a fiber $\overline{\pi}^{-1}_{B_-}(p)$ which passes through $\overline{\frak m}$ and additionally contributes $m$ to $\sum_{j=-b}^a n^*(\overline{v}_j)$, a contradiction.

(3) By Lemma~\ref{lemma: cherimoya2}, a boundary point at $z_\infty$ additionally contributes $\geq 2g$ to $\sum_{j=-b}^a n^*(\overline{v}_j)$. This
is again a contradiction.

(4) By (3), $\overline{v}_j^\flat$ does not have any boundary point at $z_\infty$.  Since the ends of $\overline{v}_j^\flat$ are contained in $N$ or $[0,1]\times S$ and $n^*(\overline{v}_j^\flat)=0$, we conclude that $\op{Im}(\overline{v}_j^\flat)\subset W$, $W_-$, or $W$ by Lemmas~\ref{pippo}, \ref{lemma: restriction of overline W minus curve} and \ref{pluto}.

(5), (6) Let $p_j= \deg(\overline{v}_j')$ be the covering degree of $\overline{v}_j'$ over $\sigma_{\infty}'$. Then $p_0\leq p_1\leq \dots \leq p_{a'-1}$ and $p_{a'}=0$.  This follows from (1) since a negative end of $\overline{v}_{j+1}^\sharp$ at (a multiple of) $\delta_0$ is required for $p_j> p_{j+1}$.

Let $p_{a'}^-$ be the multiplicity of $\delta_0$ at the negative end of $\overline{v}_{a'}^\sharp$. The negative end of $\overline{v}_{a'}^\sharp$ contributes at least $m-p_{a'}^-$ to $n^*$ by Equation~\eqref{eqn: m minus m prime} and the positive ends of $\overline{v}_j^\sharp$, for $j=0,\dots,a'-1$,  give a total contribution of at least $p_{a'}^--p_0$ to $n^*$ by Equations~\eqref{eqn: m prime} and \eqref{four}. Hence,
\begin{equation} \label{two equations}
\sum_{j=0}^a n^*(\overline{v}_j^\sharp) \geq m-p_0.
\end{equation}
On the other hand, by Equation~\eqref{two}, the contributions of the positive ends of $\overline{v}_j^\sharp$, $j=-b,\dots,-1$, add up to
\begin{equation}\label{trallalla}
\sum_{j=-b}^{-1}n^*(\overline{v}_j^\sharp) \geq p_0.
\end{equation}

Equations~\eqref{two equations} and \eqref{trallalla} give:
\begin{equation} \label{eqn: sum of intersections greater than m}
\sum_{j=-b}^{a} n^*(\overline{v}_j^\sharp) \geq m.
\end{equation}
Equality holds by Equation~\eqref{eqn: m}. This in turn implies that:
\begin{enumerate}
\item[(i)] equality holds in both Equations~\eqref{two equations} and \eqref{trallalla}; and
\item[(ii)] $\overline{v}_j^\sharp$, $j\leq 0$, has no negative end which limits to $z_\infty$.
\end{enumerate}
Since $p_0 \le 2g$, each $\overline{v}_j^\sharp$, $j<0$, must be a union of thin strips by Lemma~\ref{intersezione 2}. This gives (5). Moreover, $\overline{v}_0^\sharp$ has $\delta_0^{p_1-p_0}$ at the positive end and no negative ends at $z_\infty$ by (ii) and
\begin{equation} \label{nutria}
n^-(\overline{v}_0^\sharp) = p_1 - p_0
\end{equation}
by Lemma~\ref{intersezione 3} and (i).

The argument to prove (6) is similar to the proof of the blocking lemma in \cite{CGH2}, and the presence of the Lagrangian boundary condition for $\overline{v}_0^\sharp$ does not change the proof in any essential way. We consider $C_{\rho}=\pi_{\overline{N}}(\overline{v}_0^\sharp)\cap T_{\rho}$, where $0<\rho<1$, as an element of $H_1(T_{\rho}, \pi_{\overline{N}}(L^-_{\overline{\bf a}})\cap T_{\rho})$; we are viewing $C_{\rho}$ as the boundary of the portion of $\pi_{\overline{N}}(\overline{v}_0^\sharp)$ outside the solid torus of radius $\rho$. We recall that $\pi_{\overline{N}}(L^-_{\overline{\bf a}})\cap T_{\rho}$ consists of $2g$ parallel segments in $T_{\rho}$ which are tangent to the Hamiltonian flow.  Suppose $\rho_1>0$ is small. Let $\cup_i\mathcal{E}_i$ be the union of positive ends of $\overline{v}_0^\sharp$ that limit to multiples of $\delta_0$, and let $C_{\rho_1}'=\pi_{\overline{N}}(\cup_i\mathcal{E}_i)\cap T_{\rho_1}$.  We view $C'_{\rho_1}$ as an element of
$$H_1(T_{\rho_1}) \hookrightarrow  H_1(T_{\rho_1}, \pi_{\overline{N}}(L^-_{\overline{\bf a}})\cap T_{\rho_1}).$$
Using ``balanced coordinates'' on $V =\overline{N}-int(N)$ (cf.\ Section~\ref{subsubsection: overline W pm}), we obtain an identification $H_1(T_{\rho_1}) \cong \Z^2$ such that --- writing vectors as rows --- $(1, 0)$ corresponds to the homology class of the meridian (i.e., the closed curve which bounds a disk in $V$) and $(1,m)$ corresponds to the class of a closed Hamiltonian orbit.

With respect to this identification, $C'_{\rho_1}=(q, p_1-p_0)$ for some $q\leq 0$ and we have
$$p_1-p_0 -qm = \det \left ( \begin{matrix} 1 & q \\ m & p_1-p_0 \end{matrix} \right )  \ge p_1-p_0  \ge 0.$$
By Equation~\eqref{nutria}, we obtain $q=0$ and $C_{\rho_1}=C_{\rho_1}'$, since any intersections besides those coming from $\cup_i\mathcal{E}_i$ would give some extra contributions to $n^-(\overline{v}_0^\sharp)$.

Next consider $C_{\rho_2}$, where $\rho_2=1-\varepsilon$, $\varepsilon>0$ small. Since there are no ends of $\overline{v}_0^\sharp$ between $T_{\rho_1}$ and $T_{\rho_2}$, it follows that  $C_{\rho_2}$ is the image of $(0, p_1 - p_0) \in H_1(T_{\rho_2})$ in $H_1(T_{\rho_2},\pi_{\overline{N}}(L^-_{\overline{\bf a}})\cap T_{\rho_2})$. Negative ends can approach $x_i$ or $x_i'$ either from the interior of $V$ or from the exterior of $V$. A negative end of $\overline{v}_0^\sharp$ approaching $x_i$ or $x_i'$ from the interior traces a segment in $T_{\rho_2}$ with both endpoints in the same connected component of $\pi_{\overline{N}}(L^-_{\overline{\bf a}})\cap T_{\rho_2}$ and whose relative homology class is the image of the class  $(0,1) \in H_1(T_{\rho_2})$.
This means that  at most $p_1-p_0$ ends of $\overline{v}_0^\sharp$ approach $x_i$ or $x_i'$ from the interior.

Finally let $\rho_3=1+\varepsilon$, $\varepsilon>0$ small. The ends of $\overline{v}_0^\sharp$ between $T_{\rho_2}$ and $T_{\rho_3}$ are negative ends at $x_i$, $x_i'$ or positive ends at closed orbits on $T_1$. 
Let $C'_{\rho_3}=\pi_{\overline{N}}(\widetilde{v})\cap T_{\rho_3}$, where $\widetilde{v}$ is obtained from $\overline{v}_0^\sharp$ by truncating the negative ends that limit to $x_i$ or $x_i'$ from the exterior of $V$. There are no positive ends of $\overline{v}_0^\sharp$ that limit to a closed orbit on $T_1$; this follows from asymptotic winding considerations of the positive ends and a comparison with $C_{\rho_2}$. Since there are no negative ends of $\overline{v}_0^\sharp$ that limit to a closed orbit on $T_1$, it follows that $[C'_{\rho_3}]=0 \in H_1(T_{\rho_3},\pi_{\overline{N}}(L^-_{\overline{\bf a}})\cap T_{\rho_3})$.

Finally, we eliminate the possibility of ends which limit to $x_i$ or $x_i'$ from the outside of $V$ by observing that nothing else can contribute to $C_{\rho_2}$ but the ends at $x_i$ or $x_i'$. This implies that there are $p_1-p_0$ of them, each representing the image of the class $(0,1)\in  H_1(T_{\rho_2})$, and leaving no room for ends limiting to $x_i$ or $x_i'$ from the outside. Then $C_{\rho_3}$ is homologically trivial, and therefore the positivity of intersections implies that it is empty. Since $\rho_3$ was arbitrarily close to $1$, it follows that the image of $\overline{v}_0^\sharp$ cannot escape $\overline{W}_- - int(W_-)$, which implies (6).

(7) This is similar to (6) and is left to the reader.
\end{proof}
\begin{rmk} \label{calcolo di alcuni indici}
One can easily compute that
$I(\overline{v}_0^\sharp)=r_0$ and  $I(\overline{v}_j^\sharp)= r_j' + 2 r_j''$ when $j>0$.
\end{rmk}

\begin{lemma}\label{claim in proof}
The ECH index of each level $\overline{v}_j$, $j=-b,\dots,a$, is nonnegative if $\overline{J}_-\in \mathcal{J}_{\overline{W}_-}^{reg}$. Moreover, if $j \ne 0$,
the ECH index of $\overline{v}_j$ is strictly positive.
\end{lemma}

\begin{proof}
Recall that the restrictions $\overline{J}$ and $\overline{J'}$ are also regular by definition.
\s \n {\em Case $j>0$.} By \cite[Proposition~7.15]{HT1}, $I_{ECH}(\overline{v}_j)\geq 1$ for $j>0$ since $\overline{J'}$ is regular and no level $\overline{v}_j$ with $j >0$ consists uniquely of trivial cylinders.

\s \n {\em Case $j<0$.} We write $\overline{v}_j=\overline{v}_j'\cup \overline{v}_j^\sharp\cup\overline{v}_j^\flat$.

Suppose that $\overline{v}_0'=\varnothing$.  Then $\overline{v}_j'=\varnothing$ and $\overline{v}_j^\sharp=\varnothing$ for all $j<0$ by Lemma~\ref{lemma: preliminary restrictions part 1}, and we are left with $\overline{v}_j^\flat$, which is simply-covered and has  positive Fredholm index by regularity and translation invariance. By the index inequality (Theorem~\ref{thm: index inequality for HF}), $I(\overline{v}_j)=I(\overline{v}_j^\flat)\geq 1$.

Next suppose that $\overline{v}_0'\not=\varnothing$. Then $\overline{v}_j^\sharp$ is a union of thin strips from $z_\infty$ to some $x_i$ or $x_i'$ by Lemma~\ref{lemma: preliminary restrictions part 2} and each thin strip has ECH index $1$. We also have $I(\overline{v}_j')=0$ by Lemma~\ref{lemma: HF index of sections at infinity} and $I(\overline{v}_j^\flat)\geq 0$ by the previous paragraph.

We claim that
\begin{equation} \label{sum of three}
I(\overline{v}_j)=I(\overline{v}_j'\cup\overline{v}_j^\sharp\cup\overline{v}_j^\flat)=I(\overline{v}_j')+I(\overline{v}_j^\sharp)+I(\overline{v}_j^\flat).
\end{equation}
Note that, although $\overline{v}_j'$, $\overline{v}_j^\sharp$ and $\overline{v}_j^\flat$ are disjoint, $\overline{v}_j'$ and $\overline{v}_j^\sharp$ are both asymptotic to (a multiple of) $z_\infty$ at the positive end and the additivity of the ECH indices of $\overline{v}_j'$ and $\overline{v}_j^\sharp$ needs to be verified.  For that purpose, recall from Section~\ref{subsection: modified indices at z infty} that each $\overline{v}_j'$ comes equipped with data $\mathcal{D}'_j=((\mathcal{D}')^{to}_j,(\mathcal{D}')^{from}_j)$ at the positive and negative ends, since $\overline{u}_\infty$ is the limit of the sequence $\{\overline{u}_i\}$. The key observation here is that $(\mathcal{D}')^{to}_j=(\mathcal{D}')^{from}_j$, since all the components of $\overline{v}_{j'}^\sharp$, $j'<j$, are thin strips whose data $(\mathcal{D}^\sharp_+)_{j'}=((\mathcal{D}^\sharp_+)^{to}_{j'},(\mathcal{D}^\sharp_+)^{from}_{j'})$ at the positive end satisfies $(\mathcal{D}^\sharp_+)^{to}_{j'}=(\mathcal{D}^\sharp_+)^{from}_{j'}$. Hence we can choose a simultaneous grooming ${\frak c}^+=\{c_k^+\}$ for both $\overline{v}'_j$ and $\overline{v}^\sharp_j$ at the positive end $z_\infty$ such that $c_k^+$ has winding number $w(c_k^+)=0$ (see Equation~\eqref{eq: winding number}) and connects $\overline{\hh}(\overline{a}_{i_k,j_k})$ to $\overline{a}_{i_k,j_k}$. If we choose a groomed multivalued trivialization $\tau$ compatible with ${\frak c}^+$, then
$$I_\tau(\overline{v}_j'\cup \overline{v}_j^\sharp)= I_\tau(\overline{v}_j') + I_\tau(\overline{v}_j^\sharp),$$
which immediately implies Equation~\eqref{sum of three}.

\s\n {\em Case $j=0$.} Suppose that $\overline{v}_0'=\varnothing$. Then, $\overline{v}_0$ is a $\overline{W}_-$-curve or a degenerate $\overline{W}_-$-curve by Lemma~\ref{lemma: preliminary restrictions part 1}. If $\overline{v}_0$ is a $\overline{W}_-$-curve, then it is simply-covered and satisfies $I(\overline{v}_0)\geq \op{ind}(\overline{v}_0)\geq 0$. If $\overline{v}_0$ is a degenerate $\overline{W}_-$-curve, then $\overline{v}_0$ can be written as a union of a fiber $C$ and a $W_-$-curve $\overline{v}_0^\flat$. The Fredholm index of $\overline{v}_0^\flat$ is nonnegative since $\overline{v}_0^\flat$ is simply-covered and hence is regular. The Fredholm index of $C$ is given by:
\begin{align*}
\op{ind}(C) &= -\chi(C)+ 2\langle c_1(TW_-),C\rangle\\
&=(2g-2)+2(2-2g)=2-2g.
\end{align*}
The algebraic intersection number $\langle C,\overline{v}_0^\flat\rangle$ is equal to $2g$ and
\begin{align} \label{persimmon}
I(\overline{v}_0)&\geq \op{ind}(C) +\op{ind}(\overline{v}_0^\flat)+2\langle C,\overline{v}_0^\flat\rangle\\
\notag &\geq (2-2g)+0 + 2(2g)=2g+2,
\end{align}
by Theorem~\ref{thm: index inequality for W+ and W-}.

Next suppose that $\overline{v}_0'\not=\varnothing$. We have $I(\overline{v}_0')=0$ by Lemma~\ref{lemma: ECH index of sections at infinity}. Next, by Lemma~\ref{lemma: preliminary restrictions part 2}, if $\overline{v}_0^\sharp\not=\varnothing$, then $\op{Im}(\overline{v}_0^\sharp)\subset \overline{W}_--int(W_-)$ and has $\delta_0^{p_1-p_0}$, $p_1-p_0>0$, at the positive end and some $(p_1-p_0)$-element subset of $\{x_1,\dots,x_{2g},x_1',\dots,x_{2g}'\}$ at the negative end. Hence $I(\overline{v}_0^\sharp)=p_1-p_0$ by Lemma~\ref{lemma: ECH index of thin wedges}.  Finally, since $\op{Im}(\overline{v}_0^\flat)\subset W_-$ by Lemma~\ref{lemma: preliminary restrictions part 2}, $\overline{v}_0^\flat$ is simply-covered and $I(\overline{v}_0^\flat)\geq 0$. The ECH indices of $\overline{v}_0'$, $\overline{v}_0^\sharp$, and $\overline{v}_0^\flat$ are additive.
\s
This completes the proof of the lemma.
\end{proof}

\subsection{Compactness theorem}
\label{subsection: compactness theorem, first version}

Let $\overline{J}_-^\Diamond(\varepsilon,\delta,p)$ be a generic almost complex structure which is $(\varepsilon,U)$-close to $\overline{J}_-\in \mathcal{J}_{\overline{W}_-}^{reg}$, where $U$ and $K_{p,2\delta}$ are as in Convention~\ref{convention}. We write:
\begin{equation}\label{strange notation}
\mathcal{M}^i_{\overline{\frak m}}(\varepsilon,\delta,p) := \mathcal{M}^{I=i,n^*=m}_{\overline{J}_-^\Diamond(\varepsilon, \delta,p)}(\boldsymbol{\gamma}, {\bf y}; \overline{\frak m}).
\end{equation}
As before, $\overline{u}_\infty \in \bdry \mathcal{M}_{\overline{\frak m}}^3(\varepsilon,\delta,p)$ is written as $\overline{u}_\infty=\overline{v}_{-b}\cup\dots\cup \overline{v}_a$ and each $\overline{v}_j$ is written as $\overline{v}_j=\overline{v}_j'\cup\overline{v}_j''=\overline{v}'_j\cup\overline{v}_j^\sharp\cup \overline{v}_j^\flat\cup\overline{v}_j^f$, where $\overline{v}_j^f=\varnothing$ for $j\not=0$. There are no ghost components by an index argument as in Lemma~\ref{lemma: no ghosts}.

We now prove the following compactness theorem, which is basically a consequence of two constraints: $I(\overline{u}_i)=3$ and $n^*(\overline{u}_i)=m$. The list of possibilities should be viewed as a preliminary list, since we will subsequently eliminate Cases (2)--(6) in Theorem~\ref{thm: complement}.

\begin{thm} \label{thm: compactness for W minus I=3, first version}
Let $\overline{J}_-^\Diamond$ be $(\varepsilon,U)$-close to $\overline{J}_-\in \mathcal{J}_{\overline{W}_-}^{reg}$ and $K_{p,2\delta}$-regular with respect to $\overline{\frak m}$, and let $\overline{u}_\infty \in \bdry \mathcal{M}_{\overline{\frak m}}^3(\varepsilon,\delta,p)$. If $\overline{v}_0'\not=\varnothing$, then $\overline{u}_\infty$ contains one of the following subbuildings:
\begin{enumerate}
\item A $3$-level building consisting of $\overline{v}_1$ with $I=2$ from $\boldsymbol{\gamma}$ to $\delta_0\boldsymbol{\gamma}'$; $\overline{v}_0'=\sigma_\infty^-$; and $\overline{v}_{-1}^\sharp$ which is a thin strip from $z_\infty$ to $x_i$ or $x_i'$.
\item A $3$-level building consisting of $\overline{v}_1$ with $I=1$ from some $\boldsymbol{\gamma}''$ to $\delta_0\boldsymbol{\gamma}'$; $\overline{v}_0'=\sigma_\infty^-$; and $\overline{v}_{-1}^\sharp$ which is a thin strip.
\item A $4$-level building consisting of $\overline{v}_1$ with $I=1$ from $\boldsymbol{\gamma}$ to $\delta^2_0\boldsymbol{\gamma}'$; $\overline{v}_0'$ which is a branched double cover of $\sigma_\infty^-$; $\overline{v}_{-1}'=\sigma_\infty$; and $\overline{v}_{-1}^\sharp$ and $\overline{v}_{-2}^\sharp$ which are both thin strips.
\item A $3$-level building consisting of $\overline{v}_1$ with $I=1$ from $\boldsymbol{\gamma}$ to $\delta^2_0\boldsymbol{\gamma}'$; $\overline{v}_0'$ which is a branched double cover of $\sigma_\infty^-$; and $\overline{v}_{-1}^\sharp$ which is the union of two thin strips.
\item A $3$-level building consisting of $\overline{v}_1$ with $I=1$ from $\boldsymbol{\gamma}$ to $\delta^2_0\boldsymbol{\gamma}'$; $\overline{v}_0'=\sigma_\infty^-$; $\overline{v}_0^\sharp$ with $I=1$ from $\delta_0$ to $x_i$ or $x_i'$; and $\overline{v}_{-1}^\sharp$ which is a thin strip.
\item A $4$-level building consisting of $\overline{v}_2$ with $I=1$ from $\boldsymbol{\gamma}$ to $\delta^2_0\boldsymbol{\gamma}'$; $\overline{v}_1'=\sigma'_\infty$; $\overline{v}_1^\sharp$ with $I=1$ which is a cylinder from $\delta_0$ to $h$; $\overline{v}_0'=\sigma_\infty^-$; and $\overline{v}_{-1}^\sharp$ which is a thin strip.
\end{enumerate}
Here we are omitting levels which are connectors.
\end{thm}

See Figure~\ref{figure: graphs1bis}.

\begin{figure}[ht]
\begin{center}
\psfragscanon
\psfrag{A}{\tiny $\R\times\overline{N}$}
\psfrag{B}{\tiny $\overline{W}_-$}
\psfrag{C}{\tiny $\overline{W}$}
\psfrag{1}{\tiny $1$}
\psfrag{2}{\tiny $2$}
\psfrag{0}{\tiny $0$}
\psfrag{a}{\small (1)}
\psfrag{b}{\small (2)}
\psfrag{c}{\small (3)}
\psfrag{d}{\small (4)}
\psfrag{e}{\small (5)}
\psfrag{f}{\small (6)}
\includegraphics[width=11cm]{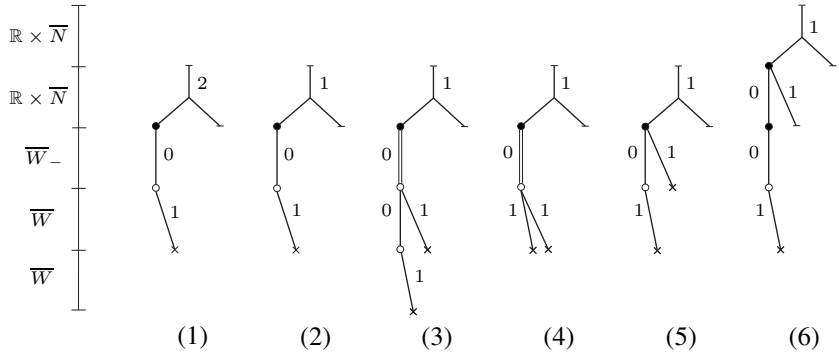}
\end{center}
\vskip.2in
\caption{Schematic diagrams for the possible types of degenerations. Here $\bullet$ represents $\delta_0$, $\circ$ represents $z_\infty$, and $\times$ represents some $x_i$ or $x_i'$. A vertical line indicates a trivial cylinder or a restriction of a trivial cylinder, and a double vertical line indicates a branched double cover of a trivial cylinder or a restriction of a trivial cylinder. The labels on the graphs are ECH indices of each component.} \label{figure: graphs1bis}
\end{figure}

\begin{proof}
By Proposition~\ref{prop: SFT compactness for W minus}, $\overline{u}_i$ converges in the sense of SFT to a holomorphic building $\overline u_\infty= \overline v_a\cup\dots\cup \overline v_{-b}$. We have three constraints:
\begin{enumerate}
\item[(i)] $\sum_{j=-b}^a I(\overline{v}_j)=3$,
\item[(ii)] $I(\overline{v}_j)\geq 0$ for all $j$, and
\item[(iii)] $\sum_{j=-b}^a n^*(\overline{v}_j)=m$.
\end{enumerate}
(i) comes from the additivity of ECH indices, (ii) comes from Lemma~\ref{claim in proof}, and (iii) comes from Equation~\eqref{eqn: m}.

Suppose that $\overline{v}'_0\not=\varnothing$, i.e., we are in the situation of Lemma~\ref{lemma: preliminary restrictions part 2}.  We have the following restrictions:
\begin{itemize}
\item the top level $\overline{v}_a$ is nontrivial and satisfies $I(\overline{v}_a)\geq 1$;
\item $\cup_{j<0}\overline{v}_j^\sharp$ consists of $p_0 \geq 1$ thin strips and contributes $\sum_{j<0} I(\overline{v}_j^\sharp)=p_0 \geq 1$ to the total ECH index;
\item $\cup_{j=-b}^a \overline{v}_j^\sharp$ consists of $p_{a'}^-$ components and each component has ECH index $\geq 1$, with exception of  $\overline{v}_0^\sharp$, which has ECH index $0$, where $a'$ and $p_{a'}^-$ are as in Lemma~\ref{lemma: preliminary restrictions part 2}.
\end{itemize}
This immediately implies $p_{a'}^- \le 2$ because $p_{a'}^-$ is also the number of
nontrivial curves with a positive end at $\delta_0$ and each of them has ECH index
$I \ge 1$ by Lemma~\ref{lemma: preliminary restrictions part 2}. We also have $p_0
\le p_{a'}^-$, and therefore we can divide the proof into three cases:
\begin{itemize}
\item Case I: \  $p_{a'}^- = p_0 =1$.
\item Case II': \  $p_{a'}^- = 2$ and $p_0=1$.
\item Case II'': \ $p_{a'}^- = 2$ and $p_0=2$.
\end{itemize}

\s\n
{\em Case I.} In this case $\overline{v}_0' = \sigma_{\infty}^-$ and
$\overline{v}_{-1}^\sharp$ is a thin strip. This leaves two possibilities for
$\overline{v}_1$: either $I(\overline{v}_1)=2$ and we are in Case (1), or
$I(\overline{v}_1)=1$ and we are in Case (2).

\s\n
{\em Case II'.}
In this case we have $\overline{v}_{j_0}^\sharp \ne \emptyset$ and with a positive end at $\delta_0$ for some $j_0 \ge 0$. Since $I(\overline{v}_{j_0}^\sharp)=1$  by Remark~\ref{calcolo di alcuni indici},  $\overline{v}_{-1}$ consists of a single thin strip and there are no other levels with $j<0$. If $j_0=0$ we are in Case (5), and if $j_0>0$ we are in Case (6). In Case (6), $\overline{v}_{1}^\sharp$ is a cylinder connecting $\delta_0$ with $h$ by Lemma~\ref{lemma: preliminary restrictions part 2} and
Remark~\ref{calcolo di alcuni indici}.

\s\n
{\em Case II''.}
In this case $\cup_{j<0} \overline{v}_j$ consists of two thin strips. If they are on the same level we are in Case (4) and if they are on different levels we are in Case (3).

\s
This completes the proof of Theorem~\ref{thm: compactness for W minus I=3, first
version}.
\end{proof}

\begin{rmk} \label{rmk: number of branch points}
In Cases (3)--(6), the total number of branch points of $\cup_{j=-b}^a \overline{v}'_j$ is one, where we are not ignoring connector components that cover $\sigma_\infty^*$: Assume without loss of generality that the only nontrivial $\overline{v}'_j$ is $\overline{v}'_0$ and that $\overline{v}'_0$ double covers $\sigma_\infty^-$. Let ${\frak b}$ be the number of branch points of $\overline{v}'_0$. Then $\op{ind}(\overline{v}'_0) = {\frak b}-1$ by Proposition \ref{prop: Fredholm index for W minus}, the Riemann-Hurwitz formula, and the proof of Lemma~\ref{lemma: regularity of curve at infinity}. The index inequality, the additivity of the indices, and the condition $I(\overline{u}_i)=3$ force ${\frak b}=1$.
\end{rmk}

The proof of the following theorem is similar and will be omitted.

\begin{thm}  \label{thm: compactness for W minus I=2, first version}
Let $\overline{J}_-^\Diamond$ be $(\varepsilon,U)$-close to $\overline{J}_-\in \mathcal{J}_{\overline{W}_-}^{reg}$ and $K_{p,2\delta}$-regular with respect to $\overline{\frak m}$, and let $\overline{u}_\infty \in \bdry \mathcal{M}_{\overline{\frak m}}^2(\varepsilon,\delta,p)$. If $\overline{v}_0'\not=\varnothing$, then $\overline{u}_\infty$ contains a $3$-level subbuilding consisting of $\overline{v}_1$ with $I=1$ from $\boldsymbol{\gamma}$ to $\delta_0\boldsymbol{\gamma}'$; $\overline{v}_0'=\sigma_\infty^-$; and $\overline{v}_{-1}^\sharp$ which is a thin strip. Here it is possible to have $I_{ECH}=0$ connectors in between.
\end{thm}

\subsection{Asymptotic eigenfunctions}
\label{subsection: asymptotic eigenfunctions}

We now collect some facts about the asymptotic operator and asymptotic eigenfunctions, referring the reader to \cite{HWZ1} and \cite{HT2}.

\subsubsection{The asymptotic operator}

We study the local behavior of a holomorphic half-cylinder which converges to a degree $l \ge 1$ multiple cover of $\delta_0$, denoted by $\delta_0^l$. Let us denote $D^2_{\rho_0}=\{\rho\leq \rho_0\}\subset D^2$ with $\rho_0>0$ small (in particular $\leq {1\over 2}$) and $K= [C, + \infty) \times (\R/2l\Z)$ with $l\in\N$. We will be using balanced coordinates on $(\R/2\Z)\times D^2_{\rho_0}$; see
Section~\ref{subsubsection: overline W pm}.  A holomorphic half-cylinder which is asymptotic to $\delta_0^l$ at the positive end restricts to a holomorphic map
$$\overline{u}: K \to \R\times (\R/2\Z)\times D^2_{\rho_0},$$ which can be written as:
$$(s,t)\mapsto (s,t,z(s,t)),$$
where $\lim \limits_{s \to + \infty} z(s,t) = 0$.
\begin{lemma}
The function $z : K \to D^2_{\rho_0}$ satisfies the equation
\begin{equation}\label{eqn: holomorphic on V}
\bdry_s z + i\bdry_t z +\varepsilon z=0,
\end{equation}
where $\varepsilon={\pi\over m}$.
\end{lemma}
\begin{proof}
The partial derivatives of $\overline{u}$ are: $\bdry_s\overline{u}=(1,0,\bdry_s z)$ and $\bdry_t\overline{u}=(0,1,\bdry_t z)$. Then $J (\bdry_s\overline{u})= R+ (0,0,i\bdry_s z)$, where $R$ is the Hamiltonian vector field. We compute that $R=\bdry_t + \varepsilon \bdry_\phi$. Hence $J (\bdry_s\overline{u})= (0,1,i\bdry_s z + i\varepsilon z)$.  This gives Equation~\eqref{eqn: holomorphic on V}.
\end{proof}

\begin{defn}
We define the {\em asymptotic operator}
$$A_l: L^2_1(\R/2l\Z,\C) \to L^2(\R/2l\Z,\C),$$
$$f\mapsto -i\bdry_t f-\varepsilon f.$$
\end{defn}
We remark that the asymptotic operator which appears in \cite{HWZ1} is $A_l$, whereas the asymptotic operator in \cite{HT2} is $-A_l$. The eigenfunctions of $A_l$ are the {\em asymptotic eigenfunctions}, and are given by $ce^{\pi int/l}$, $c\in \C$, $n\in \Z$, with corresponding eigenvalues ${\pi n\over l}-\varepsilon$.  An asymptotic eigenfunction $ce^{\pi int/l}$ is said to be {\em normalized} if $|c|=1$.  Let $E_{\pi n/l -\varepsilon}$ be the eigenspace of $A_l$ corresponding to the eigenvalue ${\pi n\over l}-\varepsilon$.

For a strip-like end asymptotic to the intersection point $z_{\infty}$, the asymptotic operator is still $f\mapsto -i\bdry_t f-\varepsilon f$, but now acting on functions $f: [0,1] \to \C$ with boundary conditions $f(0) \in e^{i (c_0+\varepsilon)} \R$ and $f(1) \in e^{ic_0}\R$ for some real constant $c_0$. The eigenfunctions are $ce^{(\pi n - \varepsilon)it + i(c_0+\varepsilon)}$, $c \in \R$, $n \in \Z$, with corresponding eigenvalues $\pi n-2\varepsilon$. As before we say that an eigenfunction is {\em normalized} if $|c|=1$.

\subsubsection{The asymptotic eigenfunction at an end}
\label{subsubsection: asymptotic eigenfunction at an end}

Let $\boldsymbol{\gamma}\in \widehat{\mathcal{O}}_{2g}$ and $\boldsymbol{\gamma}'\in \widehat{\mathcal{O}}_{2g-l}$.  As before, the modifier $*$ is placed as in $\mathcal{M}_{\overline{J}_-}^{*}(\boldsymbol{\gamma},\delta_0^l\boldsymbol{\gamma}')$ to denote the subset of $\mathcal{M}_{\overline{J}_-}(\boldsymbol{\gamma}_0,\delta_0^l\boldsymbol{\gamma}')$ satisfying $*$. The modifier $*=(l_1,\dots,l_\lambda)$ means $(l_1,\dots,l_\lambda)$ is a partition of $l$ and we restrict to curves with $\lambda$ ends at $\delta_0$ with covering multiplicities $l_1,\dots,l_\lambda$.
\nom[2l]{$*=(l_1,\dots,l_\lambda)$}{Modifier ``curve with $\lambda$ ends at $\delta_0$ with covering multiplicities $l_1,\dots,l_\lambda$''}

We consider the asymptotics of $\overline{u}\in \mathcal{M}_{\overline{J'}}^{(l_1,\dots,l_\lambda)}(\boldsymbol{\gamma},\delta_0^l\boldsymbol{\gamma}')$ near the negative end $\delta_0^{l_j}$. Let
$$\overline{\pi}_{D^2_{\rho_0}}: \R\times (\R/2\Z)\times D^2_{\rho_0}\to D^2_{\rho_0}$$
be the projection to $D^2_{\rho_0}$ with respect to balanced coordinates; we also write $\overline\pi$ for $\overline{\pi}_{D^2_{\rho_0}}$. Also let
$$z_j=\overline{\pi}_{D^2_{\rho_0}}\circ \overline{u}:(-\infty,s_0] \times (\R/2l_j\Z)\to D^2_{\rho_0}$$
be the projection of the negative end of $\overline{u}$ which corresponds to $\delta_0^{l_j}$.
\nom[1p$\pi$9]{$\overline{\pi}_{D^2_{\rho_0}}: \R\times (\R/2\Z)\times D^2_{\rho_0}\to D^2_{\rho_0}$, also $\overline{\pi}$}{Projection with respect to balanced coordinates}
The following asymptotic result is due to Hofer-Wysocki-Zehnder~\cite{HWZ1}:

\begin{lemma} \label{lemma: asymptotics}
There exist constants $C_0, C_1>0$ such that the following holds: For any
$\overline{u} \in \mathcal{M}_{\overline{J'}}^{(l_1,\dots,l_\lambda)} (\boldsymbol{\gamma},\delta_0^l\boldsymbol{\gamma}')$ with a negative end at $\delta_0^{l_j}$,
there exists an asymptotic eigenfunction $f_j:\R/2l_j\Z\to \C$ given by $f_j(t)=ce^{\pi it/l_j}$ such that
\begin{align} \label{eqn: asymptotic condition}
|z_j(s,t)-e^{(\pi/l_j-\varepsilon)s} f_j(t)| & < C_0 e^{(\pi/l_j-\varepsilon+C_1)(s-s_0)},\\
 \label{eqn: asymptotic condition C^1}
|\partial_t z_j(s,t)-e^{(\pi/l_j-\varepsilon)s} f_j'(t)| & < C_0 e^{(\pi/l_j-\varepsilon+C_1)(s-s_0)}.
\end{align}
(We allow $c=0$.)
\end{lemma}
Note that, because of the particularly simple form of the Hamiltonian flow around $\delta_0$, Lemma~\ref{lemma: asymptotics} could also be proved directly by
standard Fourier series arguments.
\begin{defn}
The function $f_j(t)$ satisfying Equation~\eqref{eqn: asymptotic condition} is called {\em the asymptotic eigenfunction for $\overline{u}$ at the negative end $\delta_0^{l_j}$.}\footnote{This is the terminology from \cite{HT2}, which is slightly different from that of the Hofer school.}
\end{defn}

The following is due to Wendl~\cite{We4} and Hutchings-Taubes~\cite[Prop.~3.2]{HT2}.  (The proof of \cite[Prop.~3.2]{HT2} essentially proves the following lemma, but is not quite stated in the same way.)

\begin{lemma} \label{lemma: nonzero asymptotic eigenfunction}
There exists an arbitrarily small perturbation $\widetilde{J'}$ of $\overline{J'}$ in $\mathcal{J}_{\overline{W'}}^{reg}$, which is supported on $\R\times \{\rho_1<\rho<\rho_2\}\subset \R\times\overline{N}$ for $0<\rho_0<\rho_1<\rho_2$ arbitrarily small, such that the following holds:
\begin{enumerate}
\item[($\star$)] for all $\boldsymbol{\gamma}$ and $\boldsymbol{\gamma}'$, partitions $(l_1,\dots,l_\lambda)$ of $l$, $j\in\{1,\dots,\lambda\}$ and $r\in \N\cup\{0\}$, the set of elements of $\mathcal{M}_{\widetilde{J'}}^{\op{ind}=r,(l_1,\dots,l_\lambda)}(\boldsymbol{\gamma},\delta_0^l\boldsymbol{\gamma}')/\R$ with vanishing asymptotic eigenfunction at $\delta_0^{l_j}$ is a real codimension $2$ submanifold.
\end{enumerate}
\end{lemma}
The real codimension $2$ condition is due to the fact that $\dim_\R E_{\pi/l_i +\varepsilon}=
2$. Let $\mathcal{J}_{\overline{W'}}^\star\subset \mathcal{J}_{\overline{W'}}^{reg}$ be the (dense) subset of almost complex structures which satisfy ($\star$).

\subsubsection{Definition of $\widetilde{U}_{m-1}$} \label{subsubsection: melon d'eau}

Let $\boldsymbol{\gamma}\in \widehat{\mathcal{O}}_{2g}$, $\boldsymbol{\gamma}'\in \widehat{\mathcal{O}}_{2g-1}$, and $\widetilde{J'}\in \mathcal{J}_{\overline{W'}}^\star$. Let $\overline{u} \in \mathcal{M}_{\widetilde{J'}}^{I=\op{ind}=2,n^*=m-1}(\boldsymbol{\gamma}, \delta_0\boldsymbol{\gamma}')$. Since $\widetilde{J'}$ satisfies ($\star$), the asymptotic eigenfunction of $\overline{u}$ at $\delta_0$ is nonzero. Hence we can associate a normalized asymptotic eigenfunction at $\delta_0$ to $\overline{u}$.

Now let
$$\langle \widetilde{U}_{m-1}(\boldsymbol{\gamma}),\delta_0^p\boldsymbol{\gamma}'\rangle= \left\{
\begin{array}{ll}
 0, & \text{ if } p\not=1; \\
\#\left(\mathcal{M}_{\widetilde{J'}}^{I=\op{ind}=2,n^*=m-1,f_{\delta_0}}(\boldsymbol{\gamma},\delta_0\boldsymbol{\gamma}')/\R\right),  & \text{ if } p=1,
\end{array}
\right. $$
for a generic normalized asymptotic eigenfunction $f_{\delta_0}$. The modifier $f_{\delta_0}$ stands for ``the normalized asymptotic eigenfunction at the negative end $\delta_0$ is $f_{\delta_0}$''.

\subsubsection{The limit $m\to \infty$} \label{limit m to infty}

Suppose $m\gg 0$. Let $\overline{\hh}_m: \overline{S}\stackrel\sim\to \overline{S}$ be a smooth extension of $\hh:S\stackrel\sim\to S$ via $\nu_m$, as defined in Section~\ref{subsubsection: overline W pm}, and let $\overline{\hh}_\infty:\overline{S}\stackrel\sim\to \overline{S}$ be the smooth extension of $\hh$ via $\nu_\infty$. We recall that $\overline{\hh}_\infty$ is the identity on
the disk $D^2_{1/2} \subset \overline{S}$ of radius $\frac 12$.

We will denote $\overline{N}_m$ (including $m= \infty$) the mapping torus of $\overline{\hh}_m$ and define $\overline{W'}_m= \R \times \overline{N}_m$. Of course
$\overline{W'}_m$ is diffeomorphic to $\overline{W'}$ for all $m$ (including $m= \infty$); we will write $\overline{W'}_m$ instead of $\overline{W'}$ when it is necessary to keep track of $m$. This will happen especially when $m=\infty$.
The Hamiltonian structure on  $\overline{N}_m$ will be written as $(\alpha_{0,m},\omega)$.
 We then define $\mathcal{J}_{\overline{W'}_m}$ as in Definition~\ref{defn: almost complex structures on R times overline N}, with $\alpha_0$ replaced by $\alpha_{0,m}$. Similarly, we will use the notation $\overline{W}_{-,m}$ and $\mathcal{J}_{\overline{W}_{-,m}}$ instead of
 $\overline{W}_-$ and $\mathcal{J}_{\overline{W}_-}$ when we need to specify $m$.

Although the orbit $\delta_0$ is degenerate when $m= \infty$, we still define the set $\overline{\mathcal{P}}=\widehat{\mathcal{P}}\cup\{\delta_0\}$ of orbits, the set $\widehat{\mathcal{O}}_k$ (resp.\ $\overline{\mathcal{O}}_k$) of orbit sets constructed from $\widehat{\mathcal{P}}$ (resp.\ $\overline{\mathcal{P}}$) which intersect $\overline{S}\times\{0\}$ exactly $k$ times, as in Section~\ref{acorn2}.

Let $\overline{J'}_\infty\in \mathcal{J}_{\overline{W'}_\infty}$.  In analogy with Definition~\ref{Arianna overline N}, let $\mathcal{M}_{\overline{J'}_{\infty}}(\boldsymbol{\gamma},\delta_0^l\boldsymbol{\gamma}')$ be the moduli space of $\overline{J'}_\infty$-holomorphic maps
in $\overline{W'}$ from $\boldsymbol{\gamma}$ to $\delta_0^l\boldsymbol{\gamma}'$ without fiber components.  (We recall that $\boldsymbol{\gamma}$ and $\boldsymbol{\gamma}'$ consist of orbits in $N$.)

The Hamiltonian vector field of $(\alpha_{0, \infty}, \omega)$ is degenerate, but all degenerate orbits are contained in the interior of $V$, and therefore transversality for the moduli spaces 
$\mathcal{M}_{\overline{J'}_{\infty}}(\boldsymbol{\gamma}, \delta_0^l \boldsymbol{\gamma}')$ works in the usual way. An almost complex structure $\overline{J'}_\infty\in \mathcal{J}_{\overline{W'}_\infty}$ is {\em regular} if $\mathcal{M}^s_{\overline{J'}_{\infty}}(\boldsymbol{\gamma},\delta_0^l\boldsymbol{\gamma}')$ is transversely cut out for all $0<k\leq 2g$, $l\geq 0$, $\boldsymbol{\gamma}\in \widehat{\mathcal{O}}_k$, and $\boldsymbol{\gamma}'\in \widehat{\mathcal{O}}_{k-l}$.

\begin{lemma}\label{lemma: Morse-Bott regularity uffa}
A generic $\overline{J'}_\infty\in \mathcal{J}_{\overline{W'}_\infty}$ is regular and satisfies ($\star$).
\end{lemma}

\begin{proof}
The Fredholm theory for holomorphic curves with Morse-Bott asymptotics uses Sobolev spaces with exponential weights. The regularity of simply-covered moduli spaces in this setting is treated in Wendl~\cite{We4}.
\end{proof}

As before, we write $\mathcal{J}_{\overline{W'}_{\infty}}^{reg}\subset \mathcal{J}_{\overline{W'}_{\infty}}$ for the dense subset of regular almost complex structures and $\mathcal{J}_{\overline{W'}_{\infty}}^\star\subset \mathcal{J}_{\overline{W'}_{\infty}}^{reg}$ for the dense subset of almost complex structures which satisfy ($\star$).  Note that, if $\overline{J'}_{\infty}\in  \mathcal{J}_{\overline{W'}_{\infty}}^\star$, then nearby almost complex structures $\overline{J'}_m\in \mathcal{J}_{\overline{W'}_m}$ for $m\gg 0$ are in $\mathcal{J}_{\overline{W'}_m}^\star$.

On the other hand, compactness is more subtle. In fact the Hamiltonian vector field of $(\alpha_{0, \infty}, \omega)$ has a Morse-Bott family of orbits corresponding to $int(D^2_{1/2})$ and a circle of degenerate orbits corresponding to $\partial D^2_{1/2}$, and the standard asymptotic convergence theorems do not work for degenerate orbits.
For this reason, we will need extra work to be able to apply the SFT compactness theorem; namely we will need to show that a family of degenerating $\overline{J'}_m$-holomorphic curves satisfying the appropriate topological constraints
cannot approach an orbit of $D^2_{1/2}$ besides a multiple of $\delta_0$.

Given a $\overline{J'}_\infty$-holomorphic map $u : \dot F \to \overline{W'}$, we say that an end of $\dot F$ {\em accumulates} to a Hamiltonian orbit $\gamma$ if, after parametrizing the end by $(s_0, + \infty) \times S^1$ or $(-\infty, -s_0)\times S^1$, there is a sequence $s_k \to + \infty$ such that $\lim \limits_{k \to \infty} u(\pm s_k, t) = \gamma(t)$ uniformly in $t$.

\begin{lemma} \label{insonnia}
Let $u \colon \dot F \to \overline{W'}$ be a finite energy $\overline{J'}_\infty$-holomorphic map. Then every end of $\dot F$ accumulates to a Hamiltonian orbit $\gamma$. If an end of
$\dot{F}$ does not accumulate to any degenerate orbit, than it is asymptotic to a (Morse-Bott) nondegenerate orbit.
\end{lemma}

\begin{proof}
The accumulation to a Hamiltonian orbit follows from \cite{Ho1}. If an end does not accumulate
to any degenerate orbit, then there is a neighborhood of the degenerate orbits which does not intersect the image of the end, as the set of degenerate orbits is compact. Then we can repeat the arguments of \cite{HWZ1, HWZ4} in the complement of that neighborhood.
\end{proof}

Let $\gamma_z$ be the orbit corresponding to $z\in D^2_{1/2}$ for the stable Hamiltonian structure $(\alpha_{0, \infty}, \omega)$ on $\overline{N}$. We will use intersection arguments to constrain the behavior of a sequence of $\overline{J'}_m$-holomorphic maps that accumulate to orbits $\gamma_z$ for some $z \in D^2_{1/2}$. (Recall that all other closed orbits in $int(V)$ intersect a fiber more than $2g$ times for $m \gg 0$, and therefore are automatically excluded.)

\begin{lemma} \label{referendum}
Let $F$ be a compact Riemann surface with boundary and let
$$\partial F = B^+_1 \sqcup \ldots \sqcup B^+_{l_+} \sqcup  B^-_1 \sqcup \ldots \sqcup B^-_{l_-}, \quad l_\pm \geq 1$$
be a decomposition of $\bdry F$ into connected components. We parametrize the components $B^-_j$, $j=1,\dots, l_\pm$, by $\beta_j : S^1 \to B^-_j$ in an orientation-reversing manner. Let $\eta$ be a positive real number and $v : F \to \overline{W'}$ a $\overline{J'}_m$-holomorphic map with $m \gg 0$, such that
\begin{itemize}
\item[(i)] each $\pi_{\overline{N}}\circ v \circ \beta_j$ is $\eta$-close to a cover of an orbit $\gamma_{z_j}$ with $z_j \in D^2_{1/2}$ and $\eta < \frac{|z_j|}{2}$, and
\item[(ii)] each $\pi_{\overline{N}}\circ v(B_j^+)$ is contained in a small neighborhood of $N= \overline{N} - int(V)$.
\end{itemize}
Then $n'(v) \ge m$.
\end{lemma}

\begin{proof}
Consider subsets $V_{\eta} = D^2_\eta \times S^1$ and $V_{1/2+ \eta} = D^2_{1/2+\eta} \times S^1$ of $\overline{N}$. Both $\partial V_\eta$ and $\partial V_{1/2+ \eta}$ are foliated by Hamiltonian orbits. We choose coordinates on $\partial V_\eta$ and $\partial V_{1/2+ \eta}$ such that the meridian has slope $(1,0)$, the orbits on $\partial V_\eta$ have slope $(0,1)$ and the orbits on $\partial V_{1/2+ \eta}$ have slope $(1,k)$ with $k \gg 0$ (here we are assuming that $\eta$ has been perturbed so that the foliation on $\bdry V_{1/2+\eta}$ has rational slope).

Perturb $\eta$ if necessary so that $\pi_{\overline{N}}\circ v$ is transverse to $\partial V_{\eta}$ and $\partial V_{1/2+ \eta}$ and let
$$C_-= \pi_{\overline{N}}\circ v(F) \cap \partial V_\eta ~~\mbox{ and }~~ C_+= \pi_{\overline{N}}\circ v(F) \cap \partial V_{1/2+ \eta},$$
where $C_-$ and $C_+$ are oriented as the boundaries of $\op{Im}(\pi_{\overline{N}}\circ v)\cap {V_\eta}$ and $\op{Im}(\pi_{\overline{N}}\circ v)\cap {V_{1/2+\eta}}$.
Since $\pi_{\overline{N}}\circ v$ maps no $B^+_j$ inside $V_{1/2 + \eta}$ and $C_+$ is positively transverse to the orbits on $\bdry V_{1/2+\eta}$,  it follows that $C_+$ represents some homology class $(q,p)$ with $p, q  \ge 1$. 
Then a simple homology computation implies that $C_-$ represents the homology class $(q, 0)$. Hence $n'(v) = qm \ge m$.
\end{proof}

\begin{lemma}\label{dentist}
Let $\overline{J'}_{m_i} \in \mathcal{J}_{\overline{W'}_{m_i}}^\star$ be a sequence such that $\overline{J'}_{m_i} \to \overline{J'}_\infty \in \mathcal{J}_{\overline{W'}_\infty}^\star$ as $i \to \infty$. Given a partition $(l)$ of $l$, a sequence of curves
$$\overline{u}_i \in \mathcal{M}_{\overline{J'}_{m_i}}^{I=1, (l),n'=m_i-l}(\boldsymbol{\gamma},\delta_0^l \boldsymbol{\gamma}')$$
with $i\to\infty$ converges (up to extracting a subsequence) to a curve
$$\overline{u}_\infty \in \mathcal{M}_{\overline{J'}_{\infty}}^{I=1,(l),\widetilde{n}=0} (\boldsymbol{\gamma},\delta_0^l\boldsymbol{\gamma}'),$$
where $\widetilde{n}(\overline{u}_\infty)$ is the intersection number of $\overline{u}_\infty$ with the section at infinity.
\end{lemma}

 \begin{proof}
The first steps of the compactness theorem are unchanged by the presence of degenerate orbits. Let $\dot F_i$ be the domains of the maps $\overline{u}_i$. We can remove finite sets of points $Z_i \subset \dot F_i$ so that the sequence $\overline{u}_i|_{\dot F_i - Z_i}$ satisfies the uniform gradient bound (Equation~\eqref{eq: gradient bound}).

The sequence of punctured Riemann surfaces $\dot F_i - Z_i$ converges to a punctured nodal surface $\dot F_\infty$ in the Deligne-Mumford moduli space. We denote by $F_\infty'$ the punctured Riemann surface obtained by removing the nodes from $\dot F_\infty$. The gradient bound~\eqref{eq: gradient bound} implies that we can extract a subsequence from $\overline{u}_i|_{\dot F_i - Z_i}$ which converges in $C^{\infty}_{loc}$ to a $\overline{J'}_\infty$-holomorphic map $\overline{u}_\infty' : F_\infty' \to \overline{W'}$. (The holomorphic map denoted by $\overline{u}_\infty'$ in this proof should not be confused with the holomorphic map $\overline{u}_\infty'$
introduced in Section~\ref{subsubsection: multisections}). We assume that no component of $\overline{u}_\infty'$ is a connector over $\delta_0$. The modifications required by the presence of a connector over $\delta_0$ are easy and left to the reader.

Fix $\varepsilon>0$ small and suppose that $i\gg 0$. Consider the connected component of $\op{Thin}_\varepsilon(\dot F_i - Z_i)$ with a puncture that limits to $\delta_0^l$; we identify it with $(- \infty, s_i) \times S^1$.   We claim that there is no $s<s_i$ such that $\pi_{\overline{N}}\circ \overline{u}_i|_{\{ s \} \times S^1}$ is sufficiently close to an orbit $\gamma_z$ with $z \in D^2_{1/2} - \{0 \}$.  Arguing by contradiction, if there is such an $s$, then we consider $C_0:=(-\infty,s]\times S^1$ and $C_1:=\dot F_i- (-\infty,s]\times S^1$.  Since $C_0$ limits to $\delta_0^l$, we have $n'(\overline{u}_i|_{C_0})\geq m_i-l$.  On the other hand, there exists a compact surface with boundary $F \subset C_0$ such that $\overline{u}_i|_F$ satisfies the conditions of Lemma~\ref{referendum} and $n'(\overline{u}_i|_F)\geq m_i$,  a contradiction.  We can similarly argue that, for any other component $(s_-,s_+)\times S^1$ of $\op{Thin}_\varepsilon(\dot F_i - Z_i)$, there is no $s\in (s_-,s_+)$ such that $\pi_{\overline{N}}\circ \overline{u}_i|_{\{ s \} \times S^1}$ is sufficiently close to an orbit $\gamma_z$ with $z \in D^2_{1/2} - \{0 \}$; here $s_-=-\infty$ or $s_+=+\infty$ are allowed.

We now apply the usual compactness argument (in view of Lemma~\ref{insonnia}) to obtain a limit $\overline{J'}_\infty$-holomorphic building $\overline{u}_\infty$ such that no level of the building has an end that accumulates to an orbit $\gamma_z$ for $z \in D^2_{1/2} - \{ 0 \}$. Finally, since $\overline{J'}_\infty \in \mathcal{J}_{\overline{W'}_\infty}^\star$, every level of $\overline{u}_{\infty}$ has $\op{ind}\geq 1$.  On the other hand, since $\overline{u}_\infty$ has total Fredholm index $1$, it is a single-level building.
\end{proof}

Let $\overline{J'}_{\infty}\in \mathcal{J}_{\overline{W'}_{\infty}}^\star$. By the compactness and regularity of moduli spaces, there are only finitely many curves $\mathcal{C}_1,\dots,\mathcal{C}_r$, modulo $\R$-translation, such that 
$$\mathcal{C}_i\in\mathcal{M}_{\overline{J'}_{\infty}}^{I=1,\widetilde{n}=0,(l_{i})}(\boldsymbol{\gamma}_i,\delta_0^{l_i}\boldsymbol{\gamma}'_i)$$
for some orbit sets $\boldsymbol{\gamma}_i,\boldsymbol{\gamma}'_i\in \widehat{\mathcal{O}}_*$ and partition $(l_{i})$ of $l_i$. Here if $\overline{u}$ is a curve with an end at $\delta_0^p$, then its ECH index is computed using the Conley-Zehnder index of $\delta_0^p$ with respect to $\overline{\hh}_m$, $m\gg 0$, and $\widetilde{n}(\overline{u})$ is defined as the intersection number of $\overline{u}$ and the section at infinity. Let
$$f_{i}: \R/2 l_{i}\Z\to \C, \quad t\mapsto c_{i}e^{\pi it/l_{i}},$$ be the asymptotic eigenfunction corresponding the end $\delta_0^{l_{i}}$ of ${\mathcal C}_i$. The condition $c_{i}\not=0$ follows from Lemma~\ref{lemma: Morse-Bott regularity uffa}. We may therefore assume without loss of generality that all the $f_{i}$ are normalized.

\subsubsection{Radial rays}  \label{subsubsection: radial rays}

\begin{defn}\label{defn: radial rays}
A {\em bad radial ray} is a radial ray ${\mathcal R}_{\phi_0}=\{\phi=\phi_0,\rho\geq 0\}$
in $\C$ which passes through a point in
$$\{f_{i}(t)~|~ i=1,\dots,r; 0< t < 2l_{i}; t \equiv {3/2}
\mbox{ mod } 2\}.$$
A radial ray ${\mathcal R}_{\phi_0}$ which is not bad is said to be {\em good}.
\nom[R]{$\mathcal{R}_{\phi_0}$}{Radial ray $\{\phi=\phi_0,\rho\geq 0\}\subset \C$}
\end{defn}

A good radial ray must exist and, as it is explained in the following remark, we can assume it is ${\mathcal R}_\pi$.

\begin{rmk}\label{rmk: good ray is real}
The set of bad radial rays is determined by $(\overline{W'}_\infty, \overline{J'}_\infty)$. Strictly speaking, we should choose the set of endpoints $E\subset \bdry D^2$ as in Section~\ref{coconut} such that $\mathcal{R}_{\phi_0+\pi}$ is a good radial ray and $\phi_0 < \phi(y_j(m))<\phi_0+c(m)$ where $c(m)\to 0$ as $m\to \infty$. After a rotational coordinate change of $D^2$, we may assume that $\mathcal{R}_\pi$ is a good radial ray and $0<\phi(y_j(m))<c(m)$.
\end{rmk}

\subsection{The rescaled function} \label{subsection: rescaling}

In this subsection we will use limiting arguments in which $m\to \infty$ and $\overline{\hh}_m\to \overline{\hh}_\infty$; see Section~\ref{limit m to infty}.  Hence many of the almost complex structures and moduli spaces will have an additional subscript $m$, where $m=\infty$ is also a possibility. 
Let $\overline{J'}_{\infty}\in \mathcal{J}_{\overline{W'}_{\infty}}^\star$ and let $\overline{J'}_{m}\in \mathcal{J}_{\overline{W'}_m}^\star$ be a nearby almost complex structure with respect to the integer $m\gg 0$. Let $\overline{J}_{-,m}\in \mathcal{J}_{\overline{W}_{-,m}}^{reg}$ be an almost complex structure which restricts to $\overline{J'}_{m}$ and let $\overline{J}_{-,m}^\Diamond$ be $(\varepsilon,U)$-close to $\overline{J}_{-,m}$.

Let $m_i$, $i\in \N$, and $\overline{u}_{ij}$, $i,j\in\N$, be sequences satisfying the following properties:
\begin{itemize}
\item[(S1)] $\displaystyle\lim_{i\to\infty} m_i=\infty$;
\item[(S2)] $\overline{u}_{ij}\in \mathcal{M}^{I=3,n^*=m_i}_{\overline{J}_{-,m_i}^\Diamond}(\boldsymbol{\gamma},{\bf y};\overline{\frak m})$, where ${\bf y}=\{x_l\}\cup{\bf y'}$ or $\{x_l'\}\cup {\bf y'}$ for some $l$; and
\item[(S3)] for all $i\in \N$ and $\kappa,\nu> 0$, there exists $j_{i,\kappa,\nu}$ such that, if $j\geq j_{i,\kappa,\nu}$, then $\overline{u}_{ij}$ is $(\kappa,\nu)$-close (cf.\ Definition~\ref{defn: kappa nu close buildings}) to a $\overline{J}_{-,m_i}^\Diamond$-holomorphic building $\overline{u}_{i\infty}=\cup_l \overline{v}_{l,i}$ with $\overline{v}_{0,i}' \ne \varnothing$.
\end{itemize}
Moreover, assume that one of the following holds:
\begin{itemize}
\item[(S4')] the first negative end at $\delta_0$ in the building $\overline{u}_{i\infty}$ has multiplicity one, or
\item[(S4'')] the first negative end at $\delta_0$ in the building $\overline{u}_{i\infty}$ has multiplicity two.
\end{itemize}

Property (S3) is a consequence of the fact that the sequence $\overline{u}_{ij}$ converges in the SFT sense to the building $\overline{u}_{i\infty}$ for each fixed $i$. Property (S4') corresponds to Cases (1) and (2) in Theorem \ref{thm: compactness for W minus I=3, first version} and Property (S4'') corresponds to the remaining cases.

\subsubsection{Truncation}

Recall the projection $\overline\pi=\overline\pi_{D^2_{\rho_0}}$ using balanced coordinates.

\begin{lemma} \label{kkprime}
Let $\overline{v}$ be a $\overline{J'}_m$-holomorphic curve in 
$\overline{W'}_m$ with ECH index $I\leq 2$ for $\overline{J'}_m \in  \mathcal{J}_{\overline{W'}_m}^\star$ and let $\widetilde{v} : (- \infty, s_0] \to \overline{W'}$ be a negative end of $\overline{v}$ which is asymptotic to a multiple cover of $\delta_0$. Then for every $\eta, \kappa >0$ there exist
positive constants $\overline{R}$ and $\kappa' < \eta$ such that ${\kappa'\over \kappa}>  \frac 1 \eta$ and
\begin{align*}
\left| \overline{\pi}\circ\widetilde{v}(-\overline{R}, t)-\kappa' f(t) \right| & \leq \frac \kappa 2,\\
\left| \partial_t (\overline{\pi}\circ\widetilde{v})(-\overline{R}, t)-\kappa'  f'(t) \right| & \leq  \frac \kappa 2
\end{align*}
for all $t$, where $f$ is the normalized asymptotic eigenfunction for $\widetilde{v}$.
\end{lemma}

\begin{proof}
This is a slightly weaker rephrasing of Lemma~\ref{lemma: asymptotics}, in view of Lemma \ref{lemma: nonzero asymptotic eigenfunction}.
\end{proof}

Let $F$ be a Riemann surface and $\overline{u}: F\to (\overline{W}_-,\overline{J}_-)$ a holomorphic map.  We denote by $p: F \to B_-$ the map $\overline{\pi}_{B_-}\circ \overline{u}$ and by $s: F \to \R$,  $t: F \to S^1$ the functions obtained by composing $p$ with the coordinates $(s,t)$ on $B_-$. The functions $(s,t)$ give local coordinates on $F$ outside the critical points of $p$.

Let $\overline{g}$ be the restriction of an $s$-invariant Riemannian metric on $\overline{W'}$ to $\overline{W}_-$ and let $d$ be the distance induced by $\overline{g}$
on $\overline{W}_-$.

\begin{defn}[Truncation]\label{defn: truncation}
A {\em truncation} $(\widetilde{u},\widetilde{F},  \{R^{(l)}\}_{l=1}^{l_0},\widetilde{\frak m},{\frak e}^\pm,\kappa')$ of a holomorphic map $\overline{u}: F\to (\overline{W}_-,\overline{J}_-)$ is a tuple where:
\begin{itemize}
\item $\widetilde{F}$ is a subsurface of $F$ and $\widetilde{u}$ is the restriction of $\overline{u}$ to $\widetilde{F}$;
 \item $p(\widetilde{F})= B_-\cap \{R^{(1)}\leq s\leq R^{(l_0)}\}$;
\item for each $l=1,\dots,l_0-1$, the restriction
$$p:\widetilde{F}\cap p^{-1}(\{R^{(l)}\leq s\leq R^{(l+1)}\}) \to B_-\cap \{R^{(l)}\leq s\leq R^{(l+1)}\}$$
is a branched cover;
\item $\sup_{x\in \widetilde{u}(\widetilde{F})}d(x,\sigma_\infty^-)\leq 2\kappa'$;
\item  if $d(\overline{u}(y),\sigma_\infty^-)\leq {\kappa'\over 2}$, then $y\in \widetilde{F}$;
\item $\widetilde{\frak m}\in \widetilde{F}$ is the unique point such that $\widetilde{u}(\widetilde{\frak m})=\overline{\frak m}$;
\item $\mathfrak{e}^+$ (resp.\ $\mathfrak{e}^-$) is the union of components of $\bdry\widetilde{F}- p^{-1}(\partial B_-)$ for which $ds({\frak n})>0$ (resp.\ $<0$), where ${\frak n}$ is the outward normal vector field along $\bdry \widetilde{F}$.
\end{itemize}
\end{defn}
The map $p$ is a branched cover only restricted to portions of $\widetilde{F}$, and not globally, because there can be portions of $\partial \widetilde{F}$ which are mapped to the interior of $B_- \cap \{R^{(1)}\leq s\leq R^{(l_0)}\}$.
\begin{defn}
A {\em good truncation}
$$(\widetilde{u},\widetilde{F}, \{R^{(l)}\}_{l=1}^{l_0},\widetilde{\frak m},{\frak e}^\pm,\kappa,\kappa')$$
is a truncation $(\widetilde{u},\widetilde{F}, \{R^{(l)}\}_{l=1}^{l_0},\widetilde{\frak m},{\frak e}^\pm,\kappa')$, together with a constant $\kappa>0$, such that ${\kappa'\over \kappa}>2$ and:
\begin{enumerate}
\item[(G)]  $\displaystyle\left \| \overline{\pi}\circ\widetilde{u}|_{\mathfrak{e}^\pm} - f\right \|_{C^1} \leq \kappa$, where $f$ is an asymptotic eigenfunction of $\delta_0$ or $z_\infty$ (as appropriate) on each component of ${\frak e}^\pm$ and $\| f \|_{C^0}\geq \kappa'$.
\end{enumerate}
\end{defn}

\begin{lemma}\label{lemma: pre-riscaling estimates}
Let $m_i$ and $\overline{u}_{ij}$ be sequences satisfying (S1)--(S3) and either (S4') or (S4'').
Then there is a sequence $j(i)$ such that the following hold for all $j \ge j(i)$:
\begin{enumerate}
\item $\overline{u}_{ij}$ admits a good truncation
$$(\widetilde{u}_{ij},\widetilde{F}_{ij},\{R^{(l)}_{ij}\}_{l=1}^{l_0},\widetilde{\frak m}_{ij},{\frak e}_{ij}^\pm,\kappa_i,\kappa'_i);$$
\item $\lim \limits_{i \to \infty} \kappa_i' = \lim \limits_{i \to \infty} \kappa_i =0$ and $\lim \limits_{i \to \infty} \frac{\kappa_i'}{\kappa_i} = + \infty$;
\item $\lim\limits_{i \to \infty} R^{(l_0)}_{ij(i)}=+\infty$ and $\lim\limits_{i \to \infty} R^{(1)}_{ij(i)}=-\infty$.
\end{enumerate}
\end{lemma}

\begin{proof}
We prove the lemma in Case (S4'), where the notation is simpler, i.e., we can use $(s,t)$ as global coordinates on $\widetilde{F}_{ij}$. In this case $l_0=2$.  Case (S4'') is conceptually the same.

Fix a sequence $\eta_i \to 0$. Then, by Lemma \ref{kkprime}, for each $i$ there exist $\kappa_i'$, $\kappa_i$ and $\overline{R}_i$ which satisfy:
\begin{itemize}
\item $\kappa_i',\kappa_i < \eta_i$ and $\frac{\kappa_i'}{\kappa_i} > \frac{1}{\eta_i}$; and
\item $\left| \overline{\pi}\circ\widetilde{v}_{i}(-\overline{R}_i, t)-\kappa_i' f_i(t) \right|\leq \frac{\kappa_i}{2}$ for all $t$, where $f_i$ is the normalized asymptotic eigenfunction for a negative end $\widetilde{v}_{i}$ of a $\overline{J'}_{m_i}$-holomorphic curve $\widehat{v}_i$ of ECH index $I\leq 2$ that limits to $\delta_0$. The curve $\widehat{v}_i$ is a component of the SFT limit of the sequence $\{\overline{u}_{ij} \}_{j}$.
\end{itemize}

On the other hand, by (S3) and Definition~\ref{defn: kappa nu close buildings}, for all $i$ there exists $j(i)$ such that if $j>j(i)$ then there exists $R_{ij}'$ such that
$$\left| \overline{\pi}\circ \overline{u}_{ij} (R_{ij}'- \overline{R}_i, t) - \overline{\pi}\circ \widetilde{v}_{i} (-\overline{R}_i, t) \right | \le {\kappa_i\over 2}$$ for all $t$.
Setting $R_{ij}^{(1)}=R'_{ij(i)}-\overline{R}_i$, we obtain:
\begin{align*}
\left| \overline{\pi}\circ\overline{u}_{ij}(R_{ij}^{(1)}, t)-\kappa_i' f_i(t) \right| \le & \left| \overline{\pi}\circ \widetilde{u}_{ij} (R_{ij}^{(1)},t) - \overline{\pi}\circ \widetilde{v}_{i} (-\overline{R}_i,t) \right | \\
& + \left| \overline{\pi}\circ\widetilde{v}_{i}(-\overline{R}_i,t) -\kappa_i' f_i(t) \right| \le \kappa_i.
\end{align*}

Similar considerations also hold for the $t$-derivative and at the negative truncated end.
\end{proof}

\begin{rmk}\label{limit of normalized eigenfunctions}
As $i\to \infty$, the curves $\widehat{v}_{i}$ approach one of the
${\mathcal C}_{i_0}$ from Section~\ref{limit m to infty}, modulo $\R$-translation,
by Lemma \ref{dentist}. Hence the normalized asymptotic eigenfunctions $f_i$ limit
to the normalized eigenfunction $f_{i_0j_0}$ for ${\mathcal C}_{i_0}$ at the negative end as $i \to \infty$.
\end{rmk}

\subsubsection{Ansatz}

We define $\widetilde{z}_i =  \overline{\pi}\circ \widetilde{u}_{ij(i)}$ and, to simplify computations, we will make the ansatz
\begin{equation} \label{eqn: ansatz}
\widetilde{z}_i= e^{-\varepsilon_i s}\widetilde{w}_i,
\end{equation}
where $\varepsilon_i = \frac{\pi}{m_i}$.

\begin{lemma} \label{making holomorphic}
The functions $\widetilde{w}_i : \widetilde{F}_i \to \C$ defined by Equation~\eqref{eqn: ansatz} are holomorphic with respect to the standard complex structure on $\C$.
\end{lemma}

\begin{proof}
We give the proof for Case (S4");  Case (S4') is similar but simpler. The functions $(s,t)$ give local conformal coordinates on $\widetilde{F}_i$ outside the branch locus because $p_{i} : \widetilde{F}_i \to B_-$ is holomorphic. Then $\widetilde{z}_i : \widetilde{F}_i \to \C$ satisfies Equation~\eqref{eqn: holomorphic on V}:
$$ \partial_s \widetilde{z}_i +i \partial_t \widetilde{z}_i + \varepsilon_i \widetilde{z}_i =0.$$
If we plug in Equation~\eqref{eqn: ansatz} into Equation~\eqref{eqn: holomorphic
on V} we obtain:
$$-\varepsilon_i e^{-\varepsilon_i s}\widetilde{w}_i + e^{-\varepsilon_i s} \bdry_s
\widetilde{w}_i + ie^{-\varepsilon_i s}\bdry_t \widetilde{w}_i +\varepsilon_i
e^{-\varepsilon_i s} \widetilde{w}_i=0.$$
Hence $\bdry_s \widetilde{w}+ i\bdry_t \widetilde{w}_i=0$, so $\widetilde{w}_i$ is a
holomorphic map to the standard complex line $(\C,i)$ in the complement of the
branch point. Then it is holomorphic everywhere by the removal of singularities.
\end{proof}

\subsubsection{Rescaling}

Consider the diagonal subsequence $\overline{u}_i=\overline{u}_{ij(i)}$, $i=1,2,\dots$. We abbreviate $\widetilde{F}_i=\widetilde{F}_{ij(i)}$, $p_i = p_{ij(i)}|_{\widetilde{F}_{ij(i)}}$, $\widetilde{\frak m}_i=\widetilde{\frak m}_{ij(i)}$, etc.  Fix $R_0',R_0>0$ and let $K_i$ be the connected component of $p_i^{-1}(B_- \cap \{ - R_0' \le s \le R_0 \})$ containing $\widetilde{\mathfrak{m}}_i$.
By passing to a subsequence we may assume that the topologies of $K_i$ and $\widetilde{F}_i-K_i$ are independent of $i$. In order to simplify
the exposition, we will assume that this is the case for all $i$.

\begin{defn} \label{defn: definiton of w_i}
We define $C_i = \sup_{z\in K_i} |\widetilde{w}_i(z)|$ and $w_i = \widetilde{w}_i/C_i$.
\end{defn}

The holomorphic maps $w_i : \widetilde{F}_i \to \C$ satisfy the following properties:
\begin{enumerate}
\item $\sup_{z \in K_i} |w_i(z)|=1$;
\item there is a unique point $\widetilde{\frak m}_i\in \widetilde{F}_i$ such that $w_i(\widetilde{\frak m}_i)=0$; the zero $\widetilde{\frak m}_i$ is a simple zero;
\item if $x \in \partial \widetilde{F}_i$ and $p_i(x)=(s,t) \in \partial B_-$,  then $w_i(x) \in e^{i (\varepsilon_i t  + \phi_i)}\R^+$ with $\varepsilon_i={\pi\over m_i}$ and $\lim \limits_{i \to \infty}\phi_i =0$  for some $\phi_i$ depending on the boundary component of $\partial \widetilde{F}_i$ containing $x$;
\item for every truncated end  $\mathfrak{c}$ (i.e., a connected component of $\mathfrak{e}^+$ or $\mathfrak{e}^-$) there are constants $\widetilde{\kappa}_i$  and $\widetilde{\kappa}_i'$ such that the inequalities
\begin{equation} \label{eqn: approx}
\displaystyle\left \|w_i|_{\mathfrak{c}} - f_i \right \|_{C^1} \leq \widetilde{\kappa}_i, \quad \|f_i\|_{C^0}\geq \widetilde\kappa_i'
\end{equation}
holds, where $f_i$ is an asymptotic eigenfunction of $\delta$ or $z_\infty$ (as appropriated) on $\mathfrak{c}$, and $\widetilde{\kappa}_i$ and the constants $\widetilde{\kappa}_i'$ satisfy $\widetilde{\kappa}_i'/\widetilde{\kappa}_i \to \infty$ as $i \to \infty$.
\end{enumerate}
\begin{rmk}\label{34}
The loops of $w_i|_{\mathfrak{e}_i^\pm}$ have winding number $1$ around the origin when $i$ is sufficiently large: in fact $\widetilde{\kappa}_i' > \widetilde{\kappa}_i$, so the linear homotopy between $w_i|_{\mathfrak{e}_i^\pm}$ and $f_i(t)$ is contained in $\C^\times$.
\end{rmk}

\begin{lemma} \label{lemma: convergence on fixed compact set}
After passing to a subsequence, $w_i|_{K_i}$ converges to a nonconstant holomorphic function $w_\infty|_{K_\infty}$.
\end{lemma}

\begin{proof}
After passing to a subsequence we may assume that $K_i$ converges to a Riemann surface $K_\infty$ with boundary. Since $w_i|_{K_i}$ is uniformly bounded, the lemma follows from Montel's theorem.
\end{proof}

From now on we assume that we have passed to a subsequence so that Lemma \ref{lemma: convergence on fixed compact set} holds. In the following subsections we analyze the convergence of $w_i$, not just on the uniformly bounded part $K_i$.  This involves techniques of SFT compactness.

\subsubsection{Energy bound}

Following Hofer~\cite{Ho1}, we define an {\em energy} for holomorphic functions on Riemann surfaces with boundary and punctures.

\begin{defn}[Energy] \label{energy for w_i}
Let $\mathcal{C}$ be the set of smooth functions $\varphi : [0,\infty) \to [0,1]$ such that $\varphi(0)=0$ and $\varphi'(r) \ge 0$ for all $r \in [0,\infty)$. If $F$ is a Riemann surface and $u : F \to \C$ a holomorphic function, we define the {\em energy of $u$} as
$$E(u)=\sup \limits_{\varphi \in \mathcal{C}} \int_{F} u^* d (\varphi(r) d \theta),$$
where $(r, \theta)$ are polar coordinates on $\C$.
\end{defn}

\begin{rmk}
If we identify $\C^\times\stackrel\sim\to \R\times S^1$ using the log map, then $$\sup \limits_{\varphi \in \mathcal{C}} \int_{F-u^{-1}(0)} u^* d (\varphi(r) d \theta)$$ agrees with the usual expression for the Hofer energy of
$$\log \circ u:F-u^{-1}(0)\to \R\times S^1.$$
\end{rmk}

\begin{lemma} \label{sconforto}
The sequence $w_i$ has uniformly bounded energy.
\end{lemma}

\begin{proof}
By Stokes' theorem,
$$\int_{\widetilde{F}_i} w_i^*(d\varphi(r)d\theta) = \int_{\bdry \widetilde{F}_i} w_i^*(\varphi(r)d\theta).$$
The boundary $\bdry \widetilde{F}_i$ is the union of ${\frak e}_i^\pm$, ${\frak f}_{ik}= p_i^{-1}(\{t=k\})\cap \bdry\widetilde{F}_i$, $k=0,1$, and ${\frak f}_{i2}= p_i^{-1}(\{1< t< 2\})\cap \bdry \widetilde{F}_i$, all oriented using the boundary orientation. 
If ${\frak c}$ is a component of ${\frak e}_i^\pm$, then $\int_{{\frak c}} w_i^*d\theta = 2\pi\deg(w_i|_{\frak c})$ if ${\frak c}$ is a circle and $\int_{\frak c}  w_i^*d\theta \approx {\pi\over m}$ if ${\frak c}$ is an arc because $w_i|_{\frak c}$ is $C^l$-close to an asymptotic eigenfunction with $l\geq 1$ (compared to the $C^0$ norm of the eigenfunction). This follows from the exponential decay estimates of \cite{HWZ1}; also see \cite[Lemma~2.3]{HT2}. Moreover, for $k=0,1$,  $\int_{{\frak f}_{ik}}w_i^*(\varphi(r)d\theta)=0$ since $w|_{{\frak f}_{ik}}$, $k=0,1$, projects to a radial ray. Finally, $\int_{{\frak f}_{i2}}w_i^*(\varphi(r)d\theta)<0$ since $w|_{{\frak f}_{i2}}$ always has a component in the negative $\theta$-direction; here it is important to remember that we are projecting using balanced coordinates. This proves the lemma.
\end{proof}

\begin{rmk}
Observe that $E(cu)=E(u)$ where $c\in \C^\times$.
\end{rmk}

\subsubsection{Bubbling}

The goal of this subsection is to eliminate certain types of bubbling.

Let $(\widetilde{F}^i,\widetilde{g}_i)$, $i\in \N$, be a sequence of Riemannian surfaces which are compatible with the complex structures and have injectivity radii which are uniformly bounded below. For example, one could obtain the metrics $\widetilde{g}_i$ by scaling the compatible hyperbolic metrics by a conformal factor so that the thick parts remain hyperbolic, while the thin parts become flat cylinders. We will use these metrics to compute the pointwise norm of the differential of the functions $w_i$ which are denoted by $|w_i'|$.

\begin{lemma}\label{natale}
For every compact set $K \subset B_-$, there is a constant $C_K>0$ such that $|w_i(z)|<C_K$ and $|w'_i(z)|<C_K$ for all $z \in p_i^{-1}(K)$ and all $i\in \N$.
\end{lemma}

\begin{proof}
The lemma holds for $K\cap K_i$ by Lemma~\ref{lemma: convergence on fixed compact set}. Hence it suffices to consider $K\cap K_i^c$, where $K_i^c:=\widetilde{F}_i-int(K_i)$. Since $w_i(K_i^c) \subset \C^\times$, we compose with $\log : \C^\times \to \R \times (\R/2 \pi \Z)$
and consider $\zeta_i = \log \circ w_i|_{K_i^c}$. The sequence $\zeta_i$ has uniformly bounded energy by Lemma~\ref{sconforto}. If $|{\zeta_i}'|$, $i\in\N$, is not uniformly bounded on $K_i^c$, then the usual bubbling analysis yields a nonconstant finite energy plane $\widetilde{v}^+_\infty: \C\to \R \times (\R/2 \pi \Z)$ or half-plane $\widetilde{v}^+_\infty: \H\to\R \times (\R/2 \pi \Z)$; see the proof of \cite[Theorem~31]{Ho1}. If $\widetilde{v}^+_{\infty}$ is a half-plane, then $\widetilde{v}^+_{\infty}(\partial \H)$ is contained in a line $\R\times\{pt\}\subset \R\times(\R/2 \pi \Z)$ as a consequence of the boundary conditions for the maps $w_i$, and we can double $\widetilde{v}^+_{\infty}$ to obtain a nonconstant finite energy plane by the Schwarz reflection principle. On the other hand, by \cite[Lemma~28]{Ho1}, there are no nonconstant finite energy planes in $\R \times (\R/2 \pi \Z)$, a contradiction. Hence $|{\zeta_i}'|$, $i\in \N$, is uniformly bounded on $K_i^c$. This in turn implies uniform bounds on $|w_i|$ and $|w_i'|$ on $K\cap K_i^c$.
\end{proof}

\subsubsection{Case (S4')}

{\em In Cases (S4') and (S4'') we often pass to a subsequence without explicit mention.}

Suppose we are in Case (S4'). Recall the compactification $cl(B_-)$ of $B_-$ defined in Section~\ref{acorn}, obtained by adjoining the points at infinity  $\mathfrak{p}_+, \mathfrak{p}_-$. Here  $cl(B_-)$ is isomorphic to the closed unit disk, $\mathfrak{p}_+$ is
a marked point in the interior and $\mathfrak{p}_-$ is a marked point on the boundary under this identification.

\begin{thm}\label{capodanno}
The sequence $(w_i)_{i\in\N}$ converges uniformly on compact subsets to a holomorphic map $w_{\infty} : B_- \to \C$ such that:
\begin{enumerate}
\item $w_{\infty}(\partial B_-) \subset \R^+$ and $w_\infty(\overline{\frak m}^b)=0$;
\item $\lim \limits_{s \to + \infty} |w_{\infty}(s,t)| = + \infty$ and $\lim \limits_{s \to + \infty} \dfrac{w_{\infty}(s,t)}{|w_{\infty}(s,t)|} = f_{i_0j_0}(t)$, for some $i_0, j_0$;
\item $\lim \limits_{s \to - \infty} w_{\infty}(s,t) = c \in \R^+$;
\item $w_\infty|_{int(B_-)}$ is a biholomorphism onto its image. In particular $\overline{\mathfrak{m}}^b$ is the unique zero of $w_\infty|_{cl(B_-)}$ and is simple;
\item $w_\infty$ extends to a holomorphic map $cl(B_-)\to \C\P^1$, still called $w_\infty$, such that $w_\infty(\bdry cl(B_-))=[a_1,a_2]\subset \R^+$ and $w_\infty(\mathfrak{p}_{+})=\infty$.
\end{enumerate}
\end{thm}

Recall that $f_{i_0j_0}(t)$ is a normalized asymptotic eigenfunction of the curve $\mathcal{C}_{i_0}$ with a negative end asymptotic to a multiple of $\delta_0$.

\begin{proof}
The uniform convergence on compact subsets is a consequence of the bounds from Lemma~\ref{natale}. (1) is immediate from the convergence, together with the observation that the angles between the arcs in $\overline{\bf a}$ tend to $0$ as $m_i\to\infty$.

(2) We can either use SFT compactness and analyze the levels, or argue directly as follows: Let $K_i^+$ (resp.\ $K_i^-$) be the component of $\widetilde{F}_i-int(K_i)$ with positive (resp.\ negative) $s$-coordinates. We expand $w_i$ and $w_{\infty}$ in Fourier series on $K_i^+$:
$$w_i(s,t) = \sum \limits_{n=- \infty}^{+ \infty} a_n^i e^{\pi n (s+it)}, \quad w_{\infty}(s,t) = \sum \limits_{n=- \infty}^{+ \infty} a_n^{\infty} e^{\pi n (s+it)}.$$
By the uniform convergence of $w_i$ to $w_{\infty}$, $\lim \limits_{i\to \infty} a_n^i = a_n^{\infty}$ for all $n$. Then (2) is equivalent to the conditions:
\begin{itemize}
\item[(i)] $\lim \limits_{i \to \infty} \frac{a_1^i}{|a_1^i|} e^{i \pi t} = f_{i_0j_0}(t)$; and
\item[(ii)] $\lim \limits_{i \to \infty} a_n^i = 0$ when $n \ge 2$.
\end{itemize}

Since we are in Case (S4'),  $\{ R_i^{(l)}\}_{l=1}^{l_0}=\{R_i^{(1)},R_i^{(2)}\}$, i.e., $l_0=2$.  By Equation~\eqref{eqn: approx},
$$\left | \frac 12 \int_0^2 \left(w_i(R_i^{(2)},t) - f_i(t)\right) e^{- \pi i n t}dt \right |\le \widetilde{\kappa}_i,$$
for all $n\in \Z$. On the other hand, if we write $f_i(t) = c_i e^{\pi i t}$, where $f_i(t)$ is not necessarily normalized, then
$$ \frac 12 \int_0^2 \left(w_i(R_i^{(2)},t) - f_i(t)\right) e^{- \pi i n t}dt =a_n^i e^{\pi n R_i^{(2)}} - c_i \delta_{1n},$$
where $\delta_{1n}$ is the Kronecker delta. Hence
\begin{enumerate}
\item[(a)] $|a_n^i|\cdot e^{\pi n R_i^{(2)}} \le \widetilde{\kappa}_i$ for all $n \ne 1$; and
\item[(b)] $|a_ 1^i e^{\pi R_i^{(2)}} - c_i | \le \widetilde{\kappa}_i$.
\end{enumerate}

We prove (ii). Arguing by contradiction, suppose that $\lim \limits_{i \to \infty} |a_n^i| \ne 0$ for some $n \ge 2$. By (a), there exists $C>0$ such that $Ce^{2\pi R_i^{(2)}}<\widetilde{\kappa}_i$ for all $i$. Since $R_i^{(2)}\to \infty$ and $\lim\limits_{i\to\infty}a^i_1$ exists, we obtain $|a_1^i e^{\pi R_i^{(2)}}|<\widetilde{\kappa}_i$ for all $i$. On the other hand, $|c_i|\geq \widetilde{\kappa}_i'$ and $\lim\limits_{i\to\infty}{\widetilde{\kappa}_i'\over \widetilde{\kappa}_i}=+\infty$, which contradicts (b).

Next we prove (i). We claim that $\lim \limits_{i \to \infty} a_1^i \ne 0$ for topological reasons. Indeed, if
$\lim \limits_{i \to \infty} a_1^i = 0$, then $\lim \limits_{i\to\infty} a_n^i=0$ for all $n\geq 1$ by (ii) and the curve $w_i|_{s=R_i^{(2)}}$ has nonpositive winding number around $0$ when $i\gg 0$, a contradiction. This proves the claim.  Finally, $\lim \limits_{i \to \infty} \left | a_1^i \cdot\frac{e^{\pi R_i^{(2)}}}
{|c_i|} - {c_i\over |c_i|} \right | =0$, since $\lim\limits_{i\to\infty}{\widetilde{\kappa}_i\over \widetilde{\kappa}_i'}=\lim\limits_{i\to\infty}{\widetilde{\kappa}_i\over |c_i|}= 0$. Hence
$$\lim \limits_{i \to \infty} \frac{a_1^i}{|a_1^i|} e^{\pi i t} =   f_{i_0j_0}(t),$$
which proves (i).

(3) This is similar to (2). We expand $w_i$ in Fourier series on $K_i^-$:
$$w_i(s,t) = \sum \limits_{- \infty}^{+ \infty} a_n^i e^{\varepsilon_i i}e^{(\pi n - \varepsilon_i)(s+it)}.$$
By the uniform convergence, $\lim \limits_{s \to - \infty} a_n^i = a_n^{\infty}$ and we can similarly prove that $a_n^{\infty} = 0$ for all $n<0$. This implies (3) because the normalized eigenfunctions converge to a constant as $i \to +\infty$.

(4) Since $H_2(cl(B_-), \partial cl(B_-)) \cong H_2(\C \P^1, [a_1, a_2]) \cong \Z,$
we have a well-defined notion of degree for $w_{\infty}$. Moreover, as in the closed case,
the degree is equal to the cardinality of the inverse image of a regular
value in $\C \P^1 - [a_1, a_2]$. Hence $\deg w_\infty|_{int(B_-)}=1$ because $w_{\infty}^{-1}(\infty) = \{ 0 \}$ and $0$ is a simple pole. This implies that $w_\infty: int(B_-)\to\C - [a_1, a_2]$ is a biholomorphism.

(5) follows from (2) and (3).
\end{proof}

\subsubsection{Case (S4'')}\label{subsection: Case (S4'')}

We give a brisk treatment of the construction of the limit, mostly pointing out the differences with (S4'). The main difference is that the holomorphic maps $\overline{u}_i$ admit good
truncations with a component which
is a branched double cover over its image with a single branch point $b_i \in B_-$, and we must analyze different cases depending on the behavior of the branch point as $i\to\infty$.

There are five cases:
\begin{itemize}
\item[(a)]  $\lim \limits_{i \to \infty} b_i = b_{\infty} \in int(B_-)$,
\item[(b)]  $\lim \limits_{i \to \infty} b_i = b_{\infty} \in \partial B_-$,
\item[(c)] $\lim \limits_{i \to \infty} s(b_i) = + \infty$,
\item[(d')] $\lim \limits_{i \to \infty} s(b_i) = -\infty$ and $d(b_i,\partial B_-)\geq C$, where $C>0$ is a constant, or
\item[(d'')]  $\lim \limits_{i \to \infty} s(b_i) = -\infty$ and $d(b_i,\partial B_-)\to 0$.
\end{itemize}

The sequence $p_i:\widetilde{F}_i\to B_-$ converges in the SFT sense to a $1$- or $2$-level holomorphic building $p_\infty=p_\infty^0$, $p_\infty^0\cup p_\infty^1$, or $p_\infty^{-1}\cup p_\infty^0$, where $p_\infty^{-1}$, $p_\infty^0$, $p_\infty^1$ are branched covers of $B$, $B_-$, $B'$, respectively. We have a $1$-level building in Cases (a) and (b) and a $2$-level building in Cases (c), (d'), and (d''). Let $p_\infty^{\frak m}: \widetilde{F}_\infty^{\frak m}\to B_-$ be the restriction of $p_\infty^0$ to the component containing the limit $\widetilde{\frak m}_\infty$ of $\widetilde{\frak m}_i$ and let $p_\infty^b: \widetilde{F}_\infty^b\to B$ or $B'$ be the restriction of $p_\infty^{-1}$, $p_\infty^0$, or $p_\infty^1$ to the component containing the limit $b_\infty$ of $b_i$. If $p_\infty^{\frak m}=p_\infty^b$ (i.e., in Cases (a) and (b)) we drop the superscripts.

Near all the punctures of $\widetilde{F}_{\infty}^\star$, $\star\in \{{\frak m},b\}$, we use cylindrical or rectangular coordinates $(s,t)$ induced by the coordinates $(s,t)$ on $B_-$ by pullback, after a possible translation in the $s$-direction.

For $*\in\{\varnothing,',-\}$, we denote the compactification of $B^*$, obtained by adjoining the points at infinity ${\frak p}^*_\pm$, by $cl(B^*)$. Similarly, let $cl(\widetilde{F}_\infty^\star)$, $\star\in \{{\frak m},b\}$, be the compactification of $\widetilde{F}_\infty^\star$, obtained by adjoining ${\frak q}^\star_{\pm,j}$, where $j\in\{1\}$ or $\{1,2\}$, depending on the number of ends. (If there is only one end, we suppress the index.) The maps $p_{\infty}^\star: \widetilde{F}_\infty^\star\to B^*$, can be compactified to $p_{\infty}^\star: cl(\widetilde{F}_\infty^\star) \to cl(B^*)$, where ${\frak q}^\star_{\pm,j}$ is mapped to ${\frak p}^*_\pm$.

In Case (a), $\widetilde{F}_{\infty}$ is an annulus with a puncture ${\frak q}_+$ in the interior and a puncture ${\frak q}_{-,j}$ on each boundary component. In Case (b), $\widetilde{F}_{\infty}$ is a disk with one puncture ${\frak q}_+$ in the interior, two punctures ${\frak q}_{-,j}$, $j=1,2$, on the boundary, and two boundary points $b_{\infty,j}$, $j=1,2$, which are glued together to give $b_\infty$.  The points ${\frak q}_{-,j}$ and $b_{\infty,j}$ alternate along the boundary.  In Case (c), $\widetilde{F}_{\infty} ^{\frak m}$ is a disk with a puncture ${\frak q}_+^{\frak m}$ in the interior and a puncture ${\frak q}_-^{\frak m}$ on the boundary, and $\widetilde{F}_{\infty}^b$ is a sphere with three punctures ${\frak q}_+^b$, ${\frak q}_{-,j}^b$, $j=1,2$. In Case (d'), $\widetilde{F}_{\infty}^{\frak m}$ is a disk with a puncture ${\frak q}_+^{\frak m}$ in the interior and two punctures ${\frak q}_{-,j}^{\frak m}$ on the boundary and $\widetilde{F}_{\infty}^b$ is a disk with four punctures ${\frak q}_{\pm,j}^b$ on the boundary. In Case (d''), $\widetilde{F}_{\infty}^{\frak m}$ is a disk with a puncture ${\frak q}_+^{\frak m}$ in the interior and two punctures ${\frak q}_{-,j}^{\frak m}$ on the boundary and $\widetilde{F}_{\infty}^b$ is a disk with four punctures ${\frak q}_{\pm,j}^b$ on the boundary and two boundary points $b_{\infty,j}$, $j=1,2$, identified.

Cases (a) and (b) are similar to Case (S4'), while the situation in Cases (c), (d') and (d'') is complicated by the fact that the limit is a $2$-level holomorphic building.

\s\n {\em Cases (a) and (b).}

\begin{thm} \label{capodanno2ab}
Suppose $\lim \limits_{i \to \infty} b_i = b_{\infty} \in B_-$. Then $(w_i)_{i\in\N}$ converges to a holomorphic map $w_{\infty}: \widetilde{F}_{\infty} \to \C$ such that:
\begin{enumerate}
\item $w_{\infty}(\partial \widetilde{F}_{\infty}) \subset \R^+$;
\item at the positive puncture ${\frak q}_+$,
$$\lim \limits_{s \to + \infty} |w_{\infty}(s,t)| = + \infty, \quad \lim \limits_{s \to + \infty} \dfrac{w_{\infty}(s,t)}{|w_{\infty}(s,t)|} = f_{i_0j_0}(t);$$
\item at the negative punctures ${\frak q}_{-,j}$, $j=1,2$, $\lim \limits_{s \to - \infty} w_{\infty}(s,t) = c_j \in \R^+$;
\item[(4a)] in Case (a), $w_{\infty}$ extends to a holomorphic map $w_{\infty}: cl(\widetilde{F}_{\infty})\to \C \P^1$ such that $w_{\infty}:int(cl(\widetilde{F}_{\infty}))\to\C \P^1 - ([a_1, a_2] \sqcup [a_3, a_4])$ is a biholomorphism;
\item[(4b)] in Case (b), $w_{\infty}$ extends to a holomorphic map $w_{\infty}: cl(\widetilde{F}_{\infty})\to \C \P^1$ such that $w_{\infty}:int(cl(\widetilde{F}_{\infty}))\to \C \P^1 - [a_1, a_2]$ is a biholomorphism;
\item[(5)] $\widetilde{\mathfrak{m}}_\infty$ is the unique zero of $w_{\infty}|_{int(cl(\widetilde{F}_{\infty}))}$ and is simple; $p_\infty(\widetilde{\mathfrak{m}}_\infty) = \overline{\mathfrak{m}}^b$.
\end{enumerate}
\end{thm}

\begin{proof}
The proof of Theorem~\ref{capodanno} goes through without modification to give (1)--(3).

(4a) As in the proof of Theorem~\ref{capodanno}(4), we can define the degree for maps of pairs $(cl(\widetilde{F}_{\infty}), \partial cl(\widetilde{F}_{\infty})) \to (\C \P^1, \R^+)$. The degree of $w_{\infty}$ is $1$ because it has a unique pole of order $1$. The order of the pole at the positive puncture can be computed from the winding number of $f_{i_0j_0}$, which is the smallest one for a positive eigenvalue by Lemma~\ref{lemma: nonzero asymptotic eigenfunction}. Then $w_{\infty}$ can have no branch points in the interior of $cl(\widetilde{F}_{\infty})$. (4b) is similar. (5) follows from (4a) and (4b).
\end{proof}

\s\n
{\em Cases (c), (d') and (d'').} When the sequence $\{ b_i \}$ is unbounded, there are surfaces $\widetilde{F}_i^\star$, $\star\in \{{\frak m},b\}$, with embeddings $\iota_i^\star : \widetilde{F}_i^\star \to \widetilde{F}_i$, such that $\widetilde{F}_i^\star$ converges to  $\widetilde{F}_{\infty}^\star$. Let $p_i^{\frak m}:\widetilde{F}_i^{\frak m}\to B_-$ be the restriction of $p_i$ and let $p_i^b: \widetilde{F}_i^b\to B^*$, $*=\varnothing$ or $'$, be the composition of an $s$-translation and $p_i|_{\widetilde{F}_i^b}$ so that $s$-coordinate of $p_i^b(b_i)$ is zero.

Let $w_i$ be as in Definition~\ref{defn: definiton of w_i}. Let
$$K_i^b = (p_i \circ \iota_i^b)^{-1} (B_-\cap \{ s(b_i)-1 \le s \le s(b_i)+1 \})\subset \widetilde{F}_i^b$$
and $C_i^b = \sup_{z \in K_i^b} |w_i(\iota_i^b(z))|$.  Then we set $w_i^{\mathfrak{m}} = w_i \circ \iota_i^{\mathfrak{m}}$ and $w_i^b = (w_i \circ \iota_i^b)/C_i^b.$

\begin{lemma} \label{natale2}
For every compact set $K \subset B^*$, there is a constant $C_K>0$ such that $|w_i^\star(z)|<C_K$,  and $|(w^\star_i)'(z)|<C_K$ for all $z \in (p_i^\star)^{-1}(K)$ and all $i\in \N$. Here $\star\in \{{\frak m},b\}$ and $*\in\{\varnothing,',-\}$, as appropriate.
\end{lemma}

\begin{proof}
Similar to Lemma~\ref{natale}.
\end{proof}

Lemma~\ref{natale2} implies that the limits $w_{\infty}^{\mathfrak{m}}$ and $w_{\infty}^b$ exist. The following lemma gives the behavior of $w_{\infty}^{\mathfrak{m}}$ and $w_{\infty}^b$ near the punctures.

\begin{lemma} \label{finalmente1} $\mbox{}$
\begin{enumerate}
\item Let $u : \R^+ \times (\R/ \pi \Z) \to \C^\times$ be a finite energy holomorphic map. If the map $t \mapsto u(s,t)$ has degree one for some (and therefore all) $s \in \R^+$, then $\lim \limits_{s \to + \infty} |u_{\infty}(s,t)| = + \infty$ and $\lim \limits_{s \to + \infty} \dfrac{u_{\infty}(s,t)}{|u_{\infty} (s,t)|} = c e^{\pi i t}$ with $c \ne 0$.
\item Let $u : \R^- \times (\R/ \pi \Z) \to \C^\times$ be a finite energy holomorphic map. If the map $t \mapsto u(s,t)$ has degree one for some (and therefore all) $s \in \R^-$, then $\lim \limits_{s \to - \infty} |u_{\infty}(s,t)| = 0.$
\end{enumerate}
\end{lemma}

\begin{proof}
(1) Let us view $u$ as a map $\R^+ \times (\R/ \pi \Z) \to \R \times (\R/ \pi \Z)$. As in the proof of Lemma~\ref{natale}, since $u$ has finite energy, it has bounded derivative. Let $u_n(s,t) = u(s+k_n,t)$, where $k_n\in \R^+$ and $\lim \limits_{n \to + \infty} k_n =
+ \infty$. The sequence $u_n$ has uniformly bounded derivative and converges to a finite energy holomorphic map
$$u_{\infty}:  \R \times (\R/ \pi \Z) \to \R \times (\R/ \pi \Z).$$
Such a holomorphic map is of the form $u_\infty(s,t)=(s+a,t+b)$, where $a,b$ are constants. This implies (1). (2) is similar and is left to the reader.
\end{proof}

\s\n {\em Case (c).}

\begin{thm}\label{capodanno2c}
Suppose $\lim \limits_{i \to \infty} s(b_i)= +\infty$. Then $(w_i^{\mathfrak{m}})_{i\in\N}$ converges to a holomorphic map $w_{\infty}^{\mathfrak{m}} : \widetilde{F}_{\infty}^{\mathfrak{m}} \to \C$ such that:
\begin{enumerate}
\item $w_{\infty}^{\mathfrak{m}}(\partial \widetilde{F}_{\infty}^{\frak m}) \subset \R^+$;
\item at the positive puncture ${\frak q}^{\frak m}_+$,
$$\lim \limits_{s \to + \infty} |w_{\infty}^{\mathfrak{m}}(s,t)| = + \infty, \quad \lim \limits_{s \to + \infty} \dfrac{w_{\infty}^{\mathfrak{m}}(s,t)}{|w_{\infty}^{\mathfrak{m}} (s,t)|} = f(t),$$ where $f$ is a normalized eigenfunction of the asymptotic operator at $\delta_0$ with winding number one;
\item at the negative puncture ${\frak q}^{\frak m}_-$, $\lim \limits_{s \to - \infty} w_{\infty}^{\mathfrak{m}}(s,t) = c \in \R^+$;
\item $w_\infty^{\frak m}$ extends to a holomorphic map $w_\infty^{\frak m}: cl(\widetilde{F}_\infty^{\frak m}) \to \C \P^1$ such that $w_\infty^{\frak m}:int(cl(\widetilde{F}_\infty^{\frak m}))\to\C \P^1 - [a_1, a_2]$ is a biholomorphism; and
\item $\overline{\mathfrak{m}}^b$ is the unique zero of $w_\infty^{\mathfrak{m}}|_{int(cl(\widetilde{F}_\infty^{\frak m}))}$ and is simple.
\end{enumerate}
Also $(w_i^b)_{i\in \N}$ converges to a holomorphic map $w_{\infty}^b : \widetilde{F}_{\infty}^b \to \C$ such that:
\begin{enumerate}
\item[(6)] at the positive puncture ${\frak q}^b_+$,
$$\lim \limits_{s \to + \infty} |w_{\infty}^b(s,t)| = + \infty, \quad \lim \limits_{s \to + \infty} \dfrac{w_{\infty}^b(s,t)}{|w_{\infty}^b(s,t)|} = f_{i_0j_0}(t);$$
\item[(7)] at the negative puncture ${\frak q}^b_{-,1}$ that connects to ${\frak q}^{\frak m}_+$, $\lim \limits_{s \to - \infty} w_{\infty}^b(s,t) = 0$;
\item[(8)] at the other negative puncture ${\frak q}^b_{-,2}$,  $\lim \limits_{s \to - \infty} w_{\infty}^b(s,t) =c \in \R^+$;
\item[(9)] at the punctures ${\frak q}^{\mathfrak{m}}_+$ and ${\frak q}^b_{-,1}$, $\lim \limits_{s \to + \infty} \dfrac{w_{\infty}^{\mathfrak{m}}(s,t)}{|w_{\infty}^{\mathfrak{m}} (s,t)|} = \lim \limits_{s \to  - \infty} \dfrac{w_{\infty}^b(s,t)}{|w_{\infty}^b(s,t)|}$;
\item[(10)] $w_\infty^b$ extends to a biholomorphism $w_\infty^b:\C\P^1 \to \C \P^1$; and
\item[(11)] ${\frak q}^b_{-,1}$ is the unique zero of $w_\infty^b$ and is simple.
\end{enumerate}
\end{thm}

\begin{proof}
The proof of Theorem~\ref{capodanno} goes through with little modification, in view of Lemma~\ref{finalmente1}.  We remark that (8) is a consequence of Convention~\ref{convention for h} and the proof technique of Lemma~\ref{lemma: value of widetilde Phi}.
\end{proof}

\s\n {\em Cases (d') and (d'').}

\begin{thm} \label{capodanno2d}
Suppose $\lim \limits_{i \to \infty} s(b_i) = - \infty$. Then $(w_i^{\mathfrak{m}})_{i\in\N}$ converges to a holomorphic map $w_{\infty}^{\mathfrak{m}} : \widetilde{F}_{\infty}^{\mathfrak{m}} \to \C$ such that:
\begin{enumerate}
\item $w_{\infty}^{\mathfrak{m}}(\partial \widetilde{F}_{\infty}^{\mathfrak{m}}) \subset \R^+$;
\item at the positive puncture ${\frak q}^{\frak m}_+$,
$$\lim \limits_{s \to + \infty} |w_{\infty}^{\mathfrak{m}}(s,t)| = + \infty, \quad \lim \limits_{s \to + \infty} \dfrac{w_{\infty}^{\mathfrak{m}}(s,t)} {|w_{\infty}^{\mathfrak{m}} (s,t)|} = f_{i_0j_0}(t);$$
\item at the negative punctures ${\frak q}^{\frak m}_{-,j}$, $j=1,2$, $\lim \limits_{s \to - \infty} w_{\infty}^{\mathfrak{m}}(s,t) = c_j \in \R^+$;
\item $w_\infty^{\frak m}$ extends to a holomorphic map $w_\infty^{\frak m}:cl(\widetilde{F}^{\frak m}_\infty) \to \C \P^1$ such that $w_\infty^{\frak m}: int(cl(\widetilde{F}^{\frak m}_\infty))\to \C \P^1 - [a_1, a_2]$ is a biholomorphism;
\item $\widetilde{\frak m}_\infty$ is the unique zero of $w_\infty^{\mathfrak{m}}|_{int(cl(\widetilde{F}^{\frak m}_\infty))}$ and is simple.
\end{enumerate}
Also $(w_i^b)_{i\in\N}$ converges to a constant map $w_{\infty}^b : \widetilde{F}_{\infty}^b \to \C$.
\end{thm}

\subsection{Involution lemmas}
\label{subsection: involutions}

In this subsection we collect some lemmas on holomorphic maps between Riemann surfaces with anti-holomorphic involutions.  These lemmas, collectively referred to as the {\em involution lemmas}, will play an important role in Section~\ref{subsection: theorem complement} and in the sequel \cite{CGH-II}.

Our starting point is the following observation, whose proof is straightforward.

\begin{obs} \label{obs: involution}
Let $\Sigma_1,\Sigma_2$ be Riemann surfaces with anti-holomorphic involutions
$\iota_1,\iota_2$, respectively.  If $f:\Sigma_1\to\Sigma_2$ is a holomorphic map,
then $\widetilde f:= \iota_2\circ f\circ \iota_1$ is also holomorphic. Moreover, if
$f=\widetilde{f}$, then $f(\op{Fix}(\iota_1))\subset \op{Fix}(\iota_2)$, where
$\op{Fix}(\iota_i)$ is the fixed point set of $\iota_i$.
\end{obs}

There are four versions of the involution lemma; the first two will be used in this paper and the last two only in the sequel \cite{CGH-II}. We start by introducing some common notation which will be used in all four versions: For $i=1,2$, the Riemann surface $\Sigma_i$ is an open subset of $\C \P^1$ which is invariant under complex conjugation and has finitely generated fundamental group; moreover no component of $\C\P^1-\Sigma_i$ is a single point. The complex conjugation on $\C \P^1$ restricts to an anti-holomorphic involution $\iota_i : \Sigma_i \to \Sigma_i$.
On each $\Sigma_i$ we fix ``radial rays''
$${\mathcal R}_i = \Sigma_i \cap (\R^{\le 0} \cup \{ \infty \}).$$
The {\em asymptotic marker}  $\dot{\mathcal R}_i(0)$ is the connected component of $T_0\R \P^1-\{0\}$ (i.e., a tangent half-line) consisting of vectors with negative $\bdry_x$-component; similarly, the asymptotic marker $\dot{\mathcal{R}}_i(\infty)$ is the component of $T_\infty \R\P^1-\{0\}$ that is mapped to $\dot{\mathcal{R}}_i(0)$ under the inversion $z\mapsto {1\over z}$. The radial rays ${\mathcal R}_i$ and their related asymptotic markers are invariant under the involution $\iota_i$.  In this section we will use the notation $\D$ for the open unit disk in $\C$, considered as a Riemann surface.

\begin{lemma} \label{lemma: compactification of Sigma_i}
Given $\Sigma_i$ as above, there is a compact Riemann surface with boundary
$\overline{\Sigma}_i$ with a biholomorphism $\Sigma_i \stackrel\sim\to int(\overline{\Sigma}_i)$.
Moreover there is an anti-holomorphic involution $\iota_i : \overline{\Sigma}_i \to
\overline{\Sigma}_i$ such that the diagram
$$\xymatrix{
\Sigma_i \ar@{^{(}->}[r]  \ar[d]_{\iota_i} & \overline{\Sigma}_i \ar[d]^{\iota_i} \\
\Sigma_i \ar@{^{(}->}[r] & \overline{\Sigma}_i
}$$
commutes.
\end{lemma}

\begin{proof}[Sketch of proof.]
We outline the proof of the first statement. Use the uniformization theorem to identify the universal cover of $\Sigma=\Sigma_1$ or $\Sigma_2$ with the open upper half space $\H$.  Let $G\subset PSL(2,\R)$ be the deck transformation group of $\H$ such that $\H/G=\Sigma$. If $G$ is finitely generated, then $(\bdry \H - L) /G$ is a collection of boundary circles of $\Sigma$, where $L\subset \bdry \H$ is the limit set of $G$.  Hence $\overline\Sigma=(\overline{\H}-L)/G$.
\end{proof}

Note that $\overline{\Sigma}_i$ will not necessarily be the
closure of $\Sigma_i$ in $\C \P^1$. However, when referring to points in the interior of
$\overline{\Sigma}_i$, we denote them by the corresponding
point in $\Sigma_i$. In the same way we view the radial rays ${\mathcal R}_i$ as
subsets of $\overline{\Sigma}_i$, and the asymptotic markers as tangent half-lines to
$\overline{\Sigma}_i$ (provided they make sense). Lemma \ref{lemma: compactification of Sigma_i} implies that
they are invariant by the involution on $ \overline{\Sigma}_i$.

The image of an asymptotic marker by a holomorphic function is defined by the differential at the regular points. At the singular points the local behavior of a
holomorphic map still allows us to define the image of a tangent ray.

For the first version of the involution lemma, let
$$\Sigma_1= \C\P^1-([a_1,a_2]\cup\dots\cup[a_{2p-1},a_{2p}]),$$
where $a_i\in\R^+$ and $a_1<\dots<a_{2p}$, and let $\Sigma_2= \D$. We write $\bdry\overline\Sigma_1
=\bdry_1\overline\Sigma_1\sqcup\dots \sqcup\bdry_p\overline\Sigma_1$, where
$\bdry_i\overline\Sigma_1$, $i=1,\dots,p$, corresponds to the slit $[a_{2i-1},a_{2i}]$.

\begin{lemma}[Involution Lemma, Version 1] \label{lemma: effect of involutions}
Let  $f:\overline\Sigma_1\to\overline\Sigma_2$ be a holomorphic map which is a
$q$-fold branched cover with $q\geq p$, such that:
\begin{itemize}
\item[(i)] $f(\bdry_i\overline\Sigma_1)=\bdry \overline\Sigma_2$, $i=1,\dots,p$;
\item[(ii)] $f^{-1}(0)=\{\infty\}$;  and
\item[(iii)] $f(0)\in {\mathcal R}_2$.
\end{itemize}
Then $f$ maps $\op{Fix}(\iota_1)$ to $\op{Fix}(\iota_2)$ and $\dot{\mathcal R}_{1}(\infty)$ to $\dot{\mathcal R}_{2}(0)$.
\end{lemma}

\begin{figure}[ht]
\s
\begin{overpic}[width=8cm]{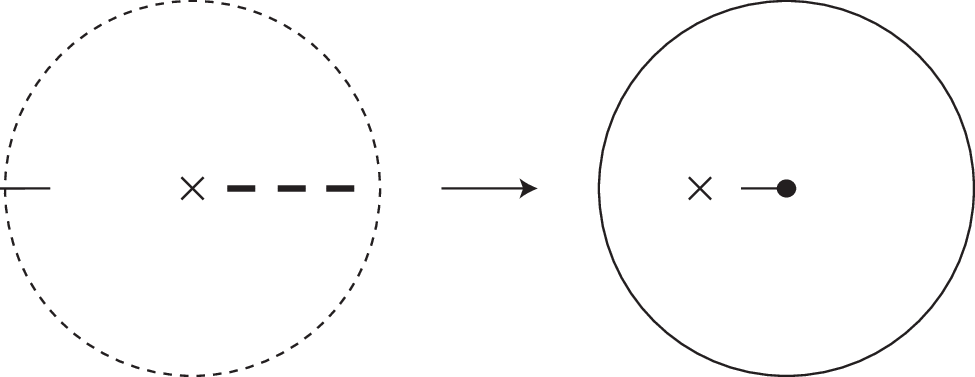}
\end{overpic}
\caption{The map $f|_{\Sigma_1}$ in Lemma~\ref{lemma: effect of involutions}. The left-hand side is $\Sigma_1$, where the point at infinity is  represented by the dotted circle, and the right-hand side is $\Sigma_2$.  The thicker lines on left-hand side are the slits $[a_{2i-1},a_{2i}]$ and the thinner lines are the asymptotic markers. The cross on the left is the origin and the cross on the right is $f(0)\in \mathcal{R}_2$.  Lemma~\ref{lemma: effect of involutions} states that $f$ maps the asymptotic marker $\dot{\mathcal R}_{1}(\infty)$ at $\infty\in\Sigma_1$ to the asymptotic marker $\dot{\mathcal R}_{2}(0)$ at $0\in \Sigma_2$.}
\label{figure: involution}
\end{figure}

\begin{proof}
We claim that $f=\widetilde{f}= \iota_2\circ f \circ \iota_1$. To compare $f$ and $\widetilde{f}$, we consider their quotient $Q=f/\widetilde{f}$. We observe three facts:
\begin{enumerate}
\item $Q$ has no poles, since $f$ and $\widetilde{f}$ have zeros only at $\infty$ and their orders agree.
\item $|Q(z)|=1$ for all $z \in \partial \overline{\Sigma}_1$,
  so the maximum modulus theorem implies that $Q(\overline{\Sigma}_1) \subset \overline{\D}=\{|z|\leq 1\}$.
\item The degree of $Q|_{\bdry_i\overline\Sigma_1}$, viewed as a map to $S^1=\partial \D$, is zero, since $$\op{deg}(f|_{\bdry_i\overline\Sigma_i})= \op{deg}(\widetilde{f}|_{\bdry_i\overline\Sigma_i}).$$
\end{enumerate}
If $Q$ is not constant, then by (3) there must be a branch point of $Q$ along $\bdry_i\overline\Sigma_1$.  In particular, $Q(\overline\Sigma_1) \not \subset \D$, which contradicts (2). Hence $Q$ is a constant map and $f=c\widetilde{f}$ for some $c\in \C-\{0\}$. Now (iii) implies that $c=1$, and we have $f=\widetilde{f}$.

Finally we apply Observation~\ref{obs: involution} to conclude that $f$ maps $\op{Fix}(\iota_1)$ to $\op{Fix}(\iota_2)$ and, by (iii), maps $\dot{\mathcal R}_1(\infty)$ to $\dot{\mathcal R}_2(0)$.
\end{proof}

The proofs of the other versions of the involution lemma are similar, and will be omitted.

\begin{lemma}[Involution Lemma, Version 2]\label{lemma: effect 2}
Let $\Sigma_1=\Sigma_2=\C\P^1$ and $a_i\in\R^{\geq 0}$ with $a_1=0<a_2<\dots<a_{p}$. Assume $f:\Sigma_1\to \Sigma_2$ is a holomorphic map which is a $q$-fold branched cover with $q\geq p$, such that:
\begin{itemize}
\item[(i)] $f^{-1}(\infty)=\{a_1,\dots,a_p\}$;
\item[(ii)] $f^{-1}(0)=\{\infty\}$; and
\item[(iii)] $f$ maps $\dot{\mathcal R}_{1}(0)$ to $\dot{\mathcal R}_{2}(\infty)$.
\end{itemize}
Then $f$ maps $\op{Fix}(\iota_1)$ to $\op{Fix}(\iota_2)$ and $\dot{\mathcal R}_{1}(\infty)$ to $\dot{\mathcal R}_{2}(0)$.
\end{lemma}

\begin{lemma}[Involution Lemma, Version 3]\label{lemma: effect 3}
Let $\Sigma_1=\Sigma_2=\D$ and $a_i\in\R^{\geq 0}$ with $a_1=0<a_2<\dots<a_{p}<1$. Assume  $f:\overline{\Sigma}_1\to \overline{\Sigma}_2$ is a holomorphic map which is a $q$-fold branched cover with $q\geq p$, such that:
\begin{itemize}
\item[(i)] $f^{-1}(0)=\{a_1,\dots,a_p\}$;
\item[(ii)] $f^{-1}(\bdry \overline{\Sigma}_1)=\bdry \overline{\Sigma}_2$; and
\item[(iii)] $f$ maps $\dot{\mathcal R}_{1}(0)$ to $\dot{\mathcal R}_{2}(0)$.
\end{itemize}
Then $f$ maps $\op{Fix}(\iota_1)$ to $\op{Fix}(\iota_2)$. Moreover $f(-1)=-1$.
\end{lemma}

For the fourth version of the involution lemma, we consider $\Sigma_1=\D-([a_1,a_2]\cup\dots\cup[a_{2p-1},a_{2p}])$, where $a_i\in(0,1)$ with
$a_1<\dots<a_{2p}$, and $\Sigma_2=\{R < |z| < 1\}$, where $0<R<1$.  We write $\bdry\overline{\Sigma}_1= \bdry_0\overline\Sigma_1\sqcup\dots
\sqcup \bdry_p \overline\Sigma_1$, where $\bdry_0\overline\Sigma_1$ is the boundary component which can identified with $\{|z|=1\}$ and $\bdry_i\overline\Sigma_1$, $i=1,\dots,p$, corresponds to the slit $[a_{2i-1},a_{2i}]$.

\begin{lemma}[Involution Lemma, Version 4] \label{lemma: effect 4}
Let  $f:\overline\Sigma_1\to \overline{\Sigma}_2$ be a holomorphic map which is a
$q$-fold branched cover with $q\geq p$, and $\mathcal{I}\subset \{1,\dots,p\}$, such that:
\begin{itemize}
\item[(i)] $f(\bdry_i\overline\Sigma_1)=\{|z|=1\}$ when $i\in \mathcal{I}$;
\item[(ii)] $f(\bdry_i\overline\Sigma_1)=\{|z|=R\}$ when $i=0$ or $i\not\in \mathcal{I}$; and
\item[(iii)] $f(0)\in{\mathcal R}_2$.
\end{itemize}
Then $f$ maps  $\op{Fix}(\iota_1)$ to $\op{Fix}(\iota_2)$. Moreover $f(-1)=-R$.
\end{lemma}

See Figure~\ref{figure: involution2}.

\begin{figure}[ht]
\s
\begin{overpic}[width=8cm]{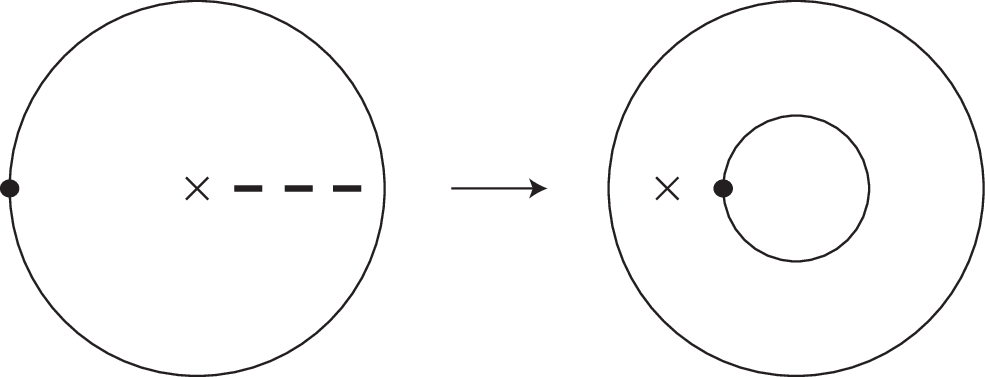}
\end{overpic}
\caption{The map $f|_{\Sigma_1}$ in Lemma~\ref{lemma: effect 4}. The left-hand side is $\Sigma_1$ and the right-hand side is $\Sigma_2$.  The thick lines on the left-hand side are the slits $[a_{2i-1},a_{2i}]$.  The cross on the left is the origin and the cross on the right is $f(0)\in \mathcal{R}_2$. The dots represent $-1\in \bdry\Sigma_1$ and $f(-1)\in \bdry\Sigma_2$. Lemma~\ref{lemma: effect 4} states that $f(-1)=-R$.}
\label{figure: involution2}
\end{figure}

\subsection{Elimination of some cases}
\label{subsection: theorem complement}

We are now in a position to eliminate some of possibilities that appear in
Theorem~\ref{thm: compactness for W minus I=3, first version}:

\begin{thm} \label{thm: complement}
Suppose $m\gg 0$, $\varepsilon,\delta>0$ are sufficiently small, $\overline{u}_\infty\in \bdry \mathcal{M}^i_{\overline{\frak m}}(\varepsilon,\delta,p)$, and $\overline{v}'_0\not=\varnothing$.
\begin{enumerate}
\item[(i)] If $i=2$, then the $3$-level subbuilding described in Theorem~\ref{thm: compactness for W minus I=2, first version} does not occur.
\item[(ii)] If $i=3$, then Cases (2)--(6) of Theorem~\ref{thm: compactness for W minus I=3, first version} do not occur.
\end{enumerate}
\end{thm}

We briefly sketch the idea of the proof. In all the cases that are eliminated by Theorem~\ref{thm: complement}, the (unique) component of $\cup_{j=1}^a \overline{v}_j^\sharp\subset \overline{u}_\infty$ which is asymptotic to a multiple of $\delta_0$ at the negative end has ECH index $I=1$.

Suppose that, for $m\gg 0$, there is a sequence of holomorphic curves which converges to a configuration $\overline{u}_\infty$ that we want to exclude. In Section \ref{subsection: rescaling}, we applied a rescaling argument to construct a holomorphic building which keeps track of how the limit $\overline{u}_\infty$ is approached; this is similar to the layer structures of Ionel-Parker~\cite[Section 7]{IP1}. In the simplest case, this building is a holomorphic map $w_{\infty} : B_- \to \C$ which satisfies the asymptotic condition
$$\lim \limits_{s \to + \infty} \dfrac{w_{\infty}(s,t)}{|w_{\infty}(s,t)|} = f_{i_0j_0}(t),$$
where $f_{i_0j_0}(t)$ is a normalized asymptotic eigenfunction of an $I=1$ curve with a negative end asymptotic to $\delta_0$. The condition $I=1$ is used as follows: since there are only finitely many $I=1$ curves with negative ends asymptotic to $\delta_0$, we may assume that $-1 \not \in \{ f_{i_0j_0}(\frac 32) \}$ as explained in Section~\ref{subsubsection: radial rays} and Remark~\ref{rmk: good ray is real}.

\begin{rmk}\label{rmk: good identification}
In this subsection we identify $cl(B_-)\simeq \overline{\mathbb D}$ so that $\mathfrak{p}_+$ corresponds to $0$ and $\mathfrak{p}_-$ corresponds to $1$. There is an anti-holomorphic involution $\iota$ on $B_-$ that fixes the half-line $\{t= \frac 32 \}$, and $\{ t = \frac 32 \}$ corresponds to the radial ray ${\mathcal R} = \overline{\mathbb D} \cap \R^{\le 0}$ by Observation \ref{obs: involution}. In particular, $\overline{\mathfrak{m}}^b$ is mapped to a point on ${\mathcal R}$.

Similarly we identify $cl(B')\simeq \C\P^1$ so that ${\frak p}_+$ corresponds to $0$, ${\frak p}_-$ corresponds to $\infty$, and $\{t={3\over 2}\}$ corresponds to the radial ray $\R^{\leq 0}$,
\end{rmk}

We now use the involutions lemmas from Section~\ref{subsection: involutions} to obtain a contradiction.
By the involution lemmas and the symmetric placement of the basepoint $\overline{\mathfrak{m}}^b$, we obtain that $w_{\infty}\circ \iota = i\circ w_\infty$. Hence $\lim \limits_{s \to + \infty} \dfrac{w_{\infty}(s,{3\over 2})}{|w_{\infty}(s,{3\over 2})|} =-1$, which contradicts $-1 \not \in \{ f_{i_0j_0}(\frac 32) \}$.

\subsubsection{Proof of Theorem~\ref{thm: complement}}

In this subsection we use limiting arguments in which $m\to \infty$ and $\overline{\hh}_m\to \overline{\hh}_\infty$; see Section~\ref{limit m to infty}.  Hence many of the almost complex structures and moduli spaces will have an additional subscript $m$, where $m=\infty$ is also a possibility. For example, we will use the notation $\mathcal{J}'_{\overline{W'}_m}$ and $\mathcal{J}_{\overline{W}_{-,m}}$ introduced in Section~\ref{limit m to infty} to refer to $\mathcal{J}_{\overline{W'}}$ and $\mathcal{J}_{\overline{W}_-}$ with respect to $m$.

Let $\overline{J'}_{\infty}\in \mathcal{J}_{\overline{W'}_{\infty}}^\star$ and let $\overline{J'}_{m}\in \mathcal{J}_{\overline{W'}_m}^\star$ be a nearby almost complex structure with respect to the integer $m\gg 0$. Let $\overline{J}_{-,m}\in \mathcal{J}_{\overline{W}_{-,m}}^{reg}$ be an almost complex structure which restricts to $\overline{J'}_{m}$ and let $\overline{J}_{-,m}^\Diamond$ be $(\varepsilon,U)$-close to $\overline{J}_{-,m}$.

We will treat Theorem~\ref{thm: compactness for W minus
I=3, first version} in detail and leave Theorem~\ref{thm: compactness for W minus
I=2, first version} to the reader. Suppose that,
for a sequence $m_i\to \infty$, there sequences  $\overline{u}_{ij}$ of
$\overline{J}_{-,m_i}^\Diamond$-holomorphic curves which converge to a
$\overline{J}_{-,m_i}^\Diamond$-holomorphic building $\overline{u}_{i \infty}$ falling into
one of Cases (2)--(6).

\s\n
{\em Elimination of Case (2).}
Suppose for each $i$ the sequence $\overline{u}_{ij}$ converges to a building $\overline{u}_{i \infty}$ satisfying Case (2). By Theorem \ref{capodanno}, we obtain a holomorphic map $w_{\infty}: cl(B_-) \to \C\P^1$, whose restriction to $int(cl(B_-))$ is a biholomorphism onto its image.

We apply the Involution Lemma~\ref{lemma: effect of involutions} to obtain a contradiction:  Let  $\overline{\Sigma}_1$ be the compactification of $\Sigma_1= \C\P^1 - [a_1,a_2]$ and let $\overline{\Sigma}_2 = cl(B_-)$ be identified with $\overline{\D}$ as in Remark
\ref{rmk: good identification}. Let $f: \overline{\Sigma}_1 \to \overline\Sigma_2$ be the extension of $(w_{\infty}|_{int(cl(B_-))})^{-1}$. Such an extension exists because $\Sigma_1$ is biholomorphic to the open unit disk and biholomorphisms of the open unit disk extend continuously to the boundary.

By Lemma~\ref{lemma: effect of involutions}, $f$ maps $\dot{\mathcal{R}}_1(\infty)$ to $\dot{\mathcal{R}}_2(0)$ and, conversely, $w_{\infty}$ maps $\dot{\mathcal{R}}_2(0)$ to $\dot{\mathcal{R}}_1(\infty)$. Since the asymptotic marker $\dot{\mathcal{R}}_2(0)$ in $\overline\Sigma_2$ corresponds to the asymptotic marker $\{t= \frac 32\}$ for ${\frak p}_+\in cl(B_-)$ by Remark \ref{rmk: good identification},  $\dot{\mathcal{R}}_1(\infty)$ is a bad radial ray (in the sense of Definition \ref{defn: radial rays}) by Theorem~\ref{capodanno}(2). This contradicts Remark~\ref{rmk: good ray is real}, so we have eliminated Case (2).

\s\n
{\em Elimination of Cases (3) and (4).}
We will treat Case (4); Case (3) is almost identical. Suppose for each $i$ the sequence $\overline{u}_{ij}$ converges to a building $\overline{u}_{i \infty}$ satisfying Case (4).  By Remark~\ref{rmk: number of branch points}, for each $i$, the total number of branched points of $\cup_{j=-b}^a\overline{v}'_{j,i}$ is one. If we exercise some care in choosing the diagonal sequence in Lemma \ref{lemma: pre-riscaling estimates}, we can divide the argument for Case (4) further into Subcases (a), (b), (c), (d') and (d'') as in Section~\ref{subsection: Case (S4'')}, depending on the behavior of the branch points of the maps $\overline{\pi}_{B_-} \circ \overline{v}'_{0,i}$.

\s\n
{\em Subcases (a) and (b).}
By Theorem~\ref{capodanno2ab}, we obtain holomorphic maps
$$w_{\infty}: cl(\widetilde{F}_{\infty}) \to \C\P^1, \quad p_\infty: \widetilde{F}_{\infty} \to B_-,$$
where $w_\infty|_{int(cl(\widetilde{F}_\infty))}$ is a biholomorphism onto its image $\Sigma_1 = \C \P^1 - ([a_1, a_2] \cup [a_3, a_4])$ and $p_\infty$ is a branched double cover with one branch point. Let $\overline{\Sigma}_1$
be the compactification of $\Sigma_1$ as in Lemma \ref{lemma: compactification of
Sigma_i}, and let $\overline{\Sigma}_2 = cl(B_-)$ be identified with $\overline{\D}$ as in Remark \ref{rmk: good identification}. We define $f : \overline{\Sigma}_1 \to \overline\Sigma_2$ as the extension of $p_{\infty} \circ (w_{\infty}|_{int(cl(\widetilde{F}_\infty))})^{-1}$.
Such an extension exists because $\Sigma_1$ is biholomorphic to an open annulus, and
biholomorphisms of the annulus always extend to the boundary. At this point we apply Lemma~\ref{lemma: effect of
involutions} as in Case (2) to obtain a contradiction. Case (b) is completely analogous
and can be excluded in the same way.

\s\n
{\em Subcase (c).}
By Theorem \ref{capodanno2c}, we obtain pairs of holomorphic maps
\begin{align*}
w_{\infty}^b  : cl(\widetilde{F}_{\infty}^b) \to \C \P^1, & \quad p_\infty^b :  \widetilde{F}^b_{\infty} \to B', \\
w_{\infty}^\mathfrak{m}  : cl(\widetilde{F}_{\infty}^\mathfrak{m}) \to \C \P^1, & \quad p_\infty^\mathfrak{m} : \widetilde{F}^\mathfrak{m}_{\infty} \to B_.
\end{align*}
 The map $w_{\infty}^{\mathfrak{m}}$ restricts to a biholomorphism of $int(cl(\widetilde{F}_{\infty}^{\mathfrak{m}}))$ with $\Sigma_1^{\mathfrak{m}} = \C \P^1 - [a_1, a_2]$ and the map $p_\infty^{\frak m}$ is a biholomorphism because it has degree $1$.
Let $\overline{\Sigma}_1^{\mathfrak{m}}$ be the compactification of $\Sigma_1^{\mathfrak{m}}$. Identify $\overline{\Sigma}_2^{\mathfrak{m}}=cl(B_-)$ with $\overline{\D}$ as in Remark \ref{rmk: good identification}. Let $f^{\mathfrak{m}}: \overline{\Sigma}_1^{\mathfrak{m}} \to
\overline\Sigma_2^{\mathfrak{m}}$ be the extension of $p_{\infty}^{\mathfrak{m}} \circ
(w_{\infty}^{\mathfrak{m}}|_{int(cl(\widetilde{F}_{\infty}^{\frak m}))})^{-1}$. As in Case (2), Lemma~\ref{lemma: effect of involutions} implies that $f^{\mathfrak{m}}$ maps
$\dot{\mathcal{R}}_1^{\mathfrak{m}}(\infty)$ to $\dot{\mathcal{R}}^{\mathfrak{m}}_2(0)$ and, conversely, $w_{\infty}^{\mathfrak{m}}$ maps $\dot{\mathcal{R}}_2^{\mathfrak{m}}(0)$ to $\dot{\mathcal{R}}_1^{\mathfrak{m}}(\infty)$.

Next we consider the ``upper level'' $(w_\infty^b, p_\infty^b)$. The map $w_\infty^b: cl(\widetilde{F}^b_\infty)\to \C\P^1$ is a biholomorphism and the map $p_\infty^b: \widetilde{F}^b_\infty\to B'$ is a branched double cover with one branch point.  We define $f^b=p_\infty^b\circ (w^b_\infty)^{-1}: \C\P^1 \to \C\P^1$, using the identification $cl(B')\simeq \C\P^1$ from Remark~\ref{rmk: good identification}. By Theorem~\ref{capodanno2c}(9) and the previous paragraph, $f^b$ maps $\dot{\mathcal{R}}_1(0)$ to $\dot{\mathcal{R}}_2(\infty)$. Then, by
Lemma~\ref{lemma: effect 2} and Theorem~\ref{capodanno2c}(7),(8), $f^b$ maps $\dot{\mathcal{R}}_1(\infty)$ to $\dot{\mathcal{R}}_2(0)$. As in Case (2), this
is a contradiction because $\dot{\mathcal{R}}_1(\infty)$ is a good radial ray.

\s\n
{\em Subcases (d') and (d'').}
By Theorem \ref{capodanno2d}, we obtain a holomorphic map $w_{\infty}^{\mathfrak{m}}
: cl(\widetilde{F}_{\infty}^{\mathfrak{m}}) \to \C \P^1$ which restricts to a biholomorphism of
$int(cl(\widetilde{F}_{\infty}^{\mathfrak{m}}))$ with $\C \P^1 - [a_1, a_2]$. Then the proof
proceeds as in Case (2).

\s\n
{\em Elimination of Cases (5) and (6).}
The limit configurations of Cases (5) and (6) must contain a connector over $\delta_0$.
This implies that Theorem~\ref{capodanno2c} applies, and we obtain pairs of holomorphic
maps
\begin{align*}
w_{\infty}^b  : cl(\widetilde{F}_{\infty}^b) \to \C \P^1, & \quad p_\infty^b :  \widetilde{F}^b_{\infty} \to B', \\
w_{\infty}^\mathfrak{m}  : cl(\widetilde{F}_{\infty}^\mathfrak{m}) \to \C \P^1, & \quad p_\infty^\mathfrak{m} : \widetilde{F}^\mathfrak{m}_{\infty} \to B_.
\end{align*}
Then the argument of Case (4c) applies.

\s
This completes the proof of Theorem~\ref{thm: complement}.

\subsection{Proof of Lemma~\ref{claim 1}} \label{proof of lemma}

We begin with the following corollary of Theorem~\ref{thm: complement}. Recall the notation from Convention~\ref{convention}, Step 2 of Section~\ref{subsection: outline of proof} and Equation~\eqref{strange notation}.

\begin{cor}\label{alternative}
Suppose $m\gg 0$ and $\varepsilon,\delta>0$ are sufficiently small constants. If $\overline{u}\in \mathcal{M}_{\overline{\frak m}}^2(\varepsilon,\delta,p)$ or  $\mathcal{M}_{\overline{\frak m}}^{3,(r_0)}(\varepsilon,\delta,p)$, then $\op{Im}(\overline{u}) \cap K_{p,2\delta}\not=\varnothing$. In the latter case, $r_0\gg 0$ and $\varepsilon,\delta>0$ are sufficiently small constants which depend on $r_0$.
\end{cor}

\begin{proof}
Arguing by contradiction, suppose there are sequences $\varepsilon_i,\delta_i\to 0$ and $\overline{u}_{i}\in \mathcal{M}_{\overline{\frak m}}^{3,(r_0)}(\varepsilon_i,\delta_i,p)$ such that $\op{Im}(\overline{u}_i)\cap K_{p,2\delta_i}=\varnothing$. Then the limit $\overline{u}_{\infty}\in \bdry\mathcal{M}_{\overline{\frak m}}^3(0,0,p)$ of $\overline{u}_{i}$ has a nontrivial $\overline{v}_0'$ component.  By Theorems~\ref{thm: compactness for W minus I=3, first version} and \ref{thm: complement}(ii), $\overline{u}_\infty$ satisfies Case (1) of Theorem~\ref{thm: compactness for W minus I=3, first version}. Hence, for $i\gg 0$, $\overline{u}_i\in G({\frak P}_{(r_0)})$, which is a contradiction. The case of $\overline{u}\in \mathcal{M}_{\overline{\frak m}}^2(\varepsilon,\delta,p)$ is easier and is a consequence of Theorems~\ref{thm: compactness for W minus I=2, first version} and \ref{thm: complement}(i).
\end{proof}

\begin{proof}[Proof of Lemma~\ref{claim 1}]
  Suppose that $\overline{u}_\infty\in \bdry_1\mathcal{M}^{3,(r_0)}_{\overline{\frak m}}$.  We are in the situation of Lemma~\ref{lemma: preliminary restrictions part 1}. If $\overline{v}_0$ is a degenerate $\overline{W}_-$-curve, then $I(\overline{v}_0)\geq 4$ by Equation~\eqref{persimmon} since $g\geq 1$. Therefore degenerate $\overline{W}_-$-curves are ruled out by Constraint (i) in the proof of Theorem~\ref{thm: compactness for W minus I=3, first version}. Hence $\overline{v}_0$ is a $\overline{W}_-$-curve and the other levels $\overline{v}_j$, $j\not=0$, are multisections of $W'$ or $W$. In this case, the curve $\overline{v}_0$ is simply-covered. Corollary~\ref{alternative} implies that the component through $\overline{\frak m}$ also intersects $K_{p,2\delta}$. Hence passing through $\overline{\frak m}$ is a generic condition and we must have $\op{ind}(\overline{v}_0)\geq 2$ and $I(\overline{v}_0)\geq 2$. This gives us two options for $\overline{u}_\infty = \overline{v}_{-b} \cup \ldots \cup \overline{v}_a$: either
\begin{enumerate}
\item[($\alpha$)] $b=0$, $I(\overline{v}_a)=1$, and $I(\overline{v}_j)=0$, $j=1,\dots,a-1$; or
\item[($\beta$)] $a=0$, $b=1$, and $I(\overline{v}_{-1})=1$.
\end{enumerate}
Ghost components are not possible since each ghost component takes up $\op{ind}\geq 2$ by Lemma~\ref{lemma: no ghosts}. Hence $\overline{u}_\infty\in A_1\cup A_2$.
\end{proof}

\subsection{Gluing} \label{subsection: gluing for psi}

The goal of this subsection is to prove Theorem~\ref{thm: transversality of ev map}.

\subsubsection{The moduli space $\mathcal{N}$} \label{moduli space N}

Let $\varepsilon={\pi\over m}$, where $m$ is a sufficiently large integer and let $\eta_\varepsilon: [-\pi,\pi]\to \R$ be a smooth function such that:
\begin{itemize}
\item $\eta_\varepsilon(\theta)=\varepsilon$ for $-\pi\leq \theta\leq \theta_1$;
\item $\eta_\varepsilon(\theta)=0$ for $\theta_2\leq \theta \leq \pi$; and
\item $\eta_\varepsilon$ is monotonically decreasing for $\theta_1\leq \theta\leq \theta_2$;
\end{itemize}
for some $-\pi<\theta_1<\theta_2<\pi$.

We use two coordinate systems on $B_-$: (i) the usual coordinates $(s,t) \in \R \times [0,2]/ (0 \sim 2)$ and (ii) complex coordinates via an identification of $cl(B_-)$ with $\overline{\mathbb D}$ such that $\mathfrak{p}_+$ is mapped to $0$ and $\mathfrak{p}_-$ is mapped to $-1$. Then $\bdry B_-$ is parametrized in an orientation-preserving manner by a coordinate $e^{i\theta}$, $\theta \in (- \pi, \pi)$.

\begin{defn}\label{definition of N}
Let $\mathcal{N}=\mathcal{N}_{\eta_\varepsilon}$ be the space of holomorphic maps $w: B_- \to \C$ such that the following properties hold:
\begin{enumerate}
\item[(N$_1$)] $w(e^{i\theta})\in \R^+ \cdot e^{i\eta_\varepsilon(\theta)}$ for all $\theta\in(-\pi,\pi)$;
\item[(N$_2$)] $\displaystyle \lim_{s \to -\infty} \left| w(s,t)-c_1 e^{-\varepsilon(s + i t - i)} \right| < \infty$ for some $c_1\in \R^+$; and
\item[(N$_3$)] $\displaystyle \lim_{s \to +\infty} \left| w(s,t)-c_2 e^{\pi(s+it)} \right | < \infty$ for some $c_2\in \C^\times$.
\end{enumerate}
\end{defn}

In particular:
\begin{enumerate}
\item[(N$_4$)] $\deg(w)=1$ away from the sector $\{0\leq \phi\leq \varepsilon\} \subset \C\P^1$; and
\item[(N$_5$)] after composing with the chosen identification $cl(B_-) \cong \overline{\mathbb D}$, $w$ extends continuously to $\overline\D$ so that $w(0)=w(-1)=\infty$.
\end{enumerate}
Multiplication by a real constant gives an $\R^+$-action on $\mathcal{N}$.

Even though ${\mathcal N}$ is the space we are interested in, it will be convenient for technical reasons to regard ${\mathcal N}$ as an open subset of a vector space $\widetilde{\mathcal N}$ obtained by relaxing properties (N$_1$)--(N$_3$).

\begin{defn}\label{definition of N tilde}
Let $\widetilde{\mathcal N}=\widetilde{\mathcal N}_{\eta_\varepsilon}$ be the space of
holomorphic maps $w: B_- \to \C$
such that the following properties hold:
\begin{enumerate}
\item[($\widetilde{\text N}_1$)] $w(e^{i\theta})\in \R \cdot e^{i\eta_\varepsilon(\theta)}$ for all $\theta\in(-\pi,\pi)$;
\item[($\widetilde{\text N}_2$)] $\displaystyle \lim_{s \to - \infty} \left| w(s,t)-c_1 e^{- \varepsilon(s+it-i)} \right| < \infty$ for some $c_1\in \R$; and
\item[($\widetilde{\text N}_3$)] $\displaystyle \lim_{s \to +\infty} \left| w(s,t)-c_2 e^{\pi(s+it)} \right | < \infty$ for some $c_2\in \C$.
\end{enumerate}
\end{defn}
In order to compute the dimension of $\widetilde{\mathcal N}$, we identify $B_-$ with $\overline{\D} - \{ -1,0 \}$ and $\widetilde{\mathcal N}$ with the space of the holomorphic sections of a holomorphic line bundle $E \to \overline{\mathbb D}$ with values in a real rank one subbundle $F$ along $\partial \overline{\mathbb D} - \{ -1 \}$.

We construct the bundles $E$ and $F$ as follows. Consider a cover of $\overline{\mathbb D}$ by three open sets
$$U_0= \overline{\mathbb D} - \{ 0,1 \},~U_1 = \{ z \in {\mathbb D} ~|~ |z| < 1/3 \},~U_2 = \{ z \in \overline{\mathbb D} ~|~ |z+1| < 1/3 \}.$$
Over each open set we take a trivial line bundle $E_i = \C \times U_i \to U_i$ and define the bundle $E$ by gluing the bundles $E_i$ via the transition maps
\begin{align*}
\psi_1 & : E_0|_{U_0 \cap U_1} \to E_1|_{U_0 \cap U_1}, \quad \psi_1(z,v) = (z, zv), \\
\psi_2 & : E_0|_{U_0 \cap U_2} \to E_2|_{U_0 \cap U_2}, \quad \psi_2(z,v)= \left ( z, i \left (
\frac{z+1}{-z+1} \right )v \right ).
\end{align*}
Let $\pi_{E_i}: E_i=\C\times U_i\to \C$ be the projections corresponding to the trivializations.
If we parametrize $\partial \overline{\mathbb D} - \{ -1 \}$ by $\theta \in (-\pi, \pi)$, the subbundle $F$ is given, as a subbundle of $E_0$, by $F(\theta)= \R \cdot e^{i \eta_\varepsilon(\theta)}$.

It is convenient to view $\Sigma = \overline{\mathbb D} - \{ -1,0\}$ as a surface with a negative strip-like end and a positive cylindrical end and $E$ as a line bundle over $\Sigma$. Let $(- \infty, -R) \times [0,1]$ be a strip with coordinates $(s,t)$.  We identify the strip-like end $Z$ of $\Sigma$ with $(- \infty, -R) \times [0,1]$ via the map
$$\phi : (- \infty, -R) \times [0,1] \to \Sigma, \quad \phi(s,t) = \frac{e^{\pi(s+it)}-i}{e^{\pi(s+it)}+i}.$$
Here the fractional linear transformation $B(\zeta)={\zeta-i\over \zeta+i}$ maps the upper half plane $\H$ to the unit disk $\D$ and $0$ to $-1$. Observe that $A(\zeta)= i({\zeta+1\over -\zeta+1})$ which appears in the definition of $\psi_2$ is the inverse of $B(\zeta)$. Hence the gluing map $\psi_2$ becomes
$$\psi_2((s,t), v)= ((s,t), e^{\pi(s+it)}),$$
with respect to coordinates $(s,t)$.

The linear Cauchy-Riemann operator
$$D: W^{1,p}(E,F) \to L^p(T^{0,1}\Sigma \otimes_{\C}E)$$
is Fredholm for $p>2$ and its kernel consists of smooth holomorphic functions. We denote by $H^0(E,F)$ its kernel and by $H^1(E,F)$ its cokernel.

\begin{lemma} \label{N tilde and E}
There is an identification $H^0(E, F)\cong \widetilde{\mathcal N}$ for every choice of $\eta_{\varepsilon}$.
\end{lemma}

\begin{proof}
The isomorphism $H^0(E,F) \cong \widetilde{\mathcal N}$ associates to a holomorphic section $\xi \in H^0(E,F)$ the holomorphic function $\pi_{E_0}\circ \xi: \Sigma \to \C$, i.e., $\pi_{E_0}\circ \xi$ is obtained by writing $\xi|_{U_0}$ with respect to the trivialization of $E_0$. On the negative end $Z$ we can write
$$\pi_{E_2}\circ\xi(s,t) = \sum \limits_{n \ge 1} c_n e^{(n \pi - \varepsilon)(s+it)+i \varepsilon},$$
since $\pi_{E_2}\circ\xi$ is holomorphic. By applying the transition function $\psi_2^{-1}$, we obtain that the leading term of $\pi_{E_0}\circ \xi$ on $Z$ is $e^{- \varepsilon (s + it) + i \varepsilon}$, which is condition ($\widetilde{\text N}_2$) in Definition \ref{definition of N tilde}. For a similar reason $\pi_{E_0}\circ \xi$ has a pole at $0$ of order at most $1$.
\end{proof}

We will consider also the compactified surface $\check{\Sigma}$ obtained by adding the ``segment at infinity'' to the strip-like end $Z$. Alternatively, $\check{\Sigma}$ admits an identification with the truncated surface $\Sigma - Z$. Let $\check{E} \to \check{\Sigma}$ be the line bundle obtained by extending $E \to \Sigma$.

\begin{lemma}
$\op{ind} D =3$ for every choice of $\eta_{\varepsilon}$.
\end{lemma}

\begin{proof}
We decompose $\partial \check{\Sigma} = cl(\partial \Sigma) \cup (\partial \check{\Sigma} -  \partial \Sigma)$, where $\partial \check{\Sigma} - \partial \Sigma$ corresponds to the ``segment at infinity'' of the negative strip-like end $Z$ and $cl(\bdry\Sigma)$ is the closure of $\bdry\Sigma$ in $\bdry\check\Sigma$.  We parametrize $cl(\partial \Sigma)$ by $\theta \in [- \pi, \pi]$ and $\partial \check{\Sigma} - \partial \Sigma$ by $\theta' \in [0,1]$ in a manner compatible with the orientation of $\check{\Sigma}$ induced by the complex structure on $\Sigma$. In particular, $\theta = - \pi$ is identified with $\theta'=1$ and $\theta = \pi$ is identified with $\theta'=0$.

We define a trivialization $\tau$ of $\check{E}|_{\partial \check{\Sigma}}$ by:
$$ \left \{
\begin{array}{ll}
\tau(\theta) = e^{i \eta_\varepsilon (\theta)} & \text{along }  cl(\partial \Sigma), \text{ with respect to the trivialization of } E_0, \\
\tau(\theta') = - e^{i(\pi+ \varepsilon) \theta'}  & \text{along }   \partial \check{\Sigma} - \partial \Sigma, \text{ with respect to the trivialization of } E_2.
\end{array}
\right. $$
We also define a Lagrangian subbundle $\check{F} \subset \check{E}|_{\partial
\check{\Sigma}}$ by $\check{F}|_{\partial \Sigma} = F$ and rotating it in the
counterclockwise direction by the minimal amount on $\partial \check{\Sigma} -
\partial \Sigma$. This means:
$$ \left \{
\begin{array}{ll}
F(\theta) = \R \cdot e^{i \eta_\varepsilon (\theta)} & \text{along }  cl(\partial \Sigma), \text{ with respect to the
trivialization of } E_0, \\
F(\theta') = \R \cdot e^{i \varepsilon\theta'}  & \text{along }   \partial \check{\Sigma} -
\partial \Sigma, \text{ with respect to the trivialization of } E_2.
\end{array}
\right. $$
By the doubling argument of Theorem \ref{HLS}, Lemma \ref{doubling chords} (or, rather, its proof) and the formula for the index of the Cauchy--Riemann operator on line bundles over punctured surfaces (see for example \cite[Formula 2.1]{We3}), the index of $D$ is
$$\op{ind} D =\chi(\D)+ \mu_{\tau}(\check{F})+ 2 c_1(\check{E}, \tau)-1.$$
From the explicit definitions of $\tau$ and $\check{F}$ one computes that $\mu_{\tau}(\check{F})=-1$ and $c_1(\check{E}, \tau)=2$. Hence $\op{ind} D=3$.
\end{proof}

\begin{lemma}
The operator $D$ is surjective for every choice of $\eta_{\varepsilon}$.
\end{lemma}

\begin{proof}
In view of Theorem \ref{HLS}, the surjectivity of $D$ follows from \cite[Formula 2.5]{We3} and \cite[Proposition 2.2(2)]{We3} applied to the double of $D$.
\end{proof}

\begin{cor} \label{lemma: dimension of mathcal N}
The real dimension of $\widetilde{\mathcal N}$ is $3$ for every choice of $\eta_{\varepsilon}$.
\end{cor}

\subsubsection{The maps ${\frak E}$ and $\mathfrak{F}$}
\label{subsubsection: the maps E and F}

We define an $\R$-linear map:
$${\frak E} : \widetilde{\mathcal N}_{\eta_\varepsilon}\to \R \times \C,$$
$$w\mapsto (c_1,c_2),$$
where $c_1,c_2$ are the coefficients from ($\widetilde{\text N}_1$) and ($\widetilde{\rm N}_2$) of Definition \ref{definition of N tilde}.

\begin{lemma} \label{map frak E}
The map ${\frak E}$ is an isomorphism.
\end{lemma}

Hence $\mathfrak{E}({\mathcal N})$ is an open positive cone contained in $\R^+ \times \C^{\times}$.

\begin{proof}
Since $\mathfrak{E}$ is a linear map, it suffices to check that $\mathfrak{E}$ is injective.  First observe that $\ker \mathfrak{E}$ consists of continuous functions $u : \overline{\mathbb D} \to \C$ such that $u(-1)=0$, $u(e^{i \theta}) \in \R \cdot e^{\eta_{\varepsilon}(\theta)}$ for all $\theta\in (-\pi,\pi)$, and $u$ is holomorphic on $\mathbb{D}$.

Suppose we have $u \in \ker \mathfrak{E}$ which is not identically zero. In particular, $u$ vanishes at finitely many points. The winding of the restriction of $u$ to $\partial \D$ is related to the number of zeroes of $u$ in $\overline{\D}$. This would be straightforward if $u$ had no zeroes on $\partial \overline{\D}$; the presence of zeroes on $\partial \D$ makes the argument only slightly more complicated. We leave the details to the reader or, alternatively, we refer to \cite[Lemma~11.5]{Se2}.

The winding number of $u$ along $\partial \D$ is nonnegative because the lines $\R \cdot e^{i \eta_\varepsilon (\theta)}$ turn clockwise. Since the zeros of a holomorphic function are always positive, we have a contradiction.
\end{proof}

Next we prove that ${\mathcal N}$ is nonempty for any choice of $\eta_{\varepsilon}$.

\begin{lemma}
Let $u_0 = \mathfrak{E}^{-1}(1,0)$. Then $u_0 : \overline{\D}-\{-1\} \to \C$ is holomorphic, $u_0(e^{i \theta}) \in \R^+\cdot e^{i \eta_{\varepsilon}(\theta)}$, and $\lim \limits_{s \to - \infty}
e^{\varepsilon (s+it-i)}u_0(s,t) =1$ for coordinates $(s,t) \in (- \infty, 0) \times [0,1]$ on a negative strip-like end around $-1$.
\end{lemma}

\begin{proof}
The only nontrivial part of the statement is that $u_0$ has no zeros along the boundary. This follows from a winding number argument as in Lemma~\ref{map frak E}. The easy details are left to the reader.
\end{proof}

\begin{cor}
${\mathcal N}$ is nonempty for any choice of $\eta_{\varepsilon}$.
\end{cor}

\begin{proof}
Take any $u \in \widetilde{\mathcal N}$ with $c_2 \ne 0$. Then $u+cu_0 \in \mathcal{N}$ for any $c$ sufficiently large.
\end{proof}

\begin{lemma}
If $u\in \mathcal{N}$, then $u$ has a unique (simple) zero in $\D - \{ 0 \}$.
\end{lemma}
\begin{proof}
The lemma follows from a straightforward winding number argument, since $u$ maps $\bdry B_-$ to the sector $\{0\leq \theta\leq \varepsilon, \rho>0\}$ and has a pole at $0$.
\end{proof}

We denote $\P {\mathcal N} = {\mathcal N}/ \R^+$, where $\R^+$ acts on ${\mathcal N}$ by multiplication and we define the maps
$$\widehat{\mathfrak{F}}:  {\mathcal N} \to \D - \{ 0 \}, \qquad \mathfrak{F}:  \P {\mathcal N} \to \D - \{ 0 \}$$
by $\widehat{\mathfrak{F}}(w) = w^{-1}(0)$ and $\mathfrak{F}([w]) = \widehat{\mathfrak{F}}(w)$.

\begin{lemma} \label{F is a diffeo}
The map $\mathfrak{F}: \P {\mathcal N} \to \D - \{0 \}$ is a diffeomorphism.
\end{lemma}

\begin{proof}
We first prove the injectivity of $\mathfrak{F}$. Let $w_0$ and $w_1$ be maps in ${\mathcal N}$ such that
$w_0^{-1}(0)=w_1^{-1}(0)$. Then $\omega= \frac{w_0}{w_1}$ is a holomorphic map on
$\overline{\mathbb D}$ such that $\omega(\partial  \overline{\mathbb D}) \subset \R$,
so it is constant. Then $w_0$ and $w_1$ represent the same element in
$\P {\mathcal N}$.

Next we prove surjectivity. Fix an element $w_0 \in {\mathcal N}$ and let $z_0 \in \mathbb{D}$ such that $w_0(z_0) =0$. For any $z_1 \in \mathbb{D}-\{0, z_0\}$ we look for a holomorphic function $\omega_{z_1} : \overline{\mathbb{D}} \to \C\P^1$ such that $\omega_{z_1}(z_1)=0$, $\omega_{z_1}(z_0) = \infty$,
$\omega_{z_1}(\partial  \overline{\mathbb D}) \subset \R^+$, and $\omega_{z_1}|_{\mathbb{D}}$ is a biholomorphism onto its image. By the argument of the previous paragraph, if such $\omega_{z_1}$ exists, it is unique up to multiplication by a positive real constant. If we set $w= \omega_{z_1} w_0$, then $w \in {\mathcal N}$, $w({z_1})=0$, and $\mathfrak{F}(w)=z_1$.

The function $\omega_{z_1}$ with the desired properties is given by
$$\omega_{z_1}(\zeta)= f(g(\zeta)^2),$$
where $g: \overline{\D} \stackrel\sim \to \overline{\H}$ is a fractional linear transformation such that $g(z_0)=i$, $\op{Re}(g({z_1}))=0$, and $0<\op{Im}(g({z_1})) <1$, and $f:\C\P^1\stackrel\sim\to \C\P^1$ is an element of $PSL(2,\R)$ (i.e., $f$ fixes the real axis) such that $f((g({z_1}))^2)=0$, $f((g(z_0))^2)=f(-1)=\infty$, and $f(0)=1$. Note that $f$ and $g$ are uniquely determined by the above conditions. This proves the surjectivity of $\mathfrak{F}$.

Finally, we claim that the $d\mathfrak{F}([w])$ is an isomorphism for all $[w]\in \P\mathcal{N}$. To this end, let us define $\Xi(z) = \omega_z w_0 \in {\mathcal N}$.  Then $\mathfrak{F}^{-1}(z)$ is the class of $\Xi(z)$ in $\P {\mathcal N}$. The map $\Xi$ is smooth because the maps $f$, $g$, and hence $\omega_z$ are rational maps whose coefficients depend smoothly on $z$. In order to prove the claim it suffices to prove that $d\Xi(z)$ is injective for any $z \in {\mathbb D} - \{ 0 \}$.  Here, in order to distinguish the differential of $\Xi$ at $z \in {\mathbb D} -
\{ 0 \}$ from the differential of the function $\Xi(z) : \overline{\mathbb
D}- \{ 0,-1 \} \to \C$, we will use the notation $d  \Xi(z)$ for the
former and $\frac{\partial \Xi(z)}{\partial \zeta}$ for the latter. We
define the map
$$F : {\mathcal N} \times (\mathbb D - \{ 0 \}) \to \C, \quad F(w, z)=w(z).$$
By differentiating the identity $F(\Xi(z), z)=0$ we obtain
$$\frac{\partial F}{\partial w} d  \Xi(z) + \frac{\partial F}{\partial z}=0.$$
We observe that $\left. \frac{\partial F}{\partial z} \right |_{(\Xi(z), z)}=
\left. \frac{\partial \Xi(z)}{\partial \zeta} \right |_{\zeta=z}$, which is
invertible because $z$ is a simple zero of $\Xi(z)$.  Hence $d  \Xi(z)$ is injective.
 \end{proof}

\subsubsection{Proof of Theorem~\ref{thm: transversality of ev map}} \label{subsection: proof of gluing}

We refer to Section~\ref{subsection: outline of proof} for the definition of the moduli spaces $\mathcal{M}_1$, $\mathcal{M}_0$, and $\mathcal{M}_{-1}$, and of the gluing parameter space ${\frak P}$. In order to simplify the exposition we will assume (without loss of generality) that all the multiplicities of $\boldsymbol{\gamma}'$ are $1$ and that each of $\mathcal{M}_1/\R$, $\mathcal{M}_0$, and $\mathcal{M}_{-1}/\R$ is connected; in particular, both $\mathcal{M}_0$ and $\mathcal{M}_{-1}/\R$ are single points.

We choose a smooth slice $\widetilde{\mathcal M}_1$ of the $\R$-action on $\mathcal{M}_1$ such that the following hold for some $\kappa_0', \kappa_0>0$, $K_0> \pi-\varepsilon$, and for all $\overline{v}_1 \in \widetilde{\mathcal M}_1$:
\begin{itemize}
\item each component of $\overline{v}_1|_{s\leq 0}$ is $(\kappa_0'+\kappa_0,0)$-close to a cylinder over a component of $\delta_0\boldsymbol{\gamma}'$; and
\item the component $\widetilde{v}_1$ of $\overline{v}_1|_{s\leq 0}$ which is close to $\sigma_\infty'$ satisfies
\begin{equation} \label{asymp 1}
\left|\overline\pi\circ \widetilde{v}_1 - \kappa_0' e^{(\pi-\varepsilon)s}(ce^{\pi it})\right|_{C^0} \leq \kappa_0 e^{K_0s},
\end{equation}
\end{itemize}
where $ce^{\pi it}$ is the normalized asymptotic eigenfunction corresponding to the negative end $\delta_0$ of $\overline{v}_1$ and $\overline\pi=\overline\pi_{D^2_{\rho_0}}$ is the projection to $D^2_{\rho_0}$ with respect to balanced coordinates.  By a slight abuse of notation, we will refer to $\kappa_0' e^{(\pi-\varepsilon)s}(ce^{\pi it})|_{s=0}=\kappa_0'ce^{\pi it}$ as {\em the} asymptotic eigenfunction of $\overline{v}_1$ at $\delta_0$.

Similarly, we choose a smooth slice $\widetilde{\mathcal M}_{-1}$ of the $\R$-action on
${\mathcal M}_{-1}$ such that the following hold for some $\kappa_1', \kappa_1>0$, $K_1>2\varepsilon$:
\begin{itemize}
\item each component of $\overline{v}_{-1}|_{s\geq 0}$ is $(\kappa'_1+\kappa_1,0)$-close to a strip over a component of $\{z_\infty\}\cup {\bf y'}$; and
\item the component $\widetilde{v}_{-1}$ of $\overline{v}_{-1}|_{s\geq 0}$ which is close to $\sigma_\infty$ satisfies
\begin{equation} \label{asymp 2}
\left|\overline\pi\circ \widetilde{v}_{-1} - \kappa'_1 e^{-2\varepsilon s}(de^{-\varepsilon it})\right|_{C^0} \leq \kappa_1 e^{-K_1s},
\end{equation}
\end{itemize}
where $de^{-\varepsilon it}$ is a normalized asymptotic eigenfunction corresponding to the positive end $z_\infty$ of $\overline{v}_{-1}$. Note that $d$ is completely determined by the data $\{(i,j)\to (i,j)\}$ at the positive end $z_\infty$. Without loss of generality, we assume that $\overline{a}_{i,j}=\R^+$, $\overline{\hh}(\overline{a}_{i,j})=\R^+ \cdot e^{i\varepsilon}$, so that $d=e^{i\varepsilon}$. Similarly, we refer to $\kappa'_1 e^{-2\varepsilon s}de^{-\varepsilon it}|_{s=0}=\kappa'_1de^{-\varepsilon it}$ as {\em the} asymptotic eigenfunction of $\overline{v}_{-1}$ at $z_\infty$.

\s
{\em In the rest of this section, when we write $\overline{v}_i$, $i=1,-1$, we will assume that $\overline{v}_i$ is in the slice $\widetilde{\mathcal M}_i$.} Also, for $T \in \R$ and $i = \pm 1$, let $\overline{v}_{i,T}$  be the $T$-translates of $\overline{v}_i$ in the $\R$-direction; i.e., if $s : \overline{W}^* \to \R$, $* = \varnothing, '$, is the $\R$-coordinate, then $s \circ \overline{v}_{i,T} = s \circ \overline{v}_i+T$. Let $f(t)=\beta_0 e^{\pi it}$ be the asymptotic eigenfunction of $\overline{v}_1$ at $\delta_0$ with eigenvalue $\pi - \varepsilon$. Then the asymptotic eigenfunction of $\overline{v}_{1,T}$ at $\delta_0$ is $e^{-(\pi- \varepsilon)T}f$. Similarly, let $g(t)=\alpha_0 e^{i\varepsilon(1-t)}$ be the asymptotic eigenfunction of $\overline{v}_{-1}$ at $z_\infty$ with eigenvalue $2 \varepsilon$. Then the asymptotic eigenfunction of $\overline{v}_{-1,T}$ at $z_\infty$ is $e^{2 \varepsilon T}g$.

Also, for the rest of this section, we fix $\boldsymbol{\gamma} \in \widetilde{\mathcal O}_{2g}$ and $\mathbf{y} \in {\mathcal S}_{\mathbf{a}, \hh(\mathbf{a})}$ once and for all. Since there is no risk of confusion, we will write
$${\mathcal M} =  \mathcal{M}_{\overline{J}_-}^{I=3,n^-=m}(\boldsymbol{\gamma},{\bf y}), \quad {\mathcal M}^{ext}=   \mathcal{M}_{\overline{J}_-}^{I=3,n^-=m,ext}(\boldsymbol{\gamma},{\bf y}).$$
After identifying the quotient ${\mathcal M}_i / \R$ with the slice $\widetilde{\mathcal M}_i$ for $i = \pm 1$ we write
$$\mathfrak{P}= (5r,\infty)^2 \times \widetilde{\mathcal M}_1 \times \mathcal M_0 \times \widetilde{\mathcal M}_{-1};$$
compare with Equation~\eqref{gluing parameter space}.

\s
We describe the gluing map
$$G: {\frak P}\to \mathcal{M}^{ext},\quad {\frak d}=(T_\pm,\overline{v}_1,\overline{v}_0,\overline{v}_{-1})\mapsto\overline{u}({\frak d}),$$
defined in a manner similar to that of Section~\ref{subsubsection: reviw of HT}. First we form the preglued curve $\overline{u}^{\#}({\frak d})$ from the data ${\frak d}=(T_\pm,\overline{v}_1,\overline{v}_0,\overline{v}_{-1})$ by patching together
$$(\overline{v}_{1,2T_+})|_{s\geq T_+},\quad \overline{v}_0|_{-T_-\leq s\leq T_+},\quad (\overline{v}_{-1,-2T_-})|_{s\leq -T_-}.$$
Then we choose cutoff functions $\widetilde{\beta}_i$, $i=-1,0,1$, on the domains of $\overline{v}_i$. If $\psi_{-1}$, $\psi_0$, and $\psi_1$ are deformations of $\overline{v}_{-1,-2T_-}$, $\overline{v}_0$, and $\overline{v}_{1,2T_+}$, the condition for the deformation $\widetilde\beta_{-1} \psi_{-1} + \widetilde\beta_0 \psi_0 + \widetilde\beta_1 \psi_1$ of $\overline{u}^{\#}({\frak d})$ to be holomorphic is given by the following equation:
\begin{equation} \label{eqn: for deformation to be J holom 2}
\widetilde\beta_{-1} \Theta_{-1} (\psi_{-1} ,\psi_0)+\widetilde\beta_0 \Theta_0 (\psi_{-1},\psi_0,\psi_1) +\widetilde\beta_1 \Theta_1 (\psi_0,\psi_1)=0,
\end{equation}
Equation~\eqref{eqn: for deformation to be J holom 2} is analogous to Equation~\eqref{eqn: for deformation to be J holom} and is solved as in Step 3 of Section~\ref{subsubsection: reviw of HT}.  The situation considered here is simpler than that of Section~\ref{subsubsection: reviw of HT}: in fact there is no obstruction bundle because $\overline{v}_0$ is regular. For $r$ sufficiently large,  $G$ is a homeomorphism onto its image.

Given $\delta>0$ small, let $D_{\delta}(\overline{\mathfrak{m}}^b)$ be the closed ball of radius $\delta$ and center $\overline{\mathfrak{m}}^b$ in $B_-$ with respect to some fixed metric. Then let $\mathcal{M}^\delta_{(\kappa,\nu)}\subset \mathcal{M}^{ext}$ be the subset of curves $\overline{u}$ that pass through $D_\delta(\overline{\mathfrak{m}}^b) \times\{z_\infty\}\subset \overline{W}_-$ and are $(\kappa,\nu)$-close to breaking.

\begin{claim}
$\mathcal{M}^\delta_{(\kappa,\nu)}\subset \mathcal{M}$.
\end{claim}

\begin{proof}
Something stronger is actually true: if the image of $\overline{u} \in \mathcal{M}^{ext}$ intersects the section at infinity $\sigma_{\infty}^-$ at a point in the interior of $\sigma_{\infty}^-$, then $\overline{u} \in {\mathcal M}$. This is a consequence of the intersection theory developed in Section~\ref{subsection: intersection numbers}: in fact $n^-(\overline{u})=m$ and the intersection point with $\sigma_{\infty}^-$ at the interior already contributes $m$ to  $n^-(\overline{u})$. However, if $\overline{u} \in {\mathcal M}^{ext} - {\mathcal M}$, then the image of $\overline{u}$ also intersects $\sigma^-_{\infty}$ at some boundary points, which give some extra contribution to $n^-(\overline{u})$. This is a contradiction.
\end{proof}

Let ${\frak d}=(T_\pm,\overline{v}_1,\overline{v}_0,\overline{v}_{-1})\in {\frak P}$. Writing $\overline{u}^{\#}=\overline{u}^{\#}({\frak d})$ and   $\overline{u}=\overline{u}({\frak d})$, we define $w^{\#}=w^{\#}({\frak d})$ and $w=w({\frak d})$ as follows:
\begin{equation} \label{dubya}
w^{\#}= e^{\varepsilon s}\cdot ( \overline\pi\circ \overline{u}^{\#}|_{-2T_-\leq s\leq 2T_+})\ \mbox{ and } \
w=e^{\varepsilon s}\cdot ( \overline\pi\circ \overline{u}|_{-2T_-\leq s\leq 2T_+}).
\end{equation}
The map $w$ is holomorphic by Lemma~\ref{making holomorphic}.

\begin{lemma} \label{munich}
Let $\mathfrak{d}_i = (T_{\pm, i}, \overline{v}_{1,i},  \overline{v}_{0},  \overline{v}_{-1}) \in \mathfrak{P}$ be a sequence such that $\lim \limits_{i \to \infty} T_{\pm, i} = + \infty$ and $\lim\limits_{i\to \infty} \overline{v}_{1,i}=\overline{v}_{1,\infty}\in \widetilde{\mathcal{M}}_1$.  Then the sequence $\overline{u}_i = \overline{u}(\mathfrak{d}_i)$ has a subsequence which converges to $(\overline{v}_{1,\infty}, \overline{v}_{0}, \overline{v}_{-1})$.
\end{lemma}

\begin{proof}
The lemma is a consequence of the following standard gluing result (cf.\ \cite[Theorem 7.3(a)]{HT2}): Given $(\kappa,\nu)$, there exists $r\gg 0$ such that $\overline{u}({\frak d})$ is $(\kappa,\nu)$-close to $(\overline{v}_{1}, \overline{v}_{0}, \overline{v}_{-1})$ for each ${\frak d}= (T_\pm,\overline{v}_{1}, \overline{v}_{0}, \overline{v}_{-1})\in {\frak P}=(5r,\infty)^2\times \widetilde{\mathcal M}_1 \times \mathcal M_0 \times \widetilde{\mathcal M}_{-1}$.
\end{proof}

Recall that $f(t)= \beta_0 e^{\pi i t}$ is the asymptotic eigenfunction of $\overline{v}_{1}$ at the negative end $\delta_0$ and $g(t)=\alpha_0 e^{i\varepsilon(1-t)}$ is the asymptotic eigenfunction of $\overline{v}_{-1}$ at the positive end $z_\infty$.  We also define the function
$$\Pi: \frak{P}\to \R^+\times \C^\times,$$
$$\Pi(T_\pm,\overline{v}_1,\overline{v}_0,\overline{v}_{-1})= (\alpha, \beta),$$
where $\alpha e^{i \varepsilon(1-t)}$ is the asymptotic eigenfunction of $\overline{v}_{-1,-2T_-}$ at $z_{\infty}$ and $\beta e^{i \pi t}$ is the asymptotic eigenfunction of $\overline{v}_{1,2T_+}$ at $\delta_0$.

\begin{lemma}\label{rescaling last}
Let $\mathfrak{d}_i = ( T_{\pm,i}, \overline{v}_{1,i}, \overline{v}_{0},  \overline{v}_{-1})\in \mathfrak{P}$ be a sequence such that  $\lim \limits_{i \to +\infty} T_{\pm, i}= + \infty$. Let $\overline{\mathfrak{m}}_i = \overline{u}(\mathfrak{d}_i)^{-1}( \sigma_{\infty}^-)$ and let $f_i(t)=\beta_i e^{\pi i t}$ be the asymptotic eigenfunction of $\overline{v}_{1,i}$ at the negative end $\delta_0$. Then, after passing to a subsequence, $\lim\limits_{i\to +\infty} \beta_i=\beta_\infty$, and, after rescaling by positive real numbers in addition, the sequence $w_i = w(\mathfrak{d}_i)$ of holomorphic functions defined as in Equation~\eqref{dubya} converges in the $C^{\infty}_{loc}$-topology to a holomorphic map $w_{\infty} \in \widetilde{\mathcal N}$ which satisfies the following:
\begin{enumerate}
\item $\mathfrak{E}(w_{\infty})=(\lambda_1\alpha_0,\lambda_2\beta_\infty)$ for some $\lambda_1,\lambda_2\geq 0$, $\lambda_1,\lambda_2$ not both zero;
\item $\Pi({\frak d}_i)$, after rescaling by positive constants, limits to $\mathfrak{E}(w_\infty)$;
\item if $\lim \limits_{i \to + \infty}  \overline{\mathfrak{m}}_i  = \overline{\mathfrak{m}}_{\infty} \in int(B_-)$, then $w_{\infty}(\overline{\mathfrak{m}}_{\infty})=0$ and $w_{\infty} \in \mathcal{N}$.
\end{enumerate}
\end{lemma}

\begin{proof}
By Lemma \ref{munich} we can extract a subsequence (which we still call $\mathfrak{d}_i$) such that $\overline{u}(\mathfrak{d}_i)$ converges to $\overline{u}_{\infty} = (\overline{v}_{1,\infty}, \overline{v}_{0}, \overline{v}_{-1})$. By the SFT convergence of $\overline{u}(\mathfrak{d}_i)$ to $\overline{u}_{\infty}$ we obtain a sequence of good truncations satisfying the estimates of Lemma~\ref{lemma: pre-riscaling estimates}. The proof of Theorem~\ref{capodanno} then goes through essentially unchanged.
\end{proof}

Let $w_0\in{\mathcal N}$ be a map such that $w_0(\overline{\mathfrak{m}}^b)=0$ and let ${\mathcal U}_{\delta}\subset \mathcal{N}$ be the subset consisting of maps $w$ such that $w(\overline{\mathfrak{m}})=0$ for some $\overline{\mathfrak{m}} \in D_{\delta}(\overline{\mathfrak{m}}^b)$. Note that ${\mathcal U}_{\delta}$ is an $\R^+$-invariant open neighborhood of $w_0$. We now have the following diagram:
$$\xymatrix{
\mbox{ } {\frak P} \mbox{ } \ar[d]_\Pi \ar[r]^-G & \mathcal{M}^{ext} \mbox{ } \supset \mbox{ }\mathcal{M}^{\delta}_{(\kappa,\nu)} \\
\R^+\times \C^\times & \mathcal{N} \ar[l]_-{\frak E} \mbox{ } \supset \mbox{ }\mathcal{U}_\delta.
}$$
Let us write ${\frak P}_\delta:= \Pi^{-1}({\frak E}(\mathcal{U}_\delta))$. Then ${\frak P}_\delta\subset \mathfrak{P}$ is an open set since $\mathfrak{E}$ is a diffeomorphism.

\begin{lemma} \label{caltech}
If $r$ is sufficiently large, then the following hold:
\begin{enumerate}
\item $\mathcal{M}^{\delta/3}_{(\kappa,\nu)}\cap \op{Im}(G) \subset G({\frak P}_{\delta/2}).$
\item $G({\frak P}_{2\delta/3})\subset int( \mathcal{M}^\delta_{(\kappa,\nu)})$.
\end{enumerate}
\end{lemma}

\begin{proof}
(1) We show that, for $r$ sufficiently large, if
$$\overline{u}(\mathfrak{d})^{-1}(\sigma_{\infty}^- ) \in D_{\delta/3}(\overline{\mathfrak{m}}^b),$$
then $\mathfrak{d} \in \mathfrak{P}_{\delta/2}$.  Arguing by contradiction, suppose there exist sequences $r_i \to \infty$, $\mathfrak{d}_i = (T_{\pm,i}, \overline{v}_{1,i}, \overline{v}_{0}, \overline{v}_{-1}) \not \in \mathfrak{P}_{\delta/2}$, and $\overline{\mathfrak{m}}_i \in D_{\delta/3}(\overline{\mathfrak{m}}^b)$ such that $T_{\pm,i} > 5 r_i$ and $\overline{u}(\mathfrak{d}_i)^{-1}(\sigma_{\infty}^-)=\overline{\mathfrak{m}}_i$. By passing to a subsequence, we may assume that $\lim \limits_{i \to \infty} \overline{\mathfrak{m}}_i = \overline{\mathfrak{m}}_{\infty} \in D_{5 \delta/12} (\overline{\mathfrak{m}}^b)$ and that the asymptotic eigenfunctions of $\overline{v}_{1,i}$ converge to $\beta_\infty e^{\pi i t}$ for $i \to \infty$. We apply Lemma~\ref{rescaling last} to $\overline{u}({\frak d}_i)$ to obtain a holomorphic map $w_{\infty} \in {\mathcal N}$ such that $w_{\infty}(\overline{\mathfrak{m}}_{\infty})=0$ and Lemma~\ref{rescaling last}(1) and (2) hold. Since $\mathfrak{d}_i \not \in \mathfrak{P}_{\delta/2}$ for all $i$, (ii) implies that $(\lambda_1\alpha_0,\lambda_2\beta_\infty) \not \in \mathfrak{E}({\mathcal U}_{\delta/2})$.  On the other hand, $w_{\infty}$ has a zero in $D_{5 \delta/12}(\overline{\mathfrak{m}}^b)$, a contradiction.

(2) is similar to (1).  Suppose there exist sequences $r_i \to \infty$ and ${\frak d}_i  = (T_{\pm,i},$ $\overline{v}_{1,i}, \overline{v}_{0}, \overline{v}_{-1}) \in {\frak P}_{2\delta/3}$ such that $T_{\pm, i} > 5r_i$ and $\overline{u}(\mathfrak{d}_i) \not\in int(\mathcal{M}^{\delta}_{(\kappa,\nu)})$. By passing to a subsequence and applying Lemma~\ref{rescaling last} to $\overline{u}({\frak d}_i)$, we obtain a holomorphic map $w_{\infty} \in \widetilde{\mathcal N}$ such that $\mathfrak{E}(w_{\infty}) \in \mathfrak{E}({\mathcal U}_{2 \delta/3})$ by Lemma~\ref{rescaling last}(2). On the other hand, $w^{-1}_\infty(0) \not \in D_{\delta}(\overline{\mathfrak{m}}^b)$ since
$\overline{u}(\mathfrak{d}_i)^{-1}( \sigma_{\infty}^-) \not \in D_{\delta}(\overline{\mathfrak{m}}^b)$ by Lemma~\ref{rescaling last}(3). This is a contradiction.
\end{proof}

Fix $\delta>0$ and take $r_0 > 5r$ sufficiently large. We write
$$\bdry {\frak P}_{2\delta/3, (r_0)}: = {\frak P}_{2\delta/3} \cap \bdry {\frak P}_{(r_0)},$$
where ${\frak P_{(r_0)}}=\{T_+\geq r_0\}\subset {\frak P}$ and $\bdry {\frak P_{(r_0)}}= \{T_+=r_0\}$. By Lemma~\ref{caltech}(2), we can define
\begin{equation} \label{upsilon 1}
\Upsilon': \bdry {\frak P}_{2\delta/3, (r_0)}\to D_\delta(\overline{\mathfrak{m}}^b),
\end{equation}
which maps ${\frak d}$ to $G({\frak d})^{-1}(\sigma_{\infty}^-)$. Since $\bdry {\frak P}_{2\delta/3, (r_0)}$ is compact, the map $\Upsilon'$ is proper and the local degree is well-defined.

Also define
\begin{equation} \label{upsilon 2}
\Upsilon''= \widehat{{\frak F}} \circ{\frak E}^{-1} \circ \Pi: \bdry{\frak P}_{2\delta/3, (r_0)}\to D_\delta(\overline{\mathfrak{m}}^b).
\end{equation}
The map $\Upsilon''$ is well-defined by the definitions of ${\frak P}_{2\delta/3}$ and ${\mathcal U}_{\delta}$. It is also proper and the local degree is well-defined.

The maps $\Upsilon'$ and $\Upsilon''$ are sufficiently close in the following sense.

\begin{lemma}
For any $k$, $\Upsilon'$ and $\Upsilon''$ can be made arbitrarily $C^0$-close by choosing $r_0$ sufficiently large.
\end{lemma}

\begin{proof}
It follows from Lemma~\ref{rescaling last}.
\end{proof}

In particular, the local degrees of $\Upsilon'$ and $\Upsilon''$ near $\overline{\mathfrak{m}}^b$ agree. The local degree of $\Upsilon'$ is equal to the left-hand side of Equation~\eqref{two degrees}, while the local degree of $\Upsilon''$ is equal to the right-hand side of Equation~\eqref{two degrees}. (We were assuming that each of $\mathcal{M}_1/\R$, $\mathcal{M}_0$, and $\mathcal{M}_{-1}/\R$ is connected.) This completes the proof of Theorem~\ref{thm: transversality of ev map}.

\vskip0.5cm

\n {\em Acknowledgements.}
We are indebted to Michael Hutchings for many helpful conversations and for our previous collaboration which was a catalyst for the present work. We also thank Denis Auroux, Tobias Ekholm, Dusa McDuff, Ivan Smith and Jean-Yves Welschinger for illuminating exchanges.  Part of this work was done while KH and PG visited MSRI during the academic year 2009--2010. We are extremely grateful to MSRI and the organizers of the ``Symplectic and Contact Geometry and Topology'' and the ``Homology Theories of Knots and Links'' programs for their hospitality; this work probably would never have seen the light of day without the large amount of free time which was made possible by the visit to MSRI.  KH also thanks the Simons Center for Geometry and Physics for their hospitality during his visit in May 2011.

\newpage
\printnomenclature[7em]
\end{document}